\documentclass[a4paper,12pt]{amsart}
\usepackage{amsmath, amsfonts, amssymb, textcomp, mysymb-private,  latexsym,paralist,soul, tikz, mathtools}

\usepackage[utf8]{inputenc}
\usepackage[all]{xy}
\usepackage[colorlinks=true,linkcolor=blue,citecolor=blue, breaklinks=true]{hyperref}
\usepackage[shortcuts]{extdash}%hyphenated line break
\usepackage{mathrsfs}
\usepackage{cleveref}
\usepackage{enumitem}
\usepackage{quiver}
\usepackage{adjustbox}

\counterwithin{equation}{section}
\allowdisplaybreaks

\theoremstyle{plain}

\newtheorem{thm}{Theorem}[section]
\newtheorem{lem}[thm]{Lemma}
\newtheorem{cor}[thm]{Corollary}
\newtheorem{prop}[thm]{Proposition}

\theoremstyle{definition}

\newtheorem{defn}[thm]{Definition}
\newtheorem{eg}[thm]{Example}
\newtheorem{rmk}[thm]{Remark}

\newtheorem{constr}[thm]{Construction}

\theoremstyle{plain}
\newcounter{thmintroctr}

\newtheorem{thmintro}[thmintroctr]{Theorem}

\theoremstyle{definition}
\newcounter{goalintroctr}

\begin{document}

%%%%%%%%%%%%%%%%%%%%%%%%%newcommands%%%%%%%%%%%%%%%%%%%%%%%%%

%terminologies

\newcommand{\shom}{\starhyphen homomorphism}
\newcommand{\fm}{Fredholm module}

\newcommand{\ind}{\iind}

\newcommand{\suspgrad}[1]{\salg\gradtensor{#1}}

\newcommand{\alg}{\aelem}

\newcommand{\clifft}[2][]{\contfualg{#1}\left(#2, \cliffc \tanbndl{#2} \right)}
\newcommand{\clifftpare}[2][]{\contfualg{#1}\left(#2, \cliffc \tanbndl{(#2)} \right)}
\newcommand{\clifftrel}[3][]{\contfualg{#1}\left(#2, \cliffc \tanbndl_{#2}({#3}) \right)}

\newcommand{\spcofmet}{\mathrm{Met}}

\newcommand{\lrep}[1][\gam]{L_{#1}}

\newcommand{\arep}[2][\rho]{#1: \aalg \to \bh[#2]}

\newcommand{\ray}[2][]{\mathfrak{R}_{#1}^{#2}}
\newcommand{\rayinfty}[3][]{(\ray[#1]{#2}{#3})|_\infty}

\newcommand{\ksubeveng}[2][\ggrp]{\kfunctr_0^{#1}(#2)}
\newcommand{\ksuboddg}[2][\ggrp]{\kfunctr_1^{#1}(#2)}
\newcommand{\ksupeveng}[2][\ggrp]{\kfunctr^0_{#1}(#2)}
\newcommand{\ksupoddg}[2][\ggrp]{\kfunctr^1_{#1}(#2)}
\newcommand{\kk}{\kkfunctr}

\newcommand{\rkkggrp}[1][*]{\kkfunctr_{#1}^\ggrp}
\newcommand{\rkkgam}[1][*]{\kkfunctr_{#1}^\gamgrp}

\newcommand{\AofM}[1][\aContField]{\aalg\left({#1}\right)}
\newcommand{\AevenofM}[1][\aContField]{\aalg_{\mathrm{ev}}\left({#1}\right)}
\newcommand{\AofMrel}[1]{\aalg\left(\aContField,#1\right)}
\newcommand{\AevenofMrel}[1]{\aalg_{\mathrm{ev}}\left(\aContField,#1\right)}

\newcommand{\AIofM}[1][\aContField]{\aalg_{[0,1]}\left(#1\right)}

\newcommand{\cliffmult}[1][s]{C_{#1}}
\newcommand{\botthom}[1][s]{\bottmap_{#1}}

%abbreviations
\newcommand{\alfinalfind}{\alfind\in\alfindset}

\newcommand{\bh}[1][]{\boprs({#1})}
\newcommand{\kh}[1][]{\koprs({#1})}

\newcommand{\csds}{C*-dynamical system}
\newcommand{\crep}{covariant representation}
\newcommand{\cp}{crossed product}
\newcommand{\rcp}{reduced crossed product}

\newcommand{\lam}{\lambda}
\newcommand{\Lam}{\Lambda}
\newcommand{\act}[1][\gam]{\alpha_{#1}}
\newcommand{\acsds}{(\aalg\,,\Gam,\act[])}
\newcommand{\acrep}{(\rho,\uopr)}
\newcommand{\urep}[1][\gam]{\uopr_{#1}}
\newcommand{\ging}{{\gam\in\Gam}}
\newcommand{\ling}{{\lam\in\Gam}}
\newcommand{\lact}[1][\xsp]{\Gam\curvearrowright#1}

\newcommand{\trep}{\tilde{\rho}}
\newcommand{\lgh}{\lsp{\Gam}\otimes\hil}
\newcommand{\aacp}[1][]{\aalg\rtimes_{#1alg}\Gam}
\newcommand{\acp}[1][]{\aalg\rtimes_{#1}\Gam}
\newcommand{\arcp}[1][]{\aalg\rtimes_{#1r}\Gam}
\newcommand{\cpr}[1][]{\rtimes_{#1r}}

\newcommand{\chaf}[1][(\gam_i\Lam)]{X_{#1}}

\renewcommand{\salg}{{C_0(\rbbd)}}

%%%%%%%%%%%%%%%%%%%%%%%%%%%%%%%%%%%%%%%%%%%%%%%%

\newcommand{\hhs}{{Hilbert\-/Hadamard space}}
\newcommand{\admhhs}{{admissible Hilbert-Hadamard space}}

%%%%%%%%%%%%%%%%%%%%%%%%%%%%%%%%%%%%%%%%%%%%%%%%

\newcommand{\calC}{\mathcal{C}}
\newcommand{\calD}{\mathcal{D}}
\newcommand{\calM}{\mathcal{M}}
\newcommand{\calN}{\mathcal{N}}
\newcommand{\calB}{\mathcal{B}}
\newcommand{\calP}{\mathcal{P}}
\newcommand{\calQ}{\mathcal{Q}}
\newcommand{\calU}{\mathcal{U}}
\newcommand{\calV}{\mathcal{V}}

\newcommand{\rtimesred}{\rtimes_{\operatorname{r}}}

%%%%%%%%%%%%%%%%%%%%%%%%%%%%%%%%%%%%%%%%%%%%%%%%

\newcommand{\spaceOfSections}{\Gamma}
\newcommand{\baseSpace}[1]{\underline{#1}}
\newcommand{\actionBaseSpace}[1]{\underline{#1}\mathclap{_{\lrcorner}}\hspace{.1em}}
\renewcommand{\actionBaseSpace}[1]{\underline{#1\smash{_{\lrcorner}}}}
\renewcommand{\actionBaseSpace}[1]{\hspace{.15em}\underline{\hspace{-.15em}\underline{#1}\hspace{-.15em}}\hspace{.15em}}
\newcommand{\aHHS}{X}
\newcommand{\aContField}{\calC}
\newcommand{\aMeasField}{\calM}
\newcommand{\midpoint}{\operatorname{midpoint}}

\newcommand{\anotContField}{\calD}
\newcommand{\anotMeasField}{\calN}
\newcommand{\CFHHS}{\mathfrak{C}}
\newcommand{\MFHHS}{\mathfrak{M}}
\newcommand{\spaceMeas}{\operatorname{M_c}}
\newcommand{\spaceProbMeas}{\operatorname{Prob}}
\newcommand{\measAlg}{\mathcal{M}}

\newcommand{\meansym}[1]{\Theta_{#1}}

%%%%%%%%%%%%%%%%%%%%%%%%%newcommands%%%%%%%%%%%%%%%%%%%%%%%%%
%%%%%%%%%%%%%%%%%%%%%%%%%title%%%%%%%%%%%%%%%%%%%%%%%%%
\title[Novikov conjecture \& fields of Hilbert-Hadamard spaces]{The Novikov conjecture, the group of diffeomorphisms and continuous fields of Hilbert-Hadamard spaces}

\author{Sherry Gong}\address{S.~Gong: Department of Mathematics, Texas A\&M University, College Station, TX, USA}\email{sgongli@tamu.edu} \thanks{The first author is partially supported by NSF 2340465.}

\author{Jianchao Wu}\address{J.~Wu: Shanghai Center for Mathematical Sciences, Fudan University, Shanghai, China}\email{jianchao{\textunderscore}wu@fudan.edu.cn}\thanks{The second author is partially supported by National Key R\&D Program 
	of China 2022YFA100700 and NSFC Key Program No. 12231005.}

\author{Zhizhang Xie}\address{Z.~Xie: Department of Mathematics, Texas A\&M University, College Station, TX, USA}\email{xie@tamu.edu}\thanks{The third author is partially supported by NSF  1952693 and 2247322.}

\author{Guoliang Yu}\address{G.~Yu: Department of Mathematics, Texas A\&M University, College Station, TX, USA}\email{guoliangyu@tamu.edu}\thanks{The fourth author is partially supported by NSF 2247313.}

\date{}

\begin{abstract}
	In this paper, we prove the Novikov conjecture for a class of highly non-linear groups, namely discrete subgroups of the diffeomorphism group of a compact smooth manifold. 
	This removes the volume-preserving condition in a previous work. 
	This result is proved by studying  operator $K$-theory and group actions on continuous fields of infinite dimensional non-positively curved spaces.
\end{abstract}

\maketitle

%%%%%%%%%%%%%%%%%%%%%%%%%title%%%%%%%%%%%%%%%%%%%%%%%%%

\tableofcontents
%%%%%%%%%%%%%%%%%%%%%%%%%intro%%%%%%%%%%%%%%%%%%%%%%%%%

\section{Introduction}
\label{sec:intro}

The Novikov conjecture states that higher signatures are invariant under
(oriented) homotopy equivalences \cite{Novikov1970}. 
This conjecture is  a central problem in differential topology of higher dimensional manifolds, since
the classification problem for higher dimensional manifolds
can be essentially  reduced to  the Novikov conjecture  by surgery theory. 

Noncommutative geometry provides a very successful approach to the Novikov conjecture via higher index theory.  
Powerful tools such as Connes' cyclic cohomology theory and Kasparov's $KK$-theory were developed to attack the Novikov conjecture. 
In this approach, one studies the higher index of the signature operators, which turns out to be always invariant under homotopy equivalences. 
Thus in order to verify the Novikov conjecture, it suffices to show higher signatures are encoded by the higher index of the signature operators. 
This encoding is implied by the (rational) strong Novikov conjecture, which gives an algorithm for computing the
higher index of elliptic operators. 
In addition, if one applies the rational strong Novikov conjecture to the Dirac operator instead of the signature operators, it predicts that the resulting higher index contains higher $A$-genus as part of the information, and, as a result, it also implies the Gromov-Lawson conjecture on scalar curvature.

The goal of this article is to prove the rational strong Novikov conjecture for countable subgroups of the  diffeomorphism group of a compact smooth manifold which are discrete (in a sense we make precise below). As an application, our result implies the Novikov conjecture and the Gromov-Lawson conjecture for manifolds with such groups as their fundamental groups. Our result strengthens the main result in \cite{GongWuYu2021}, which verifies the Novikov conjecture for geometrically discrete subgroups of the group of all \emph{volume-preserving} diffeomorphisms of a compact smooth manifold. 

The main result of the current paper   can be viewed as an infinite-dimensional analogue of Kasparov's fundamental theorem \cite{kasparov1}, which proves the Novikov conjecture for all discrete subgroups of classical Lie groups. Indeed, diffeomorphism groups are  infinite-dimensional analogues of classical Lie groups. While Kasparov's theorem relies on analyzing discrete group actions on finite-dimensional symmetric spaces, our approach involves studying discrete group actions on their infinite-dimensional analogues—such as the spaces of $L^2$-Riemannian metrics. This shift to the setting of infinite-dimensional ``symmetric spaces" introduces significant challenges that our proof must address, rendering our main result inaccessible through existing methods such as the approach using the geometry of Hilbert spaces.   

The study of diffeomorphism groups is central to both geometry and topology. For instance, the group of symplectic diffeomorphisms of a symplectic manifold plays a pivotal role in symplectic geometry as well as in classical mechanics \cite{MR1826128}.   Although the Novikov conjecture has been established for many classes of groups—such as hyperbolic groups \cite{connesmoscovici}, groups acting properly on ``bolic" spaces \cite{kasparov1} \cite{kasparovskandalis2}, a-T-menable groups \cite{higsonkasparov}, linear groups \cite{guentnerhigsonweinberger}, and groups that coarsely embed into Hilbert space \cite{yu3}—it remains a wide open question whether the conjecture holds for  subgroups of diffeomorphism groups. A particularly exciting development in this area is Connes’ seminal result, which shows that the Novikov conjecture holds for Gelfand–Fuchs cohomology classes of diffeomorphism groups \cite{connes1986}.  In other words, when one restricts to the higher signatures that arise from Gelfand-Fuchs cohomology classes, Connes' theorem shows that these higher signatures are indeed invariant under homotopy equivalences. While Connes' proof is a tour-de-force using cyclic cohomology theory,  our proof is based on a different strategy: the study of $K$-theory for $C^\ast$-algebras modelled after infinite-dimensional spaces and the study of group actions on  continuous fields of the infinite-dimensional nonpositively curved spaces.

% To our knowledge, the paper \cite{GongWuYu2021} is the first to verify the Novikov conjecture for a particular class of subgroups of diffeomorphism groups,  namely the class of geometrically discrete subgroups of the group of all \emph{volume-preserving} diffeomorphisms of any  compact smooth manifold. The main result of the present paper proves the Novikov conjecture for a considerably broader class of subgroups by removing the volume-preserving restriction imposed in \cite{GongWuYu2021}.  

Let us first explain what we mean by ``discrete'' for countable subgroups of the diffeomorphism group of a compact smooth manifold $N$. Roughly speaking, we regard a subgroup $\Gamma$ of the group $\operatorname{Diff}(N)$ of diffeomorphisms of a closed smooth manifold $N$ to be ``discrete'' if its natural  action on the space of Riemannian metrics on $N$ is metrically proper. In particular, if an infinite subgroup $\Gamma$ is ``discrete'', then it should not fix any given Riemannian metric $g$ on $N$; hence being ``discrete'' is a strong negation of being a subgroup of the group of isometries on $(N, g)$. 

To be precise, we introduce, for a regular Borel probability measure $\mu$ on $N$, the notion of \emph{$\mu$-discreteness} for subgroups of the group $\operatorname{Diff}(N)$ of diffeomorphisms. 
It has a number of equivalent characterizations (see \Cref{lem:geometrically-discrete}), among which we now present the most explicit one. 
To this end, we temporarily fix a Riemannian metric $g$ on $N$. For any diffeomorphism $\varphi \in \operatorname{Diff}(N)$ and $x \in N$, we write $D_x \varphi \colon T_x N \to T_{\varphi(x)} N$ for the derivative of $\varphi$ at $x$, viewed as a linear operator between finite-dimensional real Hilbert spaces, and write $\left\| D_x \varphi \right\|_g$ for its operator norm. 
We then define 
a pseudometric on $\operatorname{Diff}(N)$:
\begin{align*}
d_{\mu, g} (\varphi, \psi ) := \left( \int_{x \in N}  \left( \log \left( \left\| D_{\varphi^{-1}(x)} \left(\psi^{-1} \varphi\right) \right\|_g \vee  \left\| D_{\psi^{-1}(x)} \left(\varphi^{-1} \psi\right) \right\|_g \right) \right)^2    \, \operatorname{d} \mu (x) \right)^{\frac{1}{2}} 
\end{align*}
where $\vee$ stands for the operation of taking the greater value between the two. 
Observe that when $\psi^{-1} \varphi$ fixes $g$, 
we have $d_{\mu, g} (\varphi, \psi) = 0$, which suggests that $d_{\mu, g} (\varphi, \psi)$ is a measurement of how far $\psi^{-1} \varphi$ is from being isometric. It is also not hard to see that for different Riemannian metrics $g$ and $g'$, $\left| d_{\mu, g} - d_{\mu, g'} \right|$ is uniformly bounded, and for any diffeomorphisms $\varphi, \psi, \tau \in \operatorname{Diff}(N)$, we have $d_{\tau_\ast \mu, g} (\varphi, \psi ) = d_{\mu, g} (\tau^{-1} \varphi, \tau^{-1} \psi )$, where $\tau_\ast \mu$ denotes the pushforward measure. 
See \Cref{constr:pseudometric-integral} for details. 

For a countable subgroup $\Gamma$ of $\operatorname{Diff}(N)$, we say $\Gamma$ is \emph{$\mu$-discrete} if 
\[
\inf_{\tau \in \Gamma} d_{\tau_\ast \mu, g} (\gamma, 1_\Gamma) \xrightarrow{\gamma \to \infty \text{ in } \Gamma} \infty \; ,
\]
that is, for any $N>0$, only finitely many $\gamma \in \Gamma$ makes the left-hand side above fall below $N$. 
It follows from the above that this property is independent of the choice of $g$. 
Also note that in the special case where $\mu$ is a density on $N$ that is invariant under $\Gamma$, $\mu$-discreteness corresponds exactly to the notion of geometric discreteness introduced in \cite{GongWuYu2021}.

Now we are ready to state our main theorem. 

\begin{thmintro}\label{main-theorem}
	Let $\Gamma$ be a countable subgroup of the diffeomorphism group of  a closed smooth manifold $N$. If $\Gamma$ is $\mu$-discrete for some regular Borel measure $\mu$ on $N$, then the rational strong Novikov conjecture holds for $\Gamma$. 
\end{thmintro}

This extends \cite[Theorem~1.3]{GongWuYu2021} by removing the volume\-/preserving condition and allowing for more general regular Borel measures $\mu$. As is remarked after \cite[Theorem~1.3]{GongWuYu2021}, countable subgroups of $\operatorname{Diff}(N)$ on which $d_{\mu, g}$ vanishes for some Riemannian metric $g$ and some regular Borel measure $\mu$ with full support is contained in a compact Lie group $\operatorname{Isom}(N,g)$ and thus also satisfies the strong Novikov conjecture by \cite{guentnerhigsonweinberger}. Since our condition of $\mu$-discreteness is a strong negation of $\Gamma$ being isometric, combining these two results gives us hope to verify the rational strong Novikov conjecture for all countable subgroups of $\operatorname{Diff}(N)$ with a unified approach. 

As in \cite{GongWuYu2021}, the proof of \Cref{main-theorem} involves the geometry and $K$-theory of a class of nonpositively curved, manifold-like, yet possibly infinite-dimensional spaces that we termed {\hhs}s (see \Cref{defn:hhs}). 
The link between groups of diffeomorphisms and these infinite-dimensional spaces stems from the following series of observations: 
\begin{enumerate}
	\item given a closed smooth manifold $N$ of dimension $n$, since the space of inner products on an arbitrary tangent space of $N$ \textemdash\ an $n$-dimensional Euclidean space \textemdash\ can be canonically identified with the symmetric space $GL(n,\mathbb{R}) / O(n)$, it follows that the space of Riemannian metrics on $N$ can be identified with the space of smooth sections of a $GL(n,\mathbb{R}) / O(n)$-bundle $\operatorname{Riem}(N)$ over $N$; 
	\item given a finite regular Borel measure on $N$, this space of smooth sections 
	may be completed into a {\hhs} $\operatorname{Riem}(N)_\mu$ when equipped with an $L^2$-type metric given essentially by integrating, over the measure space $(N, \mu)$, the canonical metric on the symmetric space $GL(n,\mathbb{R}) / O(n)$, which has nonpositive curvature; 
	\item the group $\operatorname{Diff}(N)$ has a canonical action on the space of all Riemannian metrics via ``pushforwards'';
	\item \label{intro-link-volume-preserving} if a subgroup $\Gamma \leq \operatorname{Diff}(N)$ fixes $\mu$ via pushforward, then said canonical action preserves the above $L^2$-type metric and thus extends to an isometry on the completion $\operatorname{Riem}(N)_\mu$. 
\end{enumerate}
Furthermore, when considering a $\mu$-preserving subgroup of $\operatorname{Diff}(N)$ that is also $\mu$-discrete, we obtain an isometric action on $\operatorname{Riem}(N)_\mu$ that is also (metrically) proper. These observations provide the link between the two main results in \cite{GongWuYu2021}.

In the current paper, however,
we need to deal with diffeomorphisms $\gamma \in \operatorname{Diff}(N)$ that are not necessarily $\mu$-preserving. 
In this more general case, item~\eqref{intro-link-volume-preserving} above may not be applicable; instead, we obtain an isometry from $\operatorname{Riem}(N)_\mu$ to a possibly different completion $\operatorname{Riem}(N)_{\gamma_* \mu}$. 
This leads us naturally to 
the study of not just a single {\hhs}, but rather a family of {\hhs}s. 
More precisely, 
if we collect the completions $\operatorname{Riem}(N)_\mu$ with $\mu$ ranging over the space $\operatorname{Prob}(N)$ of all probability measures on $N$, it constitutes an example of what we call a \emph{contiuous field of {\hhs}s}, 
denoted by $\operatorname{Riem}(N)|_{\operatorname{Prob}(N)}$, where the base space
$\operatorname{Prob}(N)$ is equipped with the weak-$^*$ topology. 
Here ``continuity'' arises naturally since this family of {\hhs}s consists of completions of the same space with regard to a family of metrics depending on a continuously varying family of measure. 
Now given any diffeomorphism $\gamma$ of $N$, the aforementioned isometries $\operatorname{Riem}(N)_\mu \to \operatorname{Riem}(N)_{\gamma_* \mu}$ combine to form a \emph{continuous isometric automorphism} of $\operatorname{Riem}(N)|_{\operatorname{Prob}(N)}$. 
Following this line of ideas, we eventually see that any $\mu$-discrete subgroup of $\operatorname{Diff}(N)$ possesses a \emph{proper} action on a suitable contiuous field of {\hhs}s. 

Therefore, the majority of this paper is devoted to the development of a theory of continuous fields of {\hhs}s and isometric group actions on them. 
This culminates in the following result regarding \emph{admissible} continuous fields of {\hhs}s, a notion whose precise meaning will be made clear in \Cref{def:admissible-field}. 
We may deduce \Cref{main-theorem} from this result by following the ideas discussed above.

\begin{thmintro}\label{main-theorem-fields}
	Let $\Gamma$ be a countable group that acts isometrically and metrically properly on an admissible continuous field of {\hhs}s. Then the rational strong Novikov conjecture holds for $\Gamma$. 
\end{thmintro}

Let us highlight a few novelties in the proof of \Cref{main-theorem-fields}, which turns out to call for a somewhat in-depth study of the structure of continuous fields of {\hhs}s. 
\begin{enumerate}
	\item On top of developing the theory of continuous fields of {\hhs}s, 
	we associate a noncommutative $C^*$-algebra to these infinite\-/dimensional objects (see \Cref{def:AofM}) in order to study their $K$-theoretic properties. 
	This directly generalizes the $C^*$-algebraic construction in \cite{GongWuYu2021} and also provides a simpler replacement for the $C^*$-algebra of Tu in the context of continuous fields of affine Hilbert spaces \cite{Tu1999La}. 
	
	\item Due to difficulty in fully computing the $K$-theory of (the $C^*$-algebras of) admissible {\hhs}s, our proof needs to deviate from the classical Dirac-dual-Dirac method as in, for example, \cite{yu3,higsonkasparov,Tu1999La}. 
	Such difficulty was already present in \cite{GongWuYu2021}, where we developed a deformation technique to help us compute the dual-Dirac map (or ``Bott map'') on the left-hand side of the assembly map (and also made use of $KK$-theory with real coefficients \cite{antoniniazzaliskandalis} to deal with groups with torsion). 
	More precisely, this technique allowed us to replace a given proper action of our group on a {\hhs} by a proper action on a ``larger'' {\hhs}, which can then be deformed to the trivial action without changing the equivariant $KK$-groups on the left-hand side of the assembly. 
	In the current paper, 
	the topological complications of continuous fields present new challenges to us, 
	so we introduce a handful new ideas to make this strategy work in the case of continuous fields. 
	
	\item The first among these ideas is a systematic way to generalize the construction of a ``larger'' {\hhs} in the last paragraph to the case of continuous fields. This new construction is now dubbed a \emph{randomization} (see \Cref{def:randomization}), which is a special case of \emph{continuum products} of continuous fields of {\hhs}s (see \Cref{def:continuum-product}). 
	For a {\hhs} $X$, its randomization is the set of all $L^2$-integrable random variables in $X$, which, under the metric of ``expected squared distances'', is a typically much larger {\hhs}. When applying this randomization construction to a continuous field $\aContField$ of {\hhs}s, it often has the effect of neutralizing various topological complications of $\aContField$. 
	For example, we show that the randomization of any locally trivial continuous field is not only locally trivial, but in fact trivial.  
	
	\item In our $KK$-theoretic computations in the proof of \Cref{main-theorem-fields}, it becomes important to extend the original continuous field so that the base space becomes contractible. 
	Such an operation also played an essential role in the proof in \cite{Tu1999La}, where the convex structure of the space of conditionally negative-type kernels (a concept closely related to the geometry of Hilbert spaces) was exploited. 
	In our proof, 
	since we have no obvious counterpart of conditionally negative-type kernels for the geometry of {\hhs}s, we need a different idea. 
	To this end, we introduce a construction dubbed \emph{variation of measures} (see \Cref{def:continuum-product-field}), 
	which 
	extends a continuous field $\aContField$ of {\hhs}s with base space $\baseSpace{\aContField}$ 
	to a new continuous field over the space $\operatorname{Prob}(\baseSpace{\aContField})$ of regular Borel probability measures on $\baseSpace{\aContField}$, where $\baseSpace{\aContField}$ embeds into $\operatorname{Prob}(\baseSpace{\aContField})$ as point masses, and the fiber at each probability measure $\mu$ is the continuum product {\hhs}  with regard to $\mu$. 
	Note that $\operatorname{Prob}(\baseSpace{\aContField})$ is contractible by a linear homotopy. 
	As a side note, if one is purely interested in the setting of \Cref{main-theorem}, one may bypass the variation-of-measures construction by directly making use of the convex structure of regular Borel probably measure on $N$. 
	
	\item As in \cite{GongWuYu2021}, the construction of randomizations holds the key to the deformation technique. However, in the setting of continuous fields, an action may be nontrivial both on the base space and also in the ``fiber direction'', and since randomization does not change the base space, it does not help us with simplifying the action on the base space at all. The best we could hope for is a deformation to a ``fiberwise trivial'' action (see \Cref{lem:deformation}). 
	Thanks to the contractibility of the base space (arranged according to the previous paragraph), a ``fiberwise trivial'' action is just as good as a trivial action, as far as equivariant $KK$-theory is concerned. 
	
	\item However, it is not clear how to even make sense of ``fiberwise trivial'' actions for general continuous fields, unless, for example, the continuous field itself is trivial. 
	Hence we develop some techniques that allow us to replace the original admissible continuous field of {\hhs}s by a trivial one. 
	In this regard, the technically most involved step is making sure triviality is preserved under the variation-of-measures construction; now this is not possible for general continuous fields of {\hhs}s, e.g., those with finite-dimensional fibers, as even local triviality may fail after the variation-of-measures construction \textemdash\ the fiber at a point mass is finite-dimensional, while at any nearby diffuse measure, the fiber becomes infinite-dimensional \textemdash\ but surprisingly, for randomizations of trivial continuous fields, the variation-of-measures construction again yields trivial continuous fields. 
	To prove this, we establish a refined version of the fact that all standard probability spaces are isomorphic (see \Cref{lem:trivialization-prob-interval-trick}) and exploit the flexibility in the functoriality of continuum products of continuous fields of {\hhs}s (see, for example, \Cref{lem:continuum-product-field-maps}). 
\end{enumerate}

The paper is organized as follows: After covering some preliminary materials in \Cref{sec:prelim}, we develop in \Cref{sec:cont-fields} and \Cref{sec:topologies} the basic theory of continuous fields of {\hhs}s. Then in \Cref{sec:meas-fields}, we do the same for measurable and measured fields of {\hhs}s, to an extent paralleling the results in the previous section. 
\Cref{sec:continuum_product} introduces a key construction called $L^2$-continuum products, on which other important constructions of the paper depend, such as variation of measures and $L^2$-continuum powers, the study of the latter being deferred to \Cref{sec:continuum-power}. \Cref{sec:trivialization} builds upon the results in the previous sections and produces a handful trivialization and deformation techniques for continuous fields of {\hhs}s; they will be instrumental in the proof of our main results. 
\Cref{sec:diffeo} discusses diffeomorphisms of a closed smooth manifold $N$ and how they are related to automorphisms of continuous fields of {\hhs}s of $L^2$-Riemannian metrics on $N$. 
\Cref{sec:AofM} constructs a noncommutative $C^*$-algebra associated to a continuous field of {\hhs}s and discusses its properties. 
Finally in \Cref{sec:main}, we provide the proofs of \Cref{main-theorem-fields} and \Cref{main-theorem}. 

%%%%%%%%%%%%%%%%%%%%%%%%%intro%%%%%%%%%%%%%%%%%%%%%%%%%
%%%%%%%%%%%%%%%%%%%%%%%%%prelim%%%%%%%%%%%%%%%%%%%%%%%%%
\section{Preliminaries}
\label{sec:prelim}

In this section, we review the basics of (equivariant) $KK$-theory, the (rational) strong Novikov conjecture and Hilbert-Hadamard spaces. 

\subsection{Equivariant $KK$-theory and the rational strong Novikov conjecture}
\label{subsec:KK}

Kasparov's equivariant $KK$-theory (cf.\,\cite{kasparov1, kasparov95}) associates to a locally compact and $\sigma$-compact group $\Gamma$ and two separable $\Gamma$-$C^*$-algebras $\aalg$ and $\balg$ the abelian group $KK^\Gamma(A, B)$. 
It may be thought of as consisting of generalized equivariant morphisms from $A$ to $B$, with motivating examples coming from elliptic pseudodifferential operators on manifolds (\cite{atiyah2}) and (semisplit) extensions of $C^*$-algebras (\cite{browndouglasfillmore1977extensions}). 
In particular, it is contravariant in $A$ and covariant in $B$, both with respect to equivariant {\shom}s. It is equivariantly homotopy-invariant, stably invariant, preserves equivariant split exact sequences, and satisfies \emph{Bott periodicity}, i.e., there are natural isomorphisms
\[
KK^\Gamma(A, B) \cong KK^\Gamma(\Sigma^2 A, B) \cong KK^\Gamma(\Sigma A, \Sigma B)  \cong KK^\Gamma(A, \Sigma^2 B) 
\]
where $\Sigma^i A$ stands for $C_0(\rbbd^i,A)$ with $i\in \nbbd$ and $\Gamma$ acting trivially on $\rbbd$. It follows that a short exact sequence $0 \to J \to E \to A \to 0$ of $\Gamma$-{\cstaralg}s and equivariant {\shom}s induces a \emph{six-term exact sequence} in the second variable, and with extra conditions such as that $E$ is a nuclear (in particular, commutative) proper $\Gamma$-{\cstaralg}, 
it also induces a six-term exact sequence in the first variable (see \cite[Appendix]{kasparovskandalis2} and \cite[Chapter~VI]{guentnerhigsontrout}), though this fails in general (\cite{Skandalis1991Le}). 
When one of the two variables is $\cbbd$, equivariant $KK$-theory recovers, respectively, equivariant $K$-theory ($KK^\Gamma(\cbbd, B) \cong K^\Gamma_0(B)$) and equivariant $K$-homology ($KK^\Gamma(A, \cbbd) \cong K_\Gamma^0(A)$). 

\begin{rmk}
	Although equivariant $KK$-theory is defined for \emph{graded {\cstaralg}s}. However, throughout this article, we treat every $C^*$-algebra as ungraded (even if it happens to carry a grading). 
\end{rmk}

The \emph{Kasparov product} is a group homomorphism
\[
KK^\Gamma(A, B) \otimes_\zbbd KK^\Gamma(B, C) \to KK^\Gamma(A, C) , \quad x \otimes_\zbbd y \mapsto x \otimes_B y
\]
for any three separable $\Gamma$-$C^*$-algebras $A$, $B$, and $C$. 
The Kasparov product is associative.

\begin{defn}\label{defn:KK-Gam-compact}
	Given a countable discrete group $\Gamma$, a Hausdorff space $X$ with a $\Gamma$-action, a $\Gamma$-$C^*$-algebra $B$, and $i \in \nbbd$, we write $KK^\Gamma_i(X, B)$ for the inductive limit of the equivariant $KK$-groups $KK^\Gamma \left(C_0(Z ), C_0(\rbbd^i, A) \right)$, where $Z$ ranges over $\Gamma$-invariant and $\Gamma$-compact subsets of $X$ and $A$ ranges over $\Gamma$-invariant separable $C^*$-subalgebras of $B$, both directed by inclusion. 
	
	We write $K^\Gamma_i(X)$ for $KK^\Gamma_i(X, \cbbd)$ and call it the \emph{$\Gamma$-equivariant $K$-homology of $X$ with $\Gamma$-compact supports}. 
\end{defn}

\begin{rmk}\label{defn:KK-Gam-trivial-coeff}
	If the action of $\Gamma$ on $B$ is trivial and the action on $X$ is free, then there is a natural isomorphism $KK^\Gamma_i(X, B) \cong KK_i(X / \Gamma, B)$. See, e.g., \cite[Remark~2.2]{GongWuYu2021}. 
\end{rmk}

It is clear from Bott periodicity that there is a natural isomorphism $KK^\Gamma_i(X, B) \cong KK^\Gamma_{i+2}(X, B)$. Thus we can view the index $i$ as an element of $\zbbd / 2 \zbbd$. Also note that this construction is covariant both in $X$ with respect to continuous maps and in $B$ with respect to equivariant {\shom}s. Partially generalizing the functoriality in the second variable, the Kasparov product gives us a natural product $KK^\Gamma_i(X, B) \otimes_\zbbd KK^\Gamma(B,C) \to KK^\Gamma_i(X, C)$ for any separable $\Gamma$-{\cstaralg}s $B$ and $C$ (the separability condition can be dropped by extending the definition of $KK^\Gamma(B,C)$ through taking limits). 

We may think of $KK^\Gamma_i(-, B)$ as an extraordinary homology theory in the sense of Eilenberg-Steenrod. In the non-equivariant case, the coefficient algebra $B$ plays a rather minor role in this picture. 

\begin{lem}\label{lem:KK-separate-variables}
	For any CW-complex $X$, any $C^*$-algebra $B$, and any $i \in \zbbd / 2 \zbbd$, there is a natural isomorphism
	\[
	KK_i(X, B) \otimes_{\zbbd} \qbbd \cong \bigoplus_{j \in \zbbd / 2 \zbbd} K_j(X) \otimes_\zbbd K_{i-j}(B)  \otimes_{\zbbd} \qbbd
	\]
\end{lem}

This follows from a version of the K\"{u}nneth Theorem \cite{RosenbergSchochet1987}. See for example \cite[Lemma~2.4]{GongWuYu2021} for a sketch of the proof. 

Let $\univspfree\Gamma$ denote a \emph{universal space} for free and proper $\Gamma$-actions, that is, $\univspfree\Gamma$ is a free and proper $\Gamma$-space such that any free and proper $\Gamma$-space $X$ admits a $\Gamma$-equivariant continuous map into $\univspfree\Gamma$ that is unique up to $\Gamma$-equivariant homotopy. 
Let $B \Gamma$ be the quotient of $\univspfree\Gamma$ by $\Gamma$. 
Similarly, $\univspproper\Gamma$ denotes a \emph{universal space} for proper $\Gamma$-actions. 
These constructions are unique up to ($\Gamma$-equivariant) homotopy equivalence, and thus there is no ambiguity in writing $KK^\Gamma_i(\univspfree\Gamma, B)$, $KK_i(B\Gamma, B)$ and $KK^\Gamma_i(\univspproper\Gamma, B)$ for a $\Gamma$-{\cstaralg} $B$. By definition, there is a $\Gamma$-equivariant continuous map $\univspfree\Gamma \to \univspproper\Gamma$, regardless of the choice of models. 

The \emph{reduced Baum-Connes assembly map} for a countable discrete group $\Gamma$ and a $\Gamma$-{\cstaralg} $B$ is a group homomorphism 
\[
\mu \colon KK^\Gamma_i(\univspproper\Gamma, B) \to K_i(B \rtimes_{\operatorname{r}} \Gamma) \; .
\]
It is natural in $B$ with respect to $\Gamma$-equivariant {\shom}s or more generally with respect to taking Kasparov products, in the sense that any element $\delta \in KK^\Gamma(B,C)$ induces a commuting diagram
\begin{equation}\label{eq:BC-assembly-natural}
\xymatrix{
	KK^\Gamma_i(\univspproper\Gamma, B) \ar[r]^\mu \ar[d]^{\delta} & K_i(B \rtimes_{\operatorname{r}} \Gamma) \ar[d]^{\delta \rtimes_{\operatorname{r}} \Gamma} \\
	KK^\Gamma_i(\univspproper\Gamma, C) \ar[r]^\mu & K_i(C \rtimes_{\operatorname{r}} \Gamma)
}
\end{equation}
where  $\delta \rtimes_{\operatorname{r}} \Gamma \colon K_i(B \rtimes_{\operatorname{r}} \Gamma) \to K_i(C \rtimes_{\operatorname{r}} \Gamma)$ is a homomorphism naturally induced by $\delta$.

The case when $B = \cbbd$ is of special interest. 
\begin{defn}\label{defn:strong-Novikov}
	The \emph{rational strong Novikov conjecture} asserts that the composition
	\[
	K^\Gamma_i(\univspfree\Gamma) \to K^\Gamma_i(\univspproper\Gamma) \overset{\mu}{\to} K_i(C^*_{\operatorname{r}} \Gamma) 
	\] 
	is injective after tensoring each term by $\qbbd$.  
\end{defn}
The {rational strong Novikov conjecture} implies the Novikov conjecture,
the Gromov-Lawson conjecture on the nonexistence of positive scalar curvature for aspherical manifolds (cf.\,\cite{rosenberg1983c}) and Gromov's zero-in-the-spectrum conjecture. We refer the reader to  \cite{MR716254,miscenko74,kasparov1,connesmoscovici,connesgromovmoscovici,MR1042862,MR1388301, kasparovskandalis91,kasparovskandalis2,higsonkasparov,MR1748916, guentnerhigsonweinberger,yu2,yu3,Higson2000Bivariant,skandalistuyu,MattheyOyono-OyonoPitsch2008Homotopy, MR1817505,MR1869625, mathai, GongWuYu2021, Weinberger1990Aspects} for more details on the progress of the (rational) strong Novikov conjecture in the past few decades. 

On the other hand, it has proven extremely useful to have the flexibility of a general $\Gamma$-algebra $B$ in the picture, largely due to the following key observation, which is based on a theorem of Green \cite{green1982} and Julg \cite{julg1981} and an equivariant cutting-and-pasting argument on $B$. 
Here saying $B$ is a proper $\Gamma$-$X$-{\cstaralg} for some locally compact Hausdorff space $X$ means that 
there is a fixed $\Gamma$-action on $X$ that is proper in the sense that for any compact subset $K$ in $X$, we have $K \cap g K = \varnothing$ for all but finitely many $g \in \Gamma$, and 
there is a $\Gamma$-equivariant $*$-homomorphism $\varphi \colon C_0(X) \to Z(M(B))$, the center of the multiplier algebra of $B$, such that $\varphi(C_0(X)) \cdot B$ is dense in $B$.

\begin{thm}[{cf.\,\cite[Theorem~13.1]{guentnerhigsontrout}}]\label{thm:proper-GHT}
	For any countable discrete group $\Gamma$, and a $\Gamma$-{\cstaralg} $B$, if $B$ is a proper $\Gamma$-$X$-{\cstaralg} for some locally compact Hausdorff space $X$, then the reduced Baum-Connes assembly map 
	\[
	\mu \colon KK^\Gamma_i(\univspproper\Gamma, B) \to K_i(B \rtimes_{\operatorname{r}} \Gamma) \; .
	\]
	is a bijection. \qed
\end{thm}

This is the basis of the \emph{Dirac-dual-Dirac} method (cf.\,\cite{kasparov1,kasparov95}; also see \cite[Chapter~9]{Valette2002} and \cite{MR4300553}), which was applied very successfully to the study of the Baum-Connes assembly map. It is based on the construction of a proper $\Gamma$-$X$-{\cstaralg} $B$ together with $KK$-elements $\alpha \in KK^\Gamma(B, \cbbd)$ and $\beta \in KK^\Gamma(\cbbd, B)$ such that $\beta \otimes_B \alpha$ is equal to the identity element in $KK^\Gamma(\cbbd, \cbbd)$. When this is possible, Theorem~\ref{thm:proper-GHT} allows us to conclude that the Baum-Connes assembly map for $\Gamma$ is an isomorpism and the rational strong Novikov conjecture for $\Gamma$ follows. 
Although we do not directly apply this method to prove the main results of the present paper, our strategy still calls for a proper $\Gamma$-$X$-{\cstaralg} $B$ and a $KK$-element $\beta \in KK^\Gamma(\cbbd, B)$. 

It is not hard to see that whenever $\Gamma$ is infinite and $B$ is a proper $\Gamma$-$X$-{\cstaralg}, there is no $\Gamma$-equivariant {\shom} from $\cbbd$ to $B$. Thus one must look beyond $\Gamma$-equivariant {\shom}s in order to construct a suitable element $\beta \in KK^\Gamma(\cbbd, B)$. Many of such elements come from \emph{$\Gamma$-equivariant asymptotic morphisms} (cf.\,\cite{ConnesHigson1990, guentnerhigsontrout}). We will only make use of a special type of such morphisms, given below. 

\begin{constr}	\label{constr:KK-facts-asymptotic} 	
	Let $B$ be a $\Gamma$-$C^*$-algebra and let $\varphi_t \colon \salg \to B$ be a family of $*$\-/homomorphisms indexed by $t \in [1, \infty)$ that is
	\begin{enumerate}
		\item \emph{pointwise continuous}, i.e., $t \mapsto \varphi_t(f)$ is continuous for any $f \in \salg$, and 
		\item \emph{asymptotically invariant}, i.e., $\lim_{t \to \infty} \left\| g \cdot \left(\varphi_t(f)\right) - \varphi_t(f) \right\| = 0$ for any $f \in \salg$ and any $g \in \Gamma$. 
	\end{enumerate}
	Then by \cite[Definition~7.4]{higsonkasparov}, there is an element 
	\[
	\left[ (\varphi_t) \right] \in KK^\Gamma_0 \left(C_0(\rbbd), B \right) \cong  KK^\Gamma_1(\cbbd, B)
	\]
	whose image under the forgetful map 
	\[
	KK^\Gamma_0(\cbbd, C_0(\rbbd, B)) \to KK(\cbbd, C_0(\rbbd, B)) \cong KK(\salg, B)
	\]
	is equal to the element $[\varphi_t]$ induced by the homomorphism $\varphi_t$, for any $t \in [0, \infty)$. 
\end{constr}

To conclude our preparation of equivariant $KK$-theory, we recall the construction of equivariant $KK$-theory with real coefficients, recently introduced by Antonini, Azzali and Skandalis. 

\begin{constr}[cf.\,\cite{antoniniazzaliskandalis2016}] \label{constr:KKR}
	The \emph{equivariant $KK$-theory with real coefficients} is a bivariant theory that associates, to each pair $(A,B)$ of $\Gamma$-$C^*$-algebras, the groups  
	\[
	KK^\Gam_{\rbbd} (A, B) = \varinjlim_{N} KK^\Gam (A, B \otimes N)
	\]
	where $\otimes$ stands for the minimal tensor product and the inductive limit is taken over all II$_1$-factors $N$ with unital $*$\-/homomorphisms as connecting maps. This theory is contravariant in the first variable and covariant in the second, and there is a natural map from $KK^\Gam(A, B) \otimes_{\zbbd} \rbbd$ to $KK^\Gam_{\rbbd} (A, B)$ since $K_0(N) \cong \rbbd$ for any II$_1$-factor $N$. This map is an isomorphism when $\Gamma$ is trivial, $A = \cbbd$ and $B$ is in the bootstrap class (i.e., the class $\mathcal{N}$ in \cite{RosenbergSchochet1987}). Moreover, the Kasparov product extends to this theory. 
	
	Given a discrete group $\Gamma$, a Hausdorff space $X$ with a $\Gamma$-action, and a $C^*$-algebra $B$ with a $\Gamma$-action, we define $KK^\Gamma_{\rbbd,*}(X, B)$ in the same way as in Construction~\ref{defn:KK-Gam-compact}. Then the universal coefficient theorem allows us to identify $KK_{\rbbd, *} (X, \cbbd)$ with $K_*(X) \otimes_{\zbbd} \rbbd$ in a natural way. 
\end{constr}

The key reason we consider $KK$-theory with real coefficients is the following convenient fact. 

\begin{lem} \label{lem:KKR-EGam-inj}
	For any discrete group $\Gam$ and $\Gam$-$C^*$-algebra $A$, the homomorphism 
	\[
	\pi_* \colon KK^\Gam_{\rbbd, *} (\univspfree\Gam, A) \to KK^\Gam_{\rbbd, *} (\univspproper\Gam, A) \; ,
	\]
	which is induced by the natural $\Gam$-equivariant continuous map $\pi \colon \univspfree\Gam \to \univspproper\Gam$, is injective. 
\end{lem}
\begin{proof}
	It follows from \cite[Section~5]{antoniniazzaliskandalis} that the above homomorphism gives rise to an isomorphism between $KK^\Gam_{\rbbd, *} (\univspfree\Gam, A)$ and $KK^\Gam_{\rbbd, *} (\univspproper\Gam, A) _\tau$, which is a subgroup of $KK^\Gam_{\rbbd, *} (\univspproper\Gam, A)$ called its $\tau$-part. 
\end{proof}

\begin{lem}\label{lem:KK-homotopy-eq-nonequivariant}
	Let $\Gamma$ be a discrete group, let $X$ be a free and proper $\Gamma$-space, let $A$ and $B$ be two $\Gamma$-$C^*$-algebras, and let $\varphi \colon A \to B$ be a $*$-homomorphism that is $\Gamma$-equivariant and a homotopy equivalence, i.e., there exists a (possibly non-equivariant) $*$-homomorphism $\psi \colon B \to A$ such that $\psi \circ \varphi$ and $\varphi \circ \psi$ are homotopic to the identity maps on $A$ and $B$, respectively. Then the homomorphism 
	\[
	\rkkgam(X, A) \to \rkkgam(X, B)
	\]
	induced by $\varphi$ is an isomorphism. 
\end{lem}

\begin{proof}
	This uses a cutting-and-pasting argument similar to \cite[Lemma~8.3]{GongWuYu2021}.
\end{proof}

\subsection{{\hhs}s}
\label{subsec:hhs}
In this section, we review some basics of  \emph{{\hhs}s} (cf. \cite[Section 3]{GongWuYu2021}).

\begin{defn}\label{defn:CAT0-equiv-defn-Bruhat-Tits}
	A metric space $(X,d)$ is \emph{CAT(0)} if for any $p,q,r,m \in X$ satisfying $d(q,m) = d(r,m) = \frac{1}{2}d(q,r)$, the following \emph{CN inequality} of Bruhat and Tits \cite{BruhatTits1972Groupes} (also called the \emph{semi parallelogram law}; see \cite[XI, {\S}3]{lang})
	\[
	d(p,q)^2 + d(p,r)^2 \geq 2 d(m,p)^2 + \frac{1}{2} d(q,r)^2 \; .
	\]
\end{defn}

\begin{rmk}\label{rmk:CAT0-facts}
	Here are some facts about CAT(0) spaces. Let $X$ be a CAT(0) space. 
	\begin{enumerate}
		\item\label{rmk:CAT0-facts-unique-geodesic} The metric space $X$ is uniquely geodesic, that is, any two points $x,y \in X$ are connected by a unique (affinely parametrized) geodesic segment 
		\begin{equation}\label{eq:notation-geodesic}
		[x,y] \colon [0,1] \to X \; ,
		\end{equation} 
		i.e., we have $[x,y](0) = x$ and $[x,y](1) = y$, and $d(x, [x,y](t)) = t d(x,y)$. 
		\item\label{rmk:CAT0-facts-bicombing} The map  
		\[
		X \times X \times [0,1] \to X \, , \qquad (x, y, t) \mapsto [x, y](t )
		\]
		is continuous and is referred to as the \emph{geodesic bicombing}. It follows that $X$ is contractible. 
		\item\label{rmk:CAT0-facts-Lipschitz} For any $x,y,x',y'$ in $X$, we have
		\[
		d \big( [x, y](t ) , [x', y'](t ) \big) \leq \max\{ d(x,y), d(x',y')\} \; .
		\] 
	\end{enumerate}
\end{rmk}

Next we review the notions of angle and tangent cone. Let $(X,d)$ be a geodesic metric space. For three distinct points $x,y,z \in X$, we define the comparison angle $\widetilde{\angle} xyz$ to be the angle at $\widetilde{y}$ of the Euclidean comparison triangle $\widetilde{xyz}$. More explicitly, we have
\[\widetilde{\angle} xyz = \arccos\left(\frac{d(x,y)^2 + d(y,z)^2 - d(x,z)^2}{2d(x,y)d(y,z)}\right).\]

Given two nontrivial geodesic segments $\alpha$ and $\beta$ emanating from a point $p$ in $X$, meaning that $\alpha(0) = \beta(0) = p$, we define the angle between them, $\angle(\alpha,\beta)$, to be 
\[\angle(\alpha,\beta) = \lim_{s,t \to 0} \widetilde{\angle} (\alpha(s),p,\beta(t)) \; ,\]
provided that the limit exists. 
By \cite[Theorem~3.6.34]{burago}, angles satisfy the triangle inequality. 

\begin{rmk}\label{rmk:CAT0-equiv-defn-thin-triangle}
	Using the notion of comparison angles, we get  the following alternative (but equivalent) definition for CAT(0) spaces. Namely, a metric space $(X,d)$ is \emph{CAT(0)} if it is geodesic and for any points $p,q,r, x,y$ in $X$ with $x$ on a geodesic segment connecting $p$ and $q$ and $y$ on a geodesic segment connecting $p$ and $r$, we have
	\[
	\widetilde{\angle} xpy \leq \widetilde{\angle} qpr \; . 
	\]
	It follows from this definition that in a CAT(0) space, the angle between any two nontrivial geodesic segments emanating at the same point exists. 
\end{rmk}

Now suppose the geodesic metric space $(X,d)$ satisfies that the angle between any two nontrivial geodesic segments emanating at the same point exists. 
For a point $p \in X$, let $\Sigma_p'$ denote the metric space consisting of all equivalence classes of geodesic segments emanating from $p$, where two geodesic segments are identified if they have zero angle and the distance $d([\alpha],[\beta])$ between two classes of geodesic segments is the angle $\angle(\alpha,\beta)$. 
Note, in particular, from our definition of angles, that $d([\alpha],[\beta]) \leq \pi$ for any geodesic segments $\alpha$ and $\beta$ emanating from $p$.
Let $\Sigma_p$ denote the completion of $\Sigma_p'$. 

The \textit{tangent cone} $T_p$ at a point $p$ in $X$ is then defined to be a metric space which is, as a topological space, the cone of $\Sigma_p$, that is,
\[\Sigma_p \times [0,\infty)/\Sigma_p \times \{0\} \;. \]
The metric on it is given as follows. For two points $p,q \in T_p$ we can express them as $p=[(\alpha,t)]$ and $q = [(\beta,s)]$. Then the metric is given by
\[d(p,q) = \sqrt{t^2+s^2-2st\cos(d([\alpha],[\beta]))} \; . \]
In other words, it is what the distance would be if we went along straight lines in a Euclidean plane with the same angle between them as the angle between the corresponding directions in $X$. A key motivation for this definition is that when $X$ is a Riemannian manifold, this construction of the tangent cone at a point recovers the tangent space equipped with the metric induced by the inner product. 

The concept of {\hhs}s is inspired by \cite[Page~2]{fishersilberman}.
Roughly speaking, this is a class of (possibly infinite\-/dimensional) non-positively curved spaces. 

\begin{defn}\label{defn:hhs} 
	A \emph{\hhs} is a complete geodesic CAT(0) metric space (i.e., a Hadamard space) all of whose tangent cones are isometrically embeddable into Hilbert spaces. 
	
	For any point $x$ in a {\hhs} $X$, we define the \emph{tangent Hilbert space} $\hhil_x M$ to be the Hilbert space $\hhil_{T_x M}$ spanned by the tangent cone $T_x M$ such that the origin is at the tip of the cone, via the following Construction~\ref{constr:Hilbert-space-span}.
\end{defn}

\begin{constr}\label{constr:Hilbert-space-span}
	It is well known from the work of Schoenberg \cite{schoenberg38} that a metric space $(X,d)$ embeds isometrically into a Hilbert space if and only if the bivariant function $(x_1, x_2) \mapsto \left(d(x_1, x_2)\right)^2$ is a conditionally negative-type kernel. Given such a metric space $(X,d)$ and a fixed base point $x_0 \in X$, there is a canonical way to construct the smallest Hilbert space that contains it with $x_0$ being the origin. See, for example, \cite[Proposition~3.1]{higsonguentner}. More precisely, we define $\hhil_{X, d, x_0}$, the \emph{Hilbert space spanned by $(X,d)$ centered at $x_0$}, to be the completion of the real vector space $\rbbd_0 [X]$, which consists of formal finite linear combinations of elements in $X$ whose coefficients sum up to zero, under the pseudometric induced from the positive semidefinite blinear form 
	\[
	\left\langle \sum_{x \in X} a_x x, \sum_{y \in X} b_y y \right\rangle = -\frac{1}{2} \sum_{x, y \in X} a_x b_y \left(d(x, y)\right)^2 \; .
	\]
	Here a completion under a pseudometric is meant to also identify elements of zero distance to each other. There is a canonical isometric embedding from $(X,d)$ into $\hhil_{X, x_0, d}$ that maps each $x \in X$ to the linear combination $x - x_0$. Given an isometric embedding from $(X, d)$ to another metric space $(Y,d')$ that maps $x_0$ to $y_0$, there is a unique isometric linear embedding from $\hhil_{X, d, x_0}$ to $\hhil_{Y, d', y_0}$ that intertwines the canonical embeddings. It is straightforward to see that these assignments form a functor from the category of pointed metric spaces and isometric base-point-fixing embeddings to the category of Hilbert spaces and linear isometric embeddings. 
\end{constr}

We mostly focus on \emph{separable} {\hhs}s, i.e., those that contain countable dense subsets.

\begin{eg}\label{ex:Riemannian-Hilbertian}
	A Riemannian manifold without boundary is a {\hhs} if and only if it is complete, connected, and simply connected, and has non-positive sectional curvature. The same statement holds for \emph{Riemannian-Hilbertian manifolds} (cf.\,\cite{lang}), which are a kind of infinite\-/dimensional generalizations of Riemannian manifolds defined using charts which are open subsets in Hilbert spaces, instead of Euclidean spaces, in a way that a large part of differential geometry, including sectional curvatures, still makes sense. To see why the statement holds, observe that in this case, every tangent cone is itself a Hilbert space, and the equivalence between the CAT(0) condition and being connected, simply connected and non-positively curved follows from \cite[XI, Proposition~3.4 and Theorem~3.5]{lang}.  
\end{eg}

\begin{constr}\label{constr:log-map}
	A CAT(0) space $X$ is always uniquely geodesic. For any $x_0 \in X$, using the notation in Equation~\eqref{eq:notation-geodesic}, we define the \emph{logarithm function at $x_0$} by
	\[
	\log_{x_0} \colon X \to T_{x_0} X \; , \quad x \mapsto [([x_0, x], d(x_0, x))] \; .
	\]
	The CAT(0) condition (e.g., Remark~\ref{rmk:CAT0-equiv-defn-thin-triangle}) implies $\log_{x_0}$ is \emph{non-expansive} (also called \emph{weakly contractive} or \emph{short} by some authors), 
	i.e., 
	\[
	d\left(  \log_{x_0} (x) ,   \log_{x_0} (x') \right) \leq d(x, x')
	\]
	for any $x, x' \in X$ and, in particular, continuous. Moreover, it preserves the metric on each geodesic emanating from $x_0$, that is, 
	\[
	d\left(  \log_{x_0} (x_0) ,   \log_{x_0} (x) \right) = d(x_0, x)
	\]
	for any $x \in X$. 
\end{constr}

Recall that a subset of a geodesic metric space is called \emph{convex} if it is again a geodesic metric space when equipped with the restricted metric. We observe that a closed convex subset of a {\hhs} is itself a {\hhs}.

\begin{defn}\label{defn:hhs-admissible}
	A separable {\hhs} $M$ is called \emph{admissible} if there is an increasing sequence of closed convex subsets isometric to finite\-/dimensional Riemannian manifolds, whose union is dense in $M$. 
\end{defn}

The notion of {\hhs}s is more general than Example~\ref{ex:Riemannian-Hilbertian}, due to the following construction. 

\begin{constr}[cf.\,\cite{fishersilberman}] \label{constr:continuum-product}
	Let $X$ be a metric space. Let $(Y, \mu)$ be a measure space with $\mu(Y) < \infty$. The \emph{($L^2$-)continuum product} of $X$ over $(Y, \mu)$ is the space $L^2(Y,\mu,X)$ of equivalence classes of measurable maps $\xi$ from $Y$ to $X$ satisfying 
	\[
	\int_Y d_X(\xi(y),x_0)^2 \, d\mu(y) < \infty \; ,
	\] 
	where $x_0$ is a fixed point in $X$ and two functions are identified if they differ only on a measure-zero subset of $Y$. It follows from the triangle inequality that the above condition does not depend on the choice of $x_0$. Moreover, the Minkowski inequality implies that the formula 
	\[
	d(\xi, \eta) = \left( \int_Y d_X(\xi(y),\eta(y))^2 \, d\mu(y) \right)^{\frac{1}{2}}
	\]
	defines a metric on $L^2(Y,\mu,X)$. 
\end{constr}

Recall that a measure space $(Y,\mu)$ is called \emph{separable} if there is a countable family $\{A_n \colon n \in \nbbd \}$ of measurable subsets such that for any $\varepsilon >0$ and any measurable subset $A$ in $Y$, we have $\mu(A \bigtriangleup A_n) < \varepsilon$ for some $n$. For example, it is easy to see that any outer regular finite measure on a separable metric space is separable. This includes, in particular, any measure induced from a density on a closed smooth manifold.

\begin{prop}[{\cite[Proposition 3.13]{GongWuYu2021}}]\label{prop:continuum-product-hhs-summary}
	Let $M$ be a {\hhs} and $(Y, \mu)$ be a finite measure space. Then
	\begin{enumerate}
		\item the continuum product $L^2(Y,\mu,M)$ is again a {\hhs}; 
		\item if $(Y, \mu)$ is separable and $M$ is admissible (respectively, separable), then $L^2(Y,\mu,M)$ is also admissible (respectively, separable).  
	\end{enumerate}
\end{prop}

%%%%%%%%%%%%%%%%%%%%%%%%%prelim%%%%%%%%%%%%%%%%%%%%%%%%%
%%%%%%%%%%%%%%%%%%%%%%%%%fields%%%%%%%%%%%%%%%%%%%%%%%%%
\section{Continuous fields of {\hhs}s}
\label{sec:cont-fields}

In this section, we introduce the notion of a continuous field of {\hhs}s and establish some basic properties and terminologies. 
The definition we adopt below is reminiscent of the notion of a continuous field of Hilbert spaces (\cite[Definition~10.1.2]{Dixmier1977C}). 

\begin{defn} \label{def:cont-field-HHS}
	A \emph{continuous field} $\aContField$ of {\hhs}s consists of a locally compact paracompact Hausdorff space $\baseSpace{\aContField}$ (called the \emph{base space} of $\aContField$), a tuple $\left( \aContField_z \right)_{z \in \baseSpace{\aContField}}$ of {\hhs}s (called the \emph{fibers}), and a subset $\spaceOfSections$ (called \emph{the set of continuous sections} and often denoted by $\spaceOfSections_{\operatorname{cont}} (\aContField)$) of $\prod_{z \in \baseSpace{\aContField}} \aContField_z$ that is 
	{
		\renewcommand{\labelenumi}{\textup{(\theenumi)}} \renewcommand{\theenumi}{\roman{enumi}}
		\begin{enumerate}
			\item \label{def:cont-field-HHS::convex} 
			\emph{convex} in the sense that for any $s_{-1}, s_1 \in \spaceOfSections$ and any $s_{0} \in \prod_{z \in \baseSpace{\aContField}} \aContField_z$, if for any $z \in \baseSpace{\aContField}$, $s_{0} (z)$ is the midpoint of $s_{-1} (z)$ and $s_1 (z)$, then $s_{0} \in \spaceOfSections$, 
			
			\item \label{def:cont-field-HHS::dense} 
			\emph{pointwise dense} in the sense that 
			for any $z_0 \in \baseSpace{\aContField}$, the quotient map $\prod_{z \in \baseSpace{\aContField}} \aContField_z \to \aContField_{z_0}$ maps $\spaceOfSections$ to a dense subset, 
			
			\item \label{def:cont-field-HHS::continuous} 
			\emph{mutually co-continuous} in the sense that 
			any two elements $s, s' \in \spaceOfSections$ are \emph{co-continuous}, which means the function $\baseSpace{\aContField} \ni z \allowbreak \mapsto d_{\aContField_z} \left( s(z) , s'(z) \right) \in [0, \infty)$ is continuous, and
			
			\item \label{def:cont-field-HHS::complete} 
			\emph{co-continuity-closed} in the sense that 
			for any $s \in \prod_{z \in \baseSpace{\aContField}} \aContField_z$ satisfying 
			$\spaceOfSections \cup \{s\}$ is mutually co-continuous, 
			we have $s \in \spaceOfSections$.  
		\end{enumerate}
	}
	
	Sometimes, to emphasize the base space $\baseSpace{\aContField}$ is equal to a given topological space $Z$, we write $\aContField_Z$ instead of $\aContField$ and $\spaceOfSections_{\operatorname{cont}} (Z, \aContField)$ instead of $\spaceOfSections_{\operatorname{cont}} (\aContField)$, and say $\aContField$ is a continuous field of {\hhs}s \emph{over $Z$}. 
\end{defn}

\begin{rmk}\label{rmk:cont-field-HHS-Tu}
	Note that the conditions in \Cref{def:cont-field-HHS} make sense for general tuples of metric spaces, as opposed to just {\hhs}s\footnote{However, condition~\eqref{def:cont-field-HHS::convex} seems to be appropriate only for uniquely geodesic spaces.}. 
	
	We may also extend this definition beyond the case of locally compact paracompact Hausdorff base spaces, though following \cite{Tu1999La}, it may be more appropriate to use local sections instead of global sections, i.e., to let $\spaceOfSections_{\operatorname{cont}} (\aContField)$ be a subset of $ \bigsqcup_{U \subseteq \baseSpace{\aContField} \text{ open}} \left(\prod_{z \in U} \aContField_z\right)$ that is closed under taking restrictions and satisfies versions of conditions~\eqref{def:cont-field-HHS::convex}-\eqref{def:cont-field-HHS::complete} adapted to the setting instead. Almost all results below about continuous fields of {\hhs}s extend to this setting, with more verbose proofs. Since we do not need this generality, we choose to stay within the case of locally compact paracompact Hausdorff base spaces. 
\end{rmk}

The following lemma demonstrates how \Cref{def:cont-field-HHS}\eqref{def:cont-field-HHS::complete} interacts with conditions~\eqref{def:cont-field-HHS::dense} and~\eqref{def:cont-field-HHS::continuous}. 

\begin{lem}\label{lem:cont-field-HHS-Diximier-generate}
	Let $Z$ be a topological space and let $\left( \aContField_z \right)_{z \in Z}$ be a tuple of metric spaces. 
	Let $\spaceOfSections$ be a subset of $\prod_{z \in Z} \aContField_z$ with $\spaceOfSections$.  
	Let $s \in \prod_{z \in \baseSpace{\aContField}} \aContField_z$. 
	Consider the following conditions: 
	\begin{enumerate}
		\item \label{lem:cont-field-HHS-Diximier-generate::complete} 
		$\spaceOfSections \cup \{s\}$ is mutually co-continuous;  
		
		\item \label{lem:cont-field-HHS-Diximier-generate::approx} 
		for any $z \in Z$ and any $\varepsilon > 0$, there exists $s' \in \spaceOfSections$ and a neighborhood $U$ of $z$ such that 
		\[
		d_{\aContField_{z'}} \left( s \left( z' \right) , s' \left( z' \right) \right) \leq \varepsilon \quad \text{ for any } z' \in U \; .
		\]
	\end{enumerate}
	Then 
	\begin{itemize}
		\item \eqref{lem:cont-field-HHS-Diximier-generate::complete} $\Rightarrow$ \eqref{lem:cont-field-HHS-Diximier-generate::approx} if $\spaceOfSections$ satisfies \Cref{def:cont-field-HHS}\eqref{def:cont-field-HHS::dense} (with $\baseSpace{\aContField}$ replaced by $Z$);  
		
		\item \eqref{lem:cont-field-HHS-Diximier-generate::approx} $\Rightarrow$ \eqref{lem:cont-field-HHS-Diximier-generate::complete} if $\spaceOfSections$ satisfies \Cref{def:cont-field-HHS}\eqref{def:cont-field-HHS::continuous} (with $\baseSpace{\aContField}$ replaced by $Z$). 
	\end{itemize}
\end{lem}

\begin{proof}	
	To prove the first statement, we assume $s$ satisfies \eqref{lem:cont-field-HHS-Diximier-generate::complete}
	and fix arbitrary $z \in Z$ and $\varepsilon > 0$. 
	By \Cref{def:cont-field-HHS}\eqref{def:cont-field-HHS::dense}, there is $s' \in \spaceOfSections$ such that 
	\[d_{\calC_{z}}(s(z),s'(z))< \frac{\varepsilon}{2}.\]
	Because $s$ and $s'$ are co-continuous by assumption, i.e., $z' \mapsto d_{\calC_{z'}}(s(z'),s'(z'))$ is continuous, so there is an open neighborhood $U$ of $z$ such that for any $z' \in U$, we have $d_{\calC_{z'}}(s(z'),s'(z'))<\varepsilon$, as desired.

	To prove the second statement, we assume $s$ satisfies \eqref{lem:cont-field-HHS-Diximier-generate::approx}, 
	i.e., 
	for any $z \in Z$ and $\varepsilon>0$, there is $s'_{z, \varepsilon} \in \spaceOfSections$ and a neighborhood $U_{z, \varepsilon}$ of $z$ such that $d_{\calC_{z'}}(s(z'), s'_{z, \varepsilon}(z')) \leq \varepsilon$ for any $z' \in U_{z, \varepsilon}$. 
	Let $s''$ be an arbitrary element of $\spaceOfSections$. 
	For any $z \in Z$ and $\varepsilon>0$, since $s'_{z, \varepsilon}$ and $s''$ are co-continuous by condition~\eqref{def:cont-field-HHS::continuous}, there is a neighborhood $V_{z, \varepsilon}$ of $z$ such that 
	$$\left| d_{\calC_{z'}}(s'_{z, \varepsilon}(z'), s''(z')) - d_{\calC_{z}}(s'_{z, \varepsilon}(z), s''(z)) \right| \leq \varepsilon$$ 
	for any $z' \in V_{z, \varepsilon}$, 
	whence $\left| d_{\calC_{z'}}(s(z'), s''(z')) - d_{\calC_{z}}(s(z), s''(z)) \right| \leq 3 \varepsilon$ for any $z' \in U_{z, \varepsilon} \cap V_{z, \varepsilon}$. 
	It follows that $s$ and $s''$ are co-continuous. Since $s''$ was chosen arbitrarily, $\spaceOfSections \cup \{s\}$ is mutually co-continuous, 
	as desired. 
\end{proof}

Using the lemma above, 
we may replace condition~\eqref{def:cont-field-HHS::complete} in \Cref{def:cont-field-HHS} by one that more closely resembles the corresponding condition~(iv) in \cite[Definition~10.1.2]{Dixmier1977C}. 

\begin{cor}\label{cor:cont-field-HHS-Diximier}
	Let $\left( \aContField_z \right)_{z \in Z}$ and $\spaceOfSections$ be as in \Cref{lem:cont-field-HHS-Diximier-generate}. 
	Assume $\spaceOfSections$ satisfies conditions~\eqref{def:cont-field-HHS::dense} and~\eqref{def:cont-field-HHS::continuous} in \Cref{def:cont-field-HHS} (with $\baseSpace{\aContField}$ replaced by $Z$). 
	Then $\spaceOfSections$ satisfies \Cref{def:cont-field-HHS}\eqref{def:cont-field-HHS::complete} 
	if and only if it satisfies the following: 
	{
		\renewcommand{\labelenumi}{\textup{(\theenumi)}} \renewcommand{\theenumi}{iv'}
		\begin{enumerate}
			\item \label{def:cont-field-HHS::approx} for any $s \in \prod_{z \in Z} \aContField_z$, if for any $z \in Z$ and any $\varepsilon > 0$, there exists $s' \in \spaceOfSections$ and a neighborhood $U$ of $z$ such that 
			\[
			d_{\aContField_{z'}} \left( s \left( z' \right) , s' \left( z' \right) \right) \leq \varepsilon \quad \text{ for any } z' \in U \; ,
			\]
			then $s \in \spaceOfSections$.
		\end{enumerate}
	}
\end{cor}

\begin{proof}
	This follows directly from \Cref{lem:cont-field-HHS-Diximier-generate}. 
\end{proof}

This corollary provides an equivalent characterization for continuous fields of {\hhs}s, which will be convenient in showing certain sections belong to $\spaceOfSections_{\operatorname{cont}} (\aContField)$. A first example of these involves taking convex combinations of two continuous sections, where the weights in the convex combinations may be allowed to change continuously over the base space. 

\begin{cor}\label{cor:cont-field-HHS-geodesic-approximate}
	Let $\aContField$ be a continuous field of {\hhs}s. 
	Let $f \colon \baseSpace{\aContField} \to [0,1]$ be a continuous function. 
	Then for any $s_{0}, s_1 \in \spaceOfSections_{\operatorname{cont}} (\aContField)$, 
	if we use the notation in \Cref{rmk:CAT0-facts}\eqref{rmk:CAT0-facts-bicombing} to define $s_f \in \prod_{z \in \baseSpace{\aContField}} \aContField_z$ so that 
	\[
	s_f (z) = \left[ s_{0} \left( z \right) , s_{1} \left( z \right)  \right] ( f(z) )
	\quad \text{ for any } z \in \baseSpace{\aContField} \; ,
	\]
	then $s_f \in \spaceOfSections_{\operatorname{cont}} (\aContField)$. 
\end{cor}

\begin{proof}
	We first observe that the case when $f$ is a constant function with a dyadic value follows by repeatedly applying \Cref{def:cont-field-HHS}\eqref{def:cont-field-HHS::convex}. 
	
	For the general situation, using Condition \eqref{def:cont-field-HHS::approx}, it suffices to show that for any $z_0$ in $\baseSpace{\aContField}$ and $\varepsilon>0$, there is a section $s$ and an open neighborhood $U$ of $z_0$ such that for any $z \in U$, $d_{\aContField_z}(s(z), s_f(z)) \leq \varepsilon$.
	
	Choose a real number $M>d_{\aContField_{z_0}}(s_0(z_0), s_1(z_0))$. Because $s_0$ and $s_1$ are co-continuous, we have that there is an open neighbourhood $U_1$ of $z_0$ such that for any $z \in U_1$, $d_{\aContField_{z}}(s_0(z), s_1(z))<M$.
	
	Dyadic values are dense in $[0,1]$, so we can find a dyadic value $\alpha$ such that $|f(z_0)-\alpha|<\frac{\varepsilon}{M}$. By continuity of $f$, this means that there is an open neighborhood $U_2$ of $z_0$ such that on $U_2$, $|f(z)-\alpha|<\frac{\varepsilon}{M}$
	
	Let $U := U_1 \cap U_2$, which is an open neighborhood of $z_0$. 
	Consider $s := s_\alpha \in \prod_{z \in \baseSpace{\aContField}} \aContField_z$ with $s_\alpha (z) = \left[ s_{0} \left( z \right) , s_{1} \left( z \right)  \right] ( \alpha )$ for any $z \in \baseSpace{\aContField}$. 
	Then by our construction and \Cref{rmk:CAT0-facts}\eqref{rmk:CAT0-facts-bicombing}, we have, for any $z \in U$, 
	\[d_{\aContField_{z}}(s_\alpha(z), s_f(z)) = |f(z)-\alpha| \, d_{\aContField_{z}}(s_0(z), s_1(z)) < 
	\frac{\varepsilon}{M} \cdot M = \varepsilon,\]
	so we have found the desired $\varepsilon$ and $s$. 
\end{proof}

We then discuss how a continuous field of {\hhs}s may be ``generated'' by a collection of sections. 
In fact, any subset of tuples satisfying conditions~\eqref{def:cont-field-HHS::convex} and~\eqref{def:cont-field-HHS::continuous} in \Cref{def:cont-field-HHS} can be used to generate a continuous field of {\hhs}s, by virtue of the following observations. 

\begin{lem} \label{lem:cont-field-HHS-generate}
	Let $Z$ and $\left( \aContField_z \right)_{z \in Z}$ be as in \Cref{lem:cont-field-HHS-Diximier-generate}. 
	Let $\Sigma$ be a subset of $\prod_{z \in Z} \aContField_z$ satisfying conditions~\eqref{def:cont-field-HHS::dense} and~\eqref{def:cont-field-HHS::continuous} in \Cref{def:cont-field-HHS} (with $\baseSpace{\aContField}$ replaced by $Z$). 
	Define
	\[
	\langle \Sigma \rangle_{\operatorname{cont}} := \left\{ s \in \prod_{z \in Z} \aContField_z \colon \Sigma \cup \{s\} \text{ is mutually co-continuous} \right\} \; .
	\]
	Then $\langle \Sigma \rangle_{\operatorname{cont}}$ is the unique 
	subset of $\prod_{z \in Z} \aContField_z$ that contains $\Sigma$ and satisfies conditions~\eqref{def:cont-field-HHS::dense}-\eqref{def:cont-field-HHS::complete} in \Cref{def:cont-field-HHS} (in place of $\spaceOfSections$).  
	
	Moreover, if $\Sigma$ satisfies \Cref{def:cont-field-HHS}\eqref{def:cont-field-HHS::convex} (in place of $\spaceOfSections$), then so does $\langle \Sigma \rangle_{\operatorname{cont}}$. 
\end{lem}

\begin{proof}
	It is clear that if a subset $\spaceOfSections$ of $\prod_{z \in Z} \aContField_z$ that contains $\Sigma$ and satisfies \Cref{def:cont-field-HHS}\eqref{def:cont-field-HHS::complete}, then it satisfies \eqref{lem:cont-field-HHS-Diximier-generate::complete} (and thus also \eqref{lem:cont-field-HHS-Diximier-generate::approx}) in \Cref{lem:cont-field-HHS-Diximier-generate}, whence $\langle \Sigma \rangle_{\operatorname{cont}} \subseteq \Gamma$. This proves the part about being the ``smallest'', from which uniqueness also follows. 
	
	Since $\Sigma$ is assumed to satisfy \Cref{def:cont-field-HHS}\eqref{def:cont-field-HHS::continuous}, it is immediate that $\Sigma \subseteq \langle \Sigma \rangle_{\operatorname{cont}}$. 
	This in turn has a few consequences: 
	
	\begin{itemize}
		\item Since, by assumption, $\Sigma$ satisfies \Cref{def:cont-field-HHS}\eqref{def:cont-field-HHS::dense}, so does $\langle \Sigma \rangle_{\operatorname{cont}}$. 
		
		\item Since, by definition, $\langle \Sigma \rangle_{\operatorname{cont}}$ contains all sections $s$ such that $\Sigma \cup \{s\}$ is mutually co-continuous, it thus contains all sections $s$ such that $\langle \Sigma \rangle_{\operatorname{cont}} \cup \{s\}$ is mutually co-continuous, i.e., $\langle \Sigma \rangle_{\operatorname{cont}}$ satisfies \Cref{def:cont-field-HHS}\eqref{def:cont-field-HHS::complete}. 
		
		\item We claim that $\langle \Sigma \rangle_{\operatorname{cont}}$ satisfies \Cref{def:cont-field-HHS}\eqref{def:cont-field-HHS::continuous}. To see this, fix $s, s' \in \langle \Sigma \rangle_{\operatorname{cont}}$. 	
		By definition, $s$ satisfies \Cref{lem:cont-field-HHS-Diximier-generate}\eqref{lem:cont-field-HHS-Diximier-generate::complete} with $\Sigma$ replacing $\spaceOfSections$. 
		Since, by assumption, $\Sigma$ satisfies \Cref{def:cont-field-HHS}\eqref{def:cont-field-HHS::dense}, it follows that $s$ also satisfies \Cref{lem:cont-field-HHS-Diximier-generate}\eqref{lem:cont-field-HHS-Diximier-generate::approx} with $\Sigma$ replacing $\spaceOfSections$. 
		It then follows trivially that $s$ satisfies \Cref{lem:cont-field-HHS-Diximier-generate}\eqref{lem:cont-field-HHS-Diximier-generate::approx} with $\Sigma \cup \{s'\}$ replacing $\spaceOfSections$. 
		Since $\Sigma \cup \{s'\}$ is also mutually co-continuous by definition, it follows from \Cref{lem:cont-field-HHS-Diximier-generate} that $s$ satisfies \Cref{lem:cont-field-HHS-Diximier-generate}\eqref{lem:cont-field-HHS-Diximier-generate::complete} with $\Sigma \cup \{s'\}$ replacing $\spaceOfSections$, 
		i.e., $\Sigma \cup \{s'\} \cup \{s\}$ is mutually co-continuous. Hence $s$ and $s'$ are co-continuous. 
	\end{itemize}
	
	Let $\spaceOfSections$ be an arbitrary subset of $\prod_{z \in Z} \aContField_z$ that contains $\Sigma$ and satisfies conditions~\eqref{def:cont-field-HHS::dense}-\eqref{def:cont-field-HHS::complete} in \Cref{def:cont-field-HHS}. It follows from \Cref{def:cont-field-HHS}\eqref{def:cont-field-HHS::continuous} that for any $s \in \spaceOfSections$, $\Sigma \cup \{s\}$ is mutually co-continuous, for it is a subset of $\spaceOfSections$. Hence $\spaceOfSections \subseteq \langle \Sigma \rangle_{\operatorname{cont}}$. 
	On the other hand, It follows from \Cref{def:cont-field-HHS}\eqref{def:cont-field-HHS::complete} that 	
	$\spaceOfSections$ contains all sections $s$ such that $\spaceOfSections \cup \{s\}$ is mutually co-continuous, which include all sections $s$ such that $\langle \Sigma \rangle_{\operatorname{cont}} \cup \{s\}$ is mutually co-continuous. 
	These latter sections belong to $\langle \Sigma \rangle_{\operatorname{cont}}$, since we have shown above that $\langle \Sigma \rangle_{\operatorname{cont}}$ satisfies \Cref{def:cont-field-HHS}\eqref{def:cont-field-HHS::continuous}. 
	Therefore $\spaceOfSections = \langle \Sigma \rangle_{\operatorname{cont}}$, which proves uniqueness. 
	
	Finally, we wish to show that  if $\Sigma$ satisfies \Cref{def:cont-field-HHS}\eqref{def:cont-field-HHS::convex} (in place of $\spaceOfSections$), then so does $\langle \Sigma \rangle_{\operatorname{cont}}$. 
	More precisely, we wish to show that for any $s_{-1}, s_{0}, s_1 \in \prod_{z \in Z} \aContField_z$ satisfying that for any $z \in Z$, $s_{0} (z)$ is the midpoint of $s_{-1} (z)$ and $s_1 (z)$, the containment $\{ s_{-1}, s_1 \} \subseteq \langle \Sigma \rangle_{\operatorname{cont}}$ implies $s_{0} \in \langle \Sigma \rangle_{\operatorname{cont}}$, provided that the same holds for $\Sigma$.

	To this end, it suffices to show, in view of the definition of $\langle \Sigma \rangle_{\operatorname{cont}}$ and \Cref{lem:cont-field-HHS-Diximier-generate}, that for any $\varepsilon$ and $z \in Z$, there is $s_0'$ in $\Sigma$ and a neighborhood $U$ of $z$ such that for any $z' \in U$, $d_{\mathcal{C}_{z'}}(s_{0}(z'), s_{0}'(z'))<\varepsilon$. 
	
	To do this, we apply \Cref{lem:cont-field-HHS-Diximier-generate} again to $s_{-1}, s_1 \in \langle \Sigma \rangle_{\operatorname{cont}}$ to obtain, for $i \in \{-1, 1\}$, $s_{i}' \in \Sigma$ and a neighborhood $U_i$ of $z$ such that for any $z' \in U_i$, $d_{\mathcal{C}_{z'}}(s_{i}(z'), s_{i}'(z'))<\varepsilon$. 
	
	We let $U = U_{-1} \cap U_1$ and define $s'_0 \in \prod_{z' \in Z} \aContField_{z'}$ such that for any $z' \in Z$, $s'_{0} (z') = \left[ s'_{-1} (z') , s'_{1} (z') \right] \left( 1/2 \right)$, i.e., $s'_{0} (z')$ is the midpoint of $s'_{-1} (z')$ and $s'_1 (z')$. 
	By \Cref{rmk:CAT0-facts}\eqref{rmk:CAT0-facts-Lipschitz}, we have, for any $z' \in U$
	\begin{align*}
	d_{\mathcal{C}_{z'}} \left( s_{0} (z') ,  s'_{0} (z') \right) \\
	= & \ d_{\mathcal{C}_{z'}} \left( \left[ s_{-1} (z') , s_{1} (z') \right] \left( 1/2 \right) , \left[ s'_{-1} (z') , s'_{1} (z') \right] \left( 1/2 \right) \right) \\
	\leq & \ \max \left\{ d_{\mathcal{C}_{z'}} \left( s_{-1} (z') ,  s'_{-1} (z') \right) , d_{\mathcal{C}_{z'}} \left( s_{1} (z') ,  s'_{1} (z') \right)  \right\} < \varepsilon \; ,
	\end{align*}
	as desired. 
\end{proof}

\begin{rmk} \label{rmk:cont-field-HHS-generate}
	Let $Z$ be a locally compact paracompact Hausdorff space and let $\left( \aContField_z \right)_{z \in Z}$ be a tuple of {\hhs}s. Let $\Sigma$ be a subset $\prod_{z \in Z} \aContField_z$ satisfying conditions~\eqref{def:cont-field-HHS::convex} and~\eqref{def:cont-field-HHS::continuous} in \Cref{def:cont-field-HHS}. For each $z \in Z$, write $\Sigma_z := \left\{ s(z) \colon s \in \Sigma \right\}$, which is a convex subset of $\aContField_z$ and whose closure is thus also a {\hhs}. It follows that $\Sigma$ satisfies conditions~\eqref{def:cont-field-HHS::convex}-\eqref{def:cont-field-HHS::continuous} in \Cref{def:cont-field-HHS} when considered as a subset of $\prod_{z \in Z} \overline{\Sigma_z}$. 
\end{rmk}

\begin{defn} \label{def:cont-field-HHS-generate}
	We say the set $\Sigma$ in \Cref{rmk:cont-field-HHS-generate} is 
	a \emph{generating set} for the continuous field of {\hhs}s that consists of the base space $Z$, the tuple $\left( \overline{\Sigma_z} \right)_{z \in Z}$, and the set of continuous sections equal to $\langle \Sigma \rangle_{\operatorname{cont}}$ (as in \Cref{lem:cont-field-HHS-generate}). 
\end{defn}

\begin{eg} \label{eg:cont-fields-trivial}
	Let $Z$ be a locally compact paracompact Hausdorff space and let $\aHHS$ be a {\hhs}. The \emph{constant continuous field of {\hhs}s} with base $Z$ and fibers $\aHHS$, denoted by $(\aHHS)_Z$ or $\aHHS_Z$, is given by the constant tuple $(\aHHS)_{z \in Z}$ of {\hhs}s together with $\spaceOfSections_{\operatorname{cont}} \left( (\aHHS)_Z \right)$ given by $C(Z, \aHHS)$, the set of continuous functions from $Z$ to $\aHHS$. Note that $\spaceOfSections_{\operatorname{cont}} \left( (\aHHS)_Z \right)$ is generated by the constant sections $\left\{ (x)_{z \in Z} \in \prod_{z \in Z} \aHHS \colon x \in \aHHS  \right\}$. 
\end{eg}

\begin{eg}
	Continuous fields of affine real Euclidean spaces over locally compact paracompact Hausdorff spaces in the sense of \cite[Définition~3.2]{Tu1999La} are precisely continuous fields of {\hhs}s with each fiber being a Hilbert space. 
	Indeed, although \cite[Définition~3.2]{Tu1999La} is formulated in terms of local sections, since we restrict ourselves to locally compact paracompact Hausdorff base spaces, as pointed out by the author at the beginning of page~221, it suffices to look at global sections (or more precisely, restrictions of global sections to open subsets of the base space). 
	Upon this realization, we see that \cite[Définition~3.2]{Tu1999La} corresponds almost exactly to \Cref{def:cont-field-HHS} with \eqref{def:cont-field-HHS::complete} replaced by \eqref{def:cont-field-HHS::approx} as in \Cref{cor:cont-field-HHS-Diximier}, provided we show that for any $\lambda \in \mathbb{R}$ and any continuous sections $s_0, s_1$ in a continuous field $\aContField$ of {\hhs}s with each fiber being a Hilbert space, the section $s_\lambda$ obtained by taking the affine combination $s_\lambda (z) = (1-\lambda) s_0 (z) + \lambda s_1 (z)$ for each $z \in \baseSpace{\aContField}$ is still continuous. This latter statement is true thanks to \Cref{def:cont-field-HHS}\eqref{def:cont-field-HHS::complete} since $s_\lambda$ is cocontinuous with any continuous section $s$ in $\aContField$, for elementary Euclidean geometry yields the formula 
	\[
	d_{\aContField_z} \left( s_\lambda (z) , s(z) \right) = \sqrt{ \lambda^2 \, d_{01} (z)^2 - \lambda \left( d_{01} (z)^2 +  d_{0} (z)^2 -  d_{1} (z)^2  \right) + d_{0} (z)^2 } \; ,
	\]
	where $d_{01}, d_{0}, d_{1} \colon \baseSpace{\aContField} \to [0,\infty)$ are continuous functions with $d_{01} (z) = d_{\aContField_z} \left( s_0 (z) , s_1 (z) \right)$, $d_{0} (z) = d_{\aContField_z} \left( s (z) , s_0 (z) \right)$ and  $d_{1} (z) = d_{\aContField_z} \left( s (z) , s_1 (z) \right)$. 
	In particular, it follows that for continuous fields of {\hhs}s with each fiber being a Hilbert space, condition~\eqref{def:cont-field-HHS::convex} in \Cref{def:cont-field-HHS} is redundant. 
\end{eg}

We discuss morphisms between continuous fields of {\hhs}s. 
\begin{defn}\label{defn:cont-fields-morphism}
	Let $\aContField$ and $\anotContField$ be two continuous fields of {\hhs}s. 
	\begin{enumerate}
		\item \label{defn:cont-fields-morphism:morhpism} 
		An \emph{isometric continuous morphism} (or simply an \emph{isometric morphism} or even a \emph{morphism} if there is no risk of confusion) $\varphi \colon \aContField \to \anotContField$ is a tuple $\left( \, \baseSpace{\varphi},  \left( \varphi_y \right)_{y \in \baseSpace{\anotContField}} \right)$, where $\baseSpace{\varphi} \colon \baseSpace{\anotContField} \to \baseSpace{\aContField}$ is a continuous map and for each $y \in \baseSpace{\anotContField}$, $\varphi_y \colon \aContField_{\baseSpace{\varphi}(y)} \to \anotContField_y$ is an isometric embedding, such that 
		for any $s \in \spaceOfSections_{\operatorname{cont}} (\aContField)$, we have $\left( \varphi_y \left( s \left( \baseSpace{\varphi}(y) \right) \right) \right)_{y \in \baseSpace{\anotContField}} \in \spaceOfSections_{\operatorname{cont}} (\anotContField)$. 
		The latter element in $\spaceOfSections_{\operatorname{cont}} (\anotContField)$ is usually denoted by $\varphi(s)$. 
		\item \label{defn:cont-fields-morphism:composition} Let $\anotContField'$ be another continuous field of {\hhs}s. Given two isometric continuous morphisms $\varphi \colon \aContField \to \anotContField$ and $\psi \colon \anotContField \to \anotContField'$, their composition $\psi \circ \varphi \colon \aContField \to \anotContField'$ is given by the tuple 
		\[
		\left( \left( \, \baseSpace{\varphi} \circ \baseSpace{\psi} \colon \baseSpace{\anotContField'} \to \baseSpace{\aContField} \right),  \left( \psi_{y'} \circ \varphi_{\baseSpace{\psi}(y')} \colon \aContField_{\baseSpace{\varphi} \circ \baseSpace{\psi}(y')} \to \anotContField'_{y'} \right)_{y' \in \baseSpace{\anotContField'}} \right) \; .
		\]
		
		\item \label{defn:cont-fields-morphism:category} The category $\CFHHS$ has all continuous fields of {\hhs}s as its objects and all isometric continuous morphisms as its morphisms. Apparently, the associations $\aContField \mapsto \baseSpace{\aContField}$ and $\varphi \mapsto \baseSpace{\varphi}$ yields a contravariant functor from $\CFHHS$ to the category of topological spaces and continuous maps. 
		
		\item \label{defn:cont-fields-morphism:category-fixed-base} Given a continuous map $f \colon \baseSpace{\anotContField} \to \baseSpace{\aContField}$, we write $\CFHHS_{f}(\aContField, \anotContField)$ for the subset of $\CFHHS(\aContField, \anotContField)$ that consists of all isometric continuous morphisms $\varphi$ from $\aContField$ to $\anotContField$ such that $\baseSpace{\varphi} = f$. 
		
		\item \label{defn:cont-fields-morphism:isomorhpism} An isometric continuous morphism $\varphi \colon \aContField \to \anotContField$ is an \emph{isometric continuous isomorphism} if there is an isometric continuous morphism $\psi \colon \anotContField \to \aContField$ such that $\psi \circ \varphi$ and $\varphi \circ \psi$ are the identity morphisms on $\aContField$ and $\anotContField$, respectively. 
		
		\item \label{defn:cont-fields-morphism:trivial} We say $\aContField$ is \emph{trivial} if there is a isometric continuous isomorphism from it to a constant continuous field of {\hhs}s.  
	\end{enumerate}
\end{defn}

We establish a few simplifying criteria regarding \Cref{def:cont-field-HHS-maps}. 

\begin{lem}\label{lem:cont-fields-HHS-map-generators}
	Let $\aContField$ and $\anotContField$ be two continuous fields of {\hhs}s. 
	Let $\Sigma$ and $\Xi$ be generating sets of $\aContField$ and $\anotContField$, respectively, in the sense of \Cref{def:cont-field-HHS-generate}. 
	Let $f \colon \baseSpace{\anotContField} \to \baseSpace{\aContField}$ be a continuous map and for each $y \in \baseSpace{\anotContField}$, let $\varphi_y \colon \aContField_{f(y)} \to \anotContField_y$ be an isometric embedding. 
	Then the tuple $\left( f,  \left( \varphi_y \right)_{y \in \baseSpace{\anotContField}} \right)$ constitutes an isometric continuous morphism if and only if 
	for any $s \in \Sigma$, the section $\left( \varphi_y \left( s \left( f(y) \right) \right) \right)_{y \in \baseSpace{\anotContField}} \in \prod_{y \in \baseSpace{\anotContField}} \anotContField_y$ is co-continuous with $s'$ for any $s' \in \Xi$. 
\end{lem}

\begin{proof}
	The ``only if'' direction follows directly from \Cref{def:cont-field-HHS-maps}\eqref{defn:cont-fields-morphism:morhpism} and \Cref{def:cont-field-HHS}\eqref{def:cont-field-HHS::continuous}. 
	For the ``if'' direction, we first apply \Cref{def:cont-field-HHS-generate} and \Cref{lem:cont-field-HHS-generate} to see that for any $s \in \Sigma$, the section $\left( \varphi_y \left( s \left( f(y) \right) \right) \right)_{y \in \baseSpace{\anotContField}}$ is contained in $\langle \Xi \rangle_{\operatorname{cont}} = \spaceOfSections_{\operatorname{cont}} (\anotContField)$. 
	Now for any $s \in \spaceOfSections_{\operatorname{cont}} (\aContField) = \langle \Sigma \rangle_{\operatorname{cont}}$, since $\Sigma \cup \{s\}$ is mutually co-continuous and $\Sigma$ satisfies \Cref{def:cont-field-HHS}\eqref{def:cont-field-HHS::dense} in $\prod_{z \in \baseSpace{\aContField}} \aContField_z$, it follows from \Cref{lem:cont-field-HHS-Diximier-generate} that 
	for any $z \in \baseSpace{\aContField}$ and any $\varepsilon > 0$, there exists $s'_{z,\varepsilon} \in \Sigma$ and a neighborhood $U_{z,\varepsilon}$ of $z$ such that 
	\[
	d_{\aContField_{z'}} \left( s \left( z' \right) , s'_{z,\varepsilon} \left( z' \right) \right) \leq \varepsilon \quad \text{ for any } z' \in U_{z,\varepsilon} \; .
	\]
	Hence, for any $y \in \baseSpace{\anotContField}$ and any $\varepsilon > 0$, we have, for any $y'$ in the open neighborhood $f^{-1} \left( U_{f(y),\varepsilon} \right)$ of $y$, 
	\[
	d_{\anotContField_{y'}} \left( \varphi_{y'} \left( s \left( f(y') \right) \right) , \varphi_{y'} \left( s'_{f(y),\varepsilon}  \left( f(y') \right) \right) \right) 
	= d_{\aContField_{f(y')}} \left( s \left( f(y') \right) ,  s'_{f(y),\varepsilon}  \left( f(y') \right) \right) 
	\leq \varepsilon  \; .
	\]
	Since, by we have shown above, $\left( \varphi_{y'} \left( s'_{f(y),\varepsilon} \left( f(y') \right) \right) \right)_{y' \in \baseSpace{\anotContField}} \in  \spaceOfSections_{\operatorname{cont}} (\anotContField)$, 
	it follows from \Cref{cor:cont-field-HHS-Diximier} that $\left( \varphi_y \left( s \left( f(y) \right) \right) \right)_{y \in \baseSpace{\anotContField}}$ is contained in $\spaceOfSections_{\operatorname{cont}} (\anotContField)$. 
	Therefore $\left( f,  \left( \varphi_y \right)_{y \in \baseSpace{\anotContField}} \right)$ constitutes an isometric continuous morphism. 
\end{proof}

\begin{lem}\label{lem:cont-fields-isomorphism}
	An isometric continuous morphism $\varphi \colon \aContField \to \anotContField$ is an isometric isomorphism if and only if $\baseSpace{\varphi}$ is a homeomorphism and each $\varphi_y \colon \aContField_{\baseSpace{\varphi}(y)} \to \anotContField_y$ is an isometric bijection. 
\end{lem}

\begin{proof}
	We define an isometric continuous morphism $\psi \colon \anotContField \to \aContField$ such that $\baseSpace{\psi} = \baseSpace{\varphi}^{-1}$ and for any $z \in \baseSpace{\aContField}$, we have $\psi_z = \left( \varphi_{\baseSpace{\psi} (z)} \right)^{-1} \colon \anotContField_{\psi(z)} \to \aContField_z$. 
	Note that this is well-defined: for any $s \in \spaceOfSections_{\operatorname{cont}} (\anotContField)$, 	
	we have $\left( \psi_z \left( s \left( \baseSpace{\psi}(z) \right) \right) \right)_{z \in \baseSpace{\aContField}} \in \spaceOfSections_{\operatorname{cont}} (\aContField)$ by the co-continuity-closedness of $\aContField$, 
	as for any $s' \in \spaceOfSections_{\operatorname{cont}} (\aContField)$, the function 
	\[
	\baseSpace{\aContField} \ni z \mapsto d_{\aContField_z} \left( \psi_z \left( s \left( \baseSpace{\psi}(z) \right) \right) , s'(z) \right) 
	=  d_{\anotContField_{\psi(z)}} \left( s \left( \baseSpace{\psi}(z) \right), \varphi(s') \left( \baseSpace{\psi}(z) \right) \right)
	\]
	is continuous by the mutual co-continuity of $\anotContField$. 
	It is routine to verify that $\psi \circ \varphi$ and $\varphi \circ \psi$ are the identity morphisms on $\aContField$ and $\anotContField$, respectively. 
\end{proof}

\begin{defn} \label{def:cont-field-HHS-maps}
	Let $Y$ be a locally compact paracompact Hausdorff space and let $\aContField$ be a continuous field of {\hhs}s. 
	\begin{enumerate}
		\item \label{def:cont-field-HHS-maps:induce} Given a continuous map $f \colon Y \to \baseSpace{\aContField}$, the \emph{continuous field of {\hhs}s over $Y$ induced from $\aContField$ by $f$}, denoted by $f^* \aContField$ or $\left(f^*\aContField\right)_Y$, 
		is given by 
		defining $\left(f^*\aContField\right)_{y} = \aContField_{f(y)}$ for any $y \in Y$ and letting $\spaceOfSections_{\operatorname{cont}} (Y, f^*\aContField)$ be the set of continuous sections generated by the image of $\spaceOfSections_{\operatorname{cont}} (\aContField)$ under the map 
		\[
		\prod_{z \in \baseSpace{\aContField}} \aContField_z \to \prod_{y \in Y} \aContField_{f(y)} \, , \quad s \mapsto s \circ f \; .
		\]
		In this case, there is a canonical isometric continuous morphism $f^* \colon \aContField \to \left(f^*\aContField\right)_Y$ where $\underline{f^*} = f$ and, for each $y \in Y$, $(f^*)_y \colon \aContField_{f(y)} \to \left(f^*\aContField\right)_{y}$ is the identity map. 
		
		\item \label{def:cont-field-HHS-maps:restrict} For the sake of notational convenience, when the continuous map $f \colon Y \to \baseSpace{\aContField}$ is made clear from the context, we may sometimes also write $\aContField|_Y$ in place of $f^* \aContField$. Two primary examples of this usage are: 
		\begin{itemize}
			\item If $Y$ is a subspace of $\baseSpace{\aContField}$, then $f \colon Y \to \baseSpace{\aContField}$ is understood to be the inclusion map, and $\aContField|_Y$ is also called the \emph{restriction} or \emph{reduction} of $\aContField$ to $Y$. 
			\item If $Y = \baseSpace{\aContField} \times Z$ or $Z \times \aContField$ for some other locally compact paracompact Hausdorff space $Z$, then $f \colon Y \to \baseSpace{\aContField}$ is understood to be the projection onto the factor $\baseSpace{\aContField}$, and $\aContField|_{\baseSpace{\aContField} \times Z}$ or $\aContField|_{Z \times \baseSpace{\aContField}}$ is also called the \emph{extension} of $\aContField$ over the Cartesian product with $Z$. 
		\end{itemize}
		There is another important case where we use this notation, which will be explained in \Cref{def:continuum-product-field}. 	
		
		\item Given an isometric continuous morphism $\varphi \colon \anotContField \to \aContField$ and assuming $Y$ is a subspace of $\baseSpace{\aContField}$, we obtain an isometric continuous morphism $\varphi|_{Y} \colon \anotContField|_{\varphi(Y)} \to \aContField|_{Y}$ such that $\baseSpace{\varphi|_{Y}} := \baseSpace{\varphi}|_{Y}$ and for any $y \in Y$, $\left(\varphi |_{Y}\right)_y := \varphi_y$. 
	\end{enumerate}
\end{defn}

\begin{eg}  \label{eg:cont-field-HHS-maps-singleton}
	For any continuous field $\aContField$ of {\hhs}s and any $z \in \baseSpace{\aContField}$, we can canonically identify $\aContField|_{\{z\}}$ with $\aContField_z$, viewed as a (constant) continuous field over a singleton. 
\end{eg}

\begin{rmk} \label{rmk:cont-field-HHS-maps-trivial}
	It is clear that given a trivial continuous field $\aHHS_Z$ of {\hhs}s and a continuous map $f \colon Y \to Z$, the induced continuous field $f^* \left( \aHHS_Z \right)$ is the trivial continuous field $\aHHS_Y$ of {\hhs}s. 
\end{rmk}

\begin{rmk} \label{rmk:cont-field-HHS-maps-functorial}
	It is also clear that if $Y$, $\aContField$ and $f$ are as in \Cref{def:cont-field-HHS-maps} and $h \colon Z \to Y$ is another continuous map from a locally compact paracompact Hausdorff space, then we have a canonical isometric continuous isomorphism $h^* \left( f^* \aContField \right) \cong (f \circ h)^* \aContField$ that fits into the following commutative diagram 
	\[
	\xymatrix{
		\aContField \ar[r]^{(f \circ h)^*} \ar[d]_{f^*} & (f \circ h)^* \aContField \ar[d]^{\cong} \\
		f^* \aContField \ar[r]_{h^*} & h^* \left( f^* \aContField \right) 
	}
	\]
\end{rmk}

\begin{rmk} \label{rmk:cont-field-HHS-maps-factorization}
	It also follows from \Cref{defn:cont-fields-morphism} and \Cref{def:cont-field-HHS-maps} that any isometric continuous morphism $\varphi \colon \aContField \to \anotContField$ factors into the composition of $\baseSpace{\varphi}^* \colon \aContField \to \baseSpace{\varphi}^* \aContField = \left(\baseSpace{\varphi}^* \aContField\right)_{\baseSpace{\anotContField}}$ and an isometric continuous morphism $\dot{\varphi} \colon \left(\baseSpace{\varphi}^* \aContField\right)_{\baseSpace{\anotContField}} \to \anotContField$ defined so that $\underline{(\dot{\varphi})} = \operatorname{id}_{\baseSpace{\anotContField}}$ and, for each $y \in {\baseSpace{\anotContField}}$, we identify $\dot{\varphi}_y \colon \left( \baseSpace{\varphi}^* \aContField \right)_y  \to \anotContField_y$ with $\varphi_y \colon \aContField_{\baseSpace{\varphi}(y)} \to \anotContField_y$. 
\end{rmk}

\begin{defn}\label{defn:isometric-action}
	Let $\aContField$ be a continuous field of {\hhs}s and let $G$ be a group. An \emph{isometric (left) action} $\alpha$ of $G$ on $\aContField$ is a homomorphism from $G$ to the group $\operatorname{Isom}(\aContField)$ of isometric continuous isomorphisms from $\aContField$ to itself, 
	that is, we have $\alpha_g \circ \alpha_{g'} = \alpha_{g g'}$ for any $g, g' \in G$. 
	
	The \emph{induced (right) action} $\baseSpace{\alpha}$ of $\alpha$ on the base space $\baseSpace{\aContField}$, is the homomorphism from $G$ to $\operatorname{Homeo}(\baseSpace{\aContField})^{\operatorname{op}}$ taking $g$ to $\baseSpace{\alpha}_g := \baseSpace{\alpha_g}$ for each $g \in G$, 
	that is, we have $\baseSpace{\alpha}_{g'} \circ \baseSpace{\alpha}_{g} = \baseSpace{\alpha}_{g g'}$ for any $g, g' \in G$. 
\end{defn}

We remark that in the case where $G$ is the diffeomorphism group $\operatorname{Diff}(N)$ of a given closed manifold $N$, the induced (right) action $\baseSpace{\alpha}$ of $\operatorname{Diff}(N)$ on $N$ is given by $\baseSpace{\alpha}_\varphi(z) = \varphi^{-1}(z)$ for $\varphi \in \operatorname{Diff}(N)$ and $z\in N$. 
\section{Topological spaces associated to a continuous field of {\hhs}s}\label{sec:topologies}

In this section, we start with a continuous field $\aContField$ of {\hhs}s and construct topologies on a number of spaces associated to $\aContField$, such as the space of sections, the space of automorphisms, and the so-called total space. These constructions will be useful in the later sections, particularly when we discuss Borel measurable fields of {\hhs}s in \Cref{sec:meas-fields} and homotopies of actions in \Cref{sec:trivialization}. 

\begin{defn}\label{defn:isometric-action-topologize}
	Given a continuous field $\aContField$ of {\hhs}s, there is a natural way to topologize the group $\operatorname{Isom}(\aContField)$. To do this, let us first equip $\spaceOfSections_{\operatorname{cont}} (\aContField)$ with a compact-open topology: for a section $s \in \spaceOfSections_{\operatorname{cont}} (\aContField)$, we have a local base consisting of 
	\[
	\left\{ W_{s, K, \varepsilon} \colon K \subseteq \baseSpace{\aContField} \text{ precompact}, \varepsilon > 0 \right\}
	\]
	where 
	\[
	W_{s, K, \varepsilon} := \left\{ s' \in \spaceOfSections_{\operatorname{cont}} (\aContField) \colon d_{\aContField_{z}} \left( s (z) , s' (z) \right) < \varepsilon \text{ for any } z \in K \right\} \; .
	\]
	This allows us to equip the space $\CFHHS (\anotContField, \aContField)$ of morphisms from another continuous field $\anotContField$ to $\aContField$ (see \Cref{defn:cont-fields-morphism}\eqref{defn:cont-fields-morphism:category}) with 
	the topology of compact-open convergence on the base space and pointwise convergence on $\spaceOfSections_{\operatorname{cont}} (\aContField)$ in the compact-open topology: 
	for any $\varphi \in \CFHHS (\anotContField, \aContField)$, we have a local subbase consisting of\footnote{Using Urysohn's lemma, it is not hard to see that as long as each fiber $\anotContField_z$ has nonzero diameter, then in order to generate this topology, it suffices to look at those $U_{\varphi, L, V, s, K, \varepsilon}$ with $V = \baseSpace{\aContField}$, that is, the requirement $\baseSpace{\psi}(L) \in V$ becomes vacuous. We leave the proof of this fact to the reader as we do not need to use it. }
	\begin{align*}
	\Big\{ U_{\varphi, L, V, s, K, \varepsilon} \colon &\ V \subseteq \baseSpace{\anotContField} \text{ open},  L \subseteq \baseSpace{\varphi}^{-1} (V) \text{ compact}, \\ 
	&\ s \in \spaceOfSections_{\operatorname{cont}} (\anotContField), K \subseteq \baseSpace{\aContField} \text{ compact}, \varepsilon > 0 \Big\}
	\end{align*}
	where 
	\[
	U_{\varphi, L, V, s, K, \varepsilon} := \left\{ \psi \in \CFHHS (\anotContField, \aContField) \colon \baseSpace{\psi}(L) \in V \text{ and } \psi(s) \in W_{\varphi(s), K, \varepsilon} \right\} \; .
	\]
	
	Finally, we equip $\operatorname{Isom}(\aContField)$ with the subspace topology inherited from $\CFHHS (\aContField, \aContField)$. 
	Hence given a topological group $G$, an isometric action of $G$ on $\aContField$ is \emph{continuous} if the underlying homomorphism $G \to \operatorname{Isom}(\aContField)$ is continuous. 
\end{defn}

\begin{rmk}\label{rmk:isometric-action-topologize-fixed-base-space}
	When we restrict the topology on $\CFHHS (\anotContField , \aContField)$ to $\CFHHS_f (\anotContField , \aContField)$ defined in \Cref{defn:cont-fields-morphism}\eqref{defn:cont-fields-morphism:category-fixed-base}, we observe that for any $\varphi \in \CFHHS_f (\anotContField, \aContField)$, any $s \in \spaceOfSections_{\operatorname{cont}} (\anotContField)$, any $K \subseteq \baseSpace{\aContField}$ compact, and any $\varepsilon > 0$, we have
	\[
	U_{\varphi, L, V, s, K, \varepsilon} \cap \CFHHS_f (\anotContField , \aContField) = U_{\varphi, \varnothing, \baseSpace{\anotContField}, s, K, \varepsilon} \cap \CFHHS_f (\anotContField , \aContField) \; ,
	\]
	and thus we may simply drop $L$ and $V$ in the definition. 
\end{rmk}

\begin{rmk}\label{rmk:isometric-action-topologize-generating-set}
	In \Cref{defn:isometric-action-topologize}, if $\spaceOfSections_{\operatorname{cont}}(\anotContField)$ is generated by $\Sigma$, 
	then for any $\varphi \in \CFHHS (\anotContField, \aContField)$, the collection 
	\begin{align*}
	\Big\{ U_{\varphi, L, V, s, K, \varepsilon} \colon &\ V \subseteq \baseSpace{\anotContField} \text{ open},  L \subseteq \baseSpace{\varphi}^{-1} (V) \text{ compact}, \\ 
	&\ s \in \Sigma, K \subseteq \baseSpace{\aContField} \text{ compact}, \varepsilon > 0 \Big\}
	\end{align*}
	also forms a local subbase. 
	Indeed, this follows from \Cref{lem:cont-field-HHS-Diximier-generate}
	and the straightforward fact that for any open set $V \subseteq \baseSpace{\anotContField}$, any compact set $L \subseteq \varphi^{-1} (V)$, any $s, s_1, \ldots, s_n \in \spaceOfSections_{\operatorname{cont}} (\anotContField)$, any compact set $K \subseteq \baseSpace{\aContField}$, and any $\varepsilon, \delta > 0$, if for any $z \in K$ there is $j \in \{1, \ldots, n\}$ such that $d_{\anotContField_{\baseSpace{\varphi}(z)}} \left( s(z) , s_j (z) \right) < \delta$, then
	\[
	\varphi \in \bigcap_{j = 1}^n U_{\varphi, L, V, s_j, K, \varepsilon} \subseteq U_{\varphi, L, V, s, K, \varepsilon + 2 \delta} \; .
	\]
\end{rmk}

We establish some basic properties regarding these topologies. 

\begin{lem}\label{lem:isometric-action-topologize1}
	Let $\aContField$, $\anotContField$ and $\anotContField'$ be continuous fields of {\hhs}s. 
	Then composition of isometric continuous morphisms yields a continuous map 
	\[
	\CFHHS (\anotContField', \anotContField) \times \CFHHS(\anotContField, \aContField) \to \CFHHS (\anotContField', \aContField) \; .
	\]
\end{lem}

\begin{proof}
	This follows from the observation that for any $\varphi \in \CFHHS (\anotContField, \aContField)$, any $\varphi' \in \CFHHS (\anotContField', \anotContField)$, 
	any open set $V \subseteq \baseSpace{\anotContField}$,  any compact set $L \subseteq \baseSpace{\varphi'}^{-1} \circ \baseSpace{\varphi}^{-1} (V)$, any $s \in \spaceOfSections_{\operatorname{cont}} \left( \anotContField' \right)$, any compact set $K \subseteq \baseSpace{\aContField}$, and any $\varepsilon > 0$,  
	as soon as we exploit the regularity of $\baseSpace{\anotContField}$ to choose a precompact open set $V'$ with $\baseSpace{\varphi'} (L) \subseteq V' \subseteq \overline{V'} \subseteq \baseSpace{\varphi}^{-1} (V)$, we see that the composition map takes the open neighborhood $U_{\varphi', L, V', s, \baseSpace{\varphi} (K), \varepsilon} \times U_{\varphi, \overline{V'}, V, \varphi'(s), K, \varepsilon}$ of $(\varphi', \varphi)$ into the open neighborhood $U_{\varphi' \circ \varphi, L, V, s, K, 2 \varepsilon}$ of $\varphi' \circ \varphi$. 
\end{proof}

The next lemmas show that the topologies defined in \Cref{defn:isometric-action-topologize} behave well with regard to currying. 

\begin{lem}\label{lem:isometric-action-topologize}
	Let $\aContField$ and $\anotContField$ be continuous fields of {\hhs}s and let $Y$ be a locally compact paracompact Hausdorff space. Then the following hold: \begin{enumerate}
		\item\label{lem:isometric-action-topologize:sections} A map $f \colon Y \to \spaceOfSections_{\operatorname{cont}} (\aContField)$ is continuous if and only if the section $\displaystyle \left( f(y)(z) \right)_{(y,z) \in Y \times \baseSpace{\aContField} } \in \prod_{(y,z) \in Y \times \baseSpace{\aContField}} \aContField_{z}$ belongs to $\spaceOfSections_{\operatorname{cont}} \left( \aContField |_{Y \times \baseSpace{\aContField}} \right) $ (see \Cref{def:cont-field-HHS-maps}\eqref{def:cont-field-HHS-maps:restrict}). 
		
		\item\label{lem:isometric-action-topologize:isometries} 
		A map $g \colon Y \to \CFHHS (\anotContField, \aContField)$ is continuous if and only if the map $Y \times \baseSpace{\aContField} \to \baseSpace{\anotContField}$, $(y,z) \mapsto \baseSpace{g(y)} (z)$, is continuous and for any $s \in \spaceOfSections_{\operatorname{cont}} (\anotContField)$, the section $\displaystyle \left( \left(g(y)(s)\right)_z \right)_{(y,z) \in Y \times \baseSpace{\aContField} } \in \prod_{(y,z) \in Y \times \baseSpace{\aContField}} \aContField_{z}$ belongs to $\spaceOfSections_{\operatorname{cont}} \left( \aContField |_{Y \times \baseSpace{\aContField}} \right) $. 
	\end{enumerate}
\end{lem}

\begin{proof}
	To prove \eqref{lem:isometric-action-topologize:sections}, we start by observing 
	that by \Cref{defn:isometric-action-topologize}, $f \colon Y \to \spaceOfSections_{\operatorname{cont}} (\aContField)$ is continuous if and only if 
	for any $y \in Y$, any precompact subset $K$ in $\baseSpace{\aContField}$, and any $\varepsilon > 0$, there is an open neighborhood $U$ of $y$ in $Y$ such that 
	\begin{equation}\label{lem:isometric-action-topologize:proof:equation}
	d_{\aContField_{z'}} \left( f(y)(z') , f(y')(z') \right) < \varepsilon \quad \text{ for any } z' \in K \text{ and } y' \in U \; .
	\end{equation}
	Since $\baseSpace{\aContField}$ and $Y$ are locally compact, we may, without loss of generality, require $K$ to be open and $U$ to be precompact in the above characterization. 
	By a standard argument involving tubular neighborhoods, we see that the last characterization is equivalent to that for any $y \in Y$, any $z \in \baseSpace{\aContField}$, and any $\varepsilon > 0$, there is a precompact open neighborhood $U$ of $y$ in $Y$ and a precompact open neighborhood $K$ of $z$ in $\baseSpace{\aContField}$ such that \eqref{lem:isometric-action-topologize:proof:equation} holds. 
	Finally, observing that the term $f(y)(z')$ in \eqref{lem:isometric-action-topologize:proof:equation} is equal to $(f(y) \circ \pi) (y', z')$, where $\pi \colon Y \times \baseSpace{\aContField} \to \baseSpace{\aContField}$ is the canonical projection, we deduce by \Cref{lem:cont-field-HHS-Diximier-generate} that the last characterization is equivalent to that the collection 
	\[
	\left\{ \left( f(y)(z) \right)_{(y,z) \in Y \times \baseSpace{\aContField} } \right\} \cup \left\{ s \circ \pi \colon s \in \spaceOfSections_{\operatorname{cont}} (\aContField) \right\} \subseteq   \prod_{(y,z) \in Y \times \baseSpace{\aContField}} \aContField_{z}
	\]
	of sections is mutually co-continuous, which is in turn equivalent to that the section $\left( f(y)(z) \right)_{(y,z) \in Y \times \baseSpace{\aContField} }$ belongs to $\spaceOfSections_{\operatorname{cont}} \left( \aContField |_{Y \times \baseSpace{\aContField}} \right) $, thanks to \Cref{def:cont-field-HHS-maps}, \Cref{def:cont-field-HHS-generate}, and \Cref{lem:cont-field-HHS-generate}. 
	
	To prove \Cref{lem:isometric-action-topologize:isometries}, 
	we combine the argument above for \eqref{lem:isometric-action-topologize:sections} with the standard fact that compact-open topologies are compatible with currying: in particular, the map $Y \times \baseSpace{\aContField} \to \baseSpace{\anotContField}$, $(y,z) \mapsto \baseSpace{g(y)} (z)$, is continuous if and only if for any open subset $V \subseteq \baseSpace{\anotContField}$ and compact subset $L \subseteq \baseSpace{\aContField}$, the set $\left\{ y \in Y \colon \baseSpace{g(y)} (L) \subseteq V \right\}$ is open in $Y$. 
\end{proof}

\begin{lem}\label{lem:isometric-action-topologize-curry}
	Let $\aContField$, $\anotContField$ and $Y$ be as in \Cref{lem:isometric-action-topologize}, 
	and let $\Lambda$ be a topological space. 
	For a map $g \colon \Lambda \times Y \to \CFHHS (\anotContField, \aContField)$, the following are equivalent: 
	\begin{enumerate}
		\item\label{lem:isometric-action-topologize:isometries-curry:orig}  The map $g$ is continuous. 
		
		\item\label{lem:isometric-action-topologize:isometries-curry:one-sided} There is a continuous map $\widehat{g} \colon \Lambda \to \CFHHS \left(\anotContField, \aContField|_{Y \times \baseSpace{\aContField}} \right)$ such that for any $\lambda \in \Lambda$, $y \in Y$ and $z \in \baseSpace{\aContField}$, we have 
		\[
		\hspace{0.7cm} \baseSpace{\widehat{g} (\lambda)} (y,z)=  \baseSpace{g(\lambda,y)} (z) \  \text{ and } \  \widehat{g(\lambda)}_{(y,z)} = g(\lambda,y)_z \colon \anotContField_{\baseSpace{\widehat{g}(\lambda)} (y,z)} \to \aContField_z \, .
		\]
		
		\item\label{lem:isometric-action-topologize:isometries-curry:two-sided} There is a continuous map $\widetilde{g} \colon \Lambda \to \CFHHS \left(\anotContField|_{Y \times \baseSpace{\anotContField}}, \aContField|_{Y \times \baseSpace{\aContField}} \right)$ such that for any $\lambda \in \Lambda$, $y \in Y$ and $z \in \baseSpace{\aContField}$, we have 
		\[
		\hspace{0.2cm} \baseSpace{\widetilde{g} (\lambda)} (y,z)=  \left( y, \baseSpace{g(\lambda,y)} (z) \right) \  \text{ and } \  \widetilde{g(\lambda)}_{(y,z)} = g(\lambda,y)_z \colon \anotContField_{\baseSpace{\widetilde{g}(\lambda)} (y,z)} \to \aContField_z \, .
		\]
	\end{enumerate}
\end{lem}

\begin{proof}
	To prove \eqref{lem:isometric-action-topologize:isometries-curry:orig} implies \eqref{lem:isometric-action-topologize:isometries-curry:two-sided}, we first observe that for any $\lambda \in \Lambda$, its image $\widetilde{g}(\lambda)$ is indeed an isometric continuous morphism from $\anotContField|_{Y \times \baseSpace{\anotContField}}$ to $\aContField|_{Y \times \baseSpace{\aContField}}$, thanks to \Cref{lem:isometric-action-topologize}\eqref{lem:isometric-action-topologize:isometries} and the fact from \Cref{def:cont-field-HHS-maps} that 	
	$\spaceOfSections_{\operatorname{cont}} \left( \anotContField |_{\Lambda \times \baseSpace{\anotContField}} \right) $ is generated by $\spaceOfSections_{\operatorname{cont}} \left( \anotContField \right) $, which allows us to apply
	\Cref{lem:cont-fields-HHS-map-generators}. 
	
	To show $\widetilde{g}$ is continuous, it suffices to show that given any $\lambda_0 \in \Lambda$, 
	any open set $V \subseteq \baseSpace{\anotContField}$,  any compact set $L \subseteq \baseSpace{\widetilde{g}(\lambda_0)}^{-1} (V)$, any $s \in \spaceOfSections_{\operatorname{cont}} \left( \anotContField |_{Y \times \baseSpace{\aContField}} \right)$, any compact sets $M \subseteq Y $ and $K \subseteq \baseSpace{\aContField}$, and any $\varepsilon > 0$,  
	there exists an open neighborhood $W$ of $\lambda_0$ such that for any $\lambda \in W$, any $y \in M$ and $z \in K$, 
	we have 
	\[
	\baseSpace{\widetilde{g}(\lambda)} (L) \subseteq V 
	\quad \text{ and } \quad 
	d_{ \aContField_{z} } \left( \widetilde{g}(\lambda) (s) (y, z), \widetilde{g}(\lambda_0) (s) (y, z)  \right) < \varepsilon \; ,
	\]
	or equivalently, 
	\begin{align*}
	& \left( y', \baseSpace{g(\lambda,y')} (z') \right) \subseteq V \ \text{ for any } (y',z') \in L
	\quad \text{ and } \quad  \\
	& d_{ \aContField_{z} } \left( g(\lambda, y)_z \left( s \left( y, \baseSpace{g(\lambda,y)} (z) \right) \right),  g(\lambda_0, y)_z \left( s \left( y, \baseSpace{g(\lambda_0,y)} (z) \right) \right)  \right) < \varepsilon \; .
	\end{align*}
	In view of \Cref{rmk:isometric-action-topologize-generating-set} and \Cref{def:cont-field-HHS-maps}\eqref{def:cont-field-HHS-maps:restrict}, we may assume without loss of generality that $s \in \spaceOfSections_{\operatorname{cont}} \left( \anotContField \right)$, viewed as a subset of $\spaceOfSections_{\operatorname{cont}} \left( \anotContField |_{Y \times \baseSpace{\aContField}} \right)$ and thus the last inequality becomes 
	\[
	d_{ \aContField_{z} } \left( g(\lambda, y)_z \left( s \left( \baseSpace{g(\lambda,y)} (z) \right) \right),  g(\lambda_0, y)_z \left( s \left( \baseSpace{g(\lambda_0,y)} (z) \right) \right)  \right) < \varepsilon \; ,
	\]
	or equivalently, 
	\[
	d_{ \aContField_{z} } \left( {g}(\lambda, y) (s) (z), {g}(\lambda_0, y) (s) (z)  \right) < \varepsilon \; .
	\]
	Since $y'$ and $y$ above range over compact sets, we may apply a standard tubular neighborhood argument to find the desired open neighborhood $W$ of $\lambda_0$. 
	
	To prove \eqref{lem:isometric-action-topologize:isometries-curry:two-sided} implies \eqref{lem:isometric-action-topologize:isometries-curry:one-sided}, we simply define $\widehat{g}(\lambda) = \pi^* \circ \widetilde{g}(\lambda)$ for any $\lambda \in \Lambda$, where $\pi \colon Y \times \baseSpace{\anotContField} \to \baseSpace{\anotContField}$ is the canonical quotient map. 
	
	To prove \eqref{lem:isometric-action-topologize:isometries-curry:one-sided} implies \eqref{lem:isometric-action-topologize:isometries-curry:orig}, it suffices to show that given by $\lambda_0 \in \Lambda$, any $y_0 \in Y$, any open set $V \subseteq \baseSpace{\anotContField}$,  any compact set $L \subseteq \baseSpace{{g}(\lambda_0, y_0)}^{-1} (V)$, any $s \in \spaceOfSections_{\operatorname{cont}} \left( \anotContField \right)$, any compact set $K \subseteq \baseSpace{\aContField}$, and any $\varepsilon > 0$,  
	there exists open neighborhoods $W$ of $\lambda_0$ and $U$ of $y_0$ such that for any $\lambda \in W$, any $y \in U$ and any $z \in K$, we have 
	\[
	\baseSpace{{g}(\lambda, y)} (L) \subseteq V 
	\quad \text{ and } \quad 
	d_{ \aContField_{z} } \left( {g}(\lambda, y) (s) (z), {g}(\lambda_0, y_0) (s) (z)  \right) < \varepsilon \; ,
	\]
	or equivalently, 
	\[
	\baseSpace{\widehat{g}(\lambda)} (U \times L) \subseteq V 
	\quad \text{ and } \quad 
	d_{ \aContField_{z} } \left( \widehat{g}(\lambda) (s) (y, z), \widehat{g}(\lambda_0) (s) (y_0, z)  \right) < \varepsilon \; . 
	\]
	Hence we may first exploit the compactness of $L$ and $K$ to choose $U$ to be a precompact open neighborhood of $y_0$ such that 
	\[
	\baseSpace{\widehat{g}(\lambda_0)} (U \times L) \subseteq V 
	\quad \text{ and } \quad 
	d_{ \aContField_{z} } \left( \widehat{g}(\lambda_0) (s) (y, z), \widehat{g}(\lambda_0) (s) (y_0, z)  \right) < \varepsilon \; ,
	\]
	and then use the continuity of $\widehat{g}$ to choose $W$ to satisfy the desired conditions above. 
\end{proof}

\begin{lem}\label{lem:isometric-action-topological-group-curry}
	Let $\aContField$ and $Y$ be as in \Cref{lem:isometric-action-topologize}, and 
	let $G$ be a topological group. 
	For a map $\alpha \colon Y \times G \to \CFHHS(\aContField, \aContField)$, $(y, g) \mapsto \alpha_{y, g}$, the following are equivalent: 
	\begin{enumerate}
		\item\label{lem:isometric-action-topological-group-curry:isometries-curry:pointwise}  The map $\alpha$ is continuous and for any $y \in Y$, the assignment $g \mapsto \alpha_{y, g}$ yields a group homomorphism $G \to \operatorname{Isom}(\aContField)$. 
		
		\item\label{lem:isometric-action-topological-group-curry:isometries-curry:global} There is a continuous homomorphism $\widetilde{\alpha} \colon G \to \operatorname{Isom} \left( \aContField |_{Y \times \baseSpace{\aContField}} \right)$ such that for any $g \in G$, $y \in Y$ and $z \in \baseSpace{\aContField}$, we have 
		\[
		\baseSpace{\widetilde{\alpha}_{g}} (y, z)  = \left( y,  \baseSpace{\alpha_{y,g}} (z)  \right) \quad \text{ and } \quad \left( \widetilde{\alpha}_g \right)_{(y,z)} = \left( \alpha_{y,g} \right)_{z} 
		\colon \aContField_{\baseSpace{\alpha_{y,g}} (z)} \to \aContField_z
		\; .
		\]
	\end{enumerate}
\end{lem}

\begin{proof}
	This follows directly from \Cref{lem:isometric-action-topologize-curry} that the map $\alpha \colon G \times Y \to \CFHHS(\aContField, \aContField)$ is continuous if and only if there is an induced continuous map $\widetilde{\alpha} \colon G \to \CFHHS \left(  \aContField |_{\baseSpace{\aContField} \times Y},   \aContField |_{\baseSpace{\aContField} \times Y}\right)$ satisfying the conditions as stated in the current lemma. 
	It is then straightforward to check that $\widetilde{\alpha}$ is a group homomorphism into $\operatorname{Isom} \left( \aContField |_{\baseSpace{\aContField} \times Y} \right)$ if and only for any $y \in Y$, the assignment $g \mapsto \alpha_{y, g}$ yields a group homomorphism $G \to \operatorname{Isom}(\aContField)$. 
\end{proof}

Let us discuss a slightly different perspective on continuous fields of {\hhs}s that is sometimes convenient, e.g., when we study the Borel measurable sections associated to a continuous field of {\hhs}s below. 

\begin{defn}\label{def:cont-field-HHS-total-space}
	Let $\aContField$ be a continuous field of {\hhs}s. 
	The \emph{total space} of $\aContField$, denoted by $|\aContField|$, is a topological space with an underlying set $\bigsqcup_{z \in \baseSpace{\aContField}} \aContField_z$ and a topology generated by a base consisting of sets of the form
	\[
	V_{U, s, \varepsilon} := \left\{ (z, x) \in \{z\} \times \aContField_z \subseteq \bigsqcup_{z' \in \baseSpace{\aContField}} \aContField_{z'} \colon z \in U , \ d_{\aContField_z} (x, s(z)) < \varepsilon \right\} \; ,
	\]
	where $U$, $s$ and $\varepsilon$ range over all open subsets of $\baseSpace{\aContField}$, all continuous sections of $\aContField$, and all positive real numbers, respectively. 
	This is indeed a base: 
	\begin{enumerate}
		\item It follows from \Cref{def:cont-field-HHS}\eqref{def:cont-field-HHS::dense} that 
		\[
		X = \bigcup_{\text{open } U \subseteq X} \bigcup_{s \in \spaceOfSections_{\operatorname{cont}}(\aContField)} \bigcup_{\varepsilon > 0} V_{U, s, \varepsilon} \; .
		\]
		\item For any open subsets $U_1$ and $U_2$ of $\baseSpace{\aContField}$, any continuous sections $s_1$ and $s_2$ of $\aContField$, any $\varepsilon_1, \varepsilon_2 \in (0, \infty)$, and any $(z,x) \in V_{U_1, s_1, \varepsilon_1} \cap V_{U_2, s_2, \varepsilon_2} $, we let $\varepsilon = \frac{1}{3} \min \left\{ \varepsilon_i - d_{\aContField_z} (s_i(z), x) \colon i = 1, 2 \right\} > 0$, 
		apply \Cref{def:cont-field-HHS}\eqref{def:cont-field-HHS::dense} to obtain a continuous section $s$ of $\aContField$ with $d_{\aContField_z} (s(z), x) < \varepsilon$, 
		and apply \Cref{def:cont-field-HHS}\eqref{def:cont-field-HHS::continuous} to obtain an open neighborhood $U$ of $z$ such that $U \subseteq U_1 \cap U_2$ and for any $z' \in U$ and any $i \in \{1,2\}$, we have $d_{\aContField_{z'}} \left( s_i(z') , s(z') \right) < d_{\aContField_{z}} \left( s_i(z) , s(z) \right) + \varepsilon$. 
		Now for any $(z',x') \in V_{U, s, \varepsilon}$ and any $i \in \{1,2\}$, we have $z' \in U \subseteq U_i$ and 
		\begin{align*}
		d_{\aContField_{z'}} \left( s_i(z') , x' \right)  & <
		d_{\aContField_{z'}} \left( s_i(z') , s(z') \right) + \varepsilon \\
		& < d_{\aContField_{z}} \left( s_i(z) , s(z) \right) + 2 \varepsilon \\
		& < d_{\aContField_{z}} \left( s_i(z) , x \right) + 3 \varepsilon 
		\qquad \leq \varepsilon_i \; .
		\end{align*}
		Thus we have $(z,x) \in V_{U, s, \varepsilon} \subseteq V_{U_1, s_1, \varepsilon_1} \cap V_{U_2, s_2, \varepsilon_2}$, as desired. 
	\end{enumerate}
	
	The canonical surjection $\pi_{\aContField} \colon |\aContField| \to \baseSpace{\aContField}$ is given by $\pi_{\aContField} (z,x) = z$. 
\end{defn}

\begin{lem}\label{lem:cont-field-HHS-total-space} 
	Let $\aContField$ be a continuous field of {\hhs}s. Then the following hold: 
	\begin{enumerate}
		\item \label{lem:cont-field-HHS-total-space:quotient} The canonical surjection $\pi_{\aContField}$ is a topological quotient map and is open. 
		
		\item \label{lem:cont-field-HHS-total-space:embedding} For any generating set $\spaceOfSections_0$ of $\aContField$ as in \Cref{def:cont-field-HHS-generate}, the map 
		\[
		|\aContField| \to \baseSpace{\aContField} \times \prod_{s \in \spaceOfSections_0} [0, \infty) \, ,  \quad (z,x) \mapsto \left( z, \left( d_{\aContField_z} (x, s(z)) \right)_{s \in \spaceOfSections_0} \right)
		\]
		is an embedding of topological spaces. 
		
		\item \label{lem:cont-field-HHS-total-space:metric} Write $|\aContField| \underset{\baseSpace{\aContField}}{\times} |\aContField|$ for the subspace 
		\[
		\left\{ \left( (z,x) , (z',x') \right) \in |\aContField| {\times} |\aContField| \colon z = z' \right\} \; .
		\]
		Then the map 
		\[
		d \colon |\aContField| \underset{\baseSpace{\aContField}}{\times} |\aContField| \to [0, \infty) \, , \quad \left( (z,x) , (z,x') \right) \mapsto d_{\aContField_z} (x, x')
		\]
		is continuous. 
		
		\item \label{lem:cont-field-HHS-total-space:midpoint} 
		With notations as in \Cref{rmk:CAT0-facts}\eqref{rmk:CAT0-facts-bicombing}, 
		The map
		\[ \kappa \colon \left( |\aContField| \underset{\baseSpace{\aContField}}{\times} |\aContField| \right) \times [0,1] \to |\aContField| \, , \  \left( ((z,x_1),(z,x_2)), t \right) \mapsto \left(z, [x_1,x_2] (t) \right)\]
		is continuous.
		
		\item \label{lem:cont-field-HHS-total-space:sections} For any $s \in \prod_{z \in \baseSpace{\aContField}} \aContField_{z}$, we have $s \in \spaceOfSections_{\operatorname{cont}} (\aContField)$ if and only if the map 
		\[
		\breve{s} \colon \baseSpace{\aContField} \to |\aContField| \, , \quad z \mapsto (z, s(z))
		\]
		is continuous. 
	\end{enumerate}
\end{lem}

\begin{proof}
	\begin{enumerate}
		\item 
		The openness of $\pi_{\aContField}$ follows from the observation that $\pi_{\aContField} \left( V_{U, s, \varepsilon} \right) = U$ for any open subset $U$ of $\baseSpace{\aContField}$, any continuous section $s$ of $\aContField$ and any $\varepsilon > 0$. 
		On the other hand, $\pi_{\aContField}$ is a quotient map since for any subset $Y \subseteq \baseSpace{\aContField}$, the interior of $\pi_{\aContField} ^{-1} (Y)$ is given by 
		\[
		\bigcup_{\text{open } U \subseteq Y} \bigcup_{s \in \spaceOfSections_{\operatorname{cont}}(\aContField)} \bigcup_{\varepsilon > 0} V_{U, s, \varepsilon} \; ,
		\]
		which agrees with $\pi_{\aContField} ^{-1} (Y)$ if and only $Y$ is open. 
		
		\item The injectivity of this map follows from the requirement that $\spaceOfSections_0$ is pointwise dense. To see that it is a topological embedding, 
		we need to show the original topology $\mathcal{T}$ on $|\aContField|$ as in \Cref{def:cont-field-HHS-total-space} agrees with the topology $\mathcal{T}'$ inherited from $\baseSpace{\aContField} \times \prod_{s \in \spaceOfSections_0} [0, \infty)$. Observe that $\mathcal{T}'$ is generated by a subbase consisting of sets of the form 
		\[
		W_{U, s, I} := \left\{ (z, x) \in \{z\} \times \aContField_z \subseteq \bigsqcup_{z' \in \baseSpace{\aContField}} \aContField_{z'} \colon z \in U , \ d_{\aContField_z} (x, s(z)) \in I \right\} \; ,
		\]
		where $U$, $s$ and $I$ range over all open subsets of $\baseSpace{\aContField}$, all continuous sections of $\aContField$, and all relatively open intervals inside $[0, \infty)$, respectively. Hence clearly $\mathcal{T} \subseteq \mathcal{T}'$. On the other hand, it is routine to check that each $W_{U, s, I}$ as above is in $\mathcal{T}$, using the triangle inequality and the pointwise density of $\spaceOfSections_0$, whence $\mathcal{T} = \mathcal{T}'$. 
		
		%%%%%%%%%%%%%%%%%%%%%%%%
		
		\item To show this, it suffices to show that for any $((z_0,x_0),(z_0,x_0'))$ such that $\alpha<d_{\aContField_{z_0}}(x_0,x_0')<\beta$, there are neighborhoods $U$ of $(z_0,x_0)$ and $U'$ of $(z_0,x_0')$ in $|\aContField|$ such that for any $((z,x),(z,x')) \in U \times U'$, we have $\alpha<d_{\aContField_{z}}(x,x')<\beta$.

		Consider real numbers $\alpha_1,\beta_1$ such that 
		\[\alpha<\alpha_1<d_{\aContField_{z_0}}(x_0,x_0')<\beta_1<\beta.\]
		Note that for any $\varepsilon_1>0$ there are continuous sections $s,s'$ such that 
		\[d_{\aContField_{z_0}}(x_0,s(z_0))<\varepsilon_1 \text{, and }
		d_{\aContField_{z_0}}(x_0',s'(z_0))<\varepsilon_1,\]
		which, in particular, means by the triangle inequality that
		\[|d_{\aContField_{z_0}}(s(z_0),s'(z_0))-d_{\aContField_{z_0}}(x_0,x_0')|<2\varepsilon_1\]
		Picking
		\[\varepsilon_1<\frac{1}{2}\min(d_{\aContField_{z_0}}(x_0,x_0')-\alpha_1, \beta_1-d_{\aContField_{z_0}}(x_0,x_0'),\alpha_1-\alpha,\beta-\beta_1),\]
		ensures that $\alpha_1<d_{\aContField_{z_0}}(s(z_0),s'(z_0))<\beta_1$. (Actually to ensure these inequalities we only need $\varepsilon_1<\frac{1}{2}\min(d_{\aContField_{z_0}}(x_0,x_0')-\alpha_1, \beta_1-d_{\aContField_{z_0}}(x_0,x_0'))$, but we will use the latter terms later.)
		Because $s$ and $s'$ are co-continuous, there is an open neighborhood $U_0$ of $z_0$ in $\baseSpace{\aContField}$ such that for any $z \in U_0$, we have
		\[\alpha_1<d_{\aContField_{z}}(s(z),s'(z))<\beta_1.\]
		
		Pick $\varepsilon_2$ satisfying $0<\varepsilon_2= \frac{1}{2}\min(\alpha_1-\alpha,\beta-\beta_1)$. Then pick 
		\[U = V_{U_0,s,\varepsilon_2} = \{(z,x)|z \in U_0, d_{\aContField_{z}}(x,s(z))<\varepsilon_2)\}\]
		and
		\[U' = V_{U_0,s',\varepsilon_2} = \{(z,x)|z \in U_0, d_{\aContField_{z}}(x',s'(z))<\varepsilon_2)\}.\]
		
		Then for $((z,x),(z,x')) \in U \times U'$, 
		\[|d_{\aContField_{z}}(x,x')-d_{\aContField_{z}}(s(z),s'(z))|<2\varepsilon_2,\]
		which, combined with our earlier inequality $\alpha_1<d_{\aContField_{z}}(s(z),s'(z))<\beta_1$, implies that
		\[\alpha<d_{\aContField_{z}}(x,x')<\beta.\]
		Moreover, we have that $((z_0,x_0),(z_0,x_0'))$ is in $U \times U'$, because we chose $\varepsilon_1\leq\varepsilon_2$. Thus, we have found the desired neighborhoods $U$ and $U'$.
		
		\item 
		To show $\kappa$ is continuous, it suffices to show, for any open subset $U$ of $\baseSpace{\aContField}$, any continuous section $s$ of $\aContField$, and any $\varepsilon > 0$, the preimage $\kappa^{-1} \left( V_{U, s, \varepsilon} \right)$ is open. 
		To this end, we fix $\left( ((z,x_0),(z,x_1)), t \right) \in \kappa^{-1} \left( V_{U, s, \varepsilon} \right)$, which simply means $z \in U$ and $d_{\aContField_{z}} \left( \left[ x_0, x_1 \right] (t) , s(z)\right) < \varepsilon$. We seek an open subset $U'$ of $\baseSpace{\aContField}$, continuous sections $s_0$ and $s_1$ of $\aContField$, and $\varepsilon', \delta \in (0, \infty)$ such that $(z,x_i) \in V_{U', s_i, \varepsilon'}$ for $i \in \{1,2\}$ and $\kappa^{-1} \left( V_{U, s, \varepsilon} \right)$ contains
		\begin{align*}
		\left( \left( V_{U', s_0, \varepsilon'} \times V_{U', s_1, \varepsilon'} \right) \times (t -\delta, t+ \delta) \right) \cap \left( \left( |\aContField| \underset{\baseSpace{\aContField}}{\times} |\aContField| \right) \times [0,1] \right) \; ,
		\end{align*}
		that is, for any $z' \in U'$, $t' \in (t -\delta, t+ \delta) \cap [0,1]$ and $x'_0, x'_1 \in \aContField_{z'}$ with $d_{\aContField_{z'}} (x'_i, s_i(z')) < \varepsilon'$, we want to show $z' \in U$ and  
		\[
		d_{\aContField_{z'}} \left( \left[ x'_0, x'_1 \right] (t') , s(z')\right) < \varepsilon \; .
		\]
		
		To do this, we set  $\varepsilon' = \frac{1}{4} \left( \varepsilon - d_{\aContField_{z}} \left( \left[ x_0, x_1 \right] (t) , s(z)\right) \right) > 0$ and 
		apply \Cref{def:cont-field-HHS}\eqref{def:cont-field-HHS::dense} to obtain continuous sections $s_0$ and $s_1$ of $\aContField$ with $d_{\aContField_z} (s_i(z), x_i) < \varepsilon'$. 
		By \Cref{cor:cont-field-HHS-geodesic-approximate}, the section $s_t \colon z' \mapsto \left[ s_0 (z'), s_1 (z') \right] (t)$ is continuous. 
		Hence we may apply \Cref{def:cont-field-HHS}\eqref{def:cont-field-HHS::continuous} to obtain an open neighborhood $U'$ of $z$ such that $U' \subseteq U$ and for any $z' \in U'$, we have $d_{\aContField_{z'}} \left( s_t(z') , s(z') \right) < d_{\aContField_{z}} \left( s_t(z) , s(z) \right) + \varepsilon'$ and $d_{\aContField_{z'}} \left( s_0(z') , s_1(z') \right) < d_{\aContField_{z}} \left( s_0(z) , s_1(z) \right) + \varepsilon'$. 
		Finally we set 
		\[
		\delta = \frac{\varepsilon'}{d_{\aContField_{z}} \left( s_0(z) , s_1(z) \right) + 3 \varepsilon'} > 0 \; .
		\]
		Now in view of \Cref{rmk:CAT0-facts}\eqref{rmk:CAT0-facts-Lipschitz} we have, for any $z' \in U'$, $t' \in (t -\delta, t+ \delta) \cap [0,1]$ and $x'_0, x'_1 \in \aContField_{z'}$ with $d_{\aContField_{z'}} (x_i, s_i(z')) < \varepsilon'$,  
		\begin{align*}
		&\ d_{\aContField_{z'}} \left( \left[ x'_0, x'_1 \right] (t') , s(z')\right) \\
		\leq &\ d_{\aContField_{z'}} \left( \left[ x'_0, x'_1 \right] (t) , s(z')\right) + d_{\aContField_{z'}} \left( \left[ x'_0, x'_1 \right] (t) , \left[ x'_0, x'_1 \right] (t'))\right) \\
		= &\ d_{\aContField_{z'}} \left( \left[ x'_0, x'_1 \right] (t) , s(z')\right) + |t - t'| \, d_{\aContField_{z'}} \left( x'_0 , x'_1 \right) \\
		< &\ d_{\aContField_{z'}} \left( \left[ x'_0, x'_1 \right] (t) , s(z')\right) + \delta \, \left( d_{\aContField_{z'}} \left( s_0(z') , s_1(z') \right) + 2 \varepsilon' \right) \\
		< &\ d_{\aContField_{z'}} \left( \left[ x'_0, x'_1 \right] (t) , s(z')\right) + \delta \, \left( d_{\aContField_{z}} \left( s_0(z) , s_1(z) \right) + 3 \varepsilon' \right) \\
		= &\ d_{\aContField_{z'}} \left( \left[ x'_0, x'_1 \right] (t) , s(z')\right) + \varepsilon' \\
		\leq &\ d_{\aContField_{z'}} \left( s_t(z') , s(z')\right) + d_{\aContField_{z'}} \left( \left[ s_0 (z'), s_1 (z') \right] (t) , \left[ x'_0, x'_1 \right] (t) \right) + \varepsilon' \\
		\leq &\ d_{\aContField_{z'}} \left( s_t(z') , s(z')\right) + \max \left\{ d_{\aContField_{z'}} \left( s_i (z') ,  x'_i \right) \colon i \in \{0,1\} \right\} + \varepsilon' \\
		\leq &\ d_{\aContField_{z'}} \left( s_t(z') , s(z')\right) + 2 \varepsilon' \\
		< &\ d_{\aContField_{z}} \left( s_t(z) , s(z)\right) + 3 \varepsilon' \\
		\leq &\ d_{\aContField_{z}} \left( s_t(z) , \left[ x_0, x_1 \right] (t)\right) + d_{\aContField_{z}} \left( \left[ x_0, x_1 \right] (t) , s(z)\right) + 3\varepsilon' \\
		\leq &\ \max \left\{ d_{\aContField_{z}} \left( s_i (z) ,  x_i \right) \colon i \in \{0,1\} \right\}  + d_{\aContField_{z}} \left( \left[ x_0, x_1 \right] (t) , s(z)\right) + 3\varepsilon' \\
		< &\ d_{\aContField_{z}} \left( \left[ x_0, x_1 \right] (t) , s(z)\right) + 4\varepsilon' =  \varepsilon \; ,
		\end{align*}
		as desired. 
		%%%%%%%%%%%%%%%%%%%%%%%%
		
		\item Note that $\breve{s}: \baseSpace{\aContField} \to |\aContField|$ is continuous if and only if for each open set $V_{U, s_1, \varepsilon}$ in the subbase defining the topology of $|\aContField|$, $\breve{s}^{-1}(V_{U,s_1, \varepsilon})$ is open. Moreover,
		\[\breve{s}^{-1}(V_{U,s_1, \varepsilon}) = \{z \in U| d_{\aContField_{z}}(s(z),s_1(z))<\varepsilon\}.\]
		Thus, $\breve{s}$ is continuous if and only if for any open set $U$, continuous section $s_1$, and positive real number $\varepsilon$, the set $U \allowbreak \cap \left\{ z \in \baseSpace{\aContField}| d_{\aContField_{z}}(s(z),s_1(z))<\varepsilon \right\}$ is open. This happens if and only if for any continuous section $s_1$ and positive real number $\varepsilon$, the set
		\[\{z \in \baseSpace{\aContField}| d_{\aContField_{z}}(s(z),s_1(z))<\varepsilon\}\]
		is open. 
		
		If $s$ is in $ \spaceOfSections_{\operatorname{cont}} $, then the above set is open because $s$ and $s_1$ are co-continuous.
		
		For the other direction, if $\breve{s}$ is continuous, we want to show that $s$ is in $ \spaceOfSections_{\operatorname{cont}} $. For this it suffices to show Condition \eqref{def:cont-field-HHS::approx} for $s$, that is, it suffices to show that for any $z_0 \in \baseSpace{\aContField}$, and $\varepsilon>0$, there is an $s_1 \in \spaceOfSections_{\operatorname{cont}}$ and a neighborhood $U$ of $z_0$, such that $d_{\aContField_{z}}(s(z),s_1(z)) \leq \varepsilon$ for any $z \in U$.
		But for any $z_0 \in \baseSpace{\aContField}$, there is $s_1$ such that $d_{\aContField_{z_0}}(s(z_0),s_1(z_0)) <\frac{\varepsilon}{2}$, and by the condition that $\breve{s}$ is continuous, we have that 
		\[U=\{z \in \baseSpace{\aContField}| d_{\aContField_{z}}(s(z),s_1(z))<\varepsilon\}\]
		is open. It clearly contains $z_0$, so we have the desired continuous section and open set.
	\end{enumerate}
\end{proof}

\begin{rmk} \label{rmk:cont-field-HHS-total-space-maps}
	Let $\varphi \colon \aContField \to \anotContField$ be an isometric continuous morphism between continuous fields of {\hhs}s. Then 
	there are continuous maps 
	\begin{align*}
	\baseSpace{\varphi}_* \colon \left| \baseSpace{\varphi}^* \aContField \right| &\to \left| \aContField \right| \, , & (y, x) &\mapsto (\baseSpace{\varphi}(y) , x)  & \text{ for any } y \in \baseSpace{\anotContField} \text{ and } x \in \aContField_{\baseSpace{\varphi}(y)} \; , \\
	\dot{\varphi}_* \colon \left| \baseSpace{\varphi}^* \aContField \right| &\hookrightarrow \left| \anotContField \right| \, , & (y, x) &\mapsto (y , \varphi_y (x) )  & \text{ for any } y \in \baseSpace{\anotContField} \text{ and } x \in \aContField_{\baseSpace{\varphi}(y)} \; .
	\end{align*}
	With the notations in \Cref{rmk:cont-field-HHS-maps-factorization} and \Cref{lem:cont-field-HHS-total-space}\eqref{lem:cont-field-HHS-total-space:sections}, it follows directly from the definitions that for any $s \in \spaceOfSections_{\operatorname{cont}} (\aContField)$, 
	the section $\baseSpace{\varphi}^* (s) \in \spaceOfSections_{\operatorname{cont}} \left( \baseSpace{\varphi}^* \aContField \right) $ is the only element $s'$ in $\prod_{y \in \baseSpace{\anotContField}} \aContField_{\baseSpace{\varphi}(y)}$ 
	such that the diagram
	\[
	\xymatrix{
		\left| \aContField \right| & \left| \baseSpace{\varphi}^* \aContField \right| \ar[l]_{\baseSpace{\varphi}_*} \\
		\baseSpace{\aContField} \ar[u]^{\breve{s}} & \baseSpace{\anotContField} \ar[l]^{\baseSpace{\varphi}} \ar[u]_{\breve{s'}}
	}
	\]
	commutes, 
	while for any $s' \in \spaceOfSections_{\operatorname{cont}}\left( \baseSpace{\varphi}^* \aContField \right)$,  
	the section $\dot{\varphi} (s') \in \spaceOfSections_{\operatorname{cont}} \left( \anotContField \right) $ is the only element $s''$ in $\prod_{y \in \baseSpace{\anotContField}} \anotContField_{y}$ 
	such that the diagram
	\[
	\xymatrix{
		\left| \baseSpace{\varphi}^* \aContField \right| \ar@{^{(}->}[rr]^{\dot{\varphi}_*} && \left| \anotContField \right| \\
		& \baseSpace{\anotContField}  \ar[ul]^{\breve{s'}}  \ar[ur]_{\breve{s''}} &
	}
	\]
	commutes. 
\end{rmk}

%%%%%%%%%%%%%%%%%%%%%%%%%fields%%%%%%%%%%%%%%%%%%%%%%%%%
%%%%%%%%%%%%%%%%%%%%%%%%%meas-fields%%%%%%%%%%%%%%%%%%%%%%%%%
\section{Measurable fields of {\hhs}s}
\label{sec:meas-fields}

In this
section, we discuss 
a measure-theoretic variant of \Cref{def:cont-field-HHS}, analogous to the notion of a measurable field of Hilbert spaces (\cite[A~69]{Dixmier1977C}).

\begin{defn}  \label{def:meas-field-HHS}
	A \emph{measurable field} $\aMeasField$ of {\hhs}s consists of a measurable space $\baseSpace{\aMeasField}$ (called the \emph{base space} of $\aMeasField$), a tuple $\left( \aMeasField_z \right)_{z \in \baseSpace{\aMeasField}}$ of {\hhs}s, and a subset $\spaceOfSections_{\operatorname{meas}} (\aMeasField)$ (called \emph{the set of measurable sections}) of $\prod_{z \in \baseSpace{\aMeasField}} \aMeasField_z$ that is 
	\begin{enumerate}
		\item \label{def:meas-field-HHS::convex} 
		\emph{convex} in the sense of \Cref{def:cont-field-HHS}\eqref{def:cont-field-HHS::convex}, 
		
		\item \label{def:meas-field-HHS::dense} 
		\emph{pointwise dense} in the sense of \Cref{def:cont-field-HHS}\eqref{def:cont-field-HHS::dense},  
		
		\item \label{def:meas-field-HHS::measurable} 
		\emph{mutually co-measurable} in the sense 
		that any two elements $s, s' \in \spaceOfSections_{\operatorname{meas}} (\aMeasField)$ are \emph{co-measurable}, which means the function $\baseSpace{\aMeasField} \ni z \mapsto d_{\aMeasField_z} \left( s(z) , s'(z) \right) \in [0, \infty)$ is measurable, and
		
		\item \label{def:meas-field-HHS::complete} 
		\emph{co-measurability-closed} in the sense that 
		for any $s \in \prod_{z \in \baseSpace{\aMeasField}} \aMeasField_z$ satisfying 
		$\spaceOfSections_{\operatorname{meas}} (\aMeasField) \cup \{s\}$ is mutually co-measurable, 
		we have $s \in \spaceOfSections_{\operatorname{meas}} (\aMeasField)$.  
	\end{enumerate}
	Sometimes, to emphasize the base space $\baseSpace{\aMeasField}$ is equal to a given measurable space $(Z, \calB)$, we write $\aMeasField_Z$ or $\aMeasField_{(Z, \calB)}$ instead of $\aMeasField$ and $\spaceOfSections_{\operatorname{meas}} (Z, \aMeasField)$ or $\spaceOfSections_{\operatorname{meas}} (Z, \calB, \aMeasField)$ instead of $\spaceOfSections_{\operatorname{meas}} (\aMeasField)$, and say $\aMeasField$ is a measurable field of {\hhs}s \emph{over $(Z, \calB)$}. 
	
	A measured field of {\hhs}s is a tuple $(\aMeasField, \mu)$ where $\aMeasField$ is a measurable field of {\hhs}s and $\mu$ is a measure on $\baseSpace{\aMeasField}$. 
	In this case, we write $\spaceOfSections_{\operatorname{meas}} (\aMeasField, \mu)$ for the quotient of $\spaceOfSections_{\operatorname{meas}} (\aMeasField)$ by identifying sections that differ on null sets, that is, $s, s' \in \spaceOfSections_{\operatorname{meas}} (\aMeasField)$ are identified if and only if $\mu \left\{ z \in \baseSpace{\aMeasField} \colon d_{\aMeasField_z} \left( s(z) , s'(z) \right) > 0 \right\} = 0$. 
	Following a standard convention, when there is no danger of confusion, we do not distinguish between measurable sections from their equivalence classes, and thus we often write $s$ for $[s]$. 
\end{defn}

We can turn continuous fields of {\hhs}s into measurable fields.

\begin{lem} \label{lem:cont-field-HHS-meas}
	Let $\aContField$ be a continuous field of {\hhs}s. Let $\Sigma$ be a generating set of $\aContField$. Let $\calB$ be a $\sigma$-algebra on $\baseSpace{\aContField}$ that contains all the open subsets. Then we have 
	\begin{align*}
	& \bigg\{ s \in \prod_{z \in \baseSpace{\aContField}} \aContField_z \colon 
	\Sigma  \cup \{s\} \text{ is mutually co-measurable}
	\bigg\} \\
	= & \bigg\{ s \in \prod_{z \in \baseSpace{\aContField}} \aContField_z \colon 
	\text{the map $\baseSpace{\aContField} \to |\aContField|$, $z \mapsto (z, s(z))$ is measurable}
	\bigg\}
	\end{align*}
	and this set defines a measurable field of {\hhs}s. 
\end{lem}

\begin{proof}
	Note that for $s$ to be in 
	\[S_1 := \bigg\{ s \in \prod_{z \in \baseSpace{\aContField}} \aContField_z \colon 
	\Sigma  \cup \{s\} \text{ is mutually co-measurable}
	\bigg\},\]
	means that for any $s' \in \Sigma$, the map $z \mapsto d_{\aContField_z}(s'(z), s(z))$ is measurable, that is to say, the preimage of any open set is measurable. 	
	In other words, $s$ is in the above set if and only if for any $a,b \in \mathbb{R}$, and any $s' \in \Sigma$, the set
	\[\{z  \in \baseSpace{\aContField} \colon  a < d_{\aContField_z}(s'(z), s(z))<b\}\]
	is measurable.
	
	On the other hand, for $s$ to be in 
	\[S_2 := \bigg\{ s \in \prod_{z \in \baseSpace{\aContField}} \aContField_z \colon 
	\text{the map $\baseSpace{\aContField} \to |\aContField|$, $z \mapsto (z, s(z))$ is measurable}
	\bigg\}\]
	means that for any of the base open sets 
	\[V_{U,s',\varepsilon} = \left\{ (z,x) \in \{z\} \times \aContField_z \subset \bigsqcup_{z' \in \underline{\aContField}} \aContField_{z'} \colon  z \in U, d_{\aContField_z}(x, s'(z))< \varepsilon \right\}\]
	as in \Cref{def:cont-field-HHS-total-space}, we have that the preimage is measurable.
	That is to say, for any open set $U \subseteq \baseSpace{\aContField}$, $s' \in \Sigma$ and $\varepsilon>0$, we have 
	\[\left\{ z \in \baseSpace{\aContField} \colon  d_{\aContField_z}(s(z), s'(z))<\varepsilon \right\} \cap U\]
	is measurable. Note that we may fix $U =  \baseSpace{\aContField}$ without loss of generality here. 
	
	Comparing the two reformulations above, we see immediately that $S_1 \subseteq S_2$ (say, we take $a= -1$ and $b = \varepsilon$). In the other direction, for any $a, b \in \mathbb{R}$, we have 
	\begin{align*}
	&\ \{z  \in \baseSpace{\aContField} \colon  a < d_{\aContField_z}(s'(z), s(z))<b\} \\
	= &\  \bigcup_{k = 1}^\infty \left( \left\{ z  \in \baseSpace{\aContField} \colon  d_{\aContField_z}(s'(z), s(z))<b \right\} \setminus \left\{z  \in \baseSpace{\aContField} \colon  d_{\aContField_z}(s'(z), s(z))< a + {1}/{k} \right\} \right) \; , 
	\end{align*}
	which shows $S_1 \supseteq S_2$, as desired. 
	
	It remains to show that this set defines a measurable field of Hilbert-Hadamard space.

	To verify \Cref{def:meas-field-HHS}\eqref{def:meas-field-HHS::convex}, we show that $S_2$ is convex in the sense that for any $s_{-1}, s_1 \in S_2$, the midpoint section $s_0$ is also in $S_2$. 
	Note that with notations as in \Cref{lem:cont-field-HHS-total-space}\eqref{lem:cont-field-HHS-total-space:midpoint}  the map $\baseSpace{\aContField} \to |\aContField|$, $z \mapsto (z, s_0 (z))$ is equal to the composition 
	\[
	\baseSpace{\aContField} \xrightarrow{\left((-, s_{-1} (-)), (-, s_{1} (-)) \right)} |\aContField| \underset{\baseSpace{\aContField}}{\times} |\aContField|  \xrightarrow{(-, 1/2)} \left( |\aContField| \underset{\baseSpace{\aContField}}{\times} |\aContField| \right) \times [0,1] \xrightarrow{\kappa} |\aContField| \, ,
	\]
	where the last two maps are continuous. 
	Since $s_{-1}, s_1 \in S_2$, the first map is Borel measurable, which implies the composition is so, too, and thus $s_0 \in S_2$. 
	
	\Cref{def:meas-field-HHS}\eqref{def:meas-field-HHS::dense} follows immediately from  \Cref{def:cont-field-HHS}\eqref{def:cont-field-HHS::dense} since $S_1$ contains $\spaceOfSections_{\operatorname{cont}} (\aContField)$ by \Cref{def:cont-field-HHS}\eqref{def:cont-field-HHS::continuous}.

	To verify \Cref{def:meas-field-HHS}\eqref{def:meas-field-HHS::measurable}, 
	it suffices to fix arbitrary $s, s' \in S_2$ and show the map $\baseSpace{\aContField} \to \mathbb{R}$
	given by $z \mapsto d_{\aContField_z}(s(z),s'(z))$ is measurable. 
	With notations as in \Cref{lem:cont-field-HHS-total-space}\eqref{lem:cont-field-HHS-total-space:metric}  this map is equal to the composition 
	\[
	\baseSpace{\aContField} \xrightarrow{\left((-, s (-)), (-, s' (-)) \right)} |\aContField| \underset{\baseSpace{\aContField}}{\times} |\aContField|  \xrightarrow{d} \mathbb{R} \, ,
	\]
	where the second map is continuous. 
	Since $s, s' \in S_2$, the first map is Borel measurable, which implies the composition is so, too, as desired.

	Finally, to verify \Cref{def:meas-field-HHS}\eqref{def:meas-field-HHS::complete}, 
	notice that if a section $s$ is co-measurable with everything in $S_1$, then it is, in particular, co-measurable with everything in $\Sigma$, and is therefore in $S_1$, as desired.
\end{proof}

We remark that in the above proof, it is important to first establish the set equality in \Cref{lem:cont-field-HHS-meas} before proving the set defines a measurable field of {\hhs}s, as the various conditions in \Cref{def:meas-field-HHS} call for different characterizations of that same set. 
Something similar happens in the proof of 
\Cref{lem:meas-field-HHS-cont-maps}, where both characterizations need to be used. 

\begin{defn} \label{def:cont-field-HHS-meas}
	Let $\aContField$ be a continuous field of {\hhs}s. 
	Let $\calB$ be a $\sigma$-algebra containing all open subsets in $\baseSpace{\aContField}$. 
	Then we write $\spaceOfSections_{\operatorname{meas}}(\baseSpace{\aContField}, \calB, \aContField)$ for the set in \Cref{lem:cont-field-HHS-meas} and $\aContField_{\operatorname{meas, \calB}}$ for the resulting measurable field of {\hhs}s, which we call the measurable field of {\hhs}s \emph{associated} to $\aContField$ and $\calB$. 
	
	When $\calB$ is the $\sigma$-algebra of the Borel sets in $\baseSpace{\aContField}$, we may write $\spaceOfSections_{\operatorname{meas}}(\aContField)$ for $\spaceOfSections_{\operatorname{meas}}(\baseSpace{\aContField}, \calB, \aContField)$ and $\aContField_{\operatorname{meas}}$ for $\aContField_{\operatorname{meas, \calB}}$. 
\end{defn}

We also have measure-theoretic analogs of \Cref{defn:cont-fields-morphism}. 

\begin{defn}\label{defn:meas-fields-morphism}
	
	Let $(\aMeasField, \mu)$ and $(\anotMeasField, \nu)$ be two measured fields of {\hhs}s. 
	\begin{enumerate}
		\item An \emph{isometric measure-preserving morphism} 
		$\varphi \colon (\aMeasField, \mu) \to (\anotMeasField, \nu)$ is a tuple $\left(  \underline{\varphi},  \left( \varphi_y \right)_{y \in \operatorname{dom}\left(\baseSpace{\varphi}\right)} \right)$, where $\underline{\varphi} \colon \left(\operatorname{dom}\left(\baseSpace{\varphi}\right), \mu \right) \to (\baseSpace{\aMeasField}, \nu)$ is a measure-preserving map from a measurable subset $\operatorname{dom}\left(\baseSpace{\varphi}\right) \subseteq \baseSpace{\anotMeasField}$ with $\nu \left(\baseSpace{\anotMeasField} \setminus \operatorname{dom}\left(\baseSpace{\varphi}\right) \right) = 0$, and for any $y \in \operatorname{dom}\left(\baseSpace{\varphi}\right)$, $\varphi_y \colon \aMeasField_{\underline{\varphi}(y)} \to \anotMeasField_y$ is an isometric embedding, such that 
		for any $s \in \spaceOfSections_{\operatorname{meas}} (\aMeasField)$, the section $\left( \varphi_y \left( s \left( \underline{\varphi}(y) \right) \right) \right)_{y \in \operatorname{dom}\left(\baseSpace{\varphi}\right)}$ agrees with a section in $\spaceOfSections_{\operatorname{meas}} (\anotMeasField)$ $\nu$-almost everywhere. 
		Note that $\left( \varphi_y \left( s \left( \underline{\varphi}(y) \right) \right) \right)_{y \in \operatorname{dom}\left(\baseSpace{\varphi}\right)}$ determines a unique equivalence class in $\spaceOfSections_{\operatorname{meas}} (\anotMeasField, \nu)$ that depends only on the class $[s] \in \spaceOfSections_{\operatorname{meas}} (\aMeasField, \mu)$; thus we denote this unique equivalence class by $\varphi([s])$. 
		
		\item Two isometric measure-preserving morphisms $\varphi, \varphi' \colon (\aMeasField, \mu) \to (\anotMeasField, \nu)$ are \emph{equivalent} if $\baseSpace{\varphi} = \baseSpace{\varphi'}$ $\nu$-almost everywhere and $\varphi_y = \varphi'_y$ for $\nu$-almost every $y \in \baseSpace{\anotMeasField}$. 
		
		\item Let ${\anotMeasField}'$ be another measurable field of {\hhs}s. Given two isometric measure-preserving morphisms $\varphi \colon \aMeasField \to \anotMeasField$ and $\psi \colon \anotMeasField \to {\anotMeasField}'$, their composition $\psi \circ \varphi \colon \aMeasField \to \anotMeasField'$ is given by the tuple 
		\[
		\left( \left( \, \underline{\varphi} \circ \underline{\psi} \colon \psi^{-1} \left(\operatorname{dom}\left(\baseSpace{\varphi}\right)\right) \to \baseSpace{\aMeasField} \right),  \left( \psi_{y'} \circ \varphi_{\underline{\psi}(y')} \colon \aMeasField_{\underline{\varphi} \circ \underline{\psi}(y')} \to {\anotMeasField}'_{y'} \right)_{y' \in \psi^{-1} \left(\operatorname{dom}\left(\baseSpace{\varphi}\right)\right)} \right) \; .
		\]
		
		\item The category $\MFHHS$ has all measured fields of {\hhs}s as its objects and all isometric measure-preserving morphisms as its morphisms. It is clear that, the associations $(\aMeasField, \mu) \mapsto \left(\baseSpace{\aMeasField}, \mu\right)$ and $\varphi \mapsto \underline{\varphi}$ yields a contravariant functor from $\MFHHS$ to the category of measure spaces and measure-preserving maps. 
		
		\item An isometric measure-preserving morphism $\varphi \colon \aMeasField \to \anotMeasField$ is an \emph{isometric measure-preserving isomorphism} if there is an isometric measure-preserving morphism $\psi \colon \anotMeasField \to \aMeasField$ such that $\psi \circ \varphi$ and $\varphi \circ \psi$ are equivalent to the identity morphisms on $\aMeasField$ and $\anotMeasField$, respectively. 
	\end{enumerate}
\end{defn}

\begin{lem}\label{lem:meas-fields-isomorphism}
	An isometric measure-preserving morphism $\varphi \colon (\aMeasField, \mu) \to (\anotMeasField, \nu)$ is an isometric measure-preserving isomorphism if and only if $\baseSpace{\varphi}$ is a measure space isomorphism and for $\nu$-almost every $y \in \baseSpace{\anotMeasField}$, $\varphi_y \colon \aContField_{\baseSpace{\varphi}(y)} \to \anotContField_y$ is an isometric bijection. 
\end{lem}

\begin{proof}
	
	If $\varphi:(\aMeasField, \mu) \to (\anotMeasField, \nu)$ is an isometric measure-preserving isomorphism, then there is a measure-preserving morphism $\psi \colon (\anotMeasField, \nu) \to (\aMeasField, \mu)$. Then $\baseSpace{\psi}$ is a measure space morphism, and $\baseSpace{\varphi}$ and $\baseSpace{\psi}$ are inverses almost everywhere, so $\varphi$ is a measure space isomorphism.
	
	Moreover, for almost every $y \in \baseSpace{\anotMeasField}$, $\baseSpace{\psi} \circ \baseSpace{\varphi}(y) = y$ and
	\[\varphi_y \circ \psi_{\baseSpace{\varphi}(y)}:\anotMeasField_y = \anotMeasField_{\baseSpace{\psi} \circ \baseSpace{\varphi}(y)} \to \anotMeasField_y\]
	is the identity map.
	Also, for almost every  ${y'} \in \baseSpace{\aMeasField}$, $\baseSpace{\varphi} \circ\baseSpace{\psi} ({y'}) = {y'}$ and
	\[\psi_{y'} \circ \varphi_{\baseSpace{\psi}({y'})}:\aMeasField_{y'} = \aMeasField_{\baseSpace{\varphi} \circ \baseSpace{\psi}({y'})} \to \aMeasField_{y'},\]
	thus, for almost every $y \in \baseSpace{\anotMeasField}$,
	\[\psi_{\baseSpace{\varphi}(y)} \circ \varphi_{\baseSpace{\psi}({\baseSpace{\varphi}(y)})}:\aMeasField_{\baseSpace{\varphi}(y)}= \aMeasField_{\baseSpace{\varphi} \circ \baseSpace{\psi}({\baseSpace{\varphi}(y)})} \to \aMeasField_{\baseSpace{\varphi}(y)}\]
	is the identity map, so $\varphi_y$ is an isometric bijection, as desired.

	For the other direction, if $\varphi:(\aMeasField, \mu) \to (\anotMeasField, \nu)$ is an isometric measure-preserving morphism, $\baseSpace{\varphi}$ is a measure space isomorphism, and for $\nu$-almost every $y \in \baseSpace{\anotMeasField}$, $\varphi_y \colon \aContField_{\baseSpace{\varphi}(y)} \to \anotContField_y$ is an isometric bijection, then there is a measurable subset $\operatorname{dom}(\baseSpace{\psi}) \subset \baseSpace{\aMeasField}$ and a measure-preserving map $\baseSpace{\psi}:\operatorname{dom}\psi) \to \baseSpace{\anotMeasField}$, so that $\baseSpace{\varphi}\circ\baseSpace{\psi}$ is identity almost everywhere on $\baseSpace{\aMeasField}$ and $\baseSpace{\psi}\circ\baseSpace{\varphi}$ is identity almost everywhere on $\baseSpace{\anotMeasField}$. 
	
	Let $X_{\aMeasField,1} \subset \baseSpace{\aMeasField}$ and $X_{\anotMeasField,1} \subset \baseSpace{\anotMeasField}$ be measurable subsets of $\baseSpace{\aMeasField}$ and $\baseSpace{\anotMeasField}$ respectively with complements of measure $0$, so that $\baseSpace{\varphi}\circ\baseSpace{\psi}$ is defined and identity on $X_{\aMeasField, 1}$ and $\baseSpace{\psi}\circ\baseSpace{\varphi}$ is defined and identity on $X_{\anotMeasField, 1}$.
	
	By further restriction, we may also assume that on $X_{\anotMeasField, 1}$, $\varphi_y$ is an isometric bijection.
	
	Let
	\[X_{\aMeasField} = (\cap_n \left(\baseSpace{\varphi} \circ \baseSpace{\psi})^n X_{\aMeasField, 1}\right) \cap  (\cap_n \left(\baseSpace{\varphi} \circ \baseSpace{\psi})^n \baseSpace{\varphi}X_{\anotMeasField, 1}\right) \]
	and
	
	\[X_{\anotMeasField} = (\cap_n \left(\baseSpace{\psi} \circ \baseSpace{\varphi})^n X_{\anotMeasField, 1}\right) \cap  (\cap_n \left(\baseSpace{\psi} \circ \baseSpace{\varphi})^n \baseSpace{\psi}X_{\aMeasField, 1}\right)\]
	
	It is easy to see by checking containment in both directions that 
	$\baseSpace{\psi}(X_{\aMeasField})= X_{\anotMeasField}$
	and $\baseSpace{\varphi}(X_{\anotMeasField}) = X_{\aMeasField}$.
	
	Moreover $\baseSpace{\psi}:X_{\aMeasField} \to X_{\anotMeasField}$ and 
	$\baseSpace{\varphi}:X_{\anotMeasField} \to X_{\aMeasField}$ are mutually inverse measure-preserving maps on all of $X_{\aMeasField}$ and $X_{\anotMeasField}$.
	
	Restricting our attention to $X_{\aMeasField}$, we may define $\psi_y:\anotMeasField_{\baseSpace{\psi}(y)} \to \anotMeasField_{y}$ to be the inverse of $\varphi_{\baseSpace{\psi}(y)}: \aMeasField_y = \aMeasField_{\baseSpace{\varphi} \circ \baseSpace{\psi}(y)} \to \anotMeasField_{\baseSpace{\psi}(y)}$. 
	
	Then this $\psi$, defined on $X_\aMeasField \subset \baseSpace{\aMeasField}$ defines an inverse isometric measure-preserving morphism to $\varphi$.
\end{proof}

\begin{lem} \label{lem:meas-field-HHS-cont-maps}
	Let $\varphi = \left( \, \underline{\varphi},  \left( \varphi_y \right)_{y \in \baseSpace{\anotContField}} \right) \colon \aContField \to \anotContField$ be an isometric continuous morphism 
	between two continuous fields of {\hhs}s. 
	Let $\nu$ be a Borel measure on $\baseSpace{\anotContField}$ and let $\baseSpace{\varphi}_* \nu$ be the pushforward measure on $\baseSpace{\aContField}$. 
	Then $\varphi$ also constitutes an isometric measure-preserving morphism from $\left(\aContField_{\operatorname{meas}}, \baseSpace{\varphi}_* \nu \right)$ to $\left(\anotContField_{\operatorname{meas}} , \nu\right)$. 
\end{lem}

\begin{proof}
	Applying \Cref{rmk:cont-field-HHS-maps-factorization}, we can decompose $\varphi$ into $\dot{\varphi} \circ \baseSpace{\varphi}$,
	for $\underline{\phi}^*:\aContField \to (\underline{\phi}^* \aContField)_{\underline{\anotContField}}$ and $\dot{\phi}:(\underline{\phi}^* \aContField)_{\underline{\anotContField}} \to \anotContField$. It suffices to show that each of these induces an isometric measure-preserving morphism.

	To show that an isometric continuous morphism induces an isometric measure-preserving morphism, one needs to show that for a section $s$ of $\left(\aContField_{\operatorname{meas}}, \baseSpace{\varphi}_* \nu \right)$, the section $\phi_y(s(\underline{\phi}(y)))$ is a section of $\left(\anotContField_{\operatorname{meas}} , \nu\right)$, that is to say, that it is mutually comeasurable with every section of a generating set $\Sigma$ of $\spaceOfSections_{\operatorname{cont}} (\anotContField)$, as in \Cref{lem:cont-field-HHS-meas}.
	
	We will show this for each of the two kinds of maps in the above decomposition.
	
	In the case of $\underline{\phi}^*:\aContField \to (\underline{\phi}^* \aContField)_{\underline{\anotContField}}$, we want that for any measurable section of $\aContField_{\operatorname{meas}}$, $s \circ \underline{\phi}$ is co-measurable with any section in a generating set $\Sigma$ of $\spaceOfSections_{\operatorname{cont}} (\anotContField)$.

	Recall that $\spaceOfSections_{\operatorname{cont}} (\anotContField)$ is generated by taking sections co-continuous with all $s_1 \circ \underline{\phi}$ for $s_1 \in \spaceOfSections_{\operatorname{cont}} (\aContField)$. Thus, it suffices that for any $s_1 \in \spaceOfSections_{\operatorname{cont}} (\aContField)$, $s \circ \underline{\phi}$ is co-measurable with $s_1 \circ \underline{\phi}$. But $s$ and $s_1$ are co-measurable by definition, and $\underline{\phi}$ is continuous, so $s \circ \underline{\phi}$ and $s_1 \circ \underline{\phi}$ are co-measurable.
	
	In the case of $\dot{\phi}$, let us use the second set in the definition of \Cref{lem:cont-field-HHS-meas}. A measurable section $s$ of $(\underline{\phi}^* \aContField)_{\underline{\anotContField}}$ is a measurable map $\underline{\anotContField} \to |(\underline{\phi}^* \aContField)_{\underline{\anotContField}}|$, but by \Cref{rmk:cont-field-HHS-total-space-maps}, the map $|(\underline{\phi}^* \aContField)_{\underline{\anotContField}}| \to |\anotContField|$ is continuous, so the composed $\dot{\phi}_* \circ s:\anotContField \to |\anotContField|$ is measurable, as desired.
\end{proof}

%%%%%%%%%%%%%%%%%%%%%%%%%meas-fields%%%%%%%%%%%%%%%%%%%%%%%%%
%%%%%%%%%%%%%%%%%%%%%%%%%continuum_product%%%%%%%%%%%%%%%%%%%%%%%%%
\section{$L^2$-continuum products and variation of measures}
\label{sec:continuum_product}

We then generalize the construction of $L^2$-continuum products (\cite[Construction~3.8]{GongWuYu2021}) to the case of measurable fields of Hilbert-Hadamard spaces. 

\begin{defn}  \label{def:continuum-product}
	Let $(\aMeasField, \mu)$ be a {measured field of {\hhs}s}.  
	Fix 
	a section $s_0 \in \spaceOfSections_{\operatorname{meas}} (\aMeasField, \mu)$. 
	Then the \emph{$L^2$-continuum product of $(\aMeasField, \mu)$ based at $s_0$} is a metric space $L^2(\baseSpace{\aMeasField}, \mu, \aMeasField; s_0)$ given as follows:
	\begin{itemize}
		\item As a set, $L^2(\baseSpace{\aMeasField}, \mu, \aMeasField; s_0)$ consists of 
		elements $s \in \spaceOfSections_{\operatorname{meas}} (\aMeasField, \mu)$
		that are \emph{$L^2$-integrable with respect to $s_0$} in the sense that  
		\[
		\int_Z d_{\aMeasField_z} \left( s(z),s_0(z) \right)^2 \, d\mu(z) < \infty \; .
		\] 
		\item We equip it with the metric defined by
		\[
		d([s], [s']) = \left( \int_Z d_{\aMeasField_z} \left(s(z),s'(z) \right)^2 \, d\mu(z) \right)^{\frac{1}{2}}  \; ,
		\]
		for $s, s' \in \spaceOfSections_{\operatorname{meas}} (\aMeasField)$, 
		which is a well-defined metric thanks to the Minkowski inequality and the $L^2$-integrability condition above. 
	\end{itemize}
	
	If $\aMeasField = \aHHS_{(Z, \calB)}$, a constant measurable field of {\hhs}s with base space $(Z, \calB)$ and fibers $\aHHS$, then we typically write $L^2(Z, \mu, \aHHS; s_0)$ in place of $L^2(\baseSpace{\aMeasField}, \mu, \aMeasField; s_0)$. 
	
	If $\aContField$ is a continuous field of {\hhs}s and $\mu$ is a regular Borel measure on $\baseSpace{\aContField}$, 
	then we write $L^2(\baseSpace{\aContField}, \mu, \aContField; s_0)$ for $L^2(\baseSpace{\aContField}, \mu, \aContField_{\operatorname{meas}}; s_0)$, where $\aContField_{\operatorname{meas}}$ is the measurable field of {\hhs}s associated to $\aContField$ in the sense of  \Cref{lem:cont-field-HHS-meas} and $s_0$ is a fixed measurable section. 
	
	Observe that when $\mu$ is the zero measure, then $L^2(\baseSpace{\aMeasField}, \mu, \aMeasField; s_0)$ is a singleton. 
\end{defn}

\begin{lem} \label{lem:continuum-product-HHS}
	With $(\aMeasField, \mu)$ and $s_0$ as in \Cref{def:continuum-product}, the $L^2$-continuum product $L^2(\baseSpace{\aMeasField}, \mu, \aMeasField; s_0)$ is a {\hhs}.
\end{lem}

\begin{proof}
	The elementary but somewhat technical proof is almost identical to that of \cite[Proposition~9.3]{GongWuYu2021}. 
\end{proof}

\begin{lem} \label{lem:continuum-product-maps}
	Let $\varphi \colon (\aMeasField, \mu) \to (\anotMeasField, \nu)$ be an isometric measure-preserving morphism (respectively, isomorphism) between measured fields of {\hhs}s as in \Cref{defn:meas-fields-morphism}. 
	Let 
	$s_0 \in \spaceOfSections_{\operatorname{meas}} (\aMeasField)$. 
	Then the formula $[s] \mapsto [\varphi(s)]$ defines an isometric embedding (respectively, bijection) $\varphi_* \colon L^2 \left(\baseSpace{\aMeasField}, \mu, \aMeasField; s_0 \right) \to L^2 \left(\baseSpace{\anotMeasField}, \nu, \anotMeasField; \varphi (s_0) \right)$. 
\end{lem}

\begin{proof}
	To show that $\varphi_*$ is an isometric embedding, we first must see that if $s \in L^2 \left(\baseSpace{\aMeasField}, \mu, \aMeasField; s_0 \right)$ then $[\varphi(s)] \in L^2 \left(\baseSpace{\anotMeasField}, \nu, \anotMeasField; \varphi (s_0) \right)$.
	
	That is, we wish to see that if 
	\[\int_{\baseSpace{\aMeasField}}(d_{\aMeasField_z}(s(z), s_0(z))^2 d\mu(z) <\infty, \]
	then
	\[\int_{\baseSpace{\anotMeasField}}(d_{\anotMeasField_z}(\varphi(s)(z), \varphi(s_0)(z))^2 d\nu(z) <\infty. \]
	But the latter is by definition equal to
	\[ \int_{\baseSpace{\anotMeasField}}(d_{\anotMeasField_z}(\varphi_z(s(\baseSpace{\varphi}(z))), \varphi_z(s_0(\baseSpace{\varphi}(z))))^2 d\nu(z), \]
	and since $\varphi_z$ is an isometric embedding, it is equal to
	\[\int_{\baseSpace{\anotMeasField}}(d_{\anotMeasField_z}(s(\baseSpace{\varphi}(z)), s_0(\baseSpace{\varphi}(z)))^2 d\nu(z)\]
	which, because $\underline{\varphi}$ is measure preserving,  agrees with 
	\[\int_{\baseSpace{\aMeasField}}(d_{\aMeasField_z}(s(z), s_0(z))^2 d\mu(z).\]
	The same argument, replacing $s_0$ with $s'$,  shows that $\varphi_*$ is an isometric embedding.
	Because the construction of $\varphi_*$ respects composition it is easy to see that if $\varphi$ is an isometric measure-preserving isomorphism, then $\varphi_*$ is a bijection.
\end{proof}

Our goal is to apply this construction to measured fields of {\hhs}s that arise from continuous fields together with Borel measures on their base spaces (following \Cref{def:cont-field-HHS-meas}). 
It will suffice for our purposes to restrict to compact base spaces and finite measures. 
This simplifies the construction thanks to the following simple observation. 

\begin{lem} \label{lem:continuous-field-HHS-L2-integrability}
	Let $\aContField$ be a continuous field of {\hhs}s. Let $\mu$ be a compactly-supported regular Borel measure on $\baseSpace{\aContField}$. 
	Then for any two continuous sections $s_0, s_1 \in \spaceOfSections_{\operatorname{cont}} (\aContField)$, we have 
	\[
	\int_Z d_{\aContField_z} \left(s_0(z),s_1(z) \right)^2 \, d\mu(z) < \infty \; .
	\]
	As a result, we have $L^2(\baseSpace{\aContField}, \mu, \aContField; s_0) = L^2(\baseSpace{\aContField}, \mu, \aContField ; s_1)$.
\end{lem}

\begin{proof}
	The first claim follows from \Cref{def:cont-field-HHS}\eqref{def:cont-field-HHS::continuous} and the fact that continuous functions on a compact space are bounded. The second claim follows from the Minkowski inequality. 
\end{proof}

\begin{defn} \label{defn:continuous-field-HHS-L2-notations}
	Under the conditions of \Cref{lem:continuous-field-HHS-L2-integrability}, we write $L^2(\baseSpace{\aContField}, \mu, \aContField)$ or $\aContField_{\mu}$ for $L^2(\baseSpace{\aContField}, \mu, \aContField; s_0)$ where $s_0$ is any continuous section. 
	
	Furthermore, in the special case where $\aContField = (X)_Z$ is a constant continuous field of {\hhs}s and $\mu$ is a compactly-supported regular Borel measure on $Z$, we write $L^2(Z, \mu, X)$ for $L^2(\baseSpace{\aContField}, \mu, \aContField)$. 
\end{defn}

\begin{lem} \label{lem:continuous-field-HHS-L2-dense}
	Under the conditions of \Cref{lem:continuous-field-HHS-L2-integrability}, if $\aContField_z$ is separable for $\mu$-almost every $z \in \baseSpace{\aContField}$ and $\baseSpace{\aContField}$ is 
	second-countable, 
	then the canonical map $\spaceOfSections_{\operatorname{cont}} (\aContField) \to L^2(\baseSpace{\aContField}, \mu, \aContField)$ 
	has a dense image. 
\end{lem}

\begin{proof}
	
	Let $Z = \baseSpace{\aContField}$. 
	
	We wish to show that for any $\epsilon>0$ and any section $s$ in $L^2(\baseSpace{\aContField}, \mu, \aContField)$, there is a section $s' \in \spaceOfSections_{\operatorname{cont}} (\aContField)$ such that 
	\[\int_Z d_{\aContField_z} \left(s(z),s'(z) \right)^2 \, d\mu(z) < \epsilon.\]
	
	Let $s_0$ be an arbitrary continuous section, so that $L^2(\baseSpace{\aContField}, \mu, \aContField) = L^2(\baseSpace{\aContField}, \mu, \aContField; s_0)$.

	For any positive real number $M >100$, let $s_M \in L^2(\baseSpace{\aContField}, \mu, \aContField)$ be the section defined so that 
	\[s_M(z) = \begin{cases} s(z) & d_{\aContField_z}(s_0(z),s(z))<M \\
	s_0(z) & \text{else} \end{cases}	\]
	and 
	\[\lim_{M \to \infty} \int_{z \text{ with } d_{\aContField_z}(s(z),s_0(z))<M} d_{\aContField_z}(s(z),s_0(z))=0.\]
	
	Thus, we may choose $M>100$ and $M > \mu(Z)$ such that 
	\[\int_Z d_{\aContField_z} \left(s(z),s_M(z) \right)^2 \, d\mu(z)<\frac{\epsilon}{4}.\]
	
	It now suffices to show that for any $\epsilon>0$,  there is continuous section $s'$ such that 
	\[\int_Z d_{\aContField_z} \left(s_M(z),s'(z) \right)^2 \, d\mu(z)<\frac{\epsilon}{4},\]
	where $s_M$ is a section in $L^2(\baseSpace{\aContField}, \mu, \aContField)$ such that  $d_{\aContField_z}(s_M(z), s_0(z))<M$ for all $z\in Z$.

	Note that if we regard the section $s_M$ as a measurable map $s_M\colon \baseSpace{\aContField} \to |\aContField|$, by the general form of Luzin's theorem, for any $\epsilon_1>0$ there is a measurable closed set $E$ with $\mu(Z \backslash E)<\epsilon_1$ such that $s_M|_{E}$ is continuous.
	
	Note that by the pointwise density of continuous sections, for any $\epsilon_1>0$, for any $z_0 \in E$, there is a section $s_{z_0}$ such that $d_{\aContField_z}(s_{z_0}(z_0), s_M(z_0))<\frac{\epsilon_1}{4}$.
	
	Moreover, since $s_M|_E$ and $s_{z_0}$ are both continuous as maps $E \to |\aContField|$, the map
	\[z \mapsto d_{\aContField_z}(s_{z_0}(z), s_M(z))\]
	is a continuous map from a compact metric space to $\mathbb{R}$, it is uniformly continuous, so there is $\delta_1>0$ such that if $z \in E$ and $d_Z(z, z_0)<\delta_1$, then 
	\[|d_{\aContField_z}(s_{z_0}(z), s_M(z))-d_{\aContField_{z_0}}(s_{z_0}(z_0), s_M(z_0))|<\frac{\epsilon_1}{4}.\]

	Thus, for any $\epsilon_1>0$, for any $z_0 \in E$, there is an open set $V_{z_0}$ containing $z_0$ and a continuous section $s_{z_0}$ such that for any $z \in V_{z_0} \cap E$, we have
	\[d_{\aContField_z}(s_{z_0}(z), s_M(z)) < \epsilon_1.\]
	
	Then for any $\epsilon_1>0$, consider the above $V_{z_0}$ for all $z_0 \in E$, and $U_0 = Z \backslash E$. These open sets form an open cover of $Z$. Because $Z$ is a compact metric space, it has a finite subcover, denoted by $\{U_0, U_1, \ldots U_n\}$. Let us denote the corresponding continuous sections to $U_1, \ldots U_n$ as above  by $s_1, s_2, \ldots s_n$. Note that for $i \geq 1$ and any $z \in U_i$, we have  $d_{\aContField_z}(s_i(z), s_M(z))< \epsilon_1$.
	
	Let $\tilde s$ be the continuous section defined from the above $s_i$ for $i =1,\ldots n$, and our chosen arbitrary section $s_0$ earlier using a partition of unity subordinate to the cover $U_0,U_1,\ldots U_n$ of $Z$. (Partitions of unity can define a section by an iterated application of \Cref{cor:cont-field-HHS-geodesic-approximate}.)
	
	For any $z \in Z$,  from the fact that for any $U_i \ni z$, $d_{\aContField_z}(s_i(z), s_M(z))< \epsilon_1$, and the fact that $\aContField_z$ is a Hilbert Hadamard space, and in particular a CAT(0) space, it is straightforward to see that 
	\[d_{\aContField_z}(\tilde s(z), s_M(z))< \epsilon_1\]
	since $\tilde s(z)$ is a geodesic combination of $\{s_i(z) \mid U_i \ni z\}$. 
	
	Thus, we have that for any $z \in E$, 
	\[d_{\aContField_z}(\tilde s(z), s_M(z))< \epsilon_1.\]
	Then 
	\begin{align*}
	& \int_Z d_{\aContField_z} \left(s_M(z),\tilde s(z) \right)^2 \, d\mu(z) \\ &= \int_E d_{\aContField_z} \left(s_M(z),\tilde s(z) \right)^2 \, d\mu(z) +\int_{U_0} d_{\aContField_z} \left(s_M(z),\tilde s(z) \right)^2 \, d\mu(z)\\
	& \leq \epsilon_1^2 \mu(E) + M^2\mu(Z \backslash E) \\
	& \leq \epsilon_1^2 M + M^2\epsilon_1
	\end{align*}
	
	Now, by choosing $\epsilon_1 < \epsilon \cdot \min{\{\epsilon,1\}} (10M^4)^{-1}$, we have 
	\[\int_Z d_{\aContField_z} \left(s_M(z),\tilde s(z) \right)^2 \, d\mu(z) < \frac{\epsilon}{4},\]
	which gives the desired $s'$. 
\end{proof}

%%%%%%%%%%%%%%%%%%%%%%%%%%%%%%%%%%%%%%%%%%%%%%%%%%%%%%%
\begin{defn}  \label{def:2nd-countable}
	We say a continuous field $\aContField$ of {\hhs}s is \emph{second-countable} if the base space $\baseSpace{\aContField}$ is second-countable and for any $z \in \baseSpace{\aContField}$, the fiber $\aContField_z$ is second-countable (or equivalently for metric spaces, separable). 
\end{defn}
%%%%%%%%%%%%%%%%%%%%%%%%%%%%%%%%%%%%%%%%%%%%%%%%%%%%%%%

We are now almost ready to introduce another crucial construction for this paper, in which we start from a continuous field $\aContField$ of {\hhs}s and form another continuous field of {\hhs}s by taking
a family of $L^2$-continuum products over $\aContField$ using a continuously-varying family of finite measures on $\baseSpace{\aContField}$ (see \Cref{def:continuum-product-field}). 
In the following, 
let us denote by $\spaceMeas(Z)$ the set of all compactly-supported regular (and thus finite) Borel measures 
on a compact Hausdorff space $Z$, equipped with the topology of weak convergence, that is, the coarsest topology such that for any continuous function $f \colon Z \to \mathbb{R}$, the map $\spaceMeas(Z) \to \mathbb{R}$ defined by $\mu \mapsto \int_Z f \, \mathrm{d}\mu$ is continuous. 

\begin{lem}  \label{lem:continuum-product-field}
	Let $Z$ be a locally compact paracompact Hausdorff space, let $\aContField$ be a second-countable {continuous field of {\hhs}s}, 
	and let $f \colon Z \to \spaceMeas(\baseSpace{\aContField})$ be a continuous map. Consider the tuple 
	\[
	\left( L^2(\baseSpace{\aContField}, f(z), \aContField) \right)_{z \in Z}
	\] 
	of {\hhs}s and the diagonal map 
	\[
	\spaceOfSections_{\operatorname{cont}} (\aContField) \hookrightarrow \prod_{z \in Z} L^2(\baseSpace{\aContField}, f(z), \aContField) \, , \quad s \mapsto \left( s \right)_{z \in Z} \; .
	\]
	Then the image of this map satisfies conditions~\eqref{def:cont-field-HHS::convex}-\eqref{def:cont-field-HHS::continuous} in \Cref{def:cont-field-HHS}. 
\end{lem}

\begin{proof}

	We check each of the conditions in order. 
	
	For condition \eqref{def:cont-field-HHS::convex}, convexity: it is immediate that the images of the midpoint section of $s_1$ and $s_2$ is the midpoint section of the images of $s_1$ and $s_2$, which gives us convexity.
	
	For condition \eqref{def:cont-field-HHS::dense}, density: this is the statement that for each $z$, the image of $\spaceOfSections_{\operatorname{cont}}$ is dense in $L^2(\baseSpace{\aContField}, f(z), \aContField)$. This is exactly the statement of \Cref{lem:continuous-field-HHS-L2-dense}
	
	It remains to show condition \eqref{def:cont-field-HHS::continuous}, mutual co-continuity. For this, we wish to show that for continuous sections $s_1, s_2 \in \spaceOfSections_{\operatorname{cont}}$, the map
	\[z \mapsto \left(\int_{\baseSpace{\aContField}} d_{\aContField_x}(s_1(x),s_2(x))^2 \, \mathrm{d}(f(z))(x)\right)^{1/2}\]
	is continuous, which is to say that for a net $z_\alpha$ that converges to $z$ we need that $\int_{\baseSpace{\aContField}} d_{\aContField_x}(s_1(x),s_2(x))^2 \, \mathrm{d} (f(z_\alpha))(x)$ converges to $\int_{\baseSpace{\aContField}} d_{\aContField_x}(s_1(x),s_2(x))^2 \, \mathrm{d}(f(z))(x)$. 
	
	Because $f$ is continuous, we know that $f(z_\alpha)$ converge to $f(z)$ in the topology of weak convergence. Then, because $x \mapsto d_{\aContField_x}(s_1(x),s_2(x))$ is continuous, by the definition of the topology of weak convergence, we get that 
	\[z \mapsto \left(\int_{\baseSpace{\aContField}} d_{\aContField_x}(s_1(x),s_2(x))^2 d(f(z))(x)\right)^{1/2}\]
	is continuous, as desired.
\end{proof}

In view of \Cref{lem:cont-field-HHS-generate}, this allows us to make the following definition. 

\begin{defn}  \label{def:continuum-product-field}
	Let $Z$, $\aContField$, and $f$ be as in \Cref{lem:continuum-product-field}. 
	Then the \emph{continuous field $f^* \aContField$ of {\hhs}s over $Z$ induced from $\aContField$ by $f$} 
	is a continuous field of {\hhs}s generated (as in \Cref{def:cont-field-HHS-generate}) by the image of the diagonal map 
	\[
	\spaceOfSections_{\operatorname{cont}} (\aContField) \to \prod_{z \in Z} L^2(\baseSpace{\aContField}, f(z), \aContField) \, , \quad s \mapsto \left( s \right)_{z \in Z} \; .
	\]
	We also write $f^* \colon \spaceOfSections_{\operatorname{cont}} (\aContField) \to \spaceOfSections_{\operatorname{cont}} (f^* \aContField)$ for the resulting map. 
	
	If $Z$ is a subspace of $\spaceMeas(\baseSpace{\aContField})$, then the \emph{continuous field of $L^2$-continuum products of $\aContField$ over $Z$}, denoted by $\aContField|_Z$, is 
	the continuous field of {\hhs}s over $Z$ induced from $\aContField$ by the 
	inclusion map $Z \hookrightarrow \spaceMeas(\baseSpace{\aContField})$. 
\end{defn}

Observe that the notation $f^* \aContField$ above is compatible with the one in \Cref{def:cont-field-HHS-maps}\eqref{def:cont-field-HHS-maps:induce} when we view $\baseSpace{\aContField}$ as a closed subspace of $\spaceMeas(\baseSpace{\aContField})$ by identifying each $y \in \baseSpace{\aContField}$ with the point mass $\delta_y$ at $y$, 
since $L^2(\baseSpace{\aContField}, \delta_y, \aContField)$ is naturally identified with $\aContField_y$. 
In particular, the notation $\aContField|_Z$ is compatible with the one in \Cref{def:cont-field-HHS-maps}\eqref{def:cont-field-HHS-maps:restrict}.  

%%%%%%%%%%%%%%%%%%%%%%%%%>continuum-product-morphisms<%%%%%%%%%%%%%%%%%%%%%%%%%

We now discuss the functoriality of this construction, starting with a simple observation.

\begin{rmk}  \label{rmk:continuum-product-field-compose}
	Let $Z$, $\aContField$, and $f$ be as in \Cref{lem:continuum-product-field}. 
	Let $Y$ be another compact Hausdorff space and let $h \colon Y \to Z$ be a continuous map. 
	Then there is a canonical isometric continuous isomorphism between $h^* \left( f^* \aContField \right)$ and $(f \circ h)^* \aContField$ that intertwines the maps $h^* \circ f^* \colon \spaceOfSections_{\operatorname{cont}} (\aContField) \to \spaceOfSections_{\operatorname{cont}} \left( h^* \left( f^* \aContField \right) \right)$ and $(f \circ h)^*  \colon \spaceOfSections_{\operatorname{cont}} (\aContField) \to \spaceOfSections_{\operatorname{cont}} \left( (f \circ h)^* \aContField \right)$, 
	since the following diagram 
	\[
	\xymatrix{
		&& \left( s \right)_{z \in Z} \mathrlap{\quad \in} & \displaystyle \prod_{z \in Z} L^2(\baseSpace{\aContField}, f(z), \aContField) \ar[dd] & \mathllap{\ni \quad } s' \ar@{|->}[dd] \\
		s \quad  \ar@{|->}[rru] \ar@{|->}[rrd] & \mathllap{\in \quad }  \quad 	\spaceOfSections_{\operatorname{cont}} (\aContField)  \ar[rru] \ar[rrd] \\
		&& \left( s \right)_{y \in Y} \mathrlap{\quad \in} & \displaystyle \prod_{y \in Y} L^2(\baseSpace{\aContField}, f(h(y)), \aContField) & \mathllap{\ni \quad } s' \circ h
	}
	\]
	is obviously commutative. 
\end{rmk}

On the other hand, we can also obtain functoriality results by changing the continuous field $\aContField$ in \Cref{def:continuum-product-field}. There is quite a bit of flexibility in doing this, which we will exploit in the proof of \Cref{prop:quasitrivial-locally-trivial}.

\begin{lem} \label{lem:continuum-product-field-maps}
	Let $\aContField$ and $\anotContField$ be second-countable continuous fields of {\hhs}s and write $\aContField_{\operatorname{meas}}$ and $\anotContField_{\operatorname{meas}}$ for the associated Borel measurable fields of {\hhs}s. 
	Let $Z$ and $Y$ be locally compact paracompact Hausdorff spaces. Let $f \colon Z \to \spaceMeas(\baseSpace{\aContField})$ and $g \colon Y \to \spaceMeas(\baseSpace{\anotContField})$ be continuous maps. 
	
	Let $h \colon Y \to Z$ be a continuous map. For any $y \in Y$, let 
	\[
	\varphi^{y} \colon \left( \aContField_{\operatorname{meas}}, f(h(y)) \right) \to \left( \anotContField_{\operatorname{meas}}, g(y) \right)
	\]
	be an isometric measure-preserving morphism,  
	and let 
	\[
	\varphi^{y}_* \colon  L^2(\baseSpace{\aContField}, f(h(y)), \aContField) \to L^2(\baseSpace{\anotContField}, g(y), \anotContField)
	\]
	be the induced isometric embedding as in \Cref{lem:continuum-product-maps}. 
	
	Suppose, for any $s \in \spaceOfSections_{\operatorname{cont}} (\aContField)$, 
	the section $\left( \varphi^y_* \left( s  \right) \right)_{y \in Y} \allowbreak \in \prod_{y \in Y} L^2(\baseSpace{\anotContField}, g(y), \anotContField)$ 
	is co-continuous with any $s' \in \spaceOfSections_{\operatorname{cont}} (\anotContField) \hookrightarrow \prod_{y \in Y} L^2(\baseSpace{\anotContField}, g(y), \anotContField)$. 	
	Then the tuple $\varphi := \left( h, \left( \varphi^{y}_* \right)_{y \in Y} \right)$ forms an isometric continuous morphism from $f^* \aContField$ to $g^* \anotContField$ in the sense of \Cref{defn:cont-fields-morphism}
\end{lem}

\begin{proof}
	For the ease of understanding, we organize the various elements in the lemma into the following diagram, where wavy arrows stands for "induces":
	\[
	\begin{tikzcd}[scale=1em]
	&& {s \in \spaceOfSections_{\operatorname{cont}} (\aContField)} & {\spaceOfSections_{\operatorname{cont}} (f^*\aContField)} & {\spaceOfSections_{\operatorname{cont}} (h^*(f^*\aContField))} \\
	Z & {M(\baseSpace{\aContField})} & \aContField & {\left( \aContField_{\operatorname{meas}} , f(h(y)) \right)} & {\displaystyle \prod_{y \in Y} L^2(\baseSpace{\aContField}, f(h(y)), \aContField)} \\
	&&& {} & {} \\
	Y & {M(\baseSpace{\anotContField})} & \anotContField & {\left( \anotContField_{\operatorname{meas}} , g(y) \right)} & {\displaystyle \prod_{y \in Y} L^2(\baseSpace{\anotContField}, g(y), \anotContField)} \\
	&& {s' \in \spaceOfSections_{\operatorname{cont}} (\anotContField)} & {\spaceOfSections_{\operatorname{cont}} (g^* \anotContField)}
	\arrow["f", from=2-1, to=2-2]
	\arrow["h", from=4-1, to=2-1]
	\arrow["g"', from=4-1, to=4-2]
	\arrow["{\varphi^y}"', from=2-4, to=4-4]
	\arrow[squiggly, from=3-4, to=3-5]
	\arrow["{\displaystyle \prod_{y \in Y} \varphi^{y}_*}", from=2-5, to=4-5]
	\arrow[squiggly, from=2-3, to=2-4]
	\arrow[squiggly, from=2-3, to=2-2]
	\arrow[squiggly, from=4-3, to=4-4]
	\arrow[squiggly, from=4-3, to=4-2]
	\arrow[squiggly, from=2-3, to=1-3]
	\arrow[squiggly, from=4-3, to=5-3]
	\arrow["{f^*}", from=1-3, to=1-4]
	\arrow["{g^*}"', from=5-3, to=5-4]
	\arrow[hook, from=5-4, to=4-5]
	\arrow["{h^*}", from=1-4, to=1-5]
	\arrow[hook, from=1-5, to=2-5]
	\end{tikzcd}\]
	By \Cref{defn:cont-fields-morphism}, \Cref{lem:cont-fields-HHS-map-generators} and \Cref{def:continuum-product-field}, it suffices to verify that 
	for any $s \in \spaceOfSections_{\operatorname{cont}} (\aContField)$ and any $s' \in \spaceOfSections_{\operatorname{cont}} (\anotContField)$, we have $\left( \varphi_y \left( \left( f^* ( s)\right) \left( \baseSpace{\varphi}(y) \right) \right) \right)_{y \in \baseSpace{\anotContField}}$ and $g^*(s')$ are co-continuous in $\prod_{y \in Y} L^2(\baseSpace{\anotContField}, g(y), \anotContField)$, but this is exactly what our assumption guarantees. 
\end{proof}

Here is the special case of isomorphisms. 

\begin{cor} \label{cor:continuum-product-field-isomorphism}
	Let $\aContField$, $\anotContField$, $Z$, $Y$, $f$, $g$, $h$, and $\left( \varphi^{y} \right)_{y \in Y}$ be as in \Cref{lem:continuum-product-field-maps}. 
	Suppose, in addition, we have 
	\begin{itemize}
		\item $h$ is homeomorphic, and 
		\item for any $y \in Y$, $\baseSpace{\varphi^{y}}$ yields a measure space isomorphism from $\left( \baseSpace{\anotContField }, g(y) \right)$ to $\left( \baseSpace{\aContField }, f(h(y)) \right)$, 
		and for $g(y)$-almost every $w \in \baseSpace{\anotContField}$, the map 
		\[
		\left( \varphi^{y} \right)_w \colon \aContField_{\baseSpace{\varphi^{y}} (w)} \to \anotContField_{w}
		\]
		is surjective.
	\end{itemize}
	Then the tuple $\left( h, \left( \varphi^{y}_* \right)_{y \in Y} \right)$ in \Cref{lem:continuum-product-field-maps} forms an isometric continuous isomorphism. 
\end{cor}

\begin{proof}
	This follows directly from \Cref{lem:cont-fields-isomorphism}, \Cref{lem:meas-fields-isomorphism} and \Cref{lem:continuum-product-maps}. 
\end{proof}

We also detail a special case of \Cref{lem:continuum-product-field-maps}, where the family $\left( \varphi^{y}_* \right)_{y \in Y}$ comes from a single isometric continuous morphism. 

\begin{cor} \label{cor:continuum-product-field-isomorphism-easy}
	Let $\aContField$, $\anotContField$, $Z$, $Y$, $f$, $g$, and $h$ be as in \Cref{lem:continuum-product-field-maps}. 
	Let $\varphi \colon \aContField \to \anotContField$ be an isometric continuous morphism such that the diagram
	\[\begin{tikzcd}
	Z & {M(\baseSpace{\aContField})} \\
	Y & {M(\baseSpace{\anotContField})}
	\arrow["h", from=2-1, to=1-1]
	\arrow["f", from=1-1, to=1-2]
	\arrow["g"', from=2-1, to=2-2]
	\arrow["{\baseSpace{\varphi}_*}"', from=2-2, to=1-2]
	\end{tikzcd}\]
	commutes, i.e., $f(h(y)) = \baseSpace{\varphi}_* (g(y))$ for any $y \in Y$. 
	In view of \Cref{lem:meas-field-HHS-cont-maps}, let 
	\[
	\varphi_* \colon  L^2(\baseSpace{\aContField}, f(h(y)), \aContField) \to L^2(\baseSpace{\anotContField}, g(y), \anotContField)
	\]
	be the induced isometric embedding as in \Cref{lem:continuum-product-maps}. 
	Then the tuple $\varphi := \left( h, \left( \varphi_* \right)_{y \in Y} \right)$ forms an isometric continuous morphism from $f^* \aContField$ to $g^* \anotContField$. 
	
	Moreover, if $h$ is a homeomorphism and $\varphi$ is an isometric continuous isomorphism, then $\varphi := \left( h, \left( \varphi_* \right)_{y \in Y} \right)$ is also an isometric continuous isomorphism. 
\end{cor}

\begin{proof}
	These statements follow direct from \Cref{lem:continuum-product-field-maps} and \Cref{cor:continuum-product-field-isomorphism}. 
\end{proof}

This allows us to construct group actions on induced continuous fields of {\hhs}s. 

\begin{cor} \label{cor:continuum-product-field-actions}
	Let $Z$, $\aContField$, and $f$ be as in \Cref{lem:continuum-product-field}. Let $\alpha \colon \Gamma \curvearrowright \aContField$ be an isometric (left) action by a discrete group and 
	recall from \Cref{defn:isometric-action} that $\baseSpace{\alpha}$ is the induced right action of $\Gamma$ on $\baseSpace{\aContField}$. We also let $\baseSpace{\alpha}$ denote the induced right action of $\Gamma$ on $\spaceMeas \left( \baseSpace{\aContField} \right)$. 
	Let $\beta \colon \Gamma \curvearrowright Z$ be a right action by homeomorphisms. Suppose $f \colon Z \to \spaceMeas \left( \baseSpace{\aContField} \right)$ is $\Gamma$-equivariant. 
	Then we have an isometric action $f^*_\beta \alpha \colon \Gamma \curvearrowright f^* \aContField$ such that for any $\gamma \in \Gamma$, we have 
	$\left( f^*_\beta \alpha  \right)_\gamma = \left( \beta_\gamma, \left( \left( \alpha_\gamma \right)_* \right)_{y \in Y} \right)$. 
\end{cor}

\begin{proof}
	The fact that $\left( \beta_\gamma, \left( \left( \alpha_\gamma \right)_* \right)_{y \in Y} \right)$ defines an isometric continuous isomorphism follows directly from \Cref{cor:continuum-product-field-isomorphism-easy}. It is then routine to verify that $\left( f^*_\beta \alpha  \right)_\gamma \circ \left( f^*_\beta \alpha  \right)_{\gamma'} =  \left( f^*_\beta \alpha  \right)_{\gamma \gamma'} $ for any $\gamma, \gamma' \in \Gamma$ and that $\left( f^*_\beta \alpha  \right)_1$ is the identity  isometric continuous isomorphism on $f^*\aContField$. 
\end{proof}

Let us specialize to constant fields.

\begin{lem} \label{lem:continuum-product-field-trivial-maps}
	Let $\aHHS$ be a separable {\hhs}. 
	Let $Z$, $Y$, $Z'$ and $Y'$ be 
	compact 
	Hausdorff spaces with $Z'$ and $Y'$ also second-countable. Let $f \colon Z \to \spaceMeas(Z')$ and $g \colon Y \to \spaceMeas(Y')$ be continuous maps. 
	
	Let $h \colon Y \to Z$ be a continuous map. For any $y \in Y$, 
	$k^{y} \colon (Y', g(y)) \to (Z', f(h(y)))$ be a measure-preserving map, 
	and let 
	\[
	\left(k^{y}\right)^* \colon  L^2(Z', f(h(y)), \aHHS) \to L^2(Y', g(y), \aHHS)
	\]
	be the induced isometric embedding. 
	
	If, 
	for any functions $\xi \in C(Z')$ and $\eta \in C(Y')$, the map  
	\[
	Y \ni y \mapsto \int_{Y'} 
	\left( \xi \circ k^{y} \right) \cdot \eta
	\, \operatorname{d} g(y) \in [0, \infty)
	\]
	is continuous, 	
	then the tuple $\left( h, \left( \left(k^{y}\right)^* \right)_{y \in Y} \right)$ forms an isometric continuous morphism from $f^* \left( \aHHS_{Z'} \right)$ to $g^* \left( \aHHS_{Y'} \right)$ in the sense of \Cref{def:cont-field-HHS-maps}. 
\end{lem}

\begin{proof}
	Applying \Cref{lem:continuum-product-field-maps} to the constant continuous fields $\aHHS_{Z'}$ and $\aHHS_{Y'}$, we just need to fix arbitrary continuous maps $s \in C(Z', \aHHS)$ and $s' \in C(Y', \aHHS)$, and show the section $\left(  s \circ k^{y} \right)_{y \in Y} \in \prod_{y \in Y} L^2(Y', g(y), \aHHS)$ is co-continuous with $s'$, 
	i.e., the function 
	\[
	\delta \colon Y \to [0, \infty) \, , \quad y \mapsto \int_{y' \in Y'} d_{\aHHS} \left( s \circ k^{y} (y') , s'(y') \right)\, \operatorname{d} g(y) (y') 
	\]
	is continuous. To prove this, it suffices to show $\delta$ can be approximated by continuous real-valued functions on $Y$ in the uniform norm. To this end, we fix $\varepsilon > 0$. 
	For any $y \in Y$, we write $g(y) (Y') = \int_{Y'} 1 \, \operatorname{d} g(y)$, the total volume of the measure $g(y)$. 
	Define 
	\[
	\lambda := \max_{y \in Y} g(y) (Y') \; ,
	\]
	which exists since $g(Y)$ is a compact subset of $\spaceMeas(Y')$. 
	By the compactness of $Z'$ (respectively, $Y'$), there is a finite open cover of the compact set $s(Z')$ (respectively, $s'(Y')$) in $\aHHS$ consisting of open balls $B \left( x_i , \frac{\varepsilon}{2 \lambda} \right)$ for $i = 1, \ldots, m$ (respectively, $B \left( x'_j , \frac{\varepsilon}{2 \lambda} \right)$ for $j = 1, \ldots, n$). Since the collection 
	\begin{align*}
	& \left\{ s^{-1} \left( B \left( x_i , \frac{\varepsilon}{2 \lambda} \right) \right) \colon i = 1, \ldots, m \right\} \\
	\text{(respectively, } & \left\{ (s')^{-1} \left( B \left( x'_j , \frac{\varepsilon}{2 \lambda} \right) \right) \colon j = 1, \ldots, n \right\} \text{)}
	\end{align*}
	is a finite open cover of $Z'$ (respectively, $Y'$), it thus begets a partition of unity $\left\{ \xi_i \colon i = 1, \ldots, m \right\}$ (respectively, $\left\{ \eta_j \colon j = 1, \ldots, n \right\}$) subordinate to it. 
	For $i = 1, \ldots, m$ and $j = 1, \ldots, n$, 
	we write $d_{i j} := d_{\aHHS} \left( x_i , x'_j \right)$ and define functions 
	\[
	\zeta_{i j} \colon Y \to [0, \infty) \, , \quad  y \mapsto \int_{Y'} 
	\left( \xi_i \circ k^{y} \right) \cdot \eta_j
	\, \operatorname{d} g(y) \; ,
	\]
	which are continuous by assumption. 
	Observe that for any $y' \in Y'$, we have, for any $i \in \{1, \ldots, m\}$ and $j \in \{1, \ldots, n\}$,  
	\[
	\left| d_{\aHHS} \left( s \circ k^{y} (y') , s'(y') \right) -  d_{\aHHS} \left( x_i , x'_j \right) \right| 
	\leq d_{\aHHS} \left( s \circ k^{y} (y') , x_i  \right) +  d_{\aHHS} \left( s'(y'), x'_j \right) \; ,
	\]
	which is less than $\frac{\varepsilon}{\lambda}$ whenever $k^{y}(y') \in \operatorname{supp} \left( \xi_i \right)$ and $y' \in \operatorname{supp} \left( \eta_j \right)$, 
	whence 
	\begin{align*}
	& \Big| d_{\aHHS} \left( s \circ k^{y} (y') , s'(y') \right) - \sum_{i=1}^{m} \sum_{j=1}^{n} \left( \xi_i \circ k^{y} \right) (y') \cdot  \eta_j (y') \cdot  d_{i j} \Big| \\ 
	=&\ \Big| \sum_{i=1}^{m} \sum_{j=1}^{n} \xi_i \left( k^{y} (y') \right)  \cdot  \eta_j (y')   \cdot  \left( d_{\aHHS} \left( s \circ k^{y} (y') , s'(y') \right) -  d_{\aHHS} \left( x_i , x'_j \right) \right) \Big| \\
	<&\ \frac{\varepsilon}{\lambda} \; .
	\end{align*}
	It follows that for any $y \in Y$, we have 
	\[
	\Big| \delta (y) - \sum_{i=1}^{m} \sum_{j=1}^{n} d_{i j} \cdot \zeta_{i j} (y) \Big| \leq \frac{\varepsilon}{\lambda}  \cdot g(y) (Y') \leq \varepsilon \; . 
	\]
	Therefore $\delta$ is within $\varepsilon$-distance in the uniform norm from the continuous function $\sum_{i=1}^{m} \sum_{j=1}^{n} d_{i j} \cdot \zeta_{i j}$, as desired. 
\end{proof}

Parallel to the above, we restrict \Cref{lem:continuum-product-field-trivial-maps} to the case of isomorphisms. 

\begin{cor} \label{cor:continuum-product-field-trivial-isomorphism}
	Let $\aHHS$, $Z$, $Y$, $Z'$, $Y'$, $f \colon Z \to \spaceMeas(Z')$, $g \colon Y \to \spaceMeas(Y')$, $h \colon Y \to Z$, and $\left( k^{y} \right)_{y \in Y}$ be as in \Cref{lem:continuum-product-field-trivial-maps}. 
	Suppose, in addition, we have 
	\begin{itemize}
		\item $h$ is homeomorphic, and 
		\item for any $y \in Y$, $k^{y}$ yields a measure space isomorphism from $\left( Y' , g(y) \right)$ to $\left( Z', f(h(y)) \right)$. 
	\end{itemize}
	Then the tuple $\left( h, \left( \left(k^{y}\right)^* \right)_{y \in Y} \right)$ in \Cref{lem:continuum-product-field-trivial-maps} forms an isometric continuous isomorphism. 
\end{cor}

\begin{proof}
	This follows directly from \Cref{cor:continuum-product-field-isomorphism} and \Cref{lem:continuum-product-field-trivial-maps}. 
\end{proof}

%%%%%%%%%%%%%%%%%%%%%%%%%%%%%%%%%%
%%%%%%%%%%%%%%%%%%%%%%%%%%%%%%%%%%
%%%%%%%%%%%%%%%%%%%%%%%%%%%%%%%%%%
%%%%%%%%%%%%%%%%%%%%%%%%%%%%%%%%%%
%%%%%%%%%%%%%%%%%%%%%%%%%%%%%%%%%%
%%%%%%%%%%%%%%%%%%%%%%%%%%%%%%%%%%
%%%%%%%%%%%%%%%%%%%%%%%%%%%%%%%%%%
%%%%%%%%%%%%%%%%%%%%%%%%%%%%%%%%%%
%%%%%%%%%%%%%%%%%%%%%%%%%%%%%%%%%%
%%%%%%%%%%%%%%%%%%%%%%%%%%%%%%%%%%
%%%%%%%%%%%%%%%%%%%%%%%%%%%%%%%%%%
%%%%%%%%%%%%%%%%%%%%%%%%%%%%%%%%%%
%%%%%%%%%%%%%%%%%%%%%%%%%%%%%%%%%%
%%%%%%%%%%%%%%%%%%%%%%%%%%%%%%%%%%
%%%%%%%%%%%%%%%%%%%%%%%%%%%%%%%%%%
%%%%%%%%%%%%%%%%%%%%%%%%%%%%%%%%%%
%%%%%%%%%%%%%%%%%%%%%%%%%%%%%%%%%%
%%%%%%%%%%%%%%%%%%%%%%%%%%%%%%%%%%
%%%%%%%%%%%%%%%%%%%%%%%%%%%%%%%%%%
%%%%%%%%%%%%%%%%%%%%%%%%%%%%%%%%%%
%%%%%%%%%%%%%%%%%%%%%%%%%%%%%%%%%%
%%%%%%%%%%%%%%%%%%%%%%%%%%%%%%%%%%

\section{$L^2$-continuum powers and randomizations}\label{sec:continuum-power}

In this section, 
we highlight one special case of the construction introduced in \Cref{def:continuum-product-field}. 
This particular construction will be used to enlarge a continuous field of {\hhs}s in a canonical way, which will often turn out to reduce the complexity of the continuous field, as we shall see in \Cref{sec:trivialization}. 

\begin{defn}  \label{def:continuum-product-field-augmentation}
	Let $\aContField$ be a second-countable continuous field of {\hhs}s, let $Y$ be a compact Hausdorff space and let $\mu$ be a regular Borel measure on $Y$. 
	Recall the construction in \Cref{def:cont-field-HHS-maps} of the extension $\aContField|_{\baseSpace{\aContField} \times Y}$ of $\aContField$ over the Cartesian product $\baseSpace{\aContField} \times Y$. 
	Define a continuous map $\tau_\mu \colon \baseSpace{\aContField} \to \spaceMeas(\baseSpace{\aContField} \times Y)$, $z \mapsto \delta_z \times \mu$, 
	where $\delta_z$ is the point mass probability measure at $z$. 
	Define 
	the \emph{$L^2$-continuum power} of $\aContField$ by $(Y,\mu)$ as 
	\[
	\aContField^{(Y, \mu)} := \tau_\mu^* \left( \aContField|_{\baseSpace{\aContField} \times Y} \right)
	\]
	which is a continuous field of {\hhs}s with base space equal to $\baseSpace{\aContField}$. 
\end{defn}

As the terminology suggests, what taking $L^2$-continuum powers of a continuous field of {\hhs}s is to taking $L^2$-continuum products is akin to what taking powers of a number to taking products.  
Let us also give a more concrete description of the fibers of $\aContField^{(Y,\mu)}$. 

\begin{rmk}  \label{rmk:continuum-product-field-augmentation}
	In \Cref{def:continuum-product-field-augmentation}, we have canonical isometric isomorphisms $\left( \aContField^{(Y, \mu)} \right)_z \cong L^2(Y, \mu, \aContField_z) \cong \left( \aContField_z \right)^{(Y, \mu)}$ for any $z \in \baseSpace{\aContField}$, where we use the notation in \Cref{defn:continuous-field-HHS-L2-notations}. 
	
	Indeed, by \Cref{def:continuum-product-field}, we have 
	\[
	\left( \aContField^{(Y, \mu)} \right)_z = \left( \tau_\mu^* \left( \aContField|_{\baseSpace{\aContField} \times Y} \right) \right)_z = L^2 \left( \baseSpace{\aContField} \times Y , \delta_z \times \mu , \aContField|_{\baseSpace{\aContField} \times Y}  \right) \; .
	\]
	If we write $i_z \colon Y \hookrightarrow \baseSpace{\aContField} \times Y$ for the embedding $y \mapsto (z , y)$, then we see that $i_z \colon (Y , \mu) \hookrightarrow \left(\baseSpace{\aContField} \times Y , \delta_z \times \mu \right)$ is a measure space isomorphism and, by \Cref{rmk:cont-field-HHS-maps-functorial}, \Cref{eg:cont-field-HHS-maps-singleton} and \Cref{rmk:cont-field-HHS-maps-trivial}, $\left( i_z \right)^* \left( \aContField|_{\baseSpace{\aContField} \times Y} \right)$ is canonically isomorphic to the constant field $\left( \aContField_z \right)_Y$. Hence by \Cref{lem:continuum-product-maps}, there is a canonical isometric bijection $\left( \aContField^{(Y, \mu)} \right)_z \allowbreak = L^2 \left( \baseSpace{\aContField} \times Y , \delta_z \times \mu , \aContField|_{\baseSpace{\aContField} \times Y}  \right) \allowbreak {\cong} L^2 \left( Y , \mu , \aContField_z  \right)$. 
	
	The canonical isometric bijection $L^2(Y, \mu, \aContField_z) \cong \left( \aContField_z \right)^{(Y, \mu)}$ is established similarly. 
\end{rmk}

We now establish a few basic properties of $L^2$-continuum powers.

\begin{lem}  \label{lem:continuum-product-field-augmentation-maps}
	Let $\aContField$, $Y$, and $\mu$ be as in \Cref{def:continuum-product-field-augmentation}. 
	Let us also freely employ the canonical identifications in \Cref{rmk:continuum-product-field-augmentation}. 
	Then the following hold: 
	\begin{enumerate}

		\item \label{lem:continuum-product-field-augmentation-maps:embedding} 
		There is a unique isometric continuous morphism $\iota \colon \aContField \to \aContField^{(Y, \mu)}$ such that 
		$\baseSpace{\iota}$ is the identity map on $\baseSpace{\aContField}$ and for any $z \in \baseSpace{\aContField}$, the isometric embedding $\iota_z \colon \aContField_z \to \left( \aContField^{(Y, \mu)} \right)_z \cong L^2(Y, \mu, \aContField_z)$ takes any $x \in \aContField_z$ to the constant function $x$ in $L^2(Y, \mu, \aContField_z)$.

		\item \label{lem:continuum-product-field-augmentation-maps:morphism} 
		For any continuous field $\anotContField$ of {\hhs}s and any isometric continuous morphism $\varphi \colon \aContField \to \anotContField$, there is a unique isometric continuous morphism $\varphi^{(Y, \mu)} \colon \aContField^{(Y, \mu)} \to \anotContField^{(Y, \mu)}$ such that $\baseSpace{\varphi^{(Y,\mu)}} = \baseSpace{\varphi}$ and for any $z \in \baseSpace{\anotContField}$, the isometric embedding 
		\[
		\left( \varphi^{(Y,\mu)} \right)_z \colon L^2 \left(Y, \mu, \aContField_{\baseSpace{\varphi(z)}} \right) \cong \left( \aContField^{(Y, \mu)} \right)_{\baseSpace{\varphi(z)}} \to \left( \anotContField^{(Y, \mu)} \right)_z \cong L^2(Y, \mu, \anotContField_z)
		\]
		takes any $\xi \in L^2 \left(Y, \mu, \aContField_{\baseSpace{\varphi(z)}} \right)$ to $\varphi_z \circ \xi \in L^2(Y, \mu, \anotContField_z)$.

		\item \label{lem:continuum-product-field-augmentation-maps:isom} 
		There is a group homomorphism $\operatorname{Isom} (\aContField) \to \operatorname{Isom} \left( \aContField^{(Y, \mu)} \right)$, $\varphi \mapsto \varphi^{(Y,\mu)}$.

		\item \label{lem:continuum-product-field-augmentation-maps:action} 
		For any isometric action $\alpha \colon G \curvearrowright \aContField$, there is an isometric action $\alpha^{(Y, \mu)} \colon G \curvearrowright \aContField^{(Y, \mu)}$ such that $\left( \alpha^{(Y, \mu)} \right)_g = \left( \alpha_g \right)^{(Y, \mu)}$. 
		
		\item \label{lem:continuum-product-field-augmentation-maps:induced} 
		For any compact Hausdorff space $Z$ and any continuous map $f \colon Z \to \baseSpace{\aContField}$, there is an isometric continuous isomorphism from $f^* \left( \aContField^{(Y, \mu)} \right)$ to $ \left( f^* \aContField \right)^{(Y, \mu)}$ 
		that fits into the following commutative diagram
		\[
		\xymatrix{
			\aContField \ar[rr]^{f^*} \ar[d]_{\iota} && f^* \aContField \ar[d]^{\iota} \\
			\aContField^{(Y, \mu)} \ar[r]^{f^*} & f^* \left( \aContField^{(Y, \mu)} \right) \ar[r]^{\cong} & \left( f^* \aContField \right)^{(Y, \mu)} \ar[l]
		}
		\]
		of continuous fields of {\hhs}s and isometric continuous morphisms between them. 
		
		\item \label{lem:continuum-product-field-augmentation-maps:variation} 
		For any compact Hausdorff space $Z$ and any continuous map $f \colon Z \to \spaceMeas \left( \baseSpace{\aContField}\right)$, there is an isometric continuous isomorphism from $f^* \left( \aContField^{(Y, \mu)} \right)$ to $ \left( f^* \aContField \right)^{(Y, \mu)}$ 
		that yields the following commutative diagram
		\[
		\xymatrix{
			\spaceOfSections_{\operatorname{cont}} \left(\aContField \right) \ar[rr]^{f^*} \ar[d]_{\iota} && \spaceOfSections_{\operatorname{cont}} \left( f^* \aContField \right) \ar[d]^{\iota} \\
			\spaceOfSections_{\operatorname{cont}} \left(\aContField^{(Y, \mu)} \right) \ar[r]^{f^*} & \spaceOfSections_{\operatorname{cont}} \left(f^* \left( \aContField^{(Y, \mu)} \right) \right) \ar[r]^{\cong} & \spaceOfSections_{\operatorname{cont}} \left(\left( f^* \aContField \right)^{(Y, \mu)} \right) \ar[l]
		}
		\]
		of maps among sets of continuous sections.

		\item \label{lem:continuum-product-field-augmentation-maps:product} 
		For any compact Hausdorff space $Y'$ and any finite regular Borel measure $\mu'$ on $Y'$, there is an isometric continuous isomorphism from $\left( \aContField^{(Y, \mu)} \right)^{(Y', \mu')}$ to $\aContField^{(Y \times Y', \mu \times \mu')}$ 
		that fits into the following commutative diagram
		\[
		\xymatrix{
			\aContField \ar[r]^{\iota} \ar[d]_{\iota} & \aContField^{(Y, \mu)} \ar[d]^{\iota} \\
			\aContField^{(Y \times Y', \mu \times \mu')} \ar[r]^{\cong}  & \left( \aContField^{(Y, \mu)} \right)^{(Y', \mu')} \ar[l]
		}
		\]
		of continuous fields of {\hhs}s and isometric continuous morphisms between them.

		\item \label{lem:continuum-product-field-augmentation-maps:measure-isomorphism} 
		For any compact Hausdorff space $Y'$ and any finite regular Borel measure $\mu'$ on $Y'$, if there is a measure space isomorphism between $(Y, \mu)$ and $(Y', \mu')$, then there is an isometric continuous isomorphism from $\aContField^{(Y, \mu)}$ to $\aContField^{(Y', \mu')}$ 
		that fits into the following commutative diagram
		\[
		\xymatrix{
			&\aContField \ar[rd]^{\iota} \ar[ld]_{\iota} &  \\
			\aContField^{(Y, \mu)} \ar[rr]^{\cong}  && \aContField^{(Y', \mu')} \ar[ll]
		}
		\]
		of continuous fields of {\hhs}s and isometric continuous morphisms between them. 
	\end{enumerate}
\end{lem}

\begin{proof}
	Let $\pi \colon \baseSpace{\aContField} \times Y \to \baseSpace{\aContField}$ be the canonical projection onto the first factor.

	For \Cref{lem:continuum-product-field-augmentation-maps:embedding}, it is easy to see that $\iota_z$ is an isometric embedding $\aContField \to \aContField^{(Y, \mu)}$ if the total measure $\mu(Y) = 1$. Note that under $\iota$, a continuous section $s$ in $\aContField$ goes to the section of $\aContField^{(Y, \mu)}$ given by taking $z \in \baseSpace{\aContField}$ to the constant function $s(z)$ in $L^2(Y, \mu, \aContField_z)$. It remains to show that this is a continuous section. 
	
	Recall that by definition $\aContField|_{\baseSpace{\aContField} \times Y}$, section $s \circ \pi$ are continuous for continuous sections $s$ of $\aContField$, and by definition of $\aContField^{(Y, \mu)}$, the image of these sections in $\prod_z L^2(Y, \mu, \aContField_z)$ are continuous. But these are exactly the sections that take $z$ to the constant function $s(z)$ in $L^2(Y, \mu, \aContField_z)$, so these are continuous, as desired.

	For \Cref{lem:continuum-product-field-augmentation-maps:morphism}, note that for $\varphi \colon \aContField \to \anotContField$, there is an isometric continuous morphism
	\[\varphi^Y:\aContField|_{\baseSpace{\aContField} \times Y} \to \anotContField|_{\baseSpace{\anotContField} \times Y}\]
	where $\baseSpace{\phi^Y}: \baseSpace{\anotContField} \times Y \to \baseSpace{\aContField} \times Y$ is given by $\baseSpace{\phi^Y}(z, y) =( \baseSpace{\phi}(z), y) \in \baseSpace{\aContField} \times Y$ for $(z, y) \in \baseSpace{\anotContField} \times Y$, and the fiber map 
	\[\phi^Y_{(z,y)}:\left(\aContField|_{\baseSpace{\aContField} \times Y}\right)|_{\baseSpace{\phi^Y}(z,y)} \to \left(\anotContField|_{\baseSpace{\anotContField} \times Y}\right)|_{(z,y)}\]
	is given by $\phi_z: \aContField_{\baseSpace{\phi}(z)} \to \anotContField_{z}$ under the identification of
	\[\left(\aContField|_{\baseSpace{\aContField} \times Y}\right)_{\baseSpace{\phi^Y}(z,y)} = \left(\aContField|_{\baseSpace{\aContField} \times Y}\right)_{(\phi(z),y)} = \aContField_{\baseSpace{\phi}(z)}\]
	and 
	\[\left(\anotContField|_{\baseSpace{\anotContField} \times Y}\right)_{(z,y)} =  \anotContField_{z}.\]
	
	Note that $\phi^Y$ fits into the following diagram commutes 
	\[\begin{tikzcd}
	\baseSpace{\aContField} & {M(\baseSpace{\aContField} \times Y)} \\
	\baseSpace{\anotContField} & {M(\baseSpace{\anotContField} \times Y)}
	\arrow["\baseSpace{\phi}", from=2-1, to=1-1]
	\arrow["\tau_\mu", from=1-1, to=1-2]
	\arrow["\tau_\mu"', from=2-1, to=2-2]
	\arrow["{M(\baseSpace{\varphi^Y})}"', from=2-2, to=1-2]
	\end{tikzcd}\]
	
	To see this, for $z \in \baseSpace{\anotContField}$, and $U_1 \times U_2 \subseteq \baseSpace{C} \times Y$, note that we have
	\[\tau_\mu(\baseSpace{\phi}(z)) = \delta_{\baseSpace{\phi}(z) \times \mu}(U_1 \times U_2) = \begin{cases} 0 & \baseSpace{\phi}(z) \not\in U_1 \\ \mu(U_2) & \baseSpace{\phi}(z) \in U_1\end{cases}\]
	and 
	\[
	M(\baseSpace{\anotContField} \times Y)( \tau_\mu(z))(U_1 \times U_2)  = 
	( \delta_z \times \mu) ((\baseSpace{\phi^Y})^{-1}(U_1 \times U_2)) =
	( \delta_z \times \mu) ((\baseSpace{\phi^Y})^{-1}(U_1) \times U_2)\]
	\[ = \begin{cases} 0 & \baseSpace{\phi}(z) \not\in U_1 \\ \mu(U_2) & \baseSpace{\phi}(z) \in U_1\end{cases}\]
	so the above diagram commutes.
	
	Now applying \Cref{cor:continuum-product-field-isomorphism-easy}, we see that for the induced isometric embedding
	
	\[(\phi^Y)_*: L^2(\baseSpace{\aContField} \times Y, \tau_\mu(\baseSpace{\phi}(z)), \aContField|_{\baseSpace{\aContField} \times Y}) \to L^2(\baseSpace{\anotContField} \times Y, \tau_\mu(\baseSpace{\phi}(z)), \anotContField|_{\baseSpace{\aContField} \times Y})\]
	the tuple $(\baseSpace{\phi}, (\phi^Y)_*)$ forms an isometric continuous morphism 
	\[\tau_\mu^*(\aContField|_{\baseSpace{\aContField} \times Y}) \to \tau_\mu^*(\anotContField|_{\baseSpace{\anotContField} \times Y})\]
	where for section $s$ of $\aContField|_{\baseSpace{\aContField} \times Y}$, and $z \in \anotContField$, 
	\[((\phi^Y)_*(s))((z,y)) = \phi_z(s(\baseSpace{\phi}, y)).\]
	Recall the identifications in \Cref{rmk:continuum-product-field-augmentation}
	\[L^2(\baseSpace{\aContField} \times Y, \tau_\mu(\baseSpace{\phi}(z)), \aContField|_{\baseSpace{\aContField} \times Y}) \simeq L^2(Y, \mu, \aContField_{\baseSpace{\phi}(z)})\]
	and 
	\[ L^2(\baseSpace{\anotContField} \times Y, \tau_\mu(\baseSpace{\phi}(z)), \anotContField|_{\baseSpace{\aContField} \times Y}) \simeq  L^2(Y, \mu, \anotContField_{z})\]
	which come from the measure preserving isomorphisms $Y \to \baseSpace{\aContField} \times Y$ and $Y \to \baseSpace{\anotContField} \times Y$ given by $y \mapsto \baseSpace{\phi}(z)(y)$ and $y \mapsto (z,y)$ respectively.
	
	Under these identifications, the above map corresponds to $\xi \mapsto \phi_z \circ \xi$, as desired.
	
	Because the maps $\phi^{(Y, \mu)}$ of \Cref{lem:continuum-product-field-augmentation-maps:morphism} were constructed as an application of \Cref{cor:continuum-product-field-isomorphism-easy}, \Cref{lem:continuum-product-field-augmentation-maps:isom} follows from the last statement of \Cref{cor:continuum-product-field-isomorphism-easy}. \Cref{lem:continuum-product-field-augmentation-maps:action} then follows from observing that the isometric action can be seen as a map $G \to \operatorname{Isom}(\aContField)$, and applying \Cref{lem:continuum-product-field-augmentation-maps:isom}.

	For \Cref{lem:continuum-product-field-augmentation-maps:induced}, it is easy to check that the base spaces and fibres of $f^* \left( \aContField^{(Y, \mu)} \right)$ and $ \left( f^* \aContField \right)^{(Y, \mu)}$ are the same, as both have base space $Z$ and fibre over $z$ being $L^2(Y, \mu, \aContField_{f(z)})$. So to check \Cref{lem:continuum-product-field-augmentation-maps:induced}, we just need to check that the continuous sections on the two continuous fields agree.

	The continuous sections of $\aContField|_{(\baseSpace{\aContField} \times Y)}$ are generated by the image of the map 
	\[\spaceOfSections_{\operatorname{cont}}(\aContField) \to \prod_{(z, y) \in \baseSpace{\aContField} \times Y} \aContField_z\]
	given by $s \mapsto (s(z))_{(z,y)}$.

	The continuous sections of $\aContField^{(Y,\mu)} = \tau_\mu^*(\aContField|_{\baseSpace{\aContField} \times Y})$ are generated by the image of 
	\[\spaceOfSections_{\operatorname{cont}}\left(\aContField|_{\baseSpace{\aContField} \times Y}\right) \to \prod_{z \in \baseSpace{\aContField}} L^2(\baseSpace{\aContField} \times Y, \delta_z \times \mu, \aContField|_{\baseSpace{\aContField} \times Y})\]
	given by $s \mapsto (s)_z$ where the $s$ on the right is viewed in $L^2(\baseSpace{\aContField} \times Y, \delta_z \times \mu, \aContField|_{\baseSpace{\aContField} \times Y})$.
	
	The continuous sections of $f^*(\aContField^{(Y,\mu)})$ are then the compositions of sections of $\aContField^{(Y,\mu)}$ with $f$.
	
	Thus, continuous sections of $f^*(\aContField^{(Y,\mu)})$ are therefore generated by the image of the map 
	\[ \spaceOfSections_{\operatorname{cont}}\left(\aContField\right) 
	\to \prod_{(z, y) \in \baseSpace{\aContField} \times Y} \aContField_z 
	\to  \prod_{z \in \baseSpace{\aContField}} L^2(\baseSpace{\aContField} \times Y, \delta_z \times \mu, \aContField|_{\baseSpace{\aContField} \times Y}) \]
	\[= \prod_{z \in \baseSpace{\aContField}}  L^2(Y, \mu, \aContField_{z})
	\to \prod_{z in Z} L^2(Y, \mu, \aContField_{f(z)})\]

	given by
	\[s \mapsto (s(z))_{(z,y)} \mapsto ([(z,y) \mapsto s(z)]  \in L^2(\baseSpace{\aContField} \times Y, \delta_{z_1} \times \mu, \aContField|_{\baseSpace{\aContField} \times Y}))_{z_1 \in \baseSpace{\aContField}}\]
	\[ = ([y \mapsto s(z)]  \in L^2(Y,  \mu, \aContField_{z}))_{z \in \baseSpace{\aContField}}
	\mapsto ([y \mapsto s(f(y_0))]  \in  L^2(Y,  \mu, \aContField_{f(z_0)}))_{z_0 \in Z}\]
	
	Now let us consider the space of sections of $(f^*\aContField)^{(Y, \mu)}$. By a similar argument as the above, this is generated by the image of the map
	\[\spaceOfSections_{\operatorname{cont}}\left(\aContField\right) 
	\to \prod_{z \in Z} (f^*\aContField)_z 
	\to \prod_{(z,y) \in Z \times Y} (f^*\aContField|_{Z \times Y})_{z,y} \]
	\[\to \prod_{z \in Z} \left((f^*\aContField)^{(Y, \mu)}\right)_z
	= \prod_{z \in Z} L^2(Z \times Y, \delta_z \times \mu, f^*\aContField|_{Z \times Y})
	= \prod_{z \in Z} L^2(Y, \mu, \aContField_{f(z)})\]
	given by 
	\[s \mapsto (s(f(z)))_{z \in Z} \mapsto \left(s(f(z))\right)_{(z,y) \in Z \times Y}\]
	\[ \mapsto ([(z,y) \mapsto s(f(z))] \in L^2(Z \times Y, \delta_{z_0} \times \mu, f^*\aContField|_{Z \times Y}))_{z_0 \in Z}\]
	\[ = ([y \mapsto s(f(z_0))] \in L^2(Y, \mu, \aContField_{f(z_0)}))_{z_0 \in Z}.\]
	Thus, since the spaces of sections of $f^* \left( \aContField^{(Y, \mu)} \right)$ and $ \left( f^* \aContField \right)^{(Y, \mu)}$ have the same generating set, they are the same spaces. The commutativity of the diagram is straightforward. This finishes the proof of \Cref{lem:continuum-product-field-augmentation-maps:induced}.

	The proof of \Cref{lem:continuum-product-field-augmentation-maps:variation}, proceeds similarly to that of \Cref{lem:continuum-product-field-augmentation-maps:induced}, as follows:

	The base spaces of both $(f^*\aContField)^{(Y, \mu)}$ and $f^*(\aContField^{(Y, \mu)})$ is $Z$.  To compare the fibers, note that for $\theta \in Z$, unwrapping the definitions, the fibre of $f^*(\aContField^{(Y, \mu)})$ is 
	\[L^2(\baseSpace{\aContField^{(Y, \mu)}}, f(\theta), \aContField^{(Y, \mu)}) = 
	L^2(\baseSpace{\aContField}, f(\theta), \aContField^{(Y, \mu)}),\]
	where the fibres of $\aContField^{(Y, \mu)}$ are given by $L^2(Y, \mu, \aContField_z)$ over $z \in \baseSpace{C}$.
	
	By Fubini's theorem (which can be applied in this case, even though $f(\theta)$ and $\mu$ may not be $\sigma$-finite, because in deciding whether functions are $L^2$, the integrands are non-negative), we have that this fibre is the same as
	\[L^2 (\baseSpace{\aContField} \times Y, f(z) \times \mu, \aContField|_{\baseSpace{\aContField} \times Y}).\]

	Similarly to the computation of \Cref{lem:continuum-product-field-augmentation-maps:induced}, the continuous sections are generated by the image of 
	\[\spaceOfSections_{\operatorname{cont}}(\aContField^{(Y, \mu)}) \to \prod_{\theta \in Z} L^2(\baseSpace{\aContField}, f(\theta),  \aContField^{(Y, \mu)}),\]
	which takes the continuous section $s$ to $s$ viewed in each $L^2(\baseSpace{\aContField}, f(\theta),  \aContField^{(Y, \mu)})$. 
	
	Unwrapping what the continuous sections of $ \aContField^{(Y, \mu)}$ are, we can see that the continuous sections of $f^*( \aContField^{(Y, \mu)})$ are generated by the image of the map 
	\[\spaceOfSections_{\operatorname{cont}}(\aContField) \to \spaceOfSections_{\operatorname{cont}}( \aContField^{(Y, \mu)}) \to \prod_{\theta \in Z} L^2(\baseSpace{\aContField}, f(\theta),  \aContField^{(Y, \mu)}) = \prod_{\theta \in Z} L^2 (\baseSpace{\aContField} \times Y, f(z) \times \mu, \aContField|_{\baseSpace{\aContField} \times Y})\]
	given by mapping $ s\in \spaceOfSections_{\operatorname{cont}}(\aContField) $ to 
	\[((z,y) \mapsto s(z))_{\theta \in Z},\]
	where $((z,y) \mapsto s(z)) \in L^2 (\baseSpace{\aContField} \times Y, f(z) \times \mu, \aContField|_{\baseSpace{\aContField} \times Y})$.

	Now we turn our attention to $(f^*\aContField)^{(Y, \mu)}$. The base space of this is also $Z$.
	
	Note that $f^*(\aContField)$ has fibres $L^2(\baseSpace{\aContField}, f(\theta), \aContField)$ for $\theta \in Z$, so the fibres of $(f^*\aContField)^{(Y, \mu)}$ are 
	\[L^2(Y, \mu, L^2(\baseSpace{\aContField}, f(\theta), \aContField)).\]

	We again may apply Fubini's theorem to see that this can be rewritten as 
	\[L^2 (\baseSpace{\aContField} \times Y, f(\theta) \times \mu, \aContField|_{\baseSpace{\aContField} \times Y}),\]
	so we see that fibres agree with those of $f^*(\aContField^{(Y, \mu)})$. 
	
	The continuous section of $f^*(\aContField)$ are generated by the image of the map 
	\[\spaceOfSections_{\operatorname{cont}}(\aContField) \to \prod_{\theta \in Z}L^2(\baseSpace{\aContField}, f(\theta), \aContField))\]
	given by taking $s$ to $(s)_z$

	The continuous sections of $f^*(\aContField)^{(Y, \mu)}$ are generated by the image of 
	\[\spaceOfSections_{\operatorname{cont}}(f^*(\aContField)) \to \prod_{\theta \in Z} L^2(Y, \mu, L^2(\baseSpace{\aContField}, f(\theta), \aContField))\]
	given by taking $s$ to $(y \mapsto s(\theta))$ where $s$ is a continuous section of $f^*(\aContField)$ so for $z \in Z$, $s(\theta) \in L^2(\baseSpace{\aContField}, f(\theta), \aContField))$. 
	
	Thus, overall the continuous sections are generated by the image of 
	\[\spaceOfSections_{\operatorname{cont}}(\aContField) \to \prod_{\theta \in Z} L^2(Y, \mu, L^2(\baseSpace{\aContField}, f(\theta), \aContField)) = L^2 (\baseSpace{\aContField} \times Y, f(\theta) \times \mu, \aContField|_{\baseSpace{\aContField} \times Y})\]
	given by taking the continuous section $s$ to 
	\[((z, y) \mapsto s(z)) \in L^2 (\baseSpace{\aContField} \times Y, f(\theta) \times \mu, \aContField|_{\baseSpace{\aContField} \times Y}).\]
	
	Thus the space of continuous sections agrees with that of $f^*(\aContField^{(Y, \mu)})$, as desired. Again the commutativity of the diagram is straightforward. This finishes the proof of \Cref{lem:continuum-product-field-augmentation-maps:variation}.

	For \Cref{lem:continuum-product-field-augmentation-maps:product}, by the same reasoning as in \Cref{rmk:continuum-product-field-augmentation}, we can see that the fibres of both $\left( \aContField^{(Y, \mu)} \right)^{(Y', \mu')}$ is
	\[L^2(Y', \mu', L^2(Y, \mu, \aContField_z))\]
	and the fibres of $\aContField^{(Y \times Y', \mu \times \mu')}$ are 
	\[L^2(Y \times Y', \mu \times \mu', \aContField_z)\]
	and it is easy to see that theses are isomorphic.

	Also note that in the notation of \Cref{rmk:continuum-product-field-augmentation}, the continuous sections of $ \aContField^{(Y, \mu)} $ are generated by the image of the map
	\[\spaceOfSections_{\operatorname{cont}}(\aContField)  \to \spaceOfSections_{\operatorname{cont}}(\aContField|_{\baseSpace{C} \times Y}) \to \prod_{z \in \baseSpace{\aContField}} L^2(Y, \mu, \aContField_z)\]
	that takes $s$ to $(s(z))_{z \in \baseSpace{\aContField}}$ where $s(z)$ denotes the constant function in $L^2(Y, \mu, \aContField_z)$. 
	
	Similarly, the continuous sections of $\left( \aContField^{(Y, \mu)} \right)^{(Y', \mu')}$ are generated by the image of the map
	\[\spaceOfSections_{\operatorname{cont}}(\aContField) \to \prod_{z \in \baseSpace{\aContField}} L^2(Y', \mu', L^2(Y, \mu, \aContField_z))\]
	that takes $s$ to the constant map $y' \mapsto [y \mapsto s(z)]$ in the fibre over $z \in \baseSpace{\aContField}$. 
	
	It is easy to see that the generating sets of the continuous sections are the same, so the two continuous fields are the same.
	
	The commutation of the diagram follows from the fact that both going directly downwards from $\aContField$ and going to the right and then downwards, the map is the identity on the base spaces $\baseSpace{\aContField}$ and the fibre maps are the inclusions of constant functions, taking $x \in \aContField_z$ to the corresponding constant function $[(y', y) \mapsto x]$ in the case of directly mapping downwards and $[y' \mapsto [y \mapsto x]]$ in the case of mapping to the right and then down. These clearly correspond to each other.
	
	For \Cref{lem:continuum-product-field-augmentation-maps:measure-isomorphism}, if there is a measure space isomorphism between $(Y, \mu)$ and $(Y', \mu')$, then, using the view of  \Cref{rmk:continuum-product-field-augmentation}, the fibres of $\aContField^{(Y,\mu)}$ and $\aContField^{(Y', \mu')}$ are the same, as
	\[L^2(Y, \mu, \aContField_z) \simeq L^2(Y', \mu', \aContField_z)\]
	by maps in both directions being pre-composition with the isomorphism between $(Y, \mu)$ and $(Y', \mu')$. 
	Moreover, the continuous sections of $\aContField^{(Y,\mu)}$ are generated by the sections in
	\[\prod_{z \in \baseSpace{\aContField}} L^2(Y, \mu, \aContField_z)\]
	given by $([y \mapsto s(z)])_{z \in \baseSpace{\aContField}}$, and the continuous sections of $\aContField^{(Y',\mu')}$ are generated by the sections in
	\[\prod_{z \in \baseSpace{\aContField}} L^2(Y', \mu', \aContField_z)\]
	given by $([y' \mapsto s(z)])_{z \in \baseSpace{\aContField}}$.
	
	The isomorphisms $L^2(Y, \mu, \aContField_z) \simeq L^2(Y', \mu', \aContField_z)$ clearly takes these constant functions to each other. 
	
	This completes the proof of the lemma. 
\end{proof}

\Cref{rmk:continuum-product-field-augmentation} suggests that when $(Y, \mu)$ is a standard probability space, $\spaceOfSections_{\operatorname{cont}} \left(\aContField^{(Y, \mu)}\right)$ may be seen as consisting of ``$L^2$-integrable'' random variables in $\spaceOfSections_{\operatorname{cont}} \left(\aContField\right)$. 
This inspires the following definition. 

\begin{defn} \label{def:randomization}
	Let $\aContField$ be a second-countable continuous field of {\hhs}s. Let $(\Omega, \mu)$ be a standard probability space. The \emph{randomization} of $\aContField$ is $\aContField^{(\Omega, \mu)}$ as in \Cref{def:continuum-product-field-augmentation}. 
	
	We say a continuous field $\aContField$ of {\hhs}s is \emph{randomization-stable} if it is second-countable and isometrically continuously isomorphic to $\aContField^{(\Omega, \mu)}$. 
\end{defn}

\begin{rmk} \label{rmk:randomization-stable}
	Since all standard probability spaces are isomorphic as measure spaces, it follows from \Cref{lem:continuum-product-field-augmentation-maps}\eqref{lem:continuum-product-field-augmentation-maps:measure-isomorphism}
	that the notion of randomization-stability does not depend on the choice of $(\Omega, \mu)$. 
	For example, we may pick $(\Omega, \mu) := ([0,1], m)$, where $m$ is the Lebesgue measure.  
	
	To simplify notations, we write $\aContField^{[0,1]}$ in place of $\aContField^{([0,1], m)}$ when there is no risk of confusion. 
	
	Moreover, since $(\Omega, \mu)$ and $(\Omega \times \Omega, \mu \times \mu)$ are isomorphic as measure spaces, it follows from \Cref{lem:continuum-product-field-augmentation-maps}\eqref{lem:continuum-product-field-augmentation-maps:product} and~\eqref{lem:continuum-product-field-augmentation-maps:measure-isomorphism} that for any continuous field $\aContField$ of {\hhs}s, its randomization $\aContField^{(\Omega, \mu)}$ is randomization-stable, as 
	\[
	\left( \aContField^{(\Omega, \mu)} \right)^{(\Omega, \mu)} \cong \aContField^{(\Omega \times \Omega, \mu \times \mu)} 
	\cong \aContField^{(\Omega, \mu)} \; .
	\]
	This also provides further justification to the terminology ``randomization-stable''. 
\end{rmk}

The next result (\Cref{prop:continuum-product-field-stitch-morphisms}) is the source of various deformation and trivialization techniques in \Cref{sec:trivialization}. It allow us to stitch several isometric continuous morphisms into one in a very flexible manner.  

To formulate it, let us make the following definition. 

\begin{defn}\label{def:measurable-partition}
	Let $(Y, \mu)$ be a finite regular Borel measure space and let $n$ be a positive integer. Then a \emph{measurable $n$-partition} of $(Y, \mu)$ is a measurable map $P \colon Y \to \{0, 1, \ldots, n-1\}$, where we identify two such maps if they agree $\mu$-almost everywhere. The collection of all measurable $n$-partitions of $(Y, \mu)$ is denoted by $\mathcal{P}_n(Y, \mu)$. 
	
	Let $\mathcal{B}(Y, \mu)$ be the measure algebra of all measurable subsets of $(Y, \mu)$ modulo null sets, equipped with the metric 
	\[
	(Y_1, Y_2) \mapsto \mu \left( Y_1 \Delta Y_2 \right) = \left\| \chi_{Y_1} - \chi_{Y_2} \right\|_1 \; .
	\]
	Note that $\mathcal{P}_n(Y, \mu)$ embeds into $\mathcal{B}(Y, \mu)^n$ via $P \mapsto \left( P^{-1}(k) \right)_{k \in \{0, 1, \ldots, n-1\}}$. 	
	We equip $\mathcal{P}_n(Y, \mu)$ with the subspace topology inherited from $\mathcal{B}(Y, \mu)^n$ (and thus also from $L^1(Y,\mu)^n$). 
	
	Let $\mathcal{P}_\omega (Y,\mu)$ be the union $\bigcup_{n=1}^{\infty} \mathcal{P}_n (Y,\mu)$, where $\mathcal{P}_n(Y,\mu)$ embeds into $\mathcal{P}_{n+1}(Y,\mu)$ via the inclusion $\{0, 1, \ldots, n-1\} \subseteq \{0, 1, \ldots, n\}$. We equip $\mathcal{P}_\omega (Y, \mu)$ with the weak topology, that is, a set in $\mathcal{P}_\omega (Y,\mu)$ is open if and only if its intersection with every $\mathcal{P}_n(Y,\mu)$ is open in $\mathcal{P}_n(Y,\mu)$. 
	
	For any positive integer $n$ and any natural number $k$, we write 
	\[
	\mathcal{P}_n^{(k)} (Y,\mu) := \left\{ P \in \mathcal{P}_n (Y,\mu) \colon \mu \left( P^{-1} (k) \right) > 0 \right\} \; ,
	\]
	which is an open subset of $\mathcal{P}_n (Y,\mu)$ (and empty if $k \geq n$). Define 
	\[
	\mathcal{P}_\omega^{(k)} (Y,\mu) := \bigcup_{n=1}^{\infty} \mathcal{P}_n^{(k)} (Y,\mu) \subseteq \mathcal{P}_\omega (Y,\mu) \; ,
	\]
	which is open in $\mathcal{P}_\omega (Y,\mu)$ since clearly $\mathcal{P}_\omega^{(k)} (Y,\mu) \cap \mathcal{P}_n (Y,\mu) = \mathcal{P}_n^{(k)} (Y,\mu)$ for any positive integer $n$. 
	Observe that as long as $\mu$ is nonzero, $\left( \mathcal{P}_\omega^{(k)} (Y,\mu) \right)_{k \in \mathbb{N}}$ forms an open cover of $\mathcal{P}_\omega (Y,\mu)$. 
\end{defn}

With these notions at our disposal, we can formulate our stitching result (\Cref{prop:continuum-product-field-stitch-morphisms}). Since the statement of this result is a bit lengthy, let us start with an informal discussion to put things into perspective. We will freely use notations from \Cref{prop:continuum-product-field-stitch-morphisms}. To simplify our exposition, let us assume that $\aContField = \anotContField$, $h$ is the identity map on the base space $\baseSpace{\aContField}$, and $(Y,\mu)$ is the unit interval $[0,1]$ with the Lebesgue (probability) measure $m$. The idea is that we have two (or, in general, at most countably many) partially defined isometric continuous morphisms $\varphi^{(k)}$, $k=0,1$, from $\aContField$ to itself (or, more generally, $\anotContField$) that are compatible with $\idmap_{\baseSpace{\aContField}}$ (or, more generally, $h$), that is, each $\varphi^{(k)}$ is only defined on a certain open subset $U^{(k)}$ of $\baseSpace{\aContField}$, such that these two open subsets cover $\baseSpace{\aContField}$. 
To combine these two isometric continuous morphisms at every $z \in \baseSpace{\aContField}$, we take a partition of unity $(f^{(k)})_{k = 0,1}$ subordinate to said open cover. Now the continuum power construction with regard to $(Y,\mu)$ allows us to ``mix $\varphi^{(0)}_z$ and $\varphi^{(1)}_z$ at the ratio $f^{(0)}(z)$-to-$f^{(1)}(z)$''. 
More precisely, since the fiber $\aContField^{(Y,\mu)}_z$ is canonically identified with $L^2([0,1], m, \aContField_z)$, which can be decomposed as the Cartesian product $L^2([0,f^{(0)}(z)], m, \aContField_z) \times L^2([f^{(0)}(z),1], m, \aContField_z)$ (equipped with the $\ell^2$-metric), we can apply $\varphi^{(0)}_z$ (or more precisely, $(\varphi^{(0)}_z)^{[0,f^{(0)}(z)]}$) to the former Cartesian factor and $\varphi^{(1)}_z$ to the latter. Since the ``ratio'' varies continuously with regard to $z \in \baseSpace{\aContField}$, it can be shown that the ``mixing'' yields an isometric continuous morphism on $\aContField$. 
The intuitive idea can be compared to a wipe transition in filmmaking (where $\baseSpace{\aContField}$ corresponds to the time coordinate, and $[0,1]$ corresponds to the width of the screen). 

In the more general setting, the role of the partition of unity is played by a continuous map $\Xi \colon \baseSpace{\anotContField} \to \mathcal{P}_\omega(Y, \mu)$. We also remark that in the actual statement below, in order to reduce the number of independent quantifiers, we take the open cover to be the ``tightest'' one determined by $\Xi$, written as $\left( U_\Xi^{(k)} \right)_{k \in \mathbb{N}}$. 

\begin{prop}\label{prop:continuum-product-field-stitch-morphisms}
	Let $\aContField$ and $\anotContField$ be second-countable continuous fields of {\hhs}s. 
	Let $Y$ and $\mu$ be as in \Cref{def:continuum-product-field-augmentation} with $\mu$ nonzero. 
	Let $\Xi \colon \baseSpace{\anotContField} \to \mathcal{P}_\omega(Y, \mu)$ 
	and $h \colon \baseSpace{\anotContField} \to \baseSpace{\aContField}$ be continuous maps. 
	For each $k \in \{0, 1,2,\ldots\}$, 
	write $U_\Xi^{(k)}$ for the open subset of $\baseSpace{\anotContField}$ defined as $\Xi^{-1} \left( \mathcal{P}_\omega^{(k)} (Y,\mu) \right)$, 
	and let $\CFHHS_{h|_{U_\Xi^{(k)}}}(\aContField|_{h \left( U_\Xi^{(k)} \right)} , \anotContField|_{U_\Xi^{(k)}})$ be as in \Cref{defn:cont-fields-morphism}\eqref{defn:cont-fields-morphism:category-fixed-base}.  
	Then the following hold: 
	\begin{enumerate}
		\item \label{prop:continuum-product-field-stitch-morphisms:basic}
		There is a map 
		\[
		\Sigma_{h, \Xi} \colon 
		\prod_{k=0}^\infty	\CFHHS_{h|_{U_\Xi^{(k)}}} \left(\aContField|_{h \left( U_\Xi^{(k)} \right)} , \anotContField|_{U_\Xi^{(k)}} \right) 
		\to \CFHHS_{h} \left(\aContField ^{(Y,\mu)}, \anotContField ^{(Y,\mu)} \right) 
		\]
		that takes a tuple $\left( \varphi^{(k)} \colon \aContField|_{h \left( U_\Xi^{(k)} \right)} \to \anotContField|_{U_\Xi^{(k)}} \right)_{k \in \mathbb{N}}$ with $\baseSpace{\varphi^{(k)}} = h|_{U_\Xi^{(k)}}$ to 
		an isometric continuous morphism $\psi \colon \aContField \to \anotContField$ with $\baseSpace{\psi} = h$ such that 
		for any $z \in \baseSpace{\anotContField}$, the isometric embedding $\psi_z \colon \left( \aContField^{(Y, \mu)} \right)_{h(z)} \to \left( \anotContField^{(Y, \mu)} \right)_z$ satisfies: if we identify $\left( \aContField^{(Y, \mu)} \right)_{h(z)}$ with $L^2(Y, \mu, \aContField_{h(z)}) $ and $\left( \anotContField^{(Y, \mu)} \right)_z$ with $L^2(Y, \mu, \anotContField_z) $ as in \Cref{def:continuum-product-field-augmentation}, then for any $\xi \in L^2(Y, \mu, \aContField_{h(z)}) $, we have 
		\[
		\hspace{1.5cm} \psi_z (\xi) (y) = \varphi^{( \Xi (z) (y) )}_z \left( \xi(y) \right) \quad \text{ for $\mu$-almost every } y \in Y \; ,
		\]
		where we note that for $\mu$-almost every $y \in Y$, we have $z \in U_{\Xi (z) (y)}$ and thus $\varphi^{( \Xi (z) (y) )}_z$ is well-defined.

		\item \label{prop:continuum-product-field-stitch-morphisms:functorial}
		The above construction is functorial in the sense that for any continuous field $\anotContField'$ of {\hhs}s and any continuous map $h' \colon \baseSpace{\anotContField'} \to \baseSpace{\anotContField}$, 
		observing that $h' \left( U_{\Xi \circ h'}^{(k)} \right) = U_{\Xi}^{(k)}$, 
		we have a commutative diagram 
		\[
		\xymatrix{
			\overunderset{\displaystyle \left( \prod_{k=0}^\infty	\CFHHS_{h|_{U_\Xi^{(k)}}} \left(\aContField|_{h \left( U_\Xi^{(k)} \right)} , \anotContField|_{U_\Xi^{(k)}} \right) \right)}{\displaystyle \left( \prod_{k=0}^\infty	\CFHHS_{h'|_{U_{\Xi \circ h'}^{(k)}}} \left(\anotContField|_{U_\Xi^{(k)}} , \anotContField'|_{U_{\Xi \circ h'}^{(k)}} \right) \right)}{\times}
			\ar[rr]^-{\Sigma_{h, \Xi} \times \Sigma_{h', {\Xi \circ h'}}} \ar[d] &&  \overunderset{\displaystyle \CFHHS_{h} \left(\aContField ^{(Y,\mu)}, \anotContField ^{(Y,\mu)} \right)}{\displaystyle \CFHHS_{h'}\left(\anotContField ^{(Y,\mu)} , {\anotContField' }^{(Y,\mu)}\right)}{\times} \ar[d]  \\
			\displaystyle \prod_{k=0}^\infty	\CFHHS_{(h \circ h')|_{U_{\Xi \circ h'}^{(k)}}} \left(\aContField|_{h \left( U_{\Xi}^{(k)} \right)} , \anotContField|_{U_{\Xi \circ h'}^{(k)}} \right) 
			\ar[rr]^-{\Sigma_{h \circ h', \Xi \circ h'}} && \CFHHS_{h \circ h'} \left(\aContField ^{(Y,\mu)}, {\anotContField'} ^{(Y,\mu)} \right)
		}
		\]
		where the vertical maps are given by compositions (for each $k \in \mathbb{N}$ on the left). 
		
		\item \label{prop:continuum-product-field-stitch-morphisms:isomorphism}
		If, in addition, $h$ is homeomorphic and each $\varphi^{(k)}$ in \eqref{prop:continuum-product-field-stitch-morphisms:basic} is an isometric continuous isomorphism, then $\psi$ in \eqref{prop:continuum-product-field-stitch-morphisms:basic} is also an isometric continuous isomorphism. 
	\end{enumerate}
\end{prop}

\begin{proof}
	To prove \eqref{prop:continuum-product-field-stitch-morphisms:basic}, we adopt the notations in \Cref{def:continuum-product-field-augmentation} and apply \Cref{lem:continuum-product-field-maps} with the diagram in its proof replaced by the following one: 
	\[
	\adjustbox{scale=0.75,center}{
		\begin{tikzcd}[column sep=tiny,scale=1.7em]
		&& {s \in \spaceOfSections_{\operatorname{cont}} \left( \aContField|_{\baseSpace{\aContField} \times Y} \right)} & {\spaceOfSections_{\operatorname{cont}} \left( \aContField^{(Y,\mu)} \right)} & {\spaceOfSections_{\operatorname{cont}} \left( h^* \left( \aContField^{(Y,\mu)} \right) \right)} \\
		{\baseSpace{\aContField}} & {M \left( \baseSpace{\aContField} \times Y \right)} & {\aContField|_{\baseSpace{\aContField} \times Y}} & {\left( \left( \aContField|_{\baseSpace{\aContField} \times Y} \right)_{\operatorname{meas} }, \tau_\mu (h(z)) \right)} & {\displaystyle \prod_{z \in \baseSpace{\anotContField}} L^2 \left( \baseSpace{\aContField} \times Y , \tau_\mu (h(z)) , \aContField|_{\baseSpace{\aContField} \times Y} \right)} \\
		&&& {} & {} \\
		{\baseSpace{\anotContField}} & {M \left( \baseSpace{\anotContField} \times Y \right)} & {\anotContField|_{\baseSpace{\anotContField} \times Y}} & {\left( \left( \anotContField|_{\baseSpace{\anotContField} \times Y} \right)_{\operatorname{meas} }, \tau_\mu (z) \right)} & {\displaystyle \prod_{z \in \baseSpace{\anotContField}} L^2 \left( \baseSpace{\anotContField} \times Y , \tau_\mu (z) , \anotContField|_{\baseSpace{\anotContField} \times Y} \right)} \\
		&& {s \in \spaceOfSections_{\operatorname{cont}} \left( \anotContField|_{\baseSpace{\anotContField} \times Y} \right)} & {\spaceOfSections_{\operatorname{cont}} \left( \anotContField^{(Y,\mu)} \right)}
		\arrow["h", from=4-1, to=2-1]
		\arrow[squiggly, from=2-3, to=2-2]
		\arrow[squiggly, from=4-3, to=4-2]
		\arrow["{\tau_\mu}", from=2-1, to=2-2]
		\arrow["{\tau_\mu}"', from=4-1, to=4-2]
		\arrow[squiggly, from=2-3, to=1-3]
		\arrow[squiggly, from=4-3, to=5-3]
		\arrow["{\psi^z}"', from=2-4, to=4-4]
		\arrow["{\tau_\mu^*}", from=1-3, to=1-4]
		\arrow["{\tau_\mu^*}"', from=5-3, to=5-4]
		\arrow[squiggly, from=3-4, to=3-5]
		\arrow["{\displaystyle \prod_{z \in \baseSpace{\anotContField}} \psi^z_*}", from=2-5, to=4-5]
		\arrow[hook, from=5-4, to=4-5]
		\arrow["{h^*}", from=1-4, to=1-5]
		\arrow[hook, from=1-5, to=2-5]
		\arrow[squiggly, from=2-3, to=2-4]
		\arrow[squiggly, from=4-3, to=4-4]
		\end{tikzcd}
	}	
	\]
	where for any $z \in \baseSpace{\anotContField}$, $\psi^z$ is the isometric measure-preserving morphism in the sense of \Cref{defn:meas-fields-morphism} such that 
	\begin{itemize}
		\item $\operatorname{dom} \left( \baseSpace{\psi^z} \right)$ consists of all $(z' , y) \in \baseSpace{\anotContField} \times Y$ satisfying $z' = z$ and $z \in U_{\Xi (z) (y)}$, which is co-null with regard to $\tau_\mu(z) = \delta_z \times \mu$ since 
		\begin{align*}
		&\quad \left( \delta_z \times \mu \right) \left( \left(\baseSpace{\anotContField} \times Y \right) \setminus  \operatorname{dom} \left( \baseSpace{\psi^z} \right) \right)\\
		&= \mu \left(  \left\{ y \in Y \colon z \notin U_{\Xi (z) (y)}  \right\}\right) \\
		&\leq \sum_{k = 0}^{\infty} \mu \left(  \left\{ y \in Y \colon \Xi (z) (y) = k \text{ and } \mu \left( \left( \Xi (z) \right)^{-1} \left( k \right) \right) = 0   \right\}\right) \\
		&= 0
		\end{align*}
		by our assumption on $U^{(k)}$ for $k \in \{0,1,2,\ldots\}$. 
		
		\item $\baseSpace{\psi^z} = \left(h \times \operatorname{id}_Y\right)_{\operatorname{dom} \left( \baseSpace{\psi^z} \right)} \colon \operatorname{dom} \left( \baseSpace{\psi^z} \right) \to \baseSpace{\aContField} \times Y$, which thus satisfies 
		\[
		\left( \baseSpace{\psi^z} \right)_* \tau_\mu (z) = \left( h \times \operatorname{id}_Y \right)_* \left( \delta_z \times \mu \right) =  \delta_{h(z)} \times \mu = \tau_\mu (h(z)) \; , \quad \text{ and }
		\]
		
		\item for any $(z , y) \in \operatorname{dom} \left( \baseSpace{\psi^z} \right)$, we define the isometric embedding 
		\begin{align*}
		\left(\psi^z\right)_{(z,y)} \colon \left( \aContField |_{\baseSpace{\aContField} \times Y} \right)_{\baseSpace{\psi^z}(z,y)} = \aContField_{h(z)} &\to \anotContField_{z}  = \left( \anotContField |_{\baseSpace{\anotContField} \times Y} \right)_{(z,y)} \\
		x & \mapsto \varphi^{( \Xi (z) (y) )}_{z} (x) \; .
		\end{align*}
	\end{itemize}
	To check well-definedness,
	we follow \Cref{defn:meas-fields-morphism}, 
	and verify that for any $s \in \spaceOfSections_{\operatorname{meas}} \left( \left( \aContField|_{\baseSpace{\aContField} \times Y} \right)_{\operatorname{meas}} \right)$, 
	the section $\left(\left( \psi^z\right)_{(z,y)} \left( s \left( \underline{\psi^z}(z,y) \right) \right) \right)_{(z,y) \in \operatorname{dom}\left(\baseSpace{\varphi}\right)}$ agrees, $(\delta_z \times \mu)$-almost everywhere, with a section $s_1$ in $\spaceOfSections_{\operatorname{meas}} \left( \left( \anotContField|_{\baseSpace{\anotContField} \times Y} \right)_{\operatorname{meas}} \right)$. 
	Indeed, by our construction, we have  
	\[
	\left( \psi^z\right)_{(z,y)} \left( s \left( \underline{\psi^z}(z,y) \right) \right)  = \varphi^{( \Xi (z) (y) )}_{z} \left( s(h(z), y) \right) \quad \text{ for any }  (z,y) \in \operatorname{dom}\left(\baseSpace{\varphi}\right) \; .
	\]
	Hence if we choose an arbitrary $s_0 \in  \spaceOfSections_{\operatorname{meas}} \left( \left( \anotContField|_{\baseSpace{\anotContField} \times Y} \right)_{\operatorname{meas}} \right)$ and define $s_1 \in \prod_{(z, y) \in \baseSpace{\anotContField} \times Y} \anotContField_{z}$ by 
	\[
	s_1 (z,y) = 
	\begin{cases}
	\varphi^{( k )}_{z} \left( s(h(z), y) \right) \, , & \text{ if } (z,y) \in \operatorname{dom}\left(\baseSpace{\varphi}\right) \text{ and }  \Xi (z) (y) = k \\
	s_0 (z,y) \, , & \text{ otherwise}
	\end{cases}
	\; ,
	\]
	then $s_1$ agrees, $(\delta_z \times \mu)$-almost everywhere, with $s$, and 
	since $s_1$ is measurably piecewise defined via measurable sections, it is easily seen to be co-measurable with $\spaceOfSections_{\operatorname{meas}} \left( \left( \anotContField|_{\baseSpace{\anotContField} \times Y} \right)_{\operatorname{meas}} \right)$ and thus we have $s_1 \in  \spaceOfSections_{\operatorname{meas}} \left( \left( \anotContField|_{\baseSpace{\anotContField} \times Y} \right)_{\operatorname{meas}} \right)$ by \Cref{def:meas-field-HHS}\eqref{def:meas-field-HHS::complete}.

	In order to complete the application of \Cref{lem:continuum-product-field-maps}, it remains to verify that $s \in \spaceOfSections_{\operatorname{cont}} \left(  \aContField|_{\baseSpace{\aContField} \times Y} \right)$, 
	the section $\left( \psi^z_* \left( s  \right) \right)_{z \in \baseSpace{\anotContField}} \allowbreak \in \prod_{z \in \baseSpace{\anotContField}} L^2 \left( \baseSpace{\anotContField} \times Y , \tau_\mu (z) , \anotContField|_{\baseSpace{\anotContField} \times Y} \right)$ 
	is co-continuous with any $s' \allowbreak \in \spaceOfSections_{\operatorname{cont}} \left(  \anotContField|_{\baseSpace{\anotContField} \times Y} \right) \allowbreak \hookrightarrow \prod_{z \in \baseSpace{\anotContField}} L^2 \left( \baseSpace{\anotContField} \times Y , \tau_\mu (z) , \anotContField|_{\baseSpace{\anotContField} \times Y} \right)$. 
	Indeed, for any $z \in \baseSpace{\anotContField}$, it follows from \Cref{def:continuum-product}, the equation $\tau_\mu = \delta_z \times \mu$, our construction of $\psi^z$, the assumption on $U_\Xi^{(k)}$, and the equation $\baseSpace{\varphi^{(k)}} = h|_{U_\Xi^{(k)}}$ that 
	\begin{align*}
	& \ d_{L^2 \left( \baseSpace{\anotContField} \times Y , \tau_\mu (z) , \anotContField|_{\baseSpace{\anotContField} \times Y} \right)} \left(  \psi^z_* \left( s  \right)  , s' \right) \\
	&= \int_{(z', y) \in \baseSpace{\anotContField} \times Y}  d_{\anotContField_{z'}} \left(  \psi^z_* \left( s  \right) (z',y)  , s' (z',y) \right)   \, \mathrm{d} \tau_\mu(z) (z',y) \\
	&= \int_{y \in Y}  d_{\anotContField_{z}} \left(  \psi^z_* \left( s  \right) (z,y)  , s' (z,y) \right)   \, \mathrm{d} \mu (y)	\\
	&= \int_{y \in Y}  d_{\anotContField_{z}} \left( \left( \psi^z \right)_{(z,y)} \left( s \left( \baseSpace{\psi^z} (z,y) \right) \right)  , s' (z,y) \right)   \, \mathrm{d} \mu (y)	\\
	&= \int_{y \in Y}  d_{\anotContField_{z}} \left( \varphi^{( \Xi (z) (y) )}_{z} \left( s(h(z), y) \right)  , s' (z,y) \right)   \, \mathrm{d} \mu (y)	\\
	&= \sum_{k \in \mathbb{N} \colon \mu \left( \left( \Xi (z) \right)^{-1} (k) \right) > 0 } \int_{y \in \left( \Xi (z) \right)^{-1} (k) } d_{\anotContField_{z}} \left( \varphi^{( k )}_{z} \left( s(h(z), y) \right)  , s' (z,y) \right)   \, \mathrm{d} \mu (y)	\\
	&= \sum_{k \in \mathbb{N} \colon \mu \left( \left( \Xi (z) \right)^{-1} (k) \right) > 0 } \int_{y \in \left( \Xi (z) \right)^{-1} (k) }  d_{\anotContField_{z}} \left( \varphi^{( k )} \left( s(-, y)|_{h \left( U_\Xi^{(k)} \right)} \right) (z) , s' (z,y) \right)   \, \mathrm{d} \mu (y) \; ,
	\end{align*}
	where $s(-, y)|_{h \left( U_\Xi^{(k)} \right)} \in \spaceOfSections_{\operatorname{cont}} \left( \aContField|_{h \left( U_\Xi^{(k)} \right)} \right)$ is the image of $s$ under the isometric continuous morphism induced by the continuous map $h \left( U_\Xi^{(k)} \right) \mapsto \baseSpace{\aContField} \times Y$, $z \mapsto (z,y)$ in the sense of \Cref{def:cont-field-HHS-maps}. 
	For any $k \in \mathbb{N}$, 
	it follows from
	\Cref{lem:isometric-action-topologize}\eqref{lem:isometric-action-topologize:sections}
	that the maps $Y \to \spaceOfSections_{\operatorname{cont}} \left( \anotContField|_{ U_\Xi^{(k)} } \right)$, $y \mapsto s'(-, y)|_{U_\Xi^{(k)}}$ and $Y \to \spaceOfSections_{\operatorname{cont}} \left( \aContField|_{h \left( U_\Xi^{(k)} \right)} \right)$, $y \mapsto s(-, y)|_{h \left( U_\Xi^{(k)} \right)}$ (and thus also the map $Y \to \spaceOfSections_{\operatorname{cont}} \left( \anotContField|_{ U_\Xi^{(k)} } \right)$, $y \mapsto \varphi^{( k )} \left( s(-, y)|_{h \left( U_\Xi^{(k)} \right)} \right)$) are continuous, whence the map 
	\[
	Y \to C(U_\Xi^{(k)}, [0,\infty)) \, , \quad y \mapsto \left( z \mapsto d_{\anotContField_{z}} \left( \varphi^{( k )} \left( s(-, y)|_{h \left( U_\Xi^{(k)} \right)} \right) (z) , s' (z,y) \right)  \right)
	\]
	is continuous, where $C(U_\Xi^{(k)}, [0,\infty))$ denotes the collection of all continuous functions from $U_\Xi^{(k)}$ to $[0,\infty)$ and is equipped with the compact-open topology. This last statement is equivalent to saying that the map 
	\[
	U_\Xi^{(k)} \to C(Y , [0,\infty)) \, , \quad z \mapsto  \left( y \mapsto d_{\anotContField_{z}} \left( \varphi^{( k )} \left( s(-, y)|_{h \left( U_\Xi^{(k)} \right)} \right) (z) , s' (z,y) \right)  \right)
	\]
	is continuous. 
	Since $Y$ is compact, the compact-open topology on $C(Y , [0,\infty))$ agrees with the topology of the uniform norm. 
	Since $\Xi$ is continuous, it follows from \Cref{def:measurable-partition} that the map 
	\[
	U_\Xi^{(k)} \to L^1(Y,\mu) \, , \quad z \mapsto \left( y \mapsto \chi_{\left( \Xi (z) \right)^{-1} (k) } (y)    \right)
	\]
	is continuous. 
	Combining the above, we see that the product
	\[
	U_\Xi^{(k)} \to L^1(Y,\mu) \, , \quad z \mapsto \left( y \mapsto \chi_{\left( \Xi (z) \right)^{-1} (k) } (y) \cdot d_{\anotContField_{z}} \left( \varphi^{( k )} \left( s(-, y)|_{h \left( U_\Xi^{(k)} \right)} \right) (z) , s' (z,y) \right)   \right) 
	\]
	is also continuous, 
	whence so is the map
	\[
	U_\Xi^{(k)} \to [0,\infty) \, , \quad z \mapsto \int_{y \in \left( \Xi (z) \right)^{-1} (k) }  d_{\anotContField_{z}} \left( \varphi^{( k )} \left( s(-, y)|_{h \left( U_\Xi^{(k)} \right)} \right) (z) , s' (z,y) \right)   \, \mathrm{d} \mu (y) \; .
	\]
	Combining this with the long formula above, we see that the sections $\left( \psi^z_* \left( s  \right) \right)_{z \in \baseSpace{\anotContField}} $ and $s'$ are indeed co-continuous. 
	
	Hence the existence of the desired isometric continuous morphism $\psi \colon \aContField \to \anotContField$ follows from \Cref{lem:continuum-product-field-maps}.

	The proof of \eqref{prop:continuum-product-field-stitch-morphisms:functorial} is a direct computation using \eqref{prop:continuum-product-field-stitch-morphisms:basic}. 
	
	Finally, to prove \eqref{prop:continuum-product-field-stitch-morphisms:isomorphism}, we simply observes that the additional assumptions allow us to apply \Cref{cor:continuum-product-field-isomorphism}.
\end{proof}

%%%%%%%%%%%%%%%%%%%%%%%%%%%%%%%%%%

%%%%%%%%%%%%%%%%%%%%%%%%%>continuum-product-morphisms<%%%%%%%%%%%%%%%%%%%%%%%%%

%%%%%%%%%%%%%%%%%%%%%%%%%continuum_product%%%%%%%%%%%%%%%%%%%%%%%%%
%%%%%%%%%%%%%%%%%%%%%%%%%trivialization%%%%%%%%%%%%%%%%%%%%%%%%%
\section{Deformations and trivializations}
\label{sec:trivialization}

In this section, we 
investigate a few deformation techniques that allow us to trivialize continuous fields of {\hhs}s and actions thereupon. 
More precisely, starting from an arbitrary group action $\alpha$ on a certain \emph{locally trivial} continuous field $\aContField$ of {\hhs}s, we would like to achieve two ``trivializations'':
\begin{enumerate}
	\item We would like to strengthen the local triviality of $\aContField$ to triviality. 
	\item Then, we would like to find a homotopy connecting the action $\alpha$ and a ``fiberwise trivial'' action, i.e., one that merely shuffles the base space of $\aContField$ and does not move in the fiberwise direction (this makes sense when $\aContField$ is trivial). 
\end{enumerate}
Evidently, one cannot hope to accomplish even step~(1) without making suitable assumptions, due to the possible presence of global topological obstructions of fiber bundles. 
However, it turns out that after passing to randomizations (as in \Cref{def:randomization}), both goals can be achieved. 

We start by discussing a deformation technique regarding homotopy of actions, which extends \cite[Proposition~3.18]{GongWuYu2021}.

\begin{defn} \label{def:homotopy-actions}
	A \emph{homotopy} $\left( \alpha_t \right)_{t \in [0,1]}$ of isometric actions by a topological group $G$ on $\aContField$ is a family of isometric actions indexed by $[0,1]$ such that the map $G \times [0,1] \to \operatorname{Isom}(\aContField)$ given by $(g , t) \mapsto \alpha_{t, g}$ is continuous, where $\operatorname{Isom}(\aContField)$ is topologized as in \Cref{defn:isometric-action-topologize}.

	In this case, we say $\left( \alpha_t \right)_{t \in [0,1]}$ is a homotopy between $\alpha_0$ and $\alpha_1$. 
\end{defn}

Although the above definition is conceptually simple, the following equivalent characterization using the construction in \Cref{def:cont-field-HHS-maps} is all we need in practice, so the reader may as well take it as the definition to circumvent the details of the topology on $\operatorname{Isom}(\aContField)$ introduced in \Cref{defn:isometric-action-topologize}. 

\begin{prop}\label{lem:homotopy-actions}
	Let $G$ be a topological group and let $\aContField$ be a continuous field of {\hhs}s. 
	Then a family $\left( \alpha_{t,g} \right)_{t \in [0,1], g \in G}$ of isometric continuous morphisms from $\aContField$ to itself forms a homotopy $\left( \alpha_t \right)_{t \in [0,1]}$ of isometric actions if and only if it induces a continuous homomorphism $\widetilde{\alpha} \colon G \to \operatorname{Isom} \left( \aContField |_{\baseSpace{\aContField} \times [0,1]} \right)$ 
	such that 
	for any $g \in G$, any $t \in [0,1]$ and any $z \in \baseSpace{\aContField}$, we have
	\[
	\baseSpace{\widetilde{\alpha}_{g}} (z , t)  = \left( \baseSpace{\alpha_{t, g}} (z) , t \right) \quad \text{ and } \quad \left( \widetilde{\alpha}_g \right)_{(z,t)} = \left( \alpha_{t, g} \right)_{z} 
	\colon \aContField_{\baseSpace{\alpha_{t, g}} (z)} \to \aContField_z
	\; .
	\]
\end{prop}

\begin{proof}
	This is a direct consequence of \Cref{lem:isometric-action-topological-group-curry}. 
\end{proof}

\begin{rmk}\label{rmk:homotopy-actions}
	Although we do not need this fact, we point out that for a discrete group $\Gamma$, we may use \Cref{lem:isometric-action-topologize} to derive a more concrete characterization: a family $\left( \alpha_t \right)_{t \in [0,1]}$ of isometric actions forms a homotopy if and only if 
	it is \emph{pointwise continuous} in the sense that for any $\gamma \in \Gamma$ and any continuous section $s \in \spaceOfSections_{\operatorname{cont}} (\aContField)$, 
	the map $\baseSpace{\aContField} \times [0,1] \to \baseSpace{\aContField}$, $(z,t) \mapsto \baseSpace{\alpha_{t, \gamma}} (z)$, is continuous, and
	we have 
	\[
	\left( \alpha_{t, \gamma}(s) \right)_{t \in [0,1]} \in \spaceOfSections_{\operatorname{cont}} ( \aContField |_{\baseSpace{\aContField} \times [0,1]})
	\]
	inside $\prod_{(z, t) \in \baseSpace{\aContField} \times [0,1]} \aContField_z$. 
\end{rmk}

\begin{lem}\label{lem:deformation}
	Let $\aContField$ be a second-countable 
	continuous field of {\hhs}s. 
	Let $\alpha_0$ and $\alpha_1$ be two isometric actions of a group $G$ on $\aContField$ such that $\baseSpace{\alpha_0} = \baseSpace{\alpha_1}$, i.e., they induce the same action on the base space $\baseSpace{\aContField}$. 
	Consider the induced actions $\beta_0 := \alpha_0^{(\Omega, \mu)}$ and $\beta_1 := \alpha_1^{(\Omega, \mu)}$ of $\Gamma$ on the randomization $\aContField^{(\Omega, \mu)}$ as given in \Cref{lem:continuum-product-field-augmentation-maps}\eqref{lem:continuum-product-field-augmentation-maps:action}. 
	Then there is a homotopy $\left( \beta_t \right)_{t \in [0,1]}$ of isometric actions by $G$ on $\aContField^{(\Omega, \mu)}$ 
	such that $\baseSpace{\beta_0} = \baseSpace{\beta_t}$ for any $t \in [0,1]$. 
\end{lem}

\begin{proof}
	By \Cref{rmk:randomization-stable}, it suffices to take $(\Omega, \mu)$ to be $([0,1],m)$. 
	Note that $\baseSpace{\beta_0} = \baseSpace{\alpha_0} = \baseSpace{\alpha_1} = \baseSpace{\beta_1}$. 
	In view of \Cref{lem:homotopy-actions}, it suffices to construct 
	a continuous homomorphism $\beta \colon G \to \operatorname{Isom} \left( \aContField^{[0,1]} |_{\baseSpace{\aContField} \times [0,1]} \right)$ 
	such that for any $g \in G$, any $t \in [0,1]$ and any $z \in \baseSpace{\aContField}$, we have $\baseSpace{\beta}_g(z,t) = \left( \baseSpace{\beta_{0,g}}(z), t \right)$ and $(\beta_g)_{z,k} = \left( \beta_{k, g} \right)_{z}$ for $k = 0, 1$. 
	Since $\baseSpace{\beta}$ has been determined this way, in the terminologies of \Cref{defn:cont-fields-morphism} and \Cref{def:cont-field-HHS-maps}, it remains to construct isometric bijections $(\beta_g)_{z,t} \colon \aContField^{[0,1]}_{\beta_{0,g}(z)} \to \aContField^{[0,1]}_z$ for each $g \in G$, each $z \in \baseSpace{\aContField}$ and each $t \in [0,1]$. 
	
	To this end, we first view $\alpha_0$ and $\alpha_1$ as constant homotopies, i.e., they correspond to maps $G \times [0,1] \to \operatorname{Isom} (\aContField)$ given by $(g,t) \mapsto \alpha_{0,g}$ and $(g,t) \mapsto \alpha_{1,g}$, respectively. By \Cref{lem:homotopy-actions}, they induce actions $\widetilde{\alpha}_0, \widetilde{\alpha}_1 \colon G \to \operatorname{Isom} \left( \aContField |_{\baseSpace{\aContField} \times [0,1]} \right)$, respectively.  
	Note that for any $t \in [0,1]$ and any $z \in \baseSpace{\aContField}$, we have $\baseSpace{\widetilde{\alpha}_{0,g}}(z,t) = \left( \baseSpace{\alpha_{0,g}}(z), t \right) = \baseSpace{\beta}_g(z,t) = \left( \baseSpace{\alpha_{1,g}}(z), t \right) = \baseSpace{\widetilde{\alpha}_{1,g}}(z,t)$, that is, $\baseSpace{\widetilde{\alpha}_{0}} = \baseSpace{\beta} = \baseSpace{\widetilde{\alpha}_{1}}$.

	With notations as in \Cref{def:measurable-partition}, consider the continuous map $\Xi \colon \baseSpace{\aContField} \times [0,1] \to \mathcal{P}_2([0,1], m) \subset \mathcal{P}_\omega([0,1], m)$ taking $(z,t)$ to the measurable $2$-partition that labels $[0,1-t]$ by $0$ and $(1-t,1]$ by $1$. 
	In the terminology of \Cref{prop:continuum-product-field-stitch-morphisms}, for each $g \in G$, we have, with regard to $\baseSpace{{\beta}_{g}}$, the equations $U_{\Xi}^{(0)} = \baseSpace{\aContField} \times [0,1)$ and $U_{\Xi}^{(1)} = \baseSpace{\aContField} \times (0,1]$, both sets being dense in $\baseSpace{\aContField} \times [0,1]$ and invariant under $\baseSpace{{\beta}_{g}}$, while $U_{\Xi}^{(k)} = \varnothing$ for $k \geq 2$. 
	Hence, for each $g \in G$, we may apply \Cref{prop:continuum-product-field-stitch-morphisms}\eqref{prop:continuum-product-field-stitch-morphisms:basic} to $\Xi$ and $\baseSpace{\beta}_g$ to obtain a map 
	\[
	{\Sigma}_{\baseSpace{{\beta}_{g}}, \Xi} \colon 
	\prod_{k=0}^1	\mathrm{Isom}_{\baseSpace{{\beta}_{g}|_{U_{\Xi}^{(k)}}}} \left( \aContField |_{\baseSpace{\aContField} \times [0,1]}  \right) 
	\to \mathrm{Isom}_{\baseSpace{{\beta}_{g}}} \left( \aContField^{[0,1]} |_{\baseSpace{\aContField} \times [0,1]} \right)  \; .
	\]
	Since for any $z \in \baseSpace{\aContField}$ and any $k \in \{0,1\}$, $\Xi (z,k) $ is the trivial partition that maps the entire interval $[0,1]$ to $k$ (up to a null set), we have 
	\[
	{\Sigma}_{\baseSpace{{\beta}_{g}}, \Xi} \left( \widetilde{\alpha}_{0,g} , \widetilde{\alpha}_{1,g}  \right) = \alpha_{k,g}^{[0,1]} = \beta_{k,g} \colon \aContField^{[0,1]}_{\baseSpace{\beta_{k,g}}(z)} \to \aContField^{[0,1]}_z \; .
	\]
	Now for any $g \in G$, we define 
	\[
	{\beta}_{g} = {\Sigma}_{\baseSpace{{\beta}_{g}}, \Xi} \left( \widetilde{\alpha}_{0,g} , \widetilde{\alpha}_{1,g} \right)_{z,k} \in \mathrm{Isom}_{\baseSpace{{\beta}_{g}}} \left( \aContField^{[0,1]} |_{\baseSpace{\aContField} \times [0,1]} \right) \; .
	\]
	It follows from the functorial properties in \Cref{prop:continuum-product-field-stitch-morphisms}\eqref{prop:continuum-product-field-stitch-morphisms:functorial} that the assignment $g \to \beta_g$ defines an action $\beta \colon G \to \mathrm{Isom} \left( \aContField^{[0,1]} |_{\baseSpace{\aContField} \times [0,1]} \right) $ satisfying the prescribed conditions.

	It remains to prove $\beta$ is continuous.   
	To this end, 
	we observe that since 
	for any $g \in G$, any $t \in [0,1]$ and any $z \in \baseSpace{\aContField}$, we have $\baseSpace{\beta}_g(z,t) = \left( \baseSpace{\beta_{0,g}}(z), t \right)$, which implies that, as a map from $G$ to the group 
	of homeomorphisms on $\baseSpace{\aContField} \times [0,1]$, $\baseSpace{\beta}$ is continuous in the compact open topology. 
	Hence in view of \Cref{defn:isometric-action-topologize}, it suffices to show that for any 
	$g \in G$, 
	any $s \in \spaceOfSections_{\operatorname{cont}} \left( \aContField^{[0,1]} |_{\baseSpace{\aContField} \times [0,1]} \right) $, any compact subset $K \subseteq \baseSpace{\aContField}$ and any $\varepsilon > 0$, the set $U$ given by
	\begin{align*}
	\Bigg\{ g' \in G   \colon 
	d_{\aContField^{[0,1]}_z} \left( {\Sigma}_{\baseSpace{{\beta}_{g}}, \Xi} \left( \widetilde{\alpha}_{0,g} , \widetilde{\alpha}_{1,g}  \right) (s) (z, t) , {\Sigma}_{\baseSpace{{\beta}_{g'}}, \Xi} \left( \widetilde{\alpha}_{0,g'} , \widetilde{\alpha}_{1,g'}  \right) (s) (z, t) \right) < \varepsilon &\ \\
	\text{ for any } z \in K \text{ and } t \in [0,1] &\ \Bigg\}
	\end{align*}
	is a neighborhood of $g$. 
	
	In order to help distinguish the two copies of the unit interval $[0,1]$, let us retain the notation $Y$ for the copy of $[0,1]$ used in the continuum power construction (but still mostly suppress the notation $m$ for the Lebesgue measure for the sake of convenience). 
	Now we observe that 
	by \Cref{def:continuum-product-field-augmentation} and \Cref{def:continuum-product-field}, 
	for any $(z,t) \in \baseSpace{\aContField} \times [0,1]$,    
	there are an open neighborhood $V_{z,t}$ of $(z,t)$ and 
	$s_{z,t} \in \spaceOfSections_{\operatorname{cont}} \left( \aContField|_{\baseSpace{\aContField} \times [0,1] \times Y} \right)$ such that
	for any $(z',t') \in V_z$, 
	if we 
	identify $\left( \aContField^{Y} |_{\baseSpace{\aContField} \times [0,1]} \right)_{z',t'}$ with $L^2(Y, m, \aContField_{z'}) $
	and 
	consider the commutative diagram 
	\[
	\xymatrix{
		&\spaceOfSections_{\operatorname{cont}} \left( \aContField|_{\baseSpace{\aContField} \times [0,1] \times Y } \right) \ar[r] \ar[dd] & \left( \aContField^{Y} |_{\baseSpace{\aContField} \times [0,1]} \right)_{z',t'} \ar@{=}[d]  \\
		&& L^2 \left(\baseSpace{\aContField} \times [0,1] \times Y, \delta_{z'} \times \delta_{t'} \times m, \aContField|_{\baseSpace{\aContField} \times [0,1]} \right) \ar[d]^{\cong} \\
		&\spaceOfSections_{\operatorname{cont}} \left( \aContField|_{ \left\{ (z',t') \right\} \times Y} \right) \ar[r]^{\iota} & L^2(Y, m, \aContField_{z'})
	}
	\]
	where the horizontal maps are the canonical maps as in \Cref{lem:continuous-field-HHS-L2-dense} and the left vertical map is induced by the inclusion between the base spaces, 
	then we have 
	\[
	d_{L^2(Y, m, \aContField_{z'}) } \left( s(z',t') , \iota \left(  s_{z,t}|_{\left\{ (z',t') \right\} \times Y} \right)  \right) < \frac{\varepsilon}{5} \; .
	\]
	By \Cref{def:cont-field-HHS-maps} and \Cref{lem:cont-field-HHS-Diximier-generate}, for any $y \in Y$, there are open neighborhoods $V_{z,t,y} \subseteq V_{z,t}$ of $(z,t)$ in $\baseSpace{\aContField} \times [0,1]$ and $W_{z,t,y}$ of $y$ in $Y$ as well as $s_{z,t,y} \in \spaceOfSections_{\operatorname{cont}} (\aContField)$ such that for any $(z',t') \in V_{z,t,y}$ and $y' \in W_{z,t,y}$, we have 
	\[
	d_{\aContField_{z'}} \left( s_{z,t} (z', t', y') , s_{z,t,y} (z') \right) < \frac{\varepsilon}{5 } \; .
	\] 
	By the compactness of $K \times [0,1]$ and $Y$, there are finite subsets $F' \subseteq \baseSpace{\aContField} \times [0,1]$ and $F'' \subseteq Y$ such that 
	\[
	K \times [0,1] \times Y \subseteq \bigcup_{(z,t) \in F'} \bigcup_{y \in F''} V_{z,t,y} \times W_{z,t,y} \; .
	\]
	By the continuity of the actions $\alpha_0$ and $\alpha_1$, the set $U'$ given by 
	\begin{align*}
	\Bigg\{ g' \in G \colon 
	&\ d_{\aContField_{z'}} \left(   {\alpha}_{k,g} \left( s_{z,t,y}  \right) (z') ,   {\alpha}_{k,g'}  \left( s_{z,t,y}  \right) (z') \right) < \frac{\varepsilon}{5}  \\
	&\ \text{ for any } k \in \{0,1\} , \ (z,t) \in F' , \ y \in F'', \text{ and } z' \in K  \Bigg\}
	\end{align*}
	is an open neighborhood of $g$ in $G$. 
	The desired result thus follows from the containment $U' \subseteq U$, which holds since for any $g' \in U'$ and any $(z,t) \in K \times [0,1]$, if we choose Borel maps $p \colon K \times [0,1] \to F'$ and $q \colon Y \to F''$ satisfying $p^{-1} (z,t) \subseteq V_{z,t,y}$ and $q^{-1} (y) \subseteq W_{z,t,y}$ for any $z \in K$, $t \in [0,1]$ and $y \in Y$, then we have 
	\begin{align*}
	&\ d_{\aContField^{Y}_z} \left( {\Sigma}_{\baseSpace{{\beta}_{g}}, \Xi} \left( \widetilde{\alpha}_{0,g} , \widetilde{\alpha}_{1,g}  \right) (s) (z, t) , {\Sigma}_{\baseSpace{{\beta}_{g'}}, \Xi} \left( \widetilde{\alpha}_{0,g'} , \widetilde{\alpha}_{1,g'}  \right) (s) (z, t) \right) \\
	< &\ d_{\aContField^{Y}_z} \left( {\Sigma}_{\baseSpace{{\beta}_{g}}, \Xi} \left( \widetilde{\alpha}_{0,g} , \widetilde{\alpha}_{1,g}  \right) (s_{p(z,t)}) (z,t)  , {\Sigma}_{\baseSpace{{\beta}_{g'}}, \Xi} \left( \widetilde{\alpha}_{0,g'} , \widetilde{\alpha}_{1,g'}  \right) (s_{p(z,t)}) (z,t) \right)  + \frac{2 \varepsilon}{5} \\
	= &\ \int_{Y} \, d_{\aContField_{z}} \left( \alpha_{\Xi (z) (y), g} \left( s_{p(z,t)} (z, t, y) \right) , \alpha_{\Xi (z) (y), g'}  \left( s_{p(z,t)} (z, t, y) \right)  \right) \mathrm{d} y + \frac{2 \varepsilon}{5} \\
	\leq &\ \int_{Y} \, d_{\aContField_{z}} \left( \alpha_{\Xi (z) (y), g}  \left( s_{p(z,t),q(y)} (z) \right) , \alpha_{\Xi (z) (y), g'} \left( s_{p(z,t),q(y)} (z) \right)  \right) \mathrm{d} y + \frac{4 \varepsilon}{5} \\
	< &\  \varepsilon 
	\; ,
	%%%%%%%%%%%%%%%%%%%%%%%%%%%%%%%%%%%%%%
	\end{align*}
	by our constructions above. 
\end{proof}

%%%%%%%%%%%%%%%%%%%%%%%%%%%%%%%%%%%%%%%

Next we show that analogous to \cite[Th\'{e}or\`{e}me~1 on page~252]{DixmierDouady1963Champs}, for randomization-stable continuous fields of {\hhs}s, local triviality implies triviality. This suggests that for these objects there is no global topological obstruction to triviality. Let us make this precise. 

\begin{defn} \label{def:locally-trivial}
	A continuous field $\aContField$ of {\hhs}s is said to be \emph{locally trivial} if for any $z \in \baseSpace{\aContField}$, there is an open neighborhood $U$ of $z$ in $\baseSpace{\aContField}$ such that the reduction $\aContField |_{U}$ (see \Cref{def:cont-field-HHS-maps}) is trivial, i.e., it is isometrically continuously isomorphic to a constant continuous field of {\hhs}s over $U$ (see \Cref{eg:cont-fields-trivial}). 
\end{defn}

\begin{rmk} \label{rmk:locally-trivial-randomization}
	It is clear that if a second-countable continuous field $\aContField$ of {\hhs}s is locally trivial, then so is its randomization $\aContField^{(\Omega, \mu)}$. 
\end{rmk}

\begin{prop} \label{prop:quasitrivial-locally-trivial}
	A randomization-stable continuous field $\aContField$ of {\hhs}s is trivial if and only if it is locally trivial and all fibers $\aContField_z$ are isometric to the same {\hhs}. 
\end{prop}

\begin{proof}
	The ``only if'' direction is clear. Let us prove the ``if'' direction. 
	
	By our assumption, there is a {\hhs} $X$ such that each fiber of $\aContField$ is isometric to $X$. 
	Moreover, since $\aContField$ is locally trivial, 
	by \Cref{def:locally-trivial}, there is an open cover $\mathcal{U}$ of $\baseSpace{\aContField}$ such that $\aContField|_U$ is trivial for any $U \in \mathcal{U}$, that is, there is an isometric continuous morphism $\varphi_U \colon \aContField|_U \to (X)_U$. 
	
	Since $\baseSpace{\aContField}$ is paracompact, 
	there is an open refinement $\mathcal{V} = \bigcup_{k=1}^\infty \mathcal{V}^{(k)}$ of $\mathcal{U}$ such that each $\mathcal{V}^{(k)}$ consists of disjoint open sets, and there is a partition of unity $\left( f_{V} \right)_{V \in \mathcal{V}}$ subordinate to $\mathcal{V}$, i.e., for each $V \in \mathcal{V}$, the support of $f_V$ is contained in $V$, and for each $z \in \baseSpace{\aContField}$, there is a neighborhood $W$ of $z$ that intersects only finitely many members of $\mathcal{V}$ and $\sum_{V \in \mathcal{V}} f_V (z) = 1$. 
	
	Note that since $\mathcal{V}$ refines $\mathcal{U}$, there is an isometric continuous morphism $\varphi_V \colon \aContField|_V \to (X)_V$. 
	For any positive integer $k$, since $\mathcal{V}^{(k)}$ consists of disjoint open sets, thus if we write $V_k = \bigcup \left\{ V \in \mathcal{V}^{(k)} \right\}$, then we may combine the $\varphi_V$'s to obtain an isometric continuous morphism $\varphi^{(k)} \colon \aContField|_{V_k} \to (X)_{V_k}$.
	
	Define a sequence $(F_k)_{k = 0}^\infty$ of continuous functions $\baseSpace{\aContField} \to [0,1]$ such that 
	\[
	F_k (z) = \sum_{j=1}^{k} \sum_{V \in \mathcal{V}^{(k)}} f_V(z) \; .
	\]
	Note that $0 = F_0 \leq F_1 \leq F_2 \leq \ldots \leq 1$ and this chain of inequalities stabilizes at any $z \in \baseSpace{\aContField}$. Let $m$ denote the Lebesgue measure on $[0,1]$. We then obtain a continuous map $\Xi \colon \baseSpace{\aContField} \to \mathcal{P}_\omega([0,1], m)$ such that 
	\[
	\Xi(z)(t) = \min \left\{ k \in \{0,1,2,\ldots\} \colon t \leq F_k (z) \right\} \; .
	\]
	Observe that $\mu \left( \left( \Xi (z) \right)^{-1} (k) \right) = f_k (z)$ for any $k \in \{0,1,2,\ldots\}$ and any $z \in \baseSpace{\aContField}$. 
	
	We then apply \Cref{prop:continuum-product-field-stitch-morphisms}, with $h$ 
	replaced by $\operatorname{id}_{\baseSpace{\aContField}}$, 
	to obtain an isometric continuous isomorphism $\psi \colon \aContField \to (X)_{\baseSpace{\aContField}}$. 
\end{proof}

\begin{rmk}
	\Cref{prop:quasitrivial-locally-trivial} may also be viewed as a consequence of \Cref{lem:deformation}, by the following argument (which we sketch): 
	
	If $\aContField$ is a locally trivial continuous field of {\hhs}s with all fibers isometric to the same {\hhs} $\aHHS$, then $\aContField$ is induced from an $\operatorname{Isom}(\aHHS)$-principal bundle $\calB$. It follows that $\aContField^{[0,1]}$ is induced from an $\operatorname{Isom}(\aHHS^{[0,1]})$-principal bundle $\calB'$ which arises through the monomorphism $\operatorname{Isom}(\aHHS) \hookrightarrow \operatorname{Isom}(\aHHS^{[0,1]})$. By \Cref{lem:deformation}, there is a homotopy from this monomorphism to the trivial homomorphism, which implies that $\calB'$ is trivial, whence $\aContField^{[0,1]}$ is trivial. 
\end{rmk}

\begin{cor} \label{cor:quasitrivial-locally-trivial}
	If a continuous field $\aContField$ of {\hhs}s is locally trivial and all fibers $\aContField_z$ are isometric to the same {\hhs}, then its randomization is trivial. 
\end{cor}

\begin{proof}
	Note that the randomization of $\aContField$ is randomization-stable by \Cref{rmk:randomization-stable}. 
	The statement thus follows immediately from \Cref{rmk:locally-trivial-randomization} and the ``if'' direction of \Cref{prop:quasitrivial-locally-trivial}. 
\end{proof}

Next we show that for randomization-stable continuous fields of {\hhs}s, triviality is preserved under the probability space construction as in \Cref{def:continuum-product-field}.

\begin{prop}\label{prop:trivialize-prob-field}
	A randomization-stable continuous field $\aContField$ of {\hhs}s with a compact base space is trivial if and only if $\aContField|_{\operatorname{Prob}{\baseSpace{\aContField}}}$ is trivial. 
\end{prop}

The nontrivial part is the ``only if'' direction, since even the local triviality of $\aContField|_{\operatorname{Prob}{\baseSpace{\aContField}}}$ is far from clear. 
Before we dive into the details, let us briefly explain the key idea of the proof of  \Cref{prop:trivialize-prob-field} in a special case that is most relevant for the main results of the paper. Let us start with  $\operatorname{Riem}(N)$, which is the continuous field of {\hhs}s with the base space  $N$ and fiber $\operatorname{Riem}(N)_z := \operatorname{GL}(T_z N)/\operatorname{O}(T_zN)$ for any $z \in N$. Here $\operatorname{GL}(T_z N)/\operatorname{O}(T_zN)$ is simply the space of all inner products on the tangent space $T_zN$ of $N$ at $z$. The randomization  $\aContField= \operatorname{Riem}(N)^{[0, 1]}$ of $\operatorname{Riem}(N)$ (in the sense of \Cref{def:randomization}, cf. \Cref{rmk:randomization-stable}) is a continous field of {\hhs}s with the base space  $N$ and fiber $L^2$-integrable functions from $[0, 1]$ to $\operatorname{GL}(n, \mathbb R)/\operatorname{O}(n)$, where $n = \dim N$.  Now for the continous field $\aContField|_{\operatorname{Prob}{\baseSpace{\aContField}}}$ (cf. \Cref{def:continuum-product-field}), the fiber over $\mu \in \operatorname{Prob}{\baseSpace{\aContField}}$ is the $L^2$-integrable functions from $(N\times [0, 1], \mu\times m)$ to $\operatorname{GL}(n, \mathbb R)/\operatorname{O}(n)$, where $m$ is the standard Lebesgue measure on $[0, 1]$. Recall that all standard probability spaces are isomorphic as measure spaces. In particular, $(N\times [0, 1], \mu\times m)$ is isomorphic to $([0, 1], m)$  for any $\mu \in \operatorname{Prob}{\baseSpace{\aContField}}$. Now to show that $\aContField|_{\operatorname{Prob}{\baseSpace{\aContField}}}$ is trivial, our key step is to systematically  design  the isomorphism between $(N\times [0, 1], \mu\times m)$ and  $([0, 1], m)$ so that it varies continuously as $\mu$ varies in $ \operatorname{Prob}{\baseSpace{\aContField}}$. This will be accomplished in  \Cref{lem:trivialization-prob-continuous-partition} and  \Cref{lem:trivialization-prob-interval-trick}. Finally,  \Cref{prop:trivialize-prob-field} is proved by  reducing it to \Cref{lem:trivialize-prob-field-specify}.

Let us now proceed to prove \Cref{prop:trivialize-prob-field}. Recall that 
on the unit ball of $L^\infty([0,1])$, convergence in the $L^2$-norm is equivalent to convergence in the strong operator topology, i.e., the weakest topology such that $L^\infty([0,1]) \ni f \mapsto f \xi \in L^2([0,1])$ is continuous for any $\xi \in L^2([0,1])$.

\begin{lem} \label{lem:trivialization-prob-continuous-partition}
	Let $(Z,d)$ be a compact metric space. 
	Equip $[0,1] \times Z$ with the Manhattan metric, i.e., 
	\[
	d \left( \left( t_1, z_1 \right) , \left( t_2, z_2 \right) \right) := \left| t_1 - t_2 \right| + d \left( z_1, z_2 \right) 
	\quad \text{ for } t_1, t_2 \in [0,1] \text{ and } z_1, z_2 \in Z \; .
	\]
	Then there is a sequence $\left( \calP^{(k)} \right)_{k \in \mathbb{N}}$ of finite Borel partitions of $[0,1] \times Z$ satisfying
	\begin{enumerate}
		\item \label{lem:trivialization-prob-continuous-partition:monotone} $\calP^{({k+1})}$ refines $\calP^{({k})}$ for any $k \in \mathbb{N}$, 
		
		\item \label{lem:trivialization-prob-continuous-partition:separating} $\displaystyle \lim_{k \to \infty} \sup \left\{ \operatorname{diam} (P) \colon P \in \calP^{(k)} \right\} = 0$, and 
		
		\item \label{lem:trivialization-prob-continuous-partition:continuous} for any $k \in \mathbb{N}$ and any $P \in \calP^{({k})}$, the function 
		\[
		Z \to L^\infty([0,1]) \subseteq L^2([0,1]) \, , \quad z \mapsto \chi_{P_z} \; ,
		\]
		where $P_z := \left\{ t \in [0,1] \colon (t,z) \in P \right\}$, 
		is continuous in the $L^2$-norm. 
	\end{enumerate}
\end{lem}

\begin{proof}
	We first construct a sequence $\left( \calQ^{(k)} \right)_{k \in \mathbb{N}}$ of finite Borel partitions of $[0,1] \times Z$ satisfying \eqref{lem:trivialization-prob-continuous-partition:separating} and \eqref{lem:trivialization-prob-continuous-partition:continuous} in place of $\left( \calP^{(k)} \right)_{k \in \mathbb{N}}$. 
	To this end, we choose a family of finite open covers $\calU^{(k)} = \left\{ U^{(k)}_1 , \ldots, U^{(k)}_{m^{(k)}} \right\}$ of $Z$, for $k = 1,2,\ldots$, such that $\operatorname{diam} \left( U^{(k)}_i \right) \leq \frac{1}{k+1}$ for any $k \in \mathbb{N}$ and $i \in \left\{ 0, 1, \ldots, m^{(k)}-1 \right\}$, 
	and then 
	choose a partition of unity $ \left\{ h^{(k)}_1 , \ldots, h^{(k)}_{m^{(k)}} \right\}$ subordinate to $\calU^{(k)}$. 
	Define $I^{(k)} := \left\{ 1, \ldots, k+1 \right\} \times \left\{ 1, \ldots, m^{(k)} \right\}$ and equip it with the alphabetical order, i.e., for any $(i,j), (i',j') \in I^{(k)}$, we have $(i,j) < (i',j')$ if and only if $i < i'$ or $i = i'$ and $j < j'$. 
	For any $(i,j) \in I^{(k)}$, define continuous functions 
	\begin{align*}
	H^{(k)}_{i,j,+} \colon Z \to [0,1] \, , \quad z \mapsto & \sum_{ \footnotesize{\begin{matrix} (i',j') \in I^{(k)} \colon \\ (i',j') \leq (i,j) 			\end{matrix}} } \frac{ h^{(k)}_{j'} (z) }{k+1} 
	&=& \  \frac{ i-1 }{k+1} + \sum_{j' = 1}^{j} \frac{ h^{(k)}_{j'} (z) }{k+1} \\
	H^{(k)}_{i,j,-} \colon Z \to [0,1] \, , \quad z \mapsto & \sum_{ \footnotesize{\begin{matrix} (i',j') \in I^{(k)} \colon \\ (i',j') < (i,j) 			\end{matrix}} } \frac{ h^{(k)}_{j'} (z) }{k+1} 
	&=& \  \frac{ i-1 }{k+1} + \sum_{j' = 1}^{j-1} \frac{ h^{(k)}_{j'} (z) }{k+1} 
	\end{align*}
	and define $Q^{(k)}_{i,j}$ to be the Borel set 
	\[
	\left\{ (t,z) \in [0,1] \times Z \colon H^{(k)}_{i,j,-} (z)  < t \leq H^{(k)}_{i,j,+} (z) \text{, or } t = 0 \text{ if } H^{(k)}_{i,j,-} (z) = 0 \right\} \; ,
	\]
	which has diameter no more than $\frac{2}{k+1}$, since for any $z \in Z$, we have 
	\[
	\frac{ i-1 }{k+1} \leq H^{(k)}_{i,j,-} (z) \leq H^{(k)}_{i,j,+} (z) \leq \frac{ i }{k+1} 
	\]
	and 
	\[
	H^{(k)}_{i,j,+} (z) - H^{(k)}_{i,j,-} (z) = \frac{ h^{(k)}_{j} (z) }{k+1}  \; ,
	\]
	the latter being supported in $U^{(k)}_j$. 
	For any $(i,j) \in I^{(k)}$, since both $H^{(k)}_{i,j,+}$ and $H^{(k)}_{i,j,-}$ are continuous, it is clear that the function 
	\[
	Z \to L^\infty([0,1]) \subseteq L^2([0,1]) \, , \quad z \mapsto \chi_{\left(Q^{(k)}_{i,j}\right)_z} = \chi_{ \left[ H^{(k)}_{i,j,-} (z), H^{(k)}_{i,j,+} (z) \right] } 
	\]
	is continuous in the $L^2$-norm. 
	Since we have $H^{(k)}_{1,1,-} = 0$, $H^{(k)}_{k+1,m^{(k)},+} = 1$, and $H^{(k)}_{i,j,-} = H^{(k)}_{i',j',+}$ whenever $(i,j)$ is the successor of $(i',j')$, it follows that 
	\[
	\calQ^{(k)} := \left\{ Q^{(k)}_{i,j} \colon (i,j) \in I^{(k)} \right\}
	\]
	forms a finite Borel partition of $[0,1] \times Z$, which, as we just proved, satisfies \eqref{lem:trivialization-prob-continuous-partition:separating} and \eqref{lem:trivialization-prob-continuous-partition:continuous}.

	Finally, we let $\calP^{(0)} := \calQ^{(0)}$ and inductively define, for $k = 0, 1, 2, \ldots$, the partition $\calP^{(k)}$ to be the common refinement of $\calP^{(k-1)}$ and $\calQ^{(k)}$, i.e., 
	\[
	\calP^{(k)} := \left\{ P \cap Q \colon P \in \calP^{(k-1)} , \, Q \in \calQ^{(k)} \right\} \; .
	\]
	It follows that $\left( \calP^{(k)} \right)_{k \in \mathbb{N}}$ is a sequence of finite Borel partitions of $[0,1] \times Z$ satisfying \eqref{lem:trivialization-prob-continuous-partition:monotone}, \eqref{lem:trivialization-prob-continuous-partition:separating} and \eqref{lem:trivialization-prob-continuous-partition:continuous}, where to prove \eqref{lem:trivialization-prob-continuous-partition:continuous}, we may use the remark before the lemma to turn $L^2$-continuity into continuity in the strong operator topology and then use the property that this topology is preserved by multiplication in $L^\infty([0,1])$. 
\end{proof}

\begin{lem} \label{lem:trivialization-prob-interval-trick}
	Let $Z$ be a compact metrizable space. 
	Let $m$ be the Lebesgue measure on $[0,1]$. 
	Then there is a family 
	$$\left( k^\mu \colon ([0,1], m) \to ([0,1] \times Z,  m \times \mu) \right)_{\mu \in \operatorname{Prob}(Z)}$$ 
	of measure space isomorphisms mod 0 that is continuous in the sense that for any continuous function $h \in C([0,1] \times Z)$, the assignment $\mu \mapsto h \circ k^\mu$ yields a continuous map $\operatorname{Prob}(Z) \to L^2([0,1], m)$.
\end{lem}

\begin{proof}
	
	Since both $[0,1]$ and $[0,1] \times Z$ are standard Borel spaces, by \cite[17.F]{Kechris1995}, any measure-preserving isomorphism between the measure algebras associated to $([0,1], m)$ and $([0,1] \times Z , m \times \mu )$ gives rise, contravariantly, to a measure space isomorphism mod 0 between $([0,1], m)$ and $([0,1] \times Z , m \times \mu )$. It follows that any normal $*$-isomorphism $\varphi \colon L^\infty([0,1] \times Z , m \times \mu ) \to L^\infty([0,1], m)$ of von Neumann algebras that intertwines the two measures gives rise to a measure space isomorphism (mod 0) $k \colon ([0,1], m) \to ([0,1] \times Z , m \times \mu )$ such that $\varphi (f) = f \circ h$ for any $f \in L^\infty([0,1] \times Z , m \times \mu )$. 
	Furthermore, observe that normal $*$-isomorphisms $\varphi \colon L^\infty([0,1] \times Z , m \times \mu ) \to L^\infty([0,1], m)$ intertwining the two measures are in one-to-one correspondence with $*$-homomorphisms $C([0,1] \times Z)  \to L^\infty([0,1])$ such that, 
	if we write $\tau$ for the state on $L^\infty([0,1])$ induced by the Lebesgue measure $m$ on $[0,1]$, then the induced state on $C([0,1] \times Z)$ comes from the measure $m \times \mu$ on $[0,1] \times Z$. 
	
	Hence in order to prove the lemma, it suffices to show that there exists a family 
	\[
	\left( \varphi_\mu \colon C([0,1] \times Z)  \to L^\infty([0,1]) \right)_{\mu \in \operatorname{Prob}(Z)}
	\]
	of $*$-homomorphisms
	such that 
	\begin{enumerate}[label=(\roman*)]
		\item \label{lem:trivialization-prob-interval-trick:first} 
		
		\label{lem:trivialization-prob-interval-trick:measure} for any $\mu \in \operatorname{Prob}(Z)$, 
		$\tau \circ \varphi_\mu$ agrees with the state on $C([0,1] \times Z)$ induced by the measure $m \times \mu$,

		\item \label{lem:trivialization-prob-interval-trick:dense} 
		for any $\mu \in \operatorname{Prob}(Z)$, the image of $\varphi_\mu$ is dense in $L^\infty([0,1])$ in the strong operator topology, 
		and

		\item \label{lem:trivialization-prob-interval-trick:continuous} for any $a \in C([0,1] \times Z)$, the map $\operatorname{Prob}(Z) \ni \mu \mapsto \varphi_\mu (a) \in L^\infty([0,1]) \subseteq L^2([0,1])$ is continuous in the $L^2$-norm (or equivalently, in the strong operator topology). 
		
		\label{lem:trivialization-prob-interval-trick:last} 
	\end{enumerate}
	
	We will actually construct the $*$-homomorphisms $\varphi_\mu$ on a larger $C^*$-algebra containing $C([0,1] \times Z)$. 
	To this end, consider the commutative $C^*$-algebra $\prod_{z \in Z} L^\infty ([0,1])$ of bounded functions from $Z$ to $L^\infty ([0,1])$ and a $C^*$-subalgebra $C_{\operatorname{w}} \left(Z, L^\infty ([0,1]) \right)$ defined as 
	\[
	\left\{ f \in \prod_{z \in Z} L^\infty ([0,1]) 
	\colon f \text{ is continuous in the strong operator topology} \right\} \; .
	\]
	Observe that any $\mu \in \operatorname{Prob}(Z)$ gives rise to a state $\tau_\mu$ on $C_{\operatorname{w}} \left(Z, L^\infty ([0,1]) \right)$ such that 
	\[
	\tau_\mu (f) := \int_Z \left( \int_{0}^{1} f(z) \, \mathrm{d} m \right) \mathrm{d} \mu \quad \text{ for any } f \in C_{\operatorname{w}} \left(Z, L^\infty ([0,1]) \right) \; .
	\]
	We identify $C([0,1] \times Z)$ with $C(Z, C([0,1]))$, which we view as a $C^*$-subalgebra of $C_{\operatorname{w}} \left(Z, L^\infty ([0,1]) \right)$. 
	Observe that for any $\mu \in \operatorname{Prob}(Z)$, the restriction of $\tau_\mu$ on $C([0,1] \times Z)$ agrees with the state induced by the measure $m \times \mu$.

	For the domain of the $*$-homomorphisms $\varphi_\mu$, we will construct an AF $C^*$-algebra $A$ with $C([0,1] \times Z) \subseteq A \subseteq C_{\operatorname{w}} \left(Z, L^\infty ([0,1]) \right)$. 
	To this end, let us fix a compatible metric $d$ on $Z$ and apply Lemma~\ref{lem:trivialization-prob-continuous-partition} to obtain a sequence $\left( \calP^{(k)} \right)_{k \in \mathbb{N}}$ of finite Borel partitions of $[0,1] \times Z$ satisfying \eqref{lem:trivialization-prob-continuous-partition:monotone}, \eqref{lem:trivialization-prob-continuous-partition:separating} and \eqref{lem:trivialization-prob-continuous-partition:continuous} there. 
	
	In view of Lemma~\ref{lem:trivialization-prob-continuous-partition}\eqref{lem:trivialization-prob-continuous-partition:monotone}, 
	we may, 
	for $k = 0 , 1, \ldots$, enumerate $\calP^{(k)}$ as $\left\{ P^{(k)}_1, \ldots, P^{(k)}_{n^{(k)}} \right\}$ so that 
	for any $i, i' \in \left\{ 1, \ldots, n^{(k-1)} \right\}$ with $i < i'$ and any $j, j' \in \left\{ 1, \ldots, n^{(k)} \right\}$, if $P^{(k)}_j \subseteq P^{(k-1)}_i$ and $P^{(k)}_{j'} \subseteq P^{(k-1)}_{i'}$, then $j < j'$. 
	Hence for any $k \in \mathbb{N}$, 
	the map 
	\[
	\pi^{(k)} \colon \left\{ 1, \ldots, n^{(k+1)} \right\} \to \left\{ 1, \ldots, n^{(k)} \right\} \; ,
	\]
	determined by requiring that $P^{(k+1)}_i \subseteq P^{(k)}_{\pi^{(k)} (i)}$ for any $i \in \left\{ 1, \ldots, n^{(k+1)} \right\}$ is non-decreasing. 
	
	In view of Lemma~\ref{lem:trivialization-prob-continuous-partition}\eqref{lem:trivialization-prob-continuous-partition:continuous} and the remark before Lemma~\ref{lem:trivialization-prob-continuous-partition}, we may define, 
	for any $k \in \mathbb{N}$ and any $i \in \left\{ 1 , \ldots, n^{(k)} \right\}$, 
	a projection $p^{(k)}_i := \chi_{P^{(k)}_i} \in C_{\operatorname{w}} \left(Z, L^\infty ([0,1]) \right)$, or more precisely,  
	\[
	p^{(k)}_i (z)  = \chi_{ \left( P^{(k)}_i \right)_z } \in L^\infty([0,1]) \quad \text{ for } z \in Z \; .
	\]
	Observe that for any $k \in \mathbb{N}$ and any $i \in \left\{ 1 , \ldots, n^{(k)} \right\}$, we have 
	\[
	p^{(k)}_i (z)  = \sum_{ j \in \left( \pi^{(k)} \right)^{-1} (i) } p^{(k+1)}_{j} (z) 
	\]
	For any $k \in \mathbb{N}$, let $A^{(k)}$ be the finite-dimensional $C^*$-algebra generated by the disjoint family $\left\{ p^{(k)}_1, \ldots, p^{(k)}_{n^{(k)}} \right\}$ of projections in $C_{\operatorname{w}} \left(Z, L^\infty ([0,1]) \right)$. Then we have 
	\[
	A^{(0)} \subseteq A^{(1)} \subseteq \ldots \subseteq C_{\operatorname{w}} \left(Z, L^\infty ([0,1]) \right) \; .
	\]
	Define an AF $C^*$-algebra 
	\[
	A := \overline{\bigcup_{k=0}^\infty A^{(k)}} \; .
	\]
	
	Observe that $C([0,1] \times Z) \subseteq A$. Indeed, for any $f \in C([0,1] \times Z)$ and any $\varepsilon > 0$, since $f$ is uniformly continuous, there is $\delta > 0$ such that $\left| f(t,z) - f(t',z') \right| \leq \varepsilon$ for any $(t,z) , (t',z') \in [0,1] \times Z$ with $d \left( (t,z) , (t',z') \right) \leq \delta$, where we equip $[0,1] \times Z$ with the Manhattan metric as in Lemma~\ref{lem:trivialization-prob-continuous-partition}. 
	Since $\left( \calP^{(k)} \right)_{k \in \mathbb{N}}$ satisfies Lemma~\ref{lem:trivialization-prob-continuous-partition}\eqref{lem:trivialization-prob-continuous-partition:separating}, there is $k \in \mathbb{N}$ such that 
	for any $i \in \left\{ 1, \ldots, n^{(k)} \right\}$, we have
	$\operatorname{diam} \left( P^{(k)}_i \right) \leq \delta$ and thus
	$\left| f(t,z) - f(t',z') \right| \leq \varepsilon$ for any $(t,z) , (t',z') \in P^{(k)}_i$, 
	which implies 
	\[
	\inf_{g \in A^{(k)}} \|f - g\| 
	= \inf_{ \lambda_1\vphantom{^{(k)}} , \ldots, \lambda_{n^{(k)}} \in \mathbb{C}} \Big\| f - \sum_{i = 1}^{n^{(k)}} \lambda_i p^{(k)}_i \Big\|
	= \inf_{ \lambda_1\vphantom{^{(k)}} , \ldots, \lambda_{n^{(k)}} \in \mathbb{C}} \sup_{z \in P^{(k)}_i} | f(z) - \lambda_i |
	\leq \varepsilon
	\]
	Since $\varepsilon$ was chosen arbitrarily, this implies $f \in A$.

	Now, for any $k \in \mathbb{N}$ and any $\mu \in \operatorname{Prob}(Z)$, we define an orthogonal family $\left\{ q^{(k)}_{\mu, 1} , \ldots, q^{(k)}_{\mu, n^{(k)}} \right\}$ of
	projections 
	in $L^\infty([0,1])$, where $q^{(k)}_{\mu, i}$ is the characteristic function of the interval 
	\[
	{\displaystyle \left( \sum_{j=1}^{i-1} (m \times \mu) \left( P^{(k)}_j \right) , \sum_{j=1}^{i} (m \times \mu) \left( P^{(k)}_j \right) \right]}
	\]
	for any $i \in \left\{ 1, \ldots, n^{(k)} \right\}$, and observe that 
	\[
	\sum_{i=1}^{n^{(k)}} q^{(k)}_{\mu, i} = 1 \; \text{ and } \; \tau \left( q^{(k)}_{\mu, i} \right) = (m \times \mu) \left( P^{(k)}_i \right) = \tau_\mu \left( p^{(k)}_i \right) \text{ for } i \in \left\{ 1, \ldots, n^{(k)} \right\} \; .
	\]
	It follows that there is 	
	a $*$-homomorphism $\varphi^{(k)}_\mu \colon A^{(k)} \to L^\infty([0,1])$ determined by requiring $\varphi^{(k)}_\mu \left( p^{(k)}_i \right) = q^{(k)}_{\mu, i}$ 
	for $i \in \left\{ 1, \ldots, n^{(k)} \right\}$. It satisfies $\tau \circ \varphi^{(k)}_\mu = \tau_\mu |_{A^{(k)}}$. Moreover, observe that for any $k \in \mathbb{N}$ and any $\mu \in \operatorname{Prob}(Z)$, we have an analogous identity 
	\[
	q^{(k)}_{\mu,i} (z)  = \sum_{ j \in \left( \pi^{(k)} \right)^{-1} (i) } q^{(k+1)}_{\mu,j} (z) \quad \text{ for any } i \in \left\{ 1, \ldots, n^{(k)} \right\}
	\]
	and thus $\varphi^{(k+1)}_\mu$ extends $\varphi^{(k)}_\mu$. Combining these yields a $*$-homomorphism $\varphi_\mu \colon A \to L^\infty ([0,1])$, which we may then restrict to $C([0,1] \times Z)$. 
	
	It remains to show that $\left(\varphi_\mu\right)_{\mu \in \operatorname{Prob}(Z)}$ satisfies conditions~\ref{lem:trivialization-prob-interval-trick:first}-\ref{lem:trivialization-prob-interval-trick:last} above. Note that condition~\ref{lem:trivialization-prob-interval-trick:measure} is evident since we have proved  $\tau \circ \varphi^{(k)}_\mu = \tau_\mu |_{A^{(k)}}$ for any $k \in \mathbb{N}$. 
	
	To show condition~\ref{lem:trivialization-prob-interval-trick:dense}, we 
	first observe that 
	for any $k \in \mathbb{N}$ and any $\mu \in \operatorname{Prob}(Z)$, we have 
	\[
	\sup_{i \in \left\{ 1, \ldots, n^{(k)} \right\} } \tau \left( q^{(k)}_{\mu, i} \right)
	= 	\sup_{i \in \left\{ 1, \ldots, n^{(k)} \right\} } (m \times \mu) \left( P^{(k)}_i \right) 
	\leq \sup_{i \in \left\{ 1, \ldots, n^{(k)} \right\} }  \operatorname{diam} \left( P^{(k)}_i \right)  \; ,
	\] 
	which approaches $0$ as $k \to \infty$, whence 
	$L^\infty([0,1])$ is equal to $\overline{\varphi_\mu (A)}^{\operatorname{SOT}}$, the closure of $\varphi_\mu (A)$ in the strong operator topology. 
	In view of condition~\ref{lem:trivialization-prob-interval-trick:measure} and the fact that the representation $L^\infty([0,1]) \hookrightarrow B\left( L^2 ([0,1])  \right)$ is canonically identified with the GNS representation associated to $\tau$, 
	it follows that 
	the representation $A \xrightarrow{\varphi_\mu} L^\infty([0,1]) \hookrightarrow B\left( L^2 ([0,1])  \right)$ can be identified with the GNS representation $\rho_\mu$ of $A$ associated to $\tau_\mu|_A$. 
	Observe that 
	the closure of $\rho_\mu \left( C([0,1] \times Z) \right)$ in the strong operator topology is $L^\infty \left( [0,1] \times Z , m \times \mu \right)$, which contains $\rho_\mu \left( C_{\operatorname{w}} \left(Z, L^\infty ([0,1]) \right) \right)$ and thus also $\rho_\mu \left( A \right)$. 
	Hence we have $\overline{\varphi_\mu (C([0,1] \times Z))}^{\operatorname{SOT}} = \overline{\varphi_\mu (A)}^{\operatorname{SOT}} = L^\infty([0,1])$. 
	
	To verify condition~\ref{lem:trivialization-prob-interval-trick:continuous}, it suffices to focus on the generators and show, for any $k \in \mathbb{N}$ and $i \in \left\{ 1, \ldots, n^{(k)} \right\}$, 
	the map $\operatorname{Prob}(Z) \ni \mu \mapsto \varphi_\mu \left( p^{(k)}_i \right) = q^{(k)}_{\mu, i} \in L^\infty([0,1]) \subseteq L^2([0,1])$ is continuous in the $L^2$-norm.  
	To this end, we observe that for any $k \in \mathbb{N}$ and $i \in \left\{ 1, \ldots, n^{(k)} \right\}$, we have 
	\[
	(m \times \mu) \left( P^{(k)}_i \right) = \int_Z m \left( \left( P^{(k)}_i \right)_z  \right) \mathrm{d} \mu(z) \; .
	\]
	Since, by Lemma~\ref{lem:trivialization-prob-continuous-partition}\eqref{lem:trivialization-prob-continuous-partition:continuous}, the function $Z \ni z \mapsto m \left( \left( P^{(k)}_i \right)_z  \right) \in [0,1]$ is continuous, the function $\operatorname{Prob}(Z) \ni \mu \mapsto (m \times \mu) \left( P^{(k)}_i \right) \in [0,1]$ is also continuous in the weak-* topology. It thus follows from the definition of $q^{(k)}_{\mu, i}$ above that the map $\operatorname{Prob}(Z) \ni \mu \mapsto q^{(k)}_{\mu, i} \in L^\infty([0,1]) \subseteq L^2([0,1])$ is continuous in the weak-* topology and the $L^2$-norm, as desired. 
	%%%%%%%%%%%%%%%%%%%%%%%%%%%%%%%%%%%%%%55
\end{proof}

\begin{lem}\label{lem:trivialize-prob-field-specify}
	Let $X$ be a randomization-stable {\hhs} and let $Z$ be a compact paracompact Hausdorff space. 
	Then using notations from \Cref{eg:cont-fields-trivial} and \Cref{def:continuum-product-field}, we have  $\left( (X)_Z \right) |_{\spaceProbMeas{\left(Z\right)}} \cong (X)_{\spaceProbMeas{\left(Z\right)}}$. 
\end{lem}

\begin{proof}
	We assume that $\aContField \cong  \aHHS_{\baseSpace{\aContField}}$ for some {\hhs} $\aHHS$. 
	It follows that $\aContField \cong \aContField^{[0,1]} \cong \left( \aHHS_{\baseSpace{\aContField}} \right)^{[0,1]} \cong  \left( \aHHS^{[0,1]} \right)_{\baseSpace{\aContField}}$. 
	Hence it suffices to show that  $\left( \aHHS_{\baseSpace{\aContField}} \right)^{[0,1]}|_{\operatorname{Prob} \left(\baseSpace{\aContField}\right)}$ 
	is isometric to the trivial continuous field $\left(\aHHS^{[0,1]}\right)_{\operatorname{Prob} \left(\baseSpace{\aContField}\right)}$.
	To this end, we apply \Cref{cor:continuum-product-field-trivial-isomorphism} to $Y = Z := \operatorname{Prob} \left(\baseSpace{\aContField}\right)$, 
	$Z' = [0,1] \times \baseSpace{\aContField}$, $Y' := [0,1]$, $f \colon Z \to \spaceMeas(Z')$ taking any $\mu \in \operatorname{Prob} \left(\baseSpace{\aContField}\right)$ to $m \times \mu$, $g \colon Y \to \spaceMeas(Y')$ taking the constant value $m$, $h \colon Y \to Z$ being the identity map on $\operatorname{Prob} \left(\baseSpace{\aContField}\right)$ 
	and $\left( k^{y} \right)_{y \in Y}$ as in \Cref{lem:trivialization-prob-interval-trick}. Since both of the ``additional'' hypotheses in \Cref{cor:continuum-product-field-trivial-isomorphism} are clearly satisfied, it suffices to verify the continuity hypothesis in \Cref{lem:continuum-product-field-trivial-maps}, namely, 
	for any functions $\xi \in C([0,1] \times \baseSpace{\aContField})$ and $\eta \in C([0,1])$, the map  
	\[
	\operatorname{Prob} \left(\baseSpace{\aContField}\right) \ni \mu \mapsto \int_{0}^{1} 
	\left( \xi \circ k^{\mu} \right) \cdot \eta
	\, \operatorname{d} m \in [0, \infty)
	\]
	is continuous, but this clearly follows from our conclusion in \Cref{lem:trivialization-prob-interval-trick} that for any $\xi \in C([0,1] \times \baseSpace{\aContField})$, the assignment $\mu \mapsto \xi \circ k^\mu$ yields a continuous map $\operatorname{Prob}(Z) \to L^2([0,1], m)$.
\end{proof}

\begin{proof}[Proof of \Cref{prop:trivialize-prob-field}]
	The ``if'' direction is clear since $\aContField \cong \left( \aContField |_{\operatorname{Prob}(Z)} \right) |_{Z}$, where we embed $Z$ into $\operatorname{Prob}(Z)$ as point masses. The ``only if'' direction follows from \Cref{lem:trivialize-prob-field-specify}. 
\end{proof}

%%%%%%%%%%%%%%%%%%%%%%%%%trivialization%%%%%%%%%%%%%%%%%%%%%%%%%
%%%%%%%%%%%%%%%%%%%%%%%%%diffeo%%%%%%%%%%%%%%%%%%%%%%%%%
\section{Proper actions on continuous fields of $L^2$-Riemannian metrics}
\label{sec:diffeo}

In this section, we discuss the notion of proper actions on continuous fields of {\hhs}s, which is a straightforward generalization to that of proper actions on continuous fields of affine Euclidean spaces. Then we apply the general theory to a construction of continuous fields of $L^2$-Riemannian metrics on a closed smooth manifold.

\begin{defn} \label{defn:proper-action-cont-field}
	Let $\aContField$ be a continuous field of {\hhs}s with a compact base space $\baseSpace{\aContField}$. 
	An isometric action $\alpha$ of a countable group $\Gamma$ on $\aContField$ is \emph{metrically proper} (or simply \emph{proper}) if for any continuous section $s \in \spaceOfSections_{\operatorname{cont}} (\aContField)$, we have 
	\[
	\inf_{z \in \underline{\aContField}} d_{\aContField_z} \left( s(z),  \alpha_\gamma (s)  (z) \right) \xrightarrow{\gamma \to \infty \text{ in } \Gamma} \infty \; ,
	\]
	that is, for any $N>0$, only finitely many $\gamma \in \Gamma$ makes the left-hand side fall below $N$. 
\end{defn}

\begin{rmk}  \label{rmk:proper-action-cont-field-exists}
	Under the assumption that $\aContField$ is not empty, it follows from the compactness assumption on $\baseSpace{\aContField}$ and the triangle inequality that we may replace ``for any continuous section'' by ``there exists a continuous section'' in \Cref{defn:proper-action-cont-field}. 
\end{rmk}

\begin{rmk}  \label{rmk:proper-action-cont-field-reduction}
	It is clear from \Cref{defn:proper-action-cont-field} that if an action $\alpha \colon \Gamma \curvearrowright \aContField$ is proper and $Z \subseteq \baseSpace{\aContField}$ is a closed $\Gamma$-invariant subset, then the induced action $\alpha|_{Z} \colon \Gamma \curvearrowright \aContField|_Z$ is proper. 
\end{rmk}

This notion of metrically proper actions is closely related to that of coarse embeddings. Recall that a map $f \colon X \to Y$ between metric spaces is a \emph{coarse embedding} if there exist non-decreasing functions $\rho_+, \rho_-\colon [0, \infty) \to [0, \infty)$ such that 
\begin{enumerate}
	\item $\rho_-(d_X(x, z))\leq d_Y(f(x) ,  f(z)) \leq \rho_+(d_X(x, z))$ for all $x, z \in X$, 
	\item $\lim_{r\to \infty} \rho_-(r) = \infty$. 
\end{enumerate} 
If, in addition, there is $R > 0$ such that $f(X)$ is an $R$-net in $Y$, then $f$ is called a \emph{coarse equivalence}. 

Also recall that every countable group $\Gamma$ can be equipped with a left-invariant metric $d_\Gamma$ that is \emph{proper} in the sense that $\left\{ \gamma \in \Gamma \colon d_\Gamma ( 1 , \gamma ) \leq r \right\}$ is finite for any $r \geq 0$. Such a proper left-invariant metric is unique up to coarse equivalence. When we talk about a coarse embedding of a countable group into a metric space, we are implicitly fixing a left-invariant metric on $\Gamma$ (it does not matter which one). 

Furthermore, recall that left-invariant pseduo-metrics (respectively, metrics) on $\Gamma$ are canonically in one-to-one correspondence with \emph{length functions} (respectively, \emph{faithful} length functions) on $\Gamma$, namely functions $f \colon \Gamma \to [0,\infty)$ satisfying that $f(1_\Gamma) = 0$, $f(\gamma \gamma') \leq f(\gamma) + f(\gamma')$ and $f(\gamma^{-1}) = f(\gamma)$ (and faithfulness means $f(\gamma) = 0$ only when $\gamma = 1_\Gamma$). A left-invariant pseduo-metric is proper if and only if the corresponding length function $f$ is proper in the sense that $f(\gamma) \xrightarrow{\gamma \to \infty \text{ in } \Gamma} \infty$. 

The following lemma gives an equivalent characterization of metric properness of actions in terms of coarse embeddings into Hilbert-Hadamard spaces. 
\begin{lem} \label{lem:proper-action-cont-field-coarse}
	Let $\alpha \colon \Gamma \curvearrowright \aContField$ be an isometric action of a countable group $\Gamma$ on a continuous field of {\hhs}s with a compact base space. Let $z \in \baseSpace{\aContField}$ be such that $\baseSpace{\aContField} = \overline{z \cdot \Gamma}$. Let $s \in \spaceOfSections_{\operatorname{cont}}(\aContField)$ be a continuous section. 
	Then $\alpha$ is proper in the sense of \Cref{defn:proper-action-cont-field} if and only if the map 
	\[
	\iota \colon \Gamma \to \aContField_z \, , \quad \gamma \mapsto \alpha_\gamma (s) (z)
	\]
	is a coarse embedding of $\Gamma$ into the fiber $\aContField_z$. 
\end{lem}

\begin{proof}

	Since $\baseSpace{\aContField} = \overline{z \cdot \Gamma}$, it follows that for any $\gamma_1, \gamma_2 \in \Gamma$, we have 
	\begin{align*}
	\inf_{x \in \underline{\aContField}} d_{\aContField_x} \left( s (x) ,  \alpha_{\gamma_1^{-1}\gamma_2} (s)  (x) \right)
	& = \inf_{x \in \underline{\aContField}} d_{\aContField_x} \left( \alpha_{\gamma_1} (s)  (x) ,  \alpha_{\gamma_2} (s)  (x) \right) \\
	&= \inf_{x \in z\cdot \Gamma} d_{\aContField_x} \left( \alpha_{\gamma_1} (s)  (x) ,  \alpha_{\gamma_2} (s)  (x) \right) \\
	& = \inf_{\gamma' \in  \Gamma} d_{\aContField_z} \left( \alpha_{\gamma' \gamma_1} (s)  (z) ,  \alpha_{\gamma' \gamma_2} (s)  (z) \right). 
	\end{align*}
	If $\alpha$ is proper in the sense of Definition \ref{defn:proper-action-cont-field}, then it is clear that there exists a non-decreasing function $\rho_1\colon [0, \infty) \to [0, \infty)$ that satisfies $\lim_{r\to \infty} \rho_1(r) = \infty$ and 
	\[  \rho_1(d_\Gamma(1_\Gamma, \gamma)) \leq  \inf_{x \in \underline{\aContField}} d_{\aContField_x} \left(  s(x) ,  \alpha_{\gamma} (s)  (x) \right)  \quad \text{ for any } \gamma \in \Gamma \; . \]	
	Therefore, for any $\gamma_1, \gamma_2\in \Gamma$, we have 
	\begin{align*}
	\rho_1(d_\Gamma(\gamma_1, \gamma_2)) & = \rho_1(d_\Gamma(1, \gamma_1^{-1} \gamma_2)) \\
	&\leq  \inf_{x \in \underline{\aContField}} d_{\aContField_x} \left(  s(x) ,  \alpha_{\gamma_1^{-1} \gamma_2} (s)  (x) \right) \\
	& = \inf_{\gamma' \in  \Gamma} d_{\aContField_z} \left( \alpha_{\gamma' \gamma_1} (s)  (z) ,  \alpha_{\gamma' \gamma_2} (s)  (z) \right)\\
	&  \leq  d_{\aContField_z} \left( \alpha_{\gamma_1} (s)  (z) ,  \alpha_{\gamma_2} (s)  (z) \right) \; .
	\end{align*}
	
	Then, by using the compactness of $\baseSpace{\aContField}$ and considering 
	\[ \max_{x \in \underline{\aContField}} d_{\aContField_x} \left( s  (x) ,  \alpha_{\gamma_1^{-1}\gamma_2} (s)  (x) \right) \quad  \text{ in place of } \quad \inf_{x \in \underline{\aContField}} d_{\aContField_x} \left( s (x) ,  \alpha_{\gamma_1^{-1}\gamma_2} (s)  (x) \right) \] 
	in the above, 
	a similar argument shows  
	\[  d_{\aContField_z} \left( \alpha_{\gamma_1} (s)  (z) ,  \alpha_{\gamma_2} (s)  (z) \right) \leq \max_{x \in \underline{\aContField}} d_{\aContField_x} \left( s  (x) ,  \alpha_{\gamma_1^{-1}\gamma_2} (s)  (x) \right) \text{ for any } \gamma, \gamma' \in \Gamma \; . \]
	Hence the function $\rho_2 \colon [0, \infty) \to [0, \infty)$ given by 
	\[
	\rho_2(r) := \max_{\gamma \in \Gamma \colon d_\Gamma(1, \gamma) \leq r}  \max_{x \in \underline{\aContField}} d_{\aContField_x} \left( s  (x) ,  \alpha_{\gamma} (s)  (x) \right)  \; .
	\]
	meets the requirement. This proves the ``only if'' direction. 
	
	Now let us show the ``if'' direction. If the map 
	\[
	\iota \colon \Gamma \to \aContField_z \, , \quad \gamma \mapsto \alpha_\gamma (s) (z)
	\]
	is a coarse embedding of $\Gamma$ into the fiber $\aContField_z$, then there exists a non-decreasing function $\rho_1\colon [0, \infty) \to [0, \infty)$ that satisfies $\lim_{r\to \infty} \rho_1(r) = \infty$ and 
	\begin{align*}
	\rho_1(d_\Gamma(1_\Gamma, \gamma) ) = \rho_1(d_\Gamma(\gamma' , \gamma' \gamma)) &\leq   d_{\aContField_z} \left( \alpha_{\gamma' } (s)  (z) ,  \alpha_{\gamma' \gamma} (s)  (z) \right) \text{ for any } \gamma, \gamma' \in \Gamma \; .
	\end{align*}
	It follows that 
	\[ \inf_{x \in \underline{\aContField}} d_{\aContField_x} \left( s (x) ,  \alpha_{\gamma} (s)  (x) \right)  = \inf_{\gamma' \in  \Gamma} d_{\aContField_z} \left( \alpha_{\gamma' \gamma_1} (s)  (z) ,  \alpha_{\gamma' \gamma_2} (s)  (z) \right) \geq \rho_1(d_\Gamma(1_\Gamma, \gamma) ).  \]
	Since $\lim_{r\to \infty} \rho_1(r) = \infty$, we have 
	\[
	\lim_{\gamma \to \infty} \inf_{x \in \underline{\aContField}} d_{\aContField_x} \left( s (x) ,  \alpha_{\gamma} (s)  (x) \right) = \infty \; .
	\]
	Hence the action $\alpha$ is proper in the sense of Definition \ref{defn:proper-action-cont-field}. This completes the proof. 
\end{proof}

Let us look at how the property of properness interacts with the variation-of-measure construction in \Cref{def:continuum-product-field}. To this end, recall that there is a natural convex structure on $\spaceMeas (\baseSpace{\aContField})$ through weighted sums of measures, where $\spaceMeas (\baseSpace{\aContField})$ is the set of all finite regular Borel measures on $\baseSpace{\aContField}$. The following result essentially says properness is inherited through reduction on the base space while also preserved  by taking closed convex hulls. 

\begin{lem} \label{lem:proper-action-cont-field-convex}
	Let $\alpha \colon \Gamma \curvearrowright \aContField$ be an isometric (left) action of a countable group on a continuous field of {\hhs}s with a compact base space $\baseSpace{\aContField}$. 	
	Let $Z$ and $Y$ be compact Hausdorff spaces with right $\Gamma$-actions, and let $f \colon Z \to \spaceMeas(\baseSpace{\aContField})$ and $g \colon Y \to \spaceMeas(\baseSpace{\aContField})$ be (right) $\Gamma$-equivariant continuous maps. 
	Write $f^* \alpha \colon \Gamma \curvearrowright f^* \aContField$ and $g^* \alpha \colon \Gamma \curvearrowright g^* \aContField$ for the induced isometric actions as given in \Cref{cor:continuum-product-field-actions}. 
	Suppose that $f(Z) \subseteq \overline{ \operatorname{convex} \left( g(Y) \right) }$ and $g^* \alpha$ is proper. Then $f^* \alpha$ is proper, too. 
\end{lem}

\begin{proof}
	By \Cref{rmk:proper-action-cont-field-exists} and \Cref{def:continuum-product-field}, 
	it suffices to show that for any $s \in \spaceOfSections_{\operatorname{cont}} (\aContField)$ and any $\gamma \in \Gamma$, we have
	\begin{align*}
	\inf_{z \in Z} d_{(f^* \aContField)_z} \left( f^*(s)(z),  (f^* \alpha)_\gamma  \left( f^*(s) \right)  (z) \right) 
	\geq 
	\inf_{y \in Y} d_{(g^* \aContField)_y} \left( g^*(s)(y),  (g^* \alpha)_\gamma  \left( g^*(s) \right)  (y) \right)  \; ,
	\end{align*}
	or equivalently, 
	\begin{align*}
	\inf_{z \in Z} d_{L^2(\baseSpace{\aContField}, f(z), \aContField)} \left( s,  \alpha_\gamma  \left( s\right) \right) 
	\geq 
	\inf_{y \in Y} d_{L^2(\baseSpace{\aContField}, g(y), \aContField)} \left( s,  \alpha_\gamma  \left(s\right)  \right)  \; ,
	\end{align*}
	which in turn is equivalent to showing 
	\[
	\inf_{z \in Z} \int_{x \in \baseSpace{\aContField}} d_{\aContField_x} \left( s (x),  \alpha_\gamma  \left( s\right) (x) \right) ^2 \, \mathrm{d} (f(z))  (x)
	\geq 
	\inf_{y \in Y} \int_{x \in \baseSpace{\aContField}} d_{\aContField_x} \left( s (x),  \alpha_\gamma  \left( s\right) (x) \right) ^2 \, \mathrm{d} (g(y))  (x)  \; . 
	\]
	To this end, we observe that for any $\mu \in \operatorname{convex} \left( g(Y) \right) \subseteq \spaceMeas (\aContField)$, we have 
	\begin{align*}
	&\ \int_{x \in \baseSpace{\aContField}} d_{\aContField_x} \left( s (x),  \alpha_\gamma  \left( s\right) (x) \right) ^2 \, \mathrm{d} \mu  (x) \\
	\in &\    
	\operatorname{convex} 
	\left\{   \int_{x \in \baseSpace{\aContField}} d_{\aContField_x} \left( s (x),  \alpha_\gamma  \left( s\right) (x) \right) ^2 \, \mathrm{d} (g(y))  (x)  \colon y \in Y \right\} \\
	\subseteq &\ \left[ \inf_{y \in Y} \int_{x \in \baseSpace{\aContField}} d_{\aContField_x} \left( s (x),  \alpha_\gamma  \left( s\right) (x) \right) ^2 \, \mathrm{d} (g(y))  (x)  \; , \infty \right) \; .
	\end{align*}
	Since $f(Z) \subseteq \overline{ \operatorname{convex} \left( g(Y) \right) }$ by our assumption, 
	an approximation in the weak $*$-topology thus yields the desired inequality. 
\end{proof}

\begin{cor} \label{cor:proper-action-cont-field-convex-prob}
	Let $\alpha \colon \Gamma \curvearrowright \aContField$ be an isometric action of a countable group on a continuous field of {\hhs}s with a compact base space $\baseSpace{\aContField}$. 
	Let $Z$ be a closed subset of $\operatorname{Prob}(\baseSpace{\aContField})$. 
	Then the induced action $\alpha|_{Z} \colon \Gamma \curvearrowright \aContField|_{Z}$ is proper if and only if $\alpha|_{\overline{ \operatorname{convex} \left( Z \right) }} \colon \Gamma \curvearrowright \aContField|_{\overline{ \operatorname{convex} \left( Z \right) }}$ is. 
	
	In particular, $\alpha|_{\operatorname{Prob}(\baseSpace{\aContField})} \colon \Gamma \curvearrowright \aContField|_{\operatorname{Prob}(\baseSpace{\aContField})}$ is proper if and only if $\alpha$ is. 
\end{cor}

\begin{proof}
	The ``only if'' direction follows from \Cref{rmk:proper-action-cont-field-reduction} upon embedding $Z$ into $\overline{ \operatorname{convex} \left( Z \right) }$, while the ``if'' direction follows directly from Lemma \ref{lem:proper-action-cont-field-convex} by taking $f$ and $g$ to be inclusion maps. 
	
	The ``in particular'' part then follows upon embedding $\baseSpace{\aContField}$ into $\operatorname{Prob}(\baseSpace{\aContField})$ as point masses.  
\end{proof}

\begin{cor} \label{cor:proper-action-cont-field-convex-random}
	Let $\alpha \colon \Gamma \curvearrowright \aContField$ be an isometric action of a countable group on a continuous field of {\hhs}s with a compact base space $\baseSpace{\aContField}$. 
	Let $Z$ be a compact Hausdorff space and let $\mu$ be a finite regular Borel measure on $Z$. 
	Then 
	$\alpha$ is proper if and only if the induced action $\alpha^{(Z, \mu)}$ as in \Cref{lem:continuum-product-field-augmentation-maps}\eqref{lem:continuum-product-field-augmentation-maps:action} is proper. 
\end{cor}

\begin{proof}
	This directly follows \Cref {lem:proper-action-cont-field-convex} and the construction of the induced action $\alpha^{(Z, \mu)}$. 
\end{proof}

%%%%%%%%%%%%%%%%%%%%%%%%%%%%%%%%%%%%%%

In the rest of the section, we detail the construction of a continuous field of {\hhs}s consisting of $L^2$-Riemannian metrics on a closed smooth manifold, which is sketched in the introduction and will ultimately play a pivotal role in the proof of \Cref{main-theorem} in \Cref{sec:main}. 

We start with the following construction, which is an analog of \cite[Construction~4.1]{GongWuYu2021}. 

\begin{constr}
	\label{constr:symmetric-space-GL}
	Fix a positive integer $n$ and consider the set $\calP(V)$ 
	of all inner products on an $n$-dimensional $\mathbb{R}$-vector space $V$. 
	Upon identifying $V$ via a linear isomorphism with $\mathbb{R}^n$, we can identify $\calP(V)$ with the set $\calP(n)$ (also written as $M_n(\mathbb{R})_{>0}$) of positive definite symmetric real $n \times n$-matrices, which in turn can be identified, through the congruence left action of $\operatorname{GL}(n,\mathbb{R})$ on $\calP(n)$, with the quotient space $\operatorname{GL}(n,\mathbb{R}) / \operatorname{O}(n)$, with a base point chosen to be the identity matrix $I_n$, which is identified with the class $[I_n]$ of the identity element in $\operatorname{GL}(n,\mathbb{R})$. Explicitly, this identification takes $[T] \in \operatorname{GL}(n,\mathbb{R}) / \operatorname{O}(n)$ to $T T^* \in \calP(n)$. 
	Abstractly, we may also identity $\calP(V)$ with $\operatorname{GL}(V) / \operatorname{O} (V, g)$, where $g$ is an arbitrary element of $\calP(V)$ and $\operatorname{O}(V, g)$ is the group of all linear automorphisms of $V$ preserving $g$. 
	
	As a Riemannian symmetric space of noncompact type, $\calP(V)$ is a complete simply connected Riemannian manifold with non-positive curvature, and in particular, also an {\admhhs}. 
	More precisely, the standard metric $d_{\calP(n)}$ on $\calP(n)$ is the unique left $\operatorname{GL}(n,\mathbb{R})$-invariant Riemannian metric such that on the tangent space $T_{I_n} \calP(n)$, canonically identified with the linear space of all symmetric real $n \times n$-matrices, it is given by the $\operatorname{O}(n)$-congruence-invariant inner product $\langle A, B \rangle = \operatorname{Tr} (AB)$ for $A, B \in T_{I_n} \calP(n)$. For general $V$, the standard metric $d_{\calP(V)}$ on $\calP(V)$ is induced from an (arbitrary) identification of $V$ with $\mathbb{R}^n$; it follows from the $\operatorname{GL}(n,\mathbb{R})$-invariance of $d_{\calP(n)}$ that the choice of the linear isomorphism $V \cong \mathbb{R}^n$ does not make a difference here. 
	We also write $d_{\operatorname{GL}(n,\mathbb{R}) / \operatorname{O}(n)}$ and $d_{\operatorname{GL}(V) / \operatorname{O} (V, g)}$ for the corresponding metrics on the respective reincarnations. 
	
	Any linear isomorphism $T \colon V \to V'$ of $n$-dimensional $\mathbb{R}$-vector spaces gives rise to an isometry $T^* \colon \calP (V') \to \calP (V)$ by pulling back inner products, that is, 
	\begin{equation} \label{eq:symmetric-space-GL-pullback}
		T^* (g) = g \circ (T \otimes T) \colon V \otimes V \to \mathbb{R} \quad \text{ for any } g \in \calP(V') \; .
	\end{equation}
\end{constr}

\begin{rmk}
	\label{rmk:symmetric-space-GL-metric}
	We make a few remarks on the metric on $\calP(n)$. Note that as a symmetric space, $\calP(n)$ is reducible: indeed, 
	it can be identified with $\mathbb R \times \calP_1(n)$, where $\calP_1(n)$ denotes the set of all positive definite symmetric real $n \times n$-matrices \emph{of determinant $1$}. 
	An explicit identification is given by 
	\[   \calP(n) \ni D \mapsto \left( \frac{1}{\sqrt{n}}  \log \det D, \det(D)^{-\frac{1}{n}} \, D  \right) \in \mathbb R \times \calP_1(n) \; . \]
	Under this identification, the metric on $\calP(n)$ agrees with the product Riemannian metric of the standard Euclidean metric on $\mathbb R$ and the metric on $\calP_1(n)$ inherited from $\calP(n)$. Indeed, since the tangent space $T_{I_n} \calP_1(n)$ is canonically identified with the linear space of all symmetric real $n \times n$-matrices with trace $0$, the above identification induces a linear isomorphism 
	\[
		T_{I_n} \calP(n) \to T_{0} \mathbb R \oplus T_{I_n} \calP_1(n) \, , \quad A \mapsto \left(\frac{1}{\sqrt{n}} \operatorname{Tr} (A), A - \frac{1}{n} \operatorname{Tr} (A) I_n \right) \; ,
	\]
	and for any $A, B \in T_{I_n} \calP(n)$, we have 
	\[
		\langle A, B \rangle = \operatorname{Tr} (AB) = \frac{1}{n} \operatorname{Tr} (A) \operatorname{Tr} (B) + \left\langle  A - \frac{1}{n} \operatorname{Tr} (A) I_n , B - \frac{1}{n} \operatorname{Tr} (B) I_n \right\rangle \; ,
	\]
	it follows that the inner products on the tangent spaces agree. 
	Moreover, since there is a canonical direct sum decomposition of $\operatorname{GL}(n,\mathbb{R})$ as $\mathbb{Z}/2 \times \mathbb{R} \times \operatorname{SL}(n,\mathbb{R})$, such that $\calP_1(n)$ is invariant under $\operatorname{SL}(n,\mathbb{R})$ and $\mathbb{R}^{\times} \cdot I_n$ is invariant under $\mathbb{Z}/2 \times \mathbb{R} \cong \mathbb{R}^{\times}$, it is not hard to see that the product metric is $\operatorname{GL}(n,\mathbb{R})$-invariant, which proves the two Riemannian metrics coincide under the above identification. 
\end{rmk}

The following lemma provides, roughly speaking, an alternative, simpler description of the metric on $\calP(n)$, up to some uniform error. 

\begin{lem}
	\label{lem:GL-length}
	Consider maps $\operatorname{GL}(n,\mathbb{R}) \to [0,\infty)$ given by 
	\[
		T \mapsto d_{\operatorname{GL}(n,\mathbb{R}) / \operatorname{O}(n)} ([T], [I_n]) \quad \text{ and } \quad T \mapsto \log \max \left\{ \|T\| , \| T^{-1} \| \right\} \; ,
	\]
	where $\| - \|$ denotes the operator norm. 
	Then they form length functions that map $\operatorname{O}(n)$ to $0$ and are bilipschitz to each other, that is, there are constants $c_1, c_2 > 0$ such that for any $T \in \operatorname{GL}(n,\mathbb{R})$, we have 
	\[
		c_1 d_{\operatorname{GL}(n,\mathbb{R}) / \operatorname{O}(n)} ([T], [I_n]) \leq \log \max \left\{ \|T\| , \| T^{-1} \| \right\} \leq c_2 d_{\operatorname{GL}(n,\mathbb{R}) / \operatorname{O}(n)} ([T], [I_n]) \; .
	\]
\end{lem}

\begin{proof}
	The fact that the first map is a length function follows directly from the left $\operatorname{GL}(n,\mathbb{R})$-invariance of $d_{\operatorname{GL}(n,\mathbb{R}) / \operatorname{O}(n)}$, while the same fact for the second map follows from a direct computation that utilizes the inequality $\max \left\{ a+b , a'+b' \right\} \leq \max \left\{ a , a' \right\} + \max \left\{ b , b' \right\}$ for arbitrary real numbers $a,a',b,b'$. It is also obvious that both maps take $\operatorname{O}(n)$ to $0$. 
	
	It remains to show they are bilipschitz to each other. 
	To this end, we see that under the identification $\operatorname{GL}(n,\mathbb{R}) / \operatorname{O}(n) \to \calP(n)$, $[T] \mapsto TT^*$, the ``explicit identification'' in \Cref{rmk:symmetric-space-GL-metric} can be reformulated as 
	\begin{align*}
	\operatorname{GL}(n,\mathbb{R}) / \operatorname{O}(n) & \to \mathbb R \times \operatorname{SL}(n,\mathbb{R}) / \operatorname{SO}(n)  \\
	[T] & \mapsto \left( \frac{2}{\sqrt{n}}  \log |\det (T)|, \left[\mathrm{sgn}(\det(T)) |\det(T)|^{-\frac{1}{n}} \, T \right] \right) \; .
	\end{align*}
	Writing $\widehat{T}$ for $\mathrm{sgn}(\det(T)) |\det(T)|^{-\frac{1}{n}} \, T$, we see that our first length function is bilipschitz to the ``third'' length function 
	\begin{align*}
		T \mapsto &\  \big| \log |\det (T)| \big| + d_{\operatorname{SL}(n,\mathbb{R}) / \operatorname{SO}(n)} \left( [\widehat{T}], [I_n]\right) \\
		&\  =  \log \max \left\{ |\det (T)| , |\det (T^{-1})| \right\} + d_{\operatorname{SL}(n,\mathbb{R}) / \operatorname{SO}(n)} \left( [\widehat{T}], [I_n]\right) \; ,
	\end{align*}
	where $d_{\operatorname{SL}(n,\mathbb{R}) / \operatorname{SO}(n)}$ is inherited from $d_{\operatorname{GL}(n,\mathbb{R}) / \operatorname{O}(n)}$. 
	
	It thus suffices to show the second and the third length functions are bilipschitz to each other. 
	The case of $n \leq 1$ is trivial, so let us assume $n > 1$. 
	To this end, for any $T \in \operatorname{GL}(n,\mathbb{R})$, since $\widehat{T} \in \operatorname{SL}(n,\mathbb{R})$, it follows from \cite[Remark~4.4]{GongWuYu2021}\footnote{Note that what we call $\calP_1(n)$ here is written as $P(n)$ in \cite[Remark~4.4]{GongWuYu2021}. } that 
	\begin{equation} \label{eq:bilipschitz-SL}
		0 < 2 \log (\|\widehat{T}\|) \leq d_{\operatorname{SL}(n,\mathbb{R}) / \operatorname{SO}(n)} \left( [\widehat{T}], [I_n]\right) \leq 2 \sqrt{n} \log (\|\widehat{T}\|) 
	\end{equation}
	while 
	\[
		\log (\|\widehat{T}\|) \leq ({n-1}) \log (\|(\widehat{T})^{-1}\|) \quad \text{ and vice versa, }
	\]
	whence 
	\begin{equation} \label{eq:bilipschitz-inverse-max}
		\max \left\{ \log (\|\widehat{T}\|) , \log (\|(\widehat{T})^{-1}\|) \right\} \leq ({n-1}) \log (\|\widehat{T}\|)  
	\end{equation}
	and
	\begin{equation} \label{eq:bilipschitz-inverse-min}
		\frac{1}{n-1} \log (\|\widehat{T}\|) \leq \min \left\{ \log (\|\widehat{T}\|) , \log (\|(\widehat{T})^{-1}\|) \right\} \; . 
	\end{equation}
	
	Since $\|T\| = |\det (T)|^{1/n} \cdot \left\| \widehat{T} \right\|$, we have 
	\begin{align*}
		&\ \log \max \left\{  \|T\| , \| T^{-1} \| \right\} \\
		= &\  \max \left\{ \frac{1}{n} \log |\det (T)| + \log \left\| \widehat{T} \right\| , \; \frac{1}{n} \log |\det (T^{-1})| + \log \left\| (\widehat{T})^{-1} \right\|\right\} \; ,
	\end{align*}
	which is, by \eqref{eq:bilipschitz-inverse-max}, no more than 
	\[
		\frac{1}{n} \log \max \left\{ |\det (T)| , |\det (T^{-1})| \right\} + ({n-1}) \log (\|\widehat{T}\|) 
	\]
	and, by \eqref{eq:bilipschitz-inverse-min}, no less than 
	\[
		\frac{1}{n} \log \max \left\{ |\det (T)| , |\det (T^{-1})| \right\} + \frac{1}{n-1} \log (\|\widehat{T}\|) \; .
	\]
	Combining these estimates with \eqref{eq:bilipschitz-SL}, we see that the second and the third length functions are bilipschitz to each other, as desired. 
\end{proof}

The next lemma is a coordinate-free reformulation of the last one. 

\begin{lem}
	\label{lem:GL-length-abstract}
	For any positive integer $n$, there exists $c_1, c_2 > 0$ such that for any linear isomorphism $T \colon V \to V'$ of $n$-dimensional $\mathbb{R}$-vector spaces, and for any inner products $g \in \calP(V)$ and $g' \in \calP(V')$, we have 
	\[
	c_1 d_{\calP(V)} (T^*(g'), g) \leq \log \max \left\{ \|T\|_{g,g'} , \| T^{-1} \|_{g',g} \right\} \leq c_2 d_{\calP(V)} (T^*(g'), g) \; ,
	\]
	where $T^*$ is given as in \eqref{eq:symmetric-space-GL-pullback}, and $\|-\|_{g,g'}$ is the operator norm defined for linear operators between the Hilbert spaces $(V,g)$ and $(V',g')$. 
\end{lem}

\begin{proof}
	Given a positive integer $n$, we let $c_1, c_2 > 0$ be as given in \Cref{lem:GL-length-abstract}. 
	To prove the desired inequalities, we fix (arbitrary) isometries $S$ and $S'$ connecting, respectively, $(V, g)$ and $(V',g')$ to $\mathbb{R}^n$ with the standard Euclidean inner product $g_0$. Then the desired inequalities follow from \Cref{lem:GL-length-abstract} as well as the observations that $\|T\|_{g,g'} = \| S' T S^{-1} \|$, $\| T^{-1} \|_{g',g} = \| (S' T S^{-1})^{-1} \|$, and 
	$d_{\calP(V)} (T^*(g'), g) = d_{\calP(\mathbb{R}^n)} ((S' T S^{-1})^*(g_0), g_0) = d_{\operatorname{GL}(n,\mathbb{R}) / \operatorname{O}(n)} ([S' T S^{-1}], [I_n])$, the last equation being a consequence of \Cref{constr:symmetric-space-GL}. 
\end{proof}

Now we consider, from a certain geometric perspective, the space of all Riemannian metrics on a given closed smooth manifold. 

\begin{constr} \label{constr:Riem-bundle}
	Let $N$ be an $n$-dimensional closed smooth manifold. 
	Note that a Riemannian metric on $N$ is nothing but a smooth field over $N$ of inner products on the tangent spaces of $N$. 
	Inspired by this,  we let $\operatorname{Riem}(N)$ be the continuous field of {\hhs}s with  
	\[
		\baseSpace{\operatorname{Riem}(N)} := N \, , \quad \operatorname{Riem}(N)_z := \calP(T_z N) \text{ for any } z \in N \; ,
	\]
	and $\spaceOfSections_{\operatorname{cont}} (\operatorname{Riem}(N))$ consisting of 
	all $s \in \prod_{z \in N} \calP(T_z N)$ such that for each smooth local chart $\varphi \colon \mathbb{R}^n \supseteq U \to N$, the restriction of $s$ on $U$ is a continuous function from $U\subseteq \mathbb{R}^n$ to $\calP(\mathbb R^n)$.

\end{constr}

\begin{rmk}
	We may also construct $\operatorname{Riem}(N)$ by first forming its total space (in the sense of \Cref{def:cont-field-HHS-total-space}) to be the smooth fiber bundle $\operatorname{GL}(N) \times_{\operatorname{GL}(n,\mathbb{R})} M_n(\mathbb{R})_{>0}$, where $\operatorname{GL}(N)$ is the frame bundle of $N$, and then taking the continuous sections on this smooth fiber bundle. The smooth structure does not play a significant role in this paper. 
\end{rmk}

%%%%%%%%%%%%%%%%%%%%%%%%%%%%%%%%%%%%%%%%%%%%%%%%%%%%%%%%%%%%%%%%%%%%%%%%

\begin{constr}\label{constr:diff-riem}
	The group $\operatorname{Diff}(N)$ of diffeomorphisms over a closed smooth manifold $N$ acts isometrically on the continuous field $\operatorname{Riem}(N)$ of {\hhs}s by ``pushing forward Riemannian metrics''. 
	More precisely, each $\varphi \in \operatorname{Diff}(N)$ gives rise to an isometric continuous isomorphism 
	\[
		\varphi_* := \left(\varphi^{-1} \colon N \to N, \quad \left( ( D_{z} \varphi^{-1})^{*} \colon \calP(T_{\varphi^{-1} (z)} N) \to \calP(T_{z} N) \right)_{z \in N} \right)
	\]
	in the sense of \Cref{defn:cont-fields-morphism}, where $D_{z} \varphi^{-1} \colon T_{z} N \to T_{\varphi^{-1} (z)} N$ denotes the derivative of $\varphi^{-1}$ at $z$, and $( D_{z} \varphi^{-1})^{*}$ is given as in \eqref{eq:symmetric-space-GL-pullback}. 
	
	In particular, for any $s\in \operatorname{Riem}(N)$ (e.g.~$s$ may correspond to a Riemannian metric on $N$), we have 
	\[
		\varphi_*(s)_{\varphi(z)} \left( D_z\varphi(v), D_z\varphi(w) \right) = s_z \left( v, w \right)
	\]
	for any $z \in N$ and $v,w \in T_{z} N$. It is clear that the assignment $\varphi \mapsto \varphi_*$ is homomorphic. 
	
	Furthermore, by \Cref{cor:continuum-product-field-actions} and \Cref{lem:continuum-product-field-augmentation-maps}, this left  action  canonically induces isometric left actions of $\operatorname{Diff}(N)$ on 
	$\operatorname{Riem}(N)|_{\operatorname{Prob}(N)}$ and 
	$\operatorname{Riem}(N)^{[0,1]}|_{\operatorname{Prob}(N)}$. 
	
	If $\Gamma$ is a subgroup of $\operatorname{Diff}(N)$ and $Z$ is a $\Gamma$-invariant closed convex subspace of $\operatorname{Prob}(N)$, then the above canonical action restricts to an isometric action of $\Gamma$ on $\operatorname{Riem}(N)|_{Z}$. 
\end{constr}

Key to our proof of \Cref{main-theorem} is an analysis of the properness of actions $\Gamma \curvearrowright \operatorname{Riem}(N)|_{Z}$ arising from \Cref{constr:diff-riem}. With the help of \Cref{lem:GL-length-abstract}, we shall, in \Cref{lem:geometrically-discrete} below, unpack the layers of definitions to produce a more explicit criterion of properness, formulated in terms of the following pseudometrics. 

\begin{constr}\label{constr:pseudometric-integral}
	Let $\mu$ be a regular Borel probability measure on $N$. Also fix a Riemannian metric $g$ on $N$. For any diffeomorphism $\varphi \in \operatorname{Diff}(N)$ and $x \in N$, we write $D_x \varphi \colon T_x N \to T_{\varphi(x)} N$ for the derivative of $\varphi$ at $x$, viewed as a linear operator between finite-dimensional real Hilbert spaces, and write $\left\| D_x \varphi \right\|_g$ for its operator norm. 
	Observe that $x \mapsto \left\| D_x \varphi \right\|_g$ is a smooth nonzero function on $N$. 
	We then define 
	a pseudometric on $\operatorname{Diff}(N)$
	\begin{align*}
	d_{\mu, g} \colon & \ \operatorname{Diff}(N) \times \operatorname{Diff}(N)  \to  [0, \infty) \\
	& \ (\varphi, \psi )  \mapsto  \left( \int_{x \in N}  \left( \log \left( \max \left\{\left\| D_{\varphi^{-1}(x)} \left(\psi^{-1} \varphi\right) \right\|_g ,   \left\| D_{\psi^{-1}(x)} \left(\varphi^{-1} \psi\right) \right\|_g \right\} \right) \right)^2    \, \operatorname{d} \mu (x) \right)^{\frac{1}{2}} \; .
	\end{align*}
	We will verify that this is indeed a pseudometric in \Cref{lem:pseudometric-integral}. 
	We also observe that, for a second Riemannian metric $g'$ on $N$, we have 
	\begin{equation} \label{constr:pseudometric-integral::eq:g}
	\left| d_{\mu, g} \left( \varphi, \varphi' \right) - d_{\mu, g'} \left( \varphi, \varphi' \right) \right| \leq 
	C_{g,g'} := 
	\max_{x \in N}  \left( \log \left( \left\| \operatorname{id}_{T_x X} \right\|_{g, g'}  \left\| \operatorname{id}_{T_x X} \right\|_{g', g}  \right) \right)
	\end{equation}
	where $\left\| \operatorname{id}_{T_x X} \right\|_{g, g'}$ is the operator norm of the identity map from $\left( T_x X , g_x \right)$ to $\left( T_x X , g'_x \right)$ and $\left\| \operatorname{id}_{T_x X} \right\|_{g', g}$ is the norm of its inverse. 
	Moreover, we see compute that for any $\varphi, \psi, \tau \in \operatorname{Diff}(N)$, we have 
	\begin{align*}
		&\  d_{\tau_\ast \mu, g} (\varphi, \psi ) \\
		= &\  \left( \int_{x \in N}  \left( \log \left( \max \left\{\left\| D_{\varphi^{-1} \tau(x)} \left(\psi^{-1} \varphi\right) \right\|_g ,   \left\| D_{\psi^{-1} \tau(x)} \left(\varphi^{-1} \psi\right) \right\|_g \right\} \right) \right)^2    \, \operatorname{d} \mu (x) \right)^{\frac{1}{2}} \\
		= &\  d_{\mu, g} (\tau^{-1} \varphi, \tau^{-1} \psi ) \; ,
	\end{align*}
	where $\tau_\ast \mu$ denotes the pushforward measure. 
\end{constr}

\begin{lem} \label{lem:pseudometric-integral}
	The function $d_{\mu, g}$ defined in \Cref{constr:pseudometric-integral} is indeed a pseudometric, i.e., 
	for any $\varphi$, $\varphi'$ and $\varphi''$ in $\operatorname{Diff}(N)$, we have: 
	\begin{enumerate}
		\item $d_{\mu, g} \left( \varphi, \varphi \right) = 0$;
		\item $d_{\mu, g} \left( \varphi, \varphi' \right) = d_{\mu, g} \left( \varphi' , \varphi \right)$;
		\item $d_{\mu, g} \left(\varphi , \varphi'' \right) \leq d_{\mu, g}(\varphi, \varphi') + d_{\mu, g} ( \varphi', \varphi'' ) $. 
	\end{enumerate}	
\end{lem}

\begin{proof}
	The first two axioms are immediate from the definition of $d_{\mu, g}$. To verify the last axiom, we first observe that for any $x \in N$, we have 
	\[
	1 =  \left\| \operatorname{id}_{T_{\varphi''{}^{-1} (x) }X } \right\|_g \leq \left\| D_{\varphi^{-1} (x)} \left(\varphi''{}^{-1} \varphi\right) \right\|_g  \left\| D_{\varphi''{}^{-1} (x)} \left(\varphi^{-1} \varphi''\right) \right\|_g \; , 
	\]
	whence 
	\begin{align*}
	0 \leq &\ \log \left( \max \left\{\left\| D_{\varphi^{-1} (x)} \left(\varphi''{}^{-1} \varphi\right) \right\|_g ,  \left\| D_{\varphi''{}^{-1} (x)} \left(\varphi^{-1} \varphi''\right) \right\|_g \right\} \right) \\
	= &\ \log \left( \max \left\{\left\| D_{\varphi^{-1} (x)} \left(\varphi''^{-1} \varphi'{} \varphi'{}^{-1} \varphi\right) \right\|_g , \ \left\| D_{\varphi''{}^{-1} (x)} \left(\varphi^{-1} \varphi' \varphi'{}^{-1} \varphi''\right) \right\|_g \right\} \right) \\
	= &\  \max \left\{ \log \left\| D_{\varphi'{}^{-1} (x)} \left(\varphi''{}^{-1} \varphi' \right) D_{\varphi^{-1} (x)} \left( \varphi'{}^{-1} \varphi\right) \right\|_g , \right.    \\
	& \qquad \qquad  \left.   \log  \left\| D_{\varphi'{}^{-1} (x)} \left(\varphi^{-1} \varphi'  \right) D_{\varphi''{}^{-1} (x)} \left( \varphi'{}^{-1} \varphi''\right) \right\|_g \right\}   \\
	\leq &\ \max \left\{ \log \left\| D_{\varphi'{}^{-1} (x)} \left(\varphi''{}^{-1} \varphi' \right) \right\|_g +  \log \left\| D_{\varphi^{-1} (x)} \left( \varphi'{}^{-1} \varphi\right) \right\|_g , \right.    \\
	& \qquad \quad   \left.  \log \left\| D_{\varphi'{}^{-1} (x)} \left(\varphi^{-1} \varphi'  \right)  \right\|_g + \log \left\| D_{\varphi''{}^{-1} (x)} \left( \varphi'{}^{-1} \varphi''\right) \right\|_g \right\}   \\
	\leq &\ \max \left\{ \log \left\| D_{\varphi'{}^{-1} (x)} \left(\varphi''{}^{-1} \varphi' \right) \right\|_g , \log \left\| D_{\varphi''{}^{-1} (x)} \left( \varphi'{}^{-1} \varphi''\right) \right\|_g  \right\}     \\
	&     + \max \left\{ \log   \left\| D_{\varphi'{}^{-1} (x)} \left(\varphi^{-1} \varphi'  \right)  \right\|_g , \log \left\| D_{\varphi^{-1} (x)} \left( \varphi'{}^{-1} \varphi \right) \right\|_g \right\}   
	\end{align*}
	and thus 
	\begin{align*}
	&\ \left| \log \left( \max \left\{\left\| D_{\varphi^{-1} (x)} \left(\varphi''{}^{-1} \varphi\right) \right\|_g ,  \left\| D_{\varphi''{}^{-1} (x)} \left(\varphi^{-1} \varphi''\right) \right\|_g \right\} \right) \right| \\
	\leq &\ \log \left( \max \left\{  \left\| D_{\varphi'{}^{-1} (x)} \left(\varphi''{}^{-1} \varphi' \right) \right\|_g , \left\| D_{\varphi''{}^{-1} (x)} \left( \varphi'{}^{-1} \varphi''\right) \right\|_g  \right\}   \right)   \\
	&     + \log \left( \max \left\{  \left\| D_{\varphi'{}^{-1} (x)} \left(\varphi^{-1} \varphi'  \right)  \right\|_g ,  \left\| D_{\varphi^{-1} (x)} \left( \varphi'{}^{-1} \varphi \right) \right\|_g \right\} \right) \; . 
	\end{align*}
	The desired inequality thus follows by \Cref{constr:pseudometric-integral} and the Minkowski inequality: 
	\begin{align*}
	&\ \left( \int_{x \in N}  \left( \log \left( \max \left\{\left\| D_{\varphi^{-1} (x)} \left(\varphi''{}^{-1} \varphi\right) \right\|_g ,  \left\| D_{\varphi''{}^{-1} (x)} \left(\varphi^{-1} \varphi''\right) \right\|_g \right\} \right) \right)^2    \, \operatorname{d} \mu (x) \right)^{\frac{1}{2}} \\
	\leq &\ \left( \int_{x \in N}  \left( \log \left( \max \left\{  \left\| D_{\varphi'{}^{-1} (x)} \left(\varphi''{}^{-1} \varphi' \right) \right\|_g , \left\| D_{\varphi''{}^{-1} (x)} \left( \varphi'{}^{-1} \varphi''\right) \right\|_g  \right\}   \right)  \right. \right. \\
	& \ \left. \left. +  \log \left( \max \left\{  \left\| D_{\varphi'{}^{-1} (x)} \left(\varphi^{-1} \varphi'  \right)  \right\|_g ,  \left\| D_{\varphi^{-1} (x)} \left( \varphi'{}^{-1} \varphi \right) \right\|_g \right\} \right) \right)^2    \, \operatorname{d} \mu (x) \right)^{\frac{1}{2}} \\
	\leq &\ \left( \int_{x \in N}  \left( \log \left( \max \left\{  \left\| D_{\varphi'{}^{-1} (x)} \left(\varphi''{}^{-1} \varphi' \right) \right\|_g , \left\| D_{\varphi''{}^{-1} (x)} \left( \varphi'{}^{-1} \varphi''\right) \right\|_g  \right\}   \right) \right)^2    \, \operatorname{d} \mu (x) \right)^{\frac{1}{2}} \\
	& \ +  \left( \int_{x \in N}  \left( \log \left( \max \left\{  \left\| D_{\varphi'{}^{-1} (x)} \left(\varphi^{-1} \varphi'  \right)  \right\|_g ,  \left\| D_{\varphi^{-1} (x)} \left( \varphi'{}^{-1} \varphi \right) \right\|_g \right\} \right) \right)^2    \, \operatorname{d} \mu (x) \right)^{\frac{1}{2}} 
	\end{align*}
	as desired. 
\end{proof}

\begin{lem} \label{lem:pseudometric-bilipschitz}
	Let $g$ be a Riemannian metric on $N$. 
	Denote by $s$ the continuous section of $\operatorname{Riem}(N)|_{\operatorname{Prob}(N)}$ induced from $g$ via the canonical map $\spaceOfSections_{\operatorname{cont}} (\operatorname{Riem}(N)) \to \spaceOfSections_{\operatorname{cont}} (\operatorname{Riem}(N)|_{\operatorname{Prob}(N)})$ as in \Cref{def:continuum-product-field}, that is, we have $s(\mu) = g \in \spaceOfSections_{\operatorname{cont}} (\operatorname{Riem}(N)) \subseteq L^2(N, \mu, \operatorname{Riem}(N))$ for any $\mu \in \operatorname{Prob}(N)$. 
	
	Then there are $c_1, c_2 > 0$ such that for any $\mu \in \operatorname{Prob}(N)$, we have 
	\[
		c_1 \, d_{\mu, s} (\varphi , \psi) \leq \, d_{\mu, g} (\varphi , \psi) \leq c_2 \, d_{\mu, s} (\varphi , \psi)  \quad \text{ for any } \varphi, \psi \in \operatorname{Diff}(N) \; , 
	\]
	where $d_{\mu, s}$ is the pseudometric on $\operatorname{Diff}(N)$ defined by 
	\[
	d_{\mu, s} (\varphi , \psi) = d_{(\operatorname{Riem}(N)|_{\operatorname{Prob}(N)})_\mu} ( \alpha_\varphi (s), \alpha_\psi (s) ) \quad \text{ for any } \varphi, \psi \in \operatorname{Diff}(N) \; ,
	\]
	where $\alpha$ denotes the action $\Gamma \curvearrowright \operatorname{Riem}(N)|_{\operatorname{Prob}(N)}$ from \Cref{constr:diff-riem}. 
\end{lem}

\begin{proof}
	Let $n$ be the dimension of $N$ and let $c_1, c_2 > 0$ be the constants from \Cref{lem:GL-length-abstract} that depend only on $n$. 
	For any $\mu \in \operatorname{Prob}(N)$ and $\varphi \in \operatorname{Diff}(N)$, we have $\alpha_\varphi (s) (\mu) = \varphi_* (g) \in \spaceOfSections_{\operatorname{cont}} (\operatorname{Riem}(N)) \subseteq L^2(N, \mu, \operatorname{Riem}(N))$ 
	and thus 
	\begin{align*}
	d_{\mu, s} (\varphi , \idmap_N)^2 &= d_{(\operatorname{Riem}(N)|_{\operatorname{Prob}(N)})_\mu} ( \alpha_\varphi (s), s ) \\
	&= \int_{x \in N} d_{\calP (T_x N)} \left( \varphi_*(g)_x, g_x \right) ^2 \, \operatorname{d} \mu (x) \\
	&= \int_{x \in N} d_{\calP (T_x N)} \left((D_x \varphi^{-1})^* (g_{\varphi^{-1} (x)}), g_x \right) ^2 \, \operatorname{d} \mu (x) \; ,
	\end{align*}
	where the last equation follows from \Cref{constr:diff-riem};  
	hence it follows from \Cref{lem:GL-length-abstract} and the equation $\left(D_x \varphi^{-1}\right)^{-1} = D_{\varphi^{-1}(x)} \varphi$ that
	\begin{equation} \label{eq:pseudometric-bilipschitz-special-case}
		c_1 \, d_{\mu, s} (\varphi, \idmap_N) \leq d_{\mu, g} (\varphi , \idmap_N) \leq c_2 \, d_{\mu, s} (\varphi, \idmap_N) \; .
	\end{equation}
	Now for any $\varphi, \psi \in \operatorname{Diff}(N)$, we have $d_{\mu, g} (\varphi , \psi) = d_{\psi^{-1} _\ast \mu, g} (\psi^{-1} \varphi , \idmap_N)$ by \Cref{constr:pseudometric-integral}, while it follows from the isometricity of $\alpha$ that $d_{\mu, g} (\varphi , \psi) = d_{\psi^{-1} _\ast \mu, g} (\psi^{-1} \varphi , \idmap_N)$. 
	Hence the desired inequalities follow from the special case \eqref{eq:pseudometric-bilipschitz-special-case}. 
\end{proof}

To deliver the punchline of this section, we introduce the following explicit condition, which shows up in \Cref{main-theorem}. 

\begin{defn}\label{def:geometrically-discrete}
	Let $\mu$ be a regular Borel probability measure on $N$. 
	A subgroup $\Gamma$ of $\operatorname{Diff}(N)$ is 
	\emph{$\mu$-discrete} if 
	\[
	\inf_{\tau \in \Gamma} d_{\tau_\ast\mu, g} (\gamma, 1_\Gamma) \xrightarrow{\gamma \to \infty \text{ in } \Gamma} \infty \; ,
	\]
	that is, for any $N>0$, only finitely many $\gamma \in \Gamma$ makes the left-hand side above fall below $N$. 
\end{defn}

Note that by \eqref{constr:pseudometric-integral::eq:g}, the left-hand side in the above is also equal to $\inf_{\tau \in \Gamma} d_{\mu, g} (\tau \gamma, \tau)$, and this condition is independent of the choice of $g$. 

\begin{prop}\label{lem:geometrically-discrete}
	Let $\Gamma$ be a subgroup of the group of diffeomorphisms over a closed smooth manifold $N$. As a subgroup of $\operatorname{Diff}(N)$, $\Gamma$ thus inherits actions on the continuous fields $\operatorname{Riem}(N)$ over $N$ and $\operatorname{Riem}(N) |_{\operatorname{Prob}(N)}$ of {\hhs}s  over $\operatorname{Prob}(N)$. 
	Let $\mu$ be a regular Borel probability measure and let $\overline{\operatorname{convex}(\Gamma \cdot \mu)} \subseteq \operatorname{Prob}(N)$ be the closed convex hull of the orbit of $\mu$ under the action of $\Gamma$ on $\operatorname{Prob}(N)$. 
	Note the action $\Gamma \curvearrowright \operatorname{Riem}(N) |_{\operatorname{Prob}(N)}$ restricts to actions of $\Gamma$ on $\operatorname{Riem}(N) |_{\overline{\operatorname{convex}(\Gamma \cdot \mu)}}$ and $\Gamma \curvearrowright \operatorname{Riem}(N) |_{\overline{\Gamma \cdot \mu}}$. 
	Let $g$ be a Riemannian metric on $N$. 
	Then the following are equivalent: 
	\begin{enumerate}
		\item \label{lem:geometrically-discrete:discrete} $\Gamma$ is $\mu$-discrete;
		\item \label{lem:geometrically-discrete:orbit} the action $\Gamma \curvearrowright \operatorname{Riem}(N) |_{\overline{\Gamma \cdot \mu}}$ is {metrically proper};
		\item \label{lem:geometrically-discrete:hull} the action $\Gamma \curvearrowright \operatorname{Riem}(N) |_{\overline{\operatorname{convex}(\Gamma \cdot \mu)}}$ is {metrically proper};
		\item \label{lem:geometrically-discrete:coarse} the map $\Gamma \to L^2(N, \mu, \operatorname{Riem}(N))$, $\gamma \mapsto \gamma (g)$ is a coarse embedding. 
	\end{enumerate}
\end{prop}

\begin{proof}
	The equivalence between \eqref{lem:geometrically-discrete:orbit} and \eqref{lem:geometrically-discrete:hull} is an immediate consequence of \Cref{cor:proper-action-cont-field-convex-prob}, while
	the equivalence between \eqref{lem:geometrically-discrete:orbit} and \eqref{lem:geometrically-discrete:coarse} follows directly from \Cref{lem:proper-action-cont-field-coarse}. 
	It remains to show that \eqref{lem:geometrically-discrete:discrete} and \eqref{lem:geometrically-discrete:orbit} are equivalent. For brevity, let us write  $\aContField = \operatorname{Riem}(N) |_{\overline{\Gamma \cdot \mu}}$ and  $\baseSpace{\aContField} = \overline{\Gamma \cdot \mu} $. 
	It follows from \Cref{lem:pseudometric-bilipschitz} that there are $c_1, c_2 > 0$ such that for any $\nu \in \operatorname{Prob}(N)$ and any $\gamma \in \Gamma$, we have 
	\[
	c_1 \, d_{\aContField_\nu} \left( s ( \nu ),  \alpha_\gamma (s)  (\nu) \right) \leq \, d_{\nu, g} (1_\Gamma, \gamma) \leq c_2 \, d_{\aContField_\nu} \left( s ( \nu ),  \alpha_\gamma (s)  (\nu) \right) \; , 
	\]
	where $s$ is the continuous section of $\operatorname{Riem}(N)|_{\operatorname{Prob}(N)}$ induced from $g$. 
	By the density of $\Gamma \cdot \mu$ in $\baseSpace{\aContField}$, we have  
	\begin{align*}
		\inf_{\nu \in \baseSpace{\aContField}} d_{\aContField_\nu} \left( s(\nu),  \alpha_\gamma (s)  (\nu) \right) &=   \inf_{\tau \in  \Gamma} d_{\aContField_{\tau_* \mu}} \left( s ( \tau_* \mu ),  \alpha_\gamma (s)  (\tau_* \mu) \right) \\
		&\in \left[ \frac{1}{c_2} \, \inf_{\tau \in \Gamma}  d_{\tau_* \mu, g} (1_\Gamma, \gamma),  \frac{1}{c_1} \, \inf_{\tau \in \Gamma} d_{\tau_* \mu, g} (1_\Gamma, \gamma)  \right] \; ,   
	\end{align*}
	whence we have $\inf_{\nu \in \baseSpace{\aContField}} d_{\aContField_\nu} \left( s(\nu),  \alpha_\gamma (s)  (\nu) \right) \xrightarrow{\gamma \to \infty \text{ in } \Gamma} \infty$ if and only if $\inf_{\tau \in \Gamma}  d_{\tau_* \mu, g} (1_\Gamma, \gamma) \xrightarrow{\gamma \to \infty \text{ in } \Gamma} \infty$, which is what we needed to prove according to \Cref{rmk:proper-action-cont-field-exists} and \Cref{def:geometrically-discrete}. 
\end{proof}

\begin{cor}\label{cor:geometrically-discrete-convex-hull}
	Let $\Gamma$ be a subgroup of the group of diffeomorphisms over a closed smooth manifold $N$. Then for any regular Borel probability measures $\mu$ and $\mu'$ on $N$, if $\Gamma$ is $\mu$-discrete and $\mu' \in \overline{\operatorname{convex}(\Gamma \cdot \mu)}$, then $\Gamma$ is $\mu'$-discrete. 
\end{cor}

%%%%%%%%%%%%%%%%%%%%%%%%%diffeo%%%%%%%%%%%%%%%%%%%%%%%%%
%%%%%%%%%%%%%%%%%%%%%%%%%AofM%%%%%%%%%%%%%%%%%%%%%%%%%
\section{A $C^*$-algebra associated to a continuous field of {\hhs}s}
\label{sec:AofM}

\newcommand{\piunbdd}{\Pi}
\newcommand{\pialg}{\Pi_{\operatorname{b}}}

In this section, we define a $C^*$-algebra $\AofM$ associated to a continuous field $\aContField$ of {\hhs}s.

Throughout this section, we let $\aContField$ be a continuous field of {\hhs}s with a compact base space (see \Cref{def:cont-field-HHS}). Recall that for any point $x$ in a {\hhs} $\aHHS$, we use $\hhil_x \aHHS$ to denote the tangent Hilbert space at $x$. 
Also recall the definition of the total space $|\aContField|$ and the canonical map $\pi \colon |\aContField| \to \baseSpace{\aContField}$ in \Cref{def:cont-field-HHS-total-space}. 
Thus for any $(z,x) \in |\aContField|$, where $z \in \baseSpace{\aContField}$ and $x \in \aContField_z$ as in \Cref{def:cont-field-HHS-total-space}, we write $\hhil_x \aContField_{z}$ for the tangent Hilbert space at $x$ as in Definition \ref{defn:hhs} and Construction \ref{constr:Hilbert-space-span}. 

\begin{defn}\label{defn:Pi-M}
	We define the $*$\-/algebra
	\[
	\piunbdd (\aContField) = \prod_{(z,x,t) \in  |\aContField|  \times [0,\infty) } \cliffc (\hhil_x \aContField_{z} \oplus t \rbbd ) 
	= \prod_{z \in  \baseSpace{\aContField} } \prod_{x \in  \aContField_z } \prod_{ t \in  [0,\infty) } \cliffc (\hhil_x \aContField_{z} \oplus t \rbbd ) 
	\; ,
	\]
	where 
	\[
	t \rbbd = 
	\begin{cases}
	\rbbd \, , & t > 0 \\
	\{0\} \, , & t = 0 
	\end{cases}
	\; 
	\]
	and $\rbbd$ carries the canonical inner product (independent of $t$). 
	We also define the $C^*$-algebra  
	\[
	\pialg(\aContField) = \left\{ \sigma \in \piunbdd (\aContField) \colon \sup_{(z, x,t) \in |\aContField| \times [0,\infty)} \| \sigma({z, x,t}) \| < \infty \right\}
	\] 
	equipped with pointwise algebraic operations and the uniform norm. 
\end{defn}

Although $\pialg(\aContField)$ is too large a $C^*$-algebra to be of much use, it will contain our key object in this section, $\AofM$, as a $C^*$-subalgebra. A main ingredient in its construction is given below, which can be understood as arising from ``fiberwise Euler vector fields''. 

\begin{defn}[{\cite[Definition~5.2]{GongWuYu2021}}]\label{defn_Cliffordmultiplier}\label{defn_Botthomomorphism}
	Fix a continuous section $s \in \spaceOfSections_{\operatorname{cont}} (\aContField)$.
	We define the \emph{Clifford operator $\cliffmult^{\aContField} \in \piunbdd (\aContField)$ based at $s$} by
	\[
	\cliffmult^{\aContField} (z, x , t) = \left(-\log_{\xpt} (s(z)) , t \right) \in \tanbndl_{\xpt} \aContField_z \times t \rbbd \subset \cliffc \left({\hhil_{\xpt}{\aContField_z} \oplus t \rbbd} \right)
	\]
	for any $z \in \baseSpace{\aContField}$, $x \in \aContField_z$ and $t \in [0,\infty)$. 
	We also define a map 
	\[
	\botthom^{\aContField}  \colon C(\baseSpace{\aContField}, \salg )  \to \piunbdd (\aContField)
	\]
	by
	\[
	\botthom^{\aContField} (\Ffunc) (z, x,t) := \Ffunc (z) \left(\cliffmult[s]^{\aContField}(z, x, t) \right)
	\]
	for any $z \in \baseSpace{\aContField}$, $x \in \aContField_z$, $t \in [0,\infty)$, and $\Ffunc \in C(\baseSpace{\aContField}, \salg)$, where we apply functional calculus with $\Ffunc(z) \in \salg$ to $\cliffmult[s]^{\aContField}(z, x, t) \in \cliffc \left({\hhil_{\xpt}{\aContField_z} \oplus t \rbbd} \right)$. 
	The map $\botthom^{\aContField}$ will turn out to be a graded $*$\-/homomorphism into $\pialg (\aContField)$ (see \Cref{prop_welldefinednessofBotthomomorphism}) and will be called the \emph{Bott homomorphism} based at $s$.
	
	We may drop the superscript $\aContField$ in $\cliffmult^{\aContField}$ and $\botthom^{\aContField}$ when there is no risk of confusion. 
\end{defn}

\begin{rmk} \label{rmk:botthom-geometric}
	Let us also present a more explicit formulation of $\botthom$ that avoids the use of functional calculus. 
	Observe that there is an isomorphism 
	\begin{align*}
	\salg & \cong \left\{ h \in C_0( [0,\infty), \cliffc (\rbbd) ) \colon h(0) \in \cbbd \cdot 1 \subseteq \cliffc (\rbbd) \right\} \\
	f & \mapsto f_{\operatorname{even}} \cdot 1 + f_{\operatorname{odd}} \cdot e_1
	\end{align*}
	where $1$ is the unit of $\cliffc(\rbbd)$, $e_1$ is a generator of $\cliffc(\rbbd)$ corresponding to the unit vector $1 \in \rbbd$, and $f_{\operatorname{even}}, f_{\operatorname{odd}} \in C_0([0,\infty))$ are defined by
	\[
	f_{\operatorname{even}} (t) = \frac{f(t) + f(-t)}{2} \quad \text{ and } \quad f_{\operatorname{odd}} (t) = \frac{f(t) - f(-t)}{2} \; ,
	\]
	whereby we note that $f_{\operatorname{odd}} (0) = 0$. 
	It follows from \cite[Remark~5.4]{GongWuYu2021} that for any $z \in \baseSpace{\aContField}$, $x \in \aContField_z$, $t \in [0,\infty)$, and $\Ffunc \in C(\baseSpace{\aContField}, \salg)$, we have 
	\[
	\botthom (\Ffunc) (z,x,t) = \Ffunc(z)_{\operatorname{even}} \left( \left\| \cliffmult[s]^{\aContField}(z, x, t) \right\| \right) + \Ffunc(z)_{\operatorname{odd}} \left( \left\| \cliffmult[s]^{\aContField}(z, x, t) \right\| \right)  \frac{\cliffmult[s]^{\aContField}(z, x, t)}{\left\| \cliffmult[s]^{\aContField}(z, x, t) \right\|} \; .
	\]
	Also note that $\left\| \cliffmult[s]^{\aContField}(z, x, t) \right\| = \sqrt{t^2 + d_{\aContField_z} \left( x, s(z) \right) ^2 }$. 
\end{rmk}

In the following proposition, we view both $C(\baseSpace{\aContField}, \salg )$ and $\pialg(\aContField)$ as graded $C^*$-algebras, the former via the flip on $\mathbb{R}$ and the latter induced by the grading on Clifford algebras. 
We also treat both as $C(\baseSpace{\aContField})$-$C^*$-algebras, as each $g \in C(\baseSpace{\aContField})$ can be viewed as a central multiplier of $C(\baseSpace{\aContField}, \salg )$ and $\pialg(\aContField)$ in the obvious way.

\begin{prop}\label{prop_welldefinednessofBotthomomorphism}
	For any continuous section $s \in \spaceOfSections_{\operatorname{cont}} (\aContField)$, 
	the map
	$\botthom $ is a graded $*$\-/homomorphism from $C(\baseSpace{\aContField}, \salg )$ to the $*$\-/subalgebra $\pialg(\aContField)$. 
	It is also $C(\baseSpace{\aContField})$-linear in the sense that 
	\[
	\botthom (\Ffunc g) = \botthom(\Ffunc) g \quad \text{ for any } \Ffunc \in C(\baseSpace{\aContField}, \salg ) \text{ and } g \in C(\baseSpace{\aContField}) \; .
	\]
\end{prop}
\begin{proof}
	The first statement follows from the fact that  $\cliffmult^{\aContField} (z, x , t)$ is an odd bounded self-adjoint operator for any $z \in \baseSpace{\aContField}$, $x \in \aContField_z$ and $t \in [0,\infty)$. The second statement about being $C(\baseSpace{\aContField})$-linear is an immediate consequence of the definition. 
\end{proof}

We are now ready to introduce the main definition of this section. 

\begin{defn}\label{def:AofM}
	The algebra $\AofM$ is the $\cstar$-subalgebra of $\pialg(\aContField)$ generated by 
	\[
	\left\{ \botthom^{\aContField}(\Ffunc) \colon s \in \spaceOfSections_{\operatorname{cont}} (\aContField),\ \Ffunc \in C \left(\baseSpace{\aContField}, \salg \right) \right\} \; .
	\]
	We also define $ \AevenofM $ to be the $\cstar$-subalgebra of $\AofM$ generated by 
	\[
	\left\{ \botthom^{\aContField}(\Ffunc) \colon s \in \spaceOfSections_{\operatorname{cont}} (\aContField),\ \Ffunc \in C\left(\baseSpace{\aContField}, \salg_{\operatorname{ev}}\right)  \right\} \; ,
	\]
	where $\salg_{\operatorname{ev}}$ is the subalgebra of $\salg$ consisting of even functions. 
\end{defn}

\begin{rmk}\label{rmk:AofM-single-fiber}
	The constructions we present above generalize those in \cite[Section~5]{GongWuYu2021}, in the sense that if we view a {\hhs} $X$ as a continuous field $\aContField$ with a singleton base space, then we have canonical identifications (we use the equal sign to indicate the canonicity)
	\begin{align*}
	& \spaceOfSections_{\operatorname{cont}} (\aContField) = X = |\aContField|,  &&
	\piunbdd(\aContField) = \piunbdd(X),  && \pialg(\aContField) = \pialg(X), \\ 
	& C(\baseSpace{\aContField}, \salg ) = \salg, && \cliffmult[x]^{\aContField} = \cliffmult[x]^{X} \quad \text{ and } && \botthom[x]^{\aContField} = \botthom[x]^{X} \text{ for any } x \in X , \\
	& \AofM = \AofM[X], \quad \text{ and }  &&  \AevenofM = \AevenofM[X] \; . 
	\end{align*}
\end{rmk}

\begin{rmk}\label{rmk:AofM-functoriality}
	We discuss the functoriality of the above construction with regard to the construction in \Cref{def:cont-field-HHS-maps}. 
	Let $Y$ be a compact Hausdorff space, let $\aContField$ be a continuous field of {\hhs}s, and let $f \colon Y \to \baseSpace{\aContField}$ be a continuous map. Then in the context of \Cref{defn:Pi-M}, composition with $f$ gives a graded $*$-homomorphism 
	\[
	f^* \colon \piunbdd(\aContField) \to  \piunbdd(f^*\aContField) \; ,
	\]
	which restricts to graded $*$-homomorphisms (denoted by the same symbol) $f^* \colon \pialg(\aContField) \to  \pialg(f^* \aContField)$, $f^* \colon \AofM \to \AofM[f^*\aContField]$ and $f^* \colon \AevenofM \to \AevenofM[f^*\aContField]$.  
	Indeed, the first restriction is immediate from \Cref{defn:Pi-M}, while the latter two restrictions are valid because for any $s \in \spaceOfSections_{\operatorname{cont}} (\aContField)$, we have the equation 
	\[
	f^* \left(\cliffmult^{\aContField}\right) = \cliffmult[f^*(s)]^{f^* \aContField}
	\]
	and the commuting diagram 
	\begin{equation}\label{rmk:AofM-single-fiber::eq:botthom-commute}
	\xymatrix{
		C(\baseSpace{\aContField}, \salg )  \ar[r]^-{\botthom^{\aContField}} \ar[d]_{f^*} & \pialg (\aContField) \ar[d]^{f^*} \\
		C(Y, \salg )  \ar[r]^-{\botthom[f^*(s)]^{f^* \aContField}} & \pialg (f^* \aContField)
	}
	\end{equation}
	both being immediate consequences of \Cref{defn_Cliffordmultiplier}. 
	
	In particular, if $f$ is the inclusion map of a singleton $\{z\}$ into $\baseSpace{\aContField}$, then we may write $\mathrm{ev}_z$ for $f^*$ and use the identifications in \Cref{rmk:AofM-single-fiber} to obtain graded $*$-homomorphisms 
	\begin{align*}
	& \mathrm{ev}_z \colon \piunbdd(\aContField) \to  \piunbdd(\aContField_z) \; , && \mathrm{ev}_z \colon \pialg(\aContField) \to  \pialg(\aContField_z), \\ 
	& \mathrm{ev}_z \colon \AofM \to \AofM[\aContField_z] , \quad  \text{ and }    &&  \mathrm{ev}_z \colon \AevenofM \to \AevenofM[\aContField_z]  \; ,
	\end{align*}
	as well as the equation $\mathrm{ev}_z \left(\cliffmult^{\aContField}\right) = \cliffmult[s(z)]^{\aContField_z}$ and the commuting diagram 
	\begin{equation}\label{rmk:AofM-single-fiber::eq:botthom-commute-ev}
	\xymatrix{
		C(\baseSpace{\aContField}, \salg )  \ar[r]^-{\botthom^{\aContField}} \ar[d]_{\mathrm{ev}_z} & \pialg (\aContField) \ar[d]^{\mathrm{ev}_z} \\
		\salg  \ar[r]^{\botthom[s(z)]^{\aContField_z}} & \pialg (\aContField_z)
	}
	\end{equation}
	for any $s \in \spaceOfSections_{\operatorname{cont}} (\aContField)$. 
\end{rmk}

The next few results discuss the behavior of the Clifford multipliers and the Bott homomorphisms in reaction to perturbations in the base section. 

\begin{lem}\label{lem:cliffmult-base-point-change}
	For any two sections $s_1, s_2 \in \spaceOfSections_{\operatorname{cont}}(\aContField)$, we have $\cliffmult[s_1] - \cliffmult[s_2] \in \pialg(\aContField)$ and in fact 
	\[
	\left\| \cliffmult[s_1] - \cliffmult[s_2] \right\| \leq \sup_{z \in \baseSpace{\aContField}} d_{\aContField_z} \left( s_1(z) , s_2(z) \right) \; .
	\]
\end{lem}

\begin{proof}
	By Construction~\ref{constr:log-map}, the logarithm map $\log_x$ is non-expansive for any $x$. It follows that 
	\begin{align*}
	\|\cliffmult[s_1](z, \xpt, t)-\cliffmult[s_2](z, \xpt, t)\| & = \|- \log_x(s_1(z)) + \log_x(s_2(z))\|_{\hhil_{\xpt}{\aContField_z}} \\
	& \le d_{\aContField_z} \left( s_1(z) , s_2(z) \right)
	\end{align*}
	for any $z \in \baseSpace{\aContField}$, $x \in \aContField_z$ and $t \in [0,\infty)$, whence the claims follow. 
\end{proof}

In order to convert the above error estimate for the Clifford multipliers into one for the Bott homomorphisms, we make use of the following auxiliary functions from \cite[Definition~5.9]{GongWuYu2021}. 

\begin{defn}\label{defn:Omega-Theta}
	For any $\ffunc\in\salg$ and $r \geq 0$, let us define its \emph{$\rdist$-oscillation} by 
	\begin{equation}\label{eq:oscillation}
	\oscill{\rdist}\ffunc = \sup \left\{ |\ffunc(\tvar) - \ffunc(\tvar')| : \ \tvar,\tvar'\in\rbbd, |\tvar-\tvar'|\le \rdist \right\} \; .
	\end{equation}
	For $r>0$, we also define
	\begin{equation}\label{eq:mean}
	\meansym{\rdist} \ffunc = \rdist \cdot \sup \left\{ \frac{| f(t) - f(-t) |}{2t} \colon t \geq \rdist \right\} \; .
	\end{equation}
	It is clear that for any $r > 0$ and $s \geq 0$, we have 
	\begin{align}
	\label{defn:Omega-Theta::eq:scalar} &\oscill{\rdist}(s \ffunc) = s \, \oscill{\rdist}\ffunc \, , && \meansym{\rdist} (s \ffunc) = s \, \meansym{\rdist} \ffunc \, , \\
	\label{defn:Omega-Theta::eq:bound-sup-norm} &\oscill{\rdist}\ffunc \leq 2 \|f\|   &\text{ and }\quad & \meansym{\rdist} \ffunc \leq \|f\| \; .
	\end{align}
\end{defn}

We record the following basic property for these functions. 

\begin{lem}[{see \cite[Lemma~5.10]{GongWuYu2021}}]\label{lem:Omega-Theta-lim-r}
	For any $\ffunc \in \salg$, we have
	\[
	\lim_{r \to 0} \oscill{\rdist}\ffunc = 0 = \lim_{r \to 0} \meansym{\rdist}\ffunc  \; .
	\]
\end{lem}

Then we have the following estimate. 

\begin{prop}\label{prop_Botthomandchangeofbasepoint}
	For any two sections $s_1, s_2 \in \spaceOfSections_{\operatorname{cont}}(\aContField)$ and for any $\Ffunc\in C \left(\baseSpace{\aContField}, \salg \right)$, writing $\rdist = \sup_{z \in \baseSpace{\aContField}} \dist_{z}( s_1(z), s_2 (z) )$, we then have
	\[
	\|\botthom[s_2]^{\aContField}(\Ffunc)- \botthom[s_1]^{\aContField}(\Ffunc)\| \le \max_{z \in \baseSpace{\aContField} } \left( 2 \; \oscill{\rdist}\Ffunc (z) + \max\left\{  \oscill{2\rdist}\Ffunc (z) , \, \meansym{\rdist} \Ffunc (z) \right\} \right)\; . 
	\]
\end{prop}

\begin{proof}
	We start by noting that the first maximum on the right-hand side exists thanks to \eqref{defn:Omega-Theta::eq:bound-sup-norm} and the compactness of $\baseSpace{\aContField}$. 
	Using the evaluation map from \Cref{rmk:AofM-functoriality}, it suffices to show, for any $z \in \baseSpace{\aContField}$, 
	\[
	\| \mathrm{ev}_z \circ \botthom[s_2]^{\aContField}(\Ffunc) - \mathrm{ev}_z \circ \botthom[s_1]^{\aContField}(\Ffunc) \| \le 2 \; \oscill{\rdist}\Ffunc (z) + \max\left\{  \oscill{2\rdist}\Ffunc (z) , \, \meansym{\rdist} \Ffunc (z) \right\} \; . 
	\]
	Since it follows from \eqref{rmk:AofM-single-fiber::eq:botthom-commute-ev} that $\mathrm{ev}_z \circ \botthom[s_k]^{\aContField}(\Ffunc) = \botthom[s_k(z)]^{\aContField_z}(\Ffunc(z))$ for $k \in \{1, 2\}$,  
	the desired inequality is exactly the conclusion of \cite[Proposition~5.11]{GongWuYu2021} (whose proof is just a computation using elementary planar geometry). 
\end{proof}

\begin{lem}\label{lem:botthom-homotopy}
	For any two sections $s_0, s_1 \in \spaceOfSections_{\operatorname{cont}}(\aContField)$, the Bott homomorphisms $\botthom[s_0]^{\aContField}$ and $\botthom[s_1]^{\aContField}$ are homotopic graded $*$-homomorphisms. 
\end{lem}

\begin{proof}
	For each $t\in [0, 1]$, we apply \Cref{cor:cont-field-HHS-geodesic-approximate} to obtain $s_\lambda\in \spaceOfSections_{\operatorname{cont}}(\aContField)$ such that  
	\[
	s_t (z) = \left[ s_{0} \left( z \right) , s_{1} \left( z \right)  \right] ( t )
	\quad \text{ for any } z \in \baseSpace{\aContField} \; .
	\]
	Note that the notations are compatible at $t = 0$ and $1$. 
	It follows that for any $t_1, t_2 \in [0, 1]$, we have 
	\[
	d_{\aContField_{z}} \left( s_{t_1} \left( z \right) , s_{t_2} \left( z \right) \right) = | t_1 - t_2| \; d_{\aContField_{z}} \left( s_{0} \left( z \right) , s_{1} \left( z \right) \right) 
	\quad \text{ for any } z \in \baseSpace{\aContField} \; .
	\]
	Then by \Cref{lem:Omega-Theta-lim-r} and \Cref{prop_Botthomandchangeofbasepoint}, the family $\left(\botthom[s_t]^{\aContField}\right)_{t\in [0,1 ]}$ gives a homotopy of graded $*$-homomorphisms between $\botthom[s_0]^{\aContField}$ and $\botthom[s_1]^{\aContField}$. 
\end{proof}

\begin{defn}\label{defn:non-equivariant-Bott}
	Let $[\botthom[]^{\aContField}] \in K_1(\AofM) \cong KK(\salg, \AofM)$ be the class of the following composition of $*$-homomorphisms
	\[
	\salg \xrightarrow{1 \otimes \idmap} C(\baseSpace{\aContField}) \otimes \salg \cong C(\baseSpace{\aContField}, \salg) \xrightarrow{\botthom[s]^{\aContField}} \AofM \; ,
	\]
	where $s$ is an arbitrary continuous section (by \Cref{lem:botthom-homotopy}, the class $[\botthom[]^{\aContField}]$ does not depend on the choice of $s$). We may drop the superscript $\aContField$ when there is no risk of confusion. 
\end{defn}

Now we discuss group actions on $\AofM$. 
First note that any continuous isometric automorphism $\varphi \in \isomgrp(\aContField)$ gives rise to a $*$-automorphism $\varphi_*$ of $\piunbdd(\aContField)$ by sending $\sigfunc \in \piunbdd(\aContField)$ to $\varphi_* (\sigfunc) \in \piunbdd(\aContField)$ in the following way: 
for any $z \in \baseSpace{\aContField}$, $x \in \aContField_z$, and $t \in [0,\infty)$, we let 
\begin{equation}\label{eq:action-Clifford-algebra}
\varphi_* (\sigfunc) (z, x, t) = \cliffc \left( \hhil \deriv_{x'} \varphi_z  \oplus \idmap_{t \rbbd} \right) \; \left( \sigfunc \left(z', x',  t \right) \right) 
\end{equation}
where
\begin{itemize}
	\item $z' = \baseSpace{\varphi} (z)$ and $x' = (\varphi_z)^{-1} (x)$ (see \Cref{defn:cont-fields-morphism} for the notations), 
	\item $\deriv_{x'} \varphi_z \colon \tanbndl_{x'} \aContField_{z'} \to \tanbndl_{x} \aContField_{z}$ is the derivative of $\varphi_{z} \colon \aContField_{z'} \to \aContField_{z}$ at $x'$, 
	\item $\hhil\deriv_{x'} \varphi_z \colon \hhil_{x'} \aContField_{z'} \to \hhil_{x} \aContField_{z}$ is the induced isometric isomorphism (see Construction \ref{constr:Hilbert-space-span}), and
	\item $\cliffc \left( \hhil \deriv_{x'} \varphi_z  \oplus \idmap_{t \rbbd}  \right)$ is the induced graded $*$\-/isomorphism between the corresponding Clifford algebras. 
\end{itemize} 
Since each $\cliffc \left( \hhil \deriv_{x'} \varphi_z  \oplus \idmap_{t \rbbd}  \right)$ preserves the norm, it follows that $\varphi_*$ restricts to a $*$-automorphism of $\pialg(\aContField)$. 
It is straightforward to check that the assignment $\varphi \mapsto \varphi_*$ gives rise to group homomorphisms 
\[
\isomgrp(\aContField) \to \autgrp(\piunbdd(\aContField)) \quad \text{ and } \quad \isomgrp(\aContField) \to \autgrp(\pialg(\aContField)) \; .
\]
Next we show this latter action further restricts to an action on $\AofM$.

\begin{lem}\label{lem_Botthomandgroupaction}
	For any $\varphi \in \isomgrp(\aContField)$ and any $s \in \spaceOfSections_{\operatorname{cont}} (\aContField)$, we have
	\[
	\varphi_* (\cliffmult^{\aContField}) = \cliffmult[\varphi (s)]^{\aContField}
	\quad \text{ and } \quad 
	\varphi_* \circ \botthom[s]^{\aContField} = \botthom[\varphi (s)]^{\aContField} \circ \baseSpace{\varphi}^* \; ,
	\]
	where $\baseSpace{\varphi}^* \in \autgrp \left(C\left(\baseSpace{\aContField}, \salg \right)\right) $ is induced from $\baseSpace{\varphi} \colon \baseSpace{\aContField} \to \baseSpace{\aContField}$. 
\end{lem}

\begin{proof}
	For any $z \in \baseSpace{\aContField}$, $x \in \aContField_z$ and $t \in [0,\infty)$, let us write $z' = \baseSpace{\varphi} (z)$ and $x' = (\varphi_z)^{-1} (x)$ as before (see \Cref{defn:cont-fields-morphism} for the notations) and observe that the isometry $\varphi_z \colon \aContField_{z'} \to \aContField_z$ maps the geodesic segment $[x', s(z')]$ to $[\varphi_z (x'), \varphi_z (s (z'))]$. This allows us to apply \Cref{defn_Cliffordmultiplier} to compute 
	\begin{align*}
	\varphi_* (\cliffmult^{\aContField}) (z, x, t) = \ &  \cliffc \left( \hhil \deriv_{x'} \varphi_z  \oplus \idmap_{t \rbbd} \right) \; \left( \cliffmult^{\aContField} \left(z', x',  t \right) \right) \\
	=\ & \left( \left( \deriv_{x'} \varphi_z  \right) \; \left(  -\log_{x'}  (s(z'))\right), \idmap_{t \rbbd} (t) \right)  \\
	=\ & \left( - \log_{\varphi_z (x')} (\varphi_z (s (z'))) , t \right)  \\
	=\ & \left( - \log_{x} (\varphi(s) (z)) , t \right)  \\
	=\ & \cliffmult[\varphi(s)]^{\aContField} (z, x , t) \; . 
	\end{align*}
	Since functional calculus commutes with automorphisms of $\cstar$-algebras, 
	it follows from the above computation that for any $ \Ffunc \in C \left(\baseSpace{\aContField}, \salg \right) $, we have 
	\begin{align*}
	\varphi_* \left( \botthom[s]^{\aContField} (\Ffunc) \right) (z, x, t) =\ & \cliffc \left( \hhil \deriv_{x'} \varphi_z  \oplus \idmap_{t \rbbd} \right) \; \left( \botthom^{\aContField} (\Ffunc) \left(z', x', t \right) \right)  \\
	=\ & \cliffc \left( \hhil \deriv_{x'} \varphi_z  \oplus \idmap_{t \rbbd} \right) \; \left( \Ffunc (z') \left( \cliffmult^{\aContField} (z', x' , t) \right)  \right)  \\
	=\ & \Ffunc (z') \Big( \cliffc \left( \hhil \deriv_{x'} \varphi_z  \oplus \idmap_{t \rbbd} \right) \;  \left(\cliffmult^{\aContField} (z', x' , t) \right) \Big) \\
	=\ & \baseSpace{\varphi}^*(\Ffunc) (z) \left( \cliffmult[\varphi(s)]^{\aContField} (z, x , t) \right) \\
	=\ & \botthom[\varphimap (s)]^{\aContField} \left( \baseSpace{\varphi}^*(\Ffunc) \right)(z, x, t) \; ,
	\end{align*}
	as desired. 
\end{proof}

It follows then that the action $\isomgrp(\aContField) \curvearrowright \pialg(\aContField)$ restricts to an action $\isomgrp(\aContField) \curvearrowright \AofM$.

\begin{prop}\label{prop:AofM-action}
	Let $\aContField$ be a continuous field of {\hhs}s and let $G$ be a group. Let $\alpha$ be an action of $G$ on $\aContField$ by continuous isometric isomorphisms.  Then $\AofM$ is a $G$-$Y$-$C^*$-algebra, where $Y$ is the spectrum of the center of $\AofM$. 
\end{prop}

\begin{proof}
	It follows from \Cref{lem_Botthomandgroupaction} and \Cref{def:AofM} that $\AofM$ is a $G$-invariant subalgebra of $\pialg (\aContField)$. 
	Let $Z(\AofM)$ be the center of $\AofM$. We need to show that $Z(\AofM) \cdot \AofM$ is dense in $\AofM$. Observe that $\AevenofM$ is a sub-$\cstar$-algebra of  $Z(\AofM)$.  It  suffices to show that $\AevenofM \cdot \AofM$ is dense in $\AofM$, which follows from the definition of $\AofM$ and the fact that every $\Ffunc\in C \left(\baseSpace{\aContField}, \salg \right)$ can be written as a product $\Ffunc = \Ffunc_1 \Ffunc_2$ where $\Ffunc_1 \in C \left(\baseSpace{\aContField}, \salg_\text{ev} \right)$ and $ \Ffunc_2 \in C \left(\baseSpace{\aContField}, \salg \right)$. For example, we may define 
	\[
	\Ffunc_1 (z)(t) = \sqrt{|\Ffunc (z)(t)|} + \sqrt{|\Ffunc (z)(-t)|} 
	\quad \text{ and } \quad 
	\Ffunc_2 (z)(t) = \frac{\Ffunc (z)(t)}{\Ffunc_1 (z)(t)} 
	\]
	for any $z \in \baseSpace{\aContField}$ and $t \in \mathbb{R}$. 
\end{proof}

The following result links the notion of (metric) properness in \Cref{defn:proper-action-cont-field} with the notion of proper actions on $C^*$-algebras (see the definition just above
\Cref{thm:proper-GHT}). 

\begin{cor}\label{cor:AofM-action-proper}
	Let $\aContField$ be a continuous field of {\hhs}s and let $\Gamma$ be a discrete group. Let $\alpha$ be an action of $\Gamma$ on $\aContField$ by isometric continuous isomorphisms. Assume $\alpha$ is metrically proper in the sense of \Cref{defn:proper-action-cont-field}. Then $\AofM$ is a proper $\Gamma$-$Y$-$C^*$-algebra, where $Y$ is the spectrum of the center of $\AofM$. 
\end{cor}

\begin{proof}
	It follows from \Cref{rmk:botthom-geometric} that for any $\Ffunc \in C(\baseSpace{\aContField},\salg_{\operatorname{ev}})$, $z \in \baseSpace{\aContField}$, $x \in \aContField_z$ and $t \in [0, \infty)$, we have 
	\[
	\botthom^{\aContField} (\Ffunc) (z , x, t) = \Ffunc (z) \left(\sqrt{t^2 + d_{\aContField_z} \left( x, s(z) \right) ^2 }\right) \; .
	\]
	Thus if $\Ffunc$ is compactly supported, i.e., there is $R \geq 0$ such that $\Ffunc (z) (t) = 0$ for any $t \geq R$, then the element $\botthom^{\aContField} (\Ffunc) \in \AevenofM \subset \pialg(\aContField)$ is supported in the bounded ``tube'' 
	\[
	\left\{ (z,x,t) \in |\aContField| \times [0,\infty) \colon \sqrt{t^2 + d_{\aContField_z} \left( x, s(z) \right) ^2 } \leq R \right\} \; ,
	\]
	whence by the metric properness of the action $\Gam\curvearrowright\aContField$, it follows from \Cref{lem_Botthomandgroupaction} that 
	\[
	\alfmap_{\gamelem} (\botthom^{\aContField} (\Ffunc)) \cdot \botthom^{\aContField} (\Ffunc) = \botthom[\varphi(s)]^{\aContField} (\Ffunc \circ \baseSpace{\varphi})) \cdot \botthom^{\aContField} (\Ffunc) \; ,
	\]
	which vanishes for all but finitely many elements $\gamelem$ of $\gamgrp$. 
	Since any $\Ffunc \in C(\baseSpace{\aContField},\salg_{\operatorname{ev}})$ is approximated by compactly supported ones, it follows from \Cref{def:AofM} that every element $\sigma$ of $\AevenofM$ satisfies 
	\[
	\lim_{\gamelem\to\infty} \| \alfmap_{\gamelem} (\sigma) \cdot \sigma \| = 0 \; .
	\]
	This ensures the action of $\gamgrp$ on the spectrum of $\AevenofM$ is (topologically) proper,  i.e., $\AevenofM$ is a commutative proper $\Gamma$-\cstaralg. It follows that $\AofM$ is a proper $\Gamma$-$Y$-\cstaralg.
\end{proof}

Let us lift the class $[\botthom[]^{\aContField}] \in KK(\salg, \AofM)$ introduced in \Cref{defn:non-equivariant-Bott} to the equivariant $KK$-group $\kkgam(\salg, \AofM)$, where $\Gamma$ acts on $\salg$ trivially. We closely follow the approach in \cite[Section~8]{GongWuYu2021} (which in turn is analogous to the treatment in \cite{higsonkasparov}), where the main idea is to make the representatives of $[\botthom[]^{\aContField}]$ ``almost flat'' by a rescaling procedure. 

\begin{constr}
	Let 
	\[
	\sigma \colon \rbbd_+^{\ast} \curvearrowright \salg 
	\]
	be the rescaling action given by 
	\[
	(\lambda \cdot f) (t) = f( \lambda^{-1} t )
	\]
	for any $\lambda \in \rbbd_+^{\ast}$, $f \in \salg$, and $t \in \rbbd$. This action preserves the grading by even and odd functions. 
\end{constr}

\begin{lem}[{see \cite[Lemma~8.5]{GongWuYu2021}}]\label{lem:Omega-Theta-rescale}
	Following the notations of Definition~\ref{defn:Omega-Theta}, for any $\ffunc \in \salg$ and $r \in \rbbd_+^{\ast}$, we have 
	\[
	\oscill{\rdist} (\sigma_\lambda (\ffunc)) = \oscill{\lambda^{-1} \rdist}\ffunc \qquad \text{and} \qquad \meansym{\rdist} (\sigma_\lambda (\ffunc)) = \meansym{\lambda^{-1} \rdist}\ffunc
	\]
	for all $\lambda \in \rbbd_+^{\ast}$ and thus
	\[
	\lim_{\lambda \to \infty} \sup_{\rdist' \leq \rdist} \oscill{\rdist'} (\sigma_\lambda (\ffunc)) = 0  = \lim_{\lambda \to \infty} \sup_{\rdist' \leq \rdist} \meansym{\rdist'} (\sigma_\lambda (\ffunc)) \; .
	\]
\end{lem}

\begin{lem}\label{lem:botthom-rescaled-asymptotic-invariant-01}
	For any section $s \in \spaceOfSections_{\operatorname{cont}}(\aContField)$, the family 
	\[
	\left( {\botthom^{\aContField}} \circ \left(\idmap_{C \left(\baseSpace{\aContField} \right)} \otimes  \sigma_\lambda \right) \right)_{\lambda \in [1, \infty)}
	\]
	of {\shom}s from $C \left(\baseSpace{\aContField}, \salg \right) \cong C \left(\baseSpace{\aContField} \right) \otimes \salg$ to $\AofM$ is asymptotically equivariant in the sense that for any $\Ffunc \in C \left(\baseSpace{\aContField}, \salg \right)$ and $\varphi \in \isomgrp(\aContField)$, 
	\[
	\left\| \varphi_* \circ \botthom[s]^{\aContField} \circ \left(\idmap_{C \left(\baseSpace{\aContField} \right)} \otimes  \sigma_\lambda \right) (\Ffunc) - \botthom[\varphi (s)]^{\aContField} \circ \left(\idmap_{C \left(\baseSpace{\aContField} \right)} \otimes  \sigma_\lambda \right)  \circ \baseSpace{\varphi}^* (\Ffunc) \right\| \xrightarrow{\lambda \to \infty} 0 \; .
	\]
	In particular, the family 
	\[
	\left( {\botthom^{\aContField}} \circ \left(1_{C \left(\baseSpace{\aContField} \right)} \otimes  \sigma_\lambda \right) \right)_{\lambda \in [1, \infty)}
	\]
	of {\shom}s from $\salg$ to $\AofM$ is asymptotically invariant. 
\end{lem}

\begin{proof}
	For any $\varphi \in \isomgrp(M)$, by the compactness of $\baseSpace{\aContField}$ and \Cref{def:cont-field-HHS}\eqref{def:cont-field-HHS::continuous}, we have 
	\[
	\sup_{z \in \baseSpace{\aContField}} d_{\aContField_z} (s(z) , \varphi (s) (z) ) < \infty \; .
	\] 
	We denote this finite nonnegative number by $r$. 
	Then by Proposition~\ref{prop_Botthomandchangeofbasepoint} and \Cref{lem_Botthomandgroupaction}, we have, for any $\lambda \in [1, \infty)$,  
	\begin{align*}
	&~ \left\| \botthom^{\aContField}\left(\left(\idmap_{C \left(\baseSpace{\aContField} \right)} \otimes  \sigma_\lambda \right)\left(\baseSpace{\varphi}_* (f)\right)\right) - \varphi_* \left( \botthom^{\aContField}\left(\left(\idmap_{C \left(\baseSpace{\aContField} \right)} \otimes  \sigma_\lambda \right)(f)\right) \right) \right\| \\
	= &~ \left\| \botthom^{\aContField}\left(\left(\idmap_{C \left(\baseSpace{\aContField} \right)} \otimes  \sigma_\lambda \right)(f)\right) - \botthom[\varphi(s)]^{\aContField}\left(\left(\idmap_{C \left(\baseSpace{\aContField} \right)} \otimes  \sigma_\lambda \right)(f)\right) \right\| \\
	\leq &~ \max_{z \in \baseSpace{\aContField} } \sup_{\rdist' \leq \rdist}   \left( 2 \; \oscill{\rdist'} (\sigma_\lambda(f(z))) + \max\left\{ \oscill{2\rdist'}(\sigma_\lambda(f(z)))  , \, \meansym{\rdist'} (\sigma_\lambda(f(z))) \right\} \right)\; , \
	\end{align*}
	which converges to $0$ as $\lambda \to \infty$, since the assignment
	\[
	\baseSpace{\aContField} \ni z \mapsto \sup_{\rdist' \leq \rdist} \left( 2 \; \oscill{\rdist'} (\sigma_\lambda(f(z))) + \max\left\{ \oscill{2\rdist'}(\sigma_\lambda(f(z)))  , \, \meansym{\rdist'} (\sigma_\lambda(f(z))) \right\} \right)
	\]
	gives a family (parametrized by $\lambda$) of functions on the compact space $\baseSpace{\aContField}$ that is equicontinuous by \eqref{defn:Omega-Theta::eq:bound-sup-norm} and the isometricality of each $\sigma_\lambda$ and converges pointwise to $0$ as $\lambda \to \infty$ by \Cref{lem:Omega-Theta-rescale}. 
\end{proof}

\begin{constr}\label{constr:bott-element}
	Thanks to \Cref{lem:botthom-rescaled-asymptotic-invariant-01}, we may define the \emph{(equivariant) Bott element}
	\[
	[\botthom[]^{\aContField}] \in \kkgam[1]( \cbbd, \AofM ) 
	\]
	as the one induced by the family 
	\[
	\left( {\botthom^{\aContField}} \circ \left( 1_{C \left(\baseSpace{\aContField} \right)} \otimes  \sigma_\lambda \right) \colon \salg \to \AofM \right)_{\lambda \in [1, \infty)}
	\]
	according to \Cref{constr:KK-facts-asymptotic}, for an arbitrary section $s \in \spaceOfSections_{\operatorname{cont}}(\aContField)$. This is independent of the choice of $s$ by \Cref{lem:botthom-homotopy}. Thus the forgetful map 
	\[
	\kkgam[1]( \cbbd, \AofM ) \to KK_1 ( \cbbd, \AofM ) \cong KK_0(\salg, \AofM ) \; , 
	\]
	maps the equivariant Bott element to the non-equivariant Bott element $[\botthom[]^{\aContField}] \in KK_1( \cbbd, \AofM ) $ defined in \Cref{defn:non-equivariant-Bott}. 
	We may drop the superscript $\aContField$ when there is no risk of confusion. 
\end{constr}

In the last part of this section, we discuss the structure of the $C^*$-algebra $\AofM[\aContField |_{\baseSpace{\aContField} \times Z}]$. 
To do this, we first establish a lemma which tells us that in \Cref{def:AofM}, in order to generate $\AofM$ or $\AevenofM$, it suffices to let $s$ and $\Ffunc$ range over generating sets. 

\begin{lem}\label{lem:AofM-generating-sets}
	Let $\Sigma$ be a generating set for a continuous field $\aContField$ of {\hhs}s as in \Cref{def:cont-field-HHS-generate}. Let $S$ be a generating set for the $\cstar$-algebra $C \left(\baseSpace{\aContField}, \salg \right)$. Then $\AofM$ is the $\cstar$-subalgebra of $\pialg(\aContField)$ generated by 
	\[
	\left\{ \botthom^{\aContField}(\Ffunc) \colon s \in \Sigma,\ \Ffunc \in S \right\} \; .
	\]
	Similarly, if we let $S'$ be a generating set for the $\cstar$-algebra $C \left(\baseSpace{\aContField}, \salg_{\operatorname{ev}} \right)$, then $ \AevenofM $ is the $\cstar$-subalgebra of $\AofM$ generated by 
	\[
	\left\{ \botthom^{\aContField}(\Ffunc) \colon s \in \Sigma,\ \Ffunc \in S'  \right\} \; . 
	\]
\end{lem}

\begin{proof}
	We only prove the statement about $\AofM$, the one about $\AevenofM$ being completely analogous. 
	To this end, 
	for any $s \in \Sigma$, since $\botthom^{\aContField}$ is a $*$-homomorphism, it follows that $\botthom^{\aContField} \left( C \left(\baseSpace{\aContField}, \salg \right) \right)$ is the $C^*$-algebra generated by $\botthom^{\aContField} (S)$. Hence in view of \Cref{def:AofM}, we may assume without loss of generality that $S = C \left(\baseSpace{\aContField}, \salg \right) $. 
	
	It thus suffices to show that for any fixed $s \in \spaceOfSections_{\operatorname{cont}} (\aContField)$, $\Ffunc \in C \left(\baseSpace{\aContField}, \salg \right)$ and $\varepsilon > 0$, the element $\botthom^{\aContField} (\Ffunc)$ is within distance $\varepsilon$ to a finite linear combination of elements of the form $\botthom[s']^{\aContField} (\Ffunc')$ where $s' \in \Sigma$ and $\Ffunc' \in C \left(\baseSpace{\aContField}, \salg \right)$. 
	To this end, we use the compactness of $\baseSpace{\aContField}$ and the continuity of $\Ffunc$ to choose a finite open cover $\mathcal{U}$ of $\baseSpace{\aContField}$ and points $z_U \in U$ for $U \in \mathcal{U}$ such that $\| \Ffunc(z) - \Ffunc(z_U) \| \leq \varepsilon / 12$ for any $z \in U$. Since $\mathcal{U}$ is finite, we may apply \Cref{lem:Omega-Theta-lim-r} to choose $\delta > 0$ such that for any $U \in \mathcal{U}$, we have 
	\begin{equation*}
	\sup_{r \leq \delta} \left( 2 \; \oscill{\rdist}\Ffunc (z_U) + \max\left\{  \oscill{2\rdist}\Ffunc (z_U) , \, \meansym{\rdist} \Ffunc (z_U) \right\} \right) \leq \varepsilon / 2 \; .
	\end{equation*}
	Hence it follows from \eqref{defn:Omega-Theta::eq:bound-sup-norm} that for any $z \in \baseSpace{\aContField}$, we have 
	\begin{equation}
	\sup_{r \leq \delta} \left( 2 \; \oscill{\rdist}\Ffunc (z) + \max\left\{  \oscill{2\rdist}\Ffunc (z) , \, \meansym{\rdist} \Ffunc (z) \right\} \right) \leq \varepsilon  
	\; .
	\end{equation}
	Now since $\Sigma$ is a generating set, by \Cref{def:cont-field-HHS-generate} and \Cref{lem:cont-field-HHS-Diximier-generate}, for any $z \in \baseSpace{\aContField}$, there exist $s' \in \Sigma$ and a neighborhood $V$ of $z$ such that 
	\[
	d_{\aContField_{z'}} \left( s \left( z' \right) , s' \left( z' \right) \right) \leq \delta \quad \text{ for any } z' \in V \; .
	\]
	It follows from the compactness of $\baseSpace{\aContField}$ that we may choose a finite open cover $\mathcal{V}$ of $\baseSpace{\aContField}$ together with $s_V \in \Sigma$ for each $V \in \mathcal{V}$ satisfying the last inequality with $s'$ replaced by $s_V$. Now choosing a partition of unity $(g_V)_{V \in \mathcal{V}}$ subordinate to $\mathcal{V}$, we apply the $C(\baseSpace{\aContField})$-linearity in \Cref{prop_welldefinednessofBotthomomorphism} and \eqref{defn:Omega-Theta::eq:scalar} to see that for all $z \in \baseSpace{\aContField}$, $x \in \aContField_z$, and $t \in [0,\infty)$, we have
	\begin{align*}
	& \ \left\|\botthom^{\aContField} (\Ffunc) (z,x,t) -  \sum_{V \in \mathcal{V}} \botthom[s_V]^{\aContField} (\Ffunc g_V)(z, x, t) \right\| \\
	= & \ \left\| \sum_{V \in \mathcal{V}} \left( \botthom^{\aContField} (\Ffunc g_V) (z,x,t) -   \botthom[s_V]^{\aContField} (\Ffunc g_V)(z, x, t) \right)\right\| \\
	= & \ \left\| \sum_{V \in \mathcal{V}} \left( \botthom^{\aContField} (\Ffunc) (z,x,t) -   \botthom[s_V]^{\aContField} (\Ffunc)(z, x, t) \right) g_V\right\| \\
	\leq & \ \max_{z \in \baseSpace{\aContField} } \sup_{r \leq \delta} \left( 2 \; \oscill{\rdist}\Ffunc (z) + \max\left\{  \oscill{2\rdist}\Ffunc (z) , \, \meansym{\rdist} \Ffunc (z) \right\} \right) \\
	\leq & \ \varepsilon \; ,
	\end{align*}
	whence $\left\|\botthom^{\aContField} (\Ffunc) -  \sum_{V \in \mathcal{V}} \botthom[s_V]^{\aContField} (\Ffunc g_V) \right\| \leq \varepsilon$ as desired. 
\end{proof}

The following lemma is a direct consequence of 
equation~\eqref{rmk:AofM-single-fiber::eq:botthom-commute}. 
\begin{lem}\label{lem:botthom-extension}
	Let $Z$ be a compact paracompact Hausdorff space and let $\aContField$ be a continuous field of {\hhs}s. 
	Let $s \in \spaceOfSections_{\operatorname{cont}} (\aContField)$. 
	With notation as in \Cref{def:cont-field-HHS-maps}\eqref{def:cont-field-HHS-maps:restrict}, we have Bott homomorphisms 
	\[
	\botthom^{\aContField} \colon C(\baseSpace{\aContField}, \salg )  \to \pialg (\aContField) \quad \text{ and } \quad \botthom[\pi^* (s)]^{\aContField |_{\baseSpace{\aContField} \times Z}} \colon C(\baseSpace{\aContField} \times Z, \salg )  \to \pialg \left(\aContField |_{\baseSpace{\aContField} \times Z} \right)  \; .
	\]
	Then 
	under the canonical isomorphism $C(\baseSpace{\aContField} \times Z, \salg ) \cong C(\baseSpace{\aContField}, \salg ) \otimes C(Z)$, 
	we have 
	\[
	\botthom[\pi^* (s)]^{\aContField |_{\baseSpace{\aContField} \times Z}}(\Ffunc \otimes g) ((y,z) , x, t) = g(z) \cdot  \botthom^{\aContField} (\Ffunc) (y,x,t) 
	\]
	for any $\Ffunc \in C(\baseSpace{\aContField}, \salg )$, any $g \in C(Z)$, any $y \in \baseSpace{\aContField}$, any $z \in Z$, any $x \in \aContField_y = \left( \aContField |_{\baseSpace{\aContField} \times Z}  \right)_{(y,z)}$, and any $t \in [0,\infty)$. 
\end{lem}

\begin{prop}\label{cor:AofM-extension}
	Let $Z$ be a compact paracompact Hausdorff space and let $\aContField$ be a continuous field of {\hhs}s. Then the following hold: 
	\begin{enumerate}
		\item \label{cor:AofM-extension::algebra} 
		With notation as in \Cref{def:cont-field-HHS-maps}\eqref{def:cont-field-HHS-maps:restrict} there is a $*$-isomorphism
		\[
		\AofM[\aContField |_{\baseSpace{\aContField} \times Z}] \cong \AofM[\aContField] \otimes C(Z) \; . 
		\]
		
		\item \label{cor:AofM-extension::base} If $Y$ is another compact paracompact Hausdorff space and $f \colon Y \to Z$ is a continuous function, then there is a commutative diagram
		\[\begin{tikzcd}
		{\AofM[\aContField |_{\baseSpace{\aContField} \times Z}]} & { \AofM[\aContField] \otimes C(Z)} \\
		{\AofM[\aContField |_{\baseSpace{\aContField} \times Y}]} & {\AofM[\aContField] \otimes C(Y)}
		\arrow["{\left( \operatorname{id}_{\baseSpace{\aContField}} \times f \right)^*}"', from=1-1, to=2-1]
		\arrow["{\operatorname{id}_{\AofM[\aContField]} \otimes f ^*}", from=1-2, to=2-2]
		\arrow["\cong"', from=1-2, to=1-1]
		\arrow["\cong", from=2-2, to=2-1]
		\end{tikzcd}\]
		where the horizontal maps are as given in \eqref{cor:AofM-extension::algebra}. 
		
		\item \label{cor:AofM-extension::action} If $\varphi \colon Z \to \operatorname{Isom}(\aContField)$ is a continuous function and $\widetilde{\varphi} \in \operatorname{Isom}(\aContField|_{\baseSpace{\aContField} \times Z})$ is obtained as in \Cref{lem:isometric-action-topologize-curry}\eqref{lem:isometric-action-topologize:isometries-curry:two-sided}, 
		then for any $z \in Z$, there is a commutative diagram
		\[\begin{tikzcd}
		{\AofM[\aContField |_{\baseSpace{\aContField} \times Z}]} & { \AofM[\aContField] \otimes C(Z)}  & { \AofM[\aContField] } \\
		{\AofM[\aContField |_{\baseSpace{\aContField} \times Z}]} & {\AofM[\aContField] \otimes C(Z)}  & { \AofM[\aContField] } 
		\arrow["{\widetilde{\varphi}_*}"', from=1-1, to=2-1]
		\arrow["{(\varphi(z))_* }", from=1-3, to=2-3]
		\arrow["\cong"', from=1-2, to=1-1]
		\arrow["\cong", from=2-2, to=2-1]
		\arrow["\idmap \otimes \operatorname{ev}_z", from=1-2, to=1-3]
		\arrow["\idmap \otimes \operatorname{ev}_z"', from=2-2, to=2-3]
		\end{tikzcd}\]
		where the horizontal maps on the left are as given in \eqref{cor:AofM-extension::algebra}, those on the right come from the evaluation map at $z$, and the vertical maps come from \eqref{eq:action-Clifford-algebra}. 
	\end{enumerate}
\end{prop}

\begin{proof}
	Since the map $C \left(\baseSpace{\aContField}, \salg \right) \otimes C(Z) \to C \left(\baseSpace{\aContField} \times Z, \salg \right)$ taking $\Ffunc \otimes g \in C \left(\baseSpace{\aContField}, \salg \right) \otimes C(Z)$ to the function $\Ffunc \cdot g \colon \baseSpace{\aContField} \times Z \ni (y,z) \mapsto f(y) g(z) \in \salg$ is a $*$-isomorphism, it follows from \Cref{lem:AofM-generating-sets} and \Cref{lem:botthom-extension} 
	that
	if we write $\pi \colon \baseSpace{\aContField} \times Z \to \baseSpace{\aContField}$ for the projection onto the first factor, then $\AofM[\aContField |_{\baseSpace{\aContField} \times Z}]$ is the $\cstar$-subalgebra of $\pialg\left(\aContField |_{\baseSpace{\aContField} \times Z}\right)$ generated by
	\[
	\left\{ \botthom[\pi^* (s)]^{\aContField |_{\baseSpace{\aContField} \times Z}}(\Ffunc \cdot g) \colon s \in \spaceOfSections_{\operatorname{cont}} \left(\aContField\right),\ \Ffunc \in C \left(\baseSpace{\aContField}, \salg \right) , \  g \in C(Z) \right\} \; .
	\]
	It follows that 
	\[
	\AofM[\aContField |_{\baseSpace{\aContField} \times Z}] \cong \AofM[\aContField] \otimes C(Z) \; , 
	\]
	which proves \eqref{cor:AofM-extension::algebra} 
	Moreover, \eqref{cor:AofM-extension::base} and \eqref{cor:AofM-extension::action} follow from straightforward computations using the explicit formula of the $\ast$-isomorphism $\AofM[\aContField |_{\baseSpace{\aContField} \times Z}] \cong \AofM[\aContField] \otimes C(Z)$ provided by \Cref{lem:botthom-extension} (also with the help of \Cref{lem_Botthomandgroupaction}). 
\end{proof}

%%%%%%%%%%%%%%%%%%%%%%%%%AofM%%%%%%%%%%%%%%%%%%%%%%%%%
%%%%%%%%%%%%%%%%%%%%%%%%%main%%%%%%%%%%%%%%%%%%%%%%%%%
\section{Proofs of the main theorems}
\label{sec:main}

In this section, we prove the main results of the paper, i.e., \Cref{main-theorem} and \Cref{main-theorem-fields} in the introduction. The former is a corollary of the latter. 

We first need to specify the meaning of an admissible continuous field of {\hhs}s in \Cref{main-theorem-fields}. 

\begin{defn} \label{def:admissible-field} \label{def:quasitrivial}
	Let $\aContField$ be a continuous field of {\hhs}s. Then 
	it is \emph{admissible} if it is second-countable, each fiber $\aContField_z$ is admissible in the sense of \Cref{defn:hhs-admissible}, and the randomization $\aContField^{(\Omega,\mu)}$ is a trivial continuous field of {\hhs}s. 
\end{defn}

\begin{proof}[Proof of \Cref{main-theorem-fields}]
	By \Cref{defn:strong-Novikov}, we need to prove the composition 
	\[
	K_{*+1}^{\Gam}(\univspfree\Gamma)\otimes_{\zbbd} \qbbd 
	\xrightarrow{\pi_*} K_{*+1}^{\Gam}(\univspproper\Gamma)\otimes_{\zbbd} \qbbd 
	\xrightarrow{\mu} K_{*+1}^{\Gam} (C^*(\Gamma) ) \otimes_{\zbbd} \qbbd 
	\]
	is injective, 
	where $\pi_*$ is induced by the natural map $\univspfree\Gamma \to \univspproper\Gamma$, and $\mu$ is the (rationalized) Baum-Connes assembly map. We reduce the problem in a few steps:

	\begin{enumerate}[itemindent=*,leftmargin=0.5em,label=\textit{Step~(\arabic*).},ref=\arabic*]
		\item By our assumption, there is an isometric and proper action $\alpha$ of $\Gamma$ on an admissible continuous field $\aContField$ of {\hhs}s. 
		In anticipation of our next steps, we would like to consider the induced action $\left( \alpha^{[0,1]} \right) |_{\spaceProbMeas{\left(\baseSpace{\aContField}\right)}} \colon \Gamma \curvearrowright \left( \aContField^{[0,1]} \right) |_{\spaceProbMeas{\left(\baseSpace{\aContField}\right)}}$, which is also proper by \Cref{cor:proper-action-cont-field-convex-prob} and \Cref{cor:proper-action-cont-field-convex-random}. 
		To simplify notations, we define $\anotContField := \left( \aContField^{[0,1]} \right) |_{\spaceProbMeas{\left(\baseSpace{\aContField}\right)}}$ and write $\alpha$ in place of  $\left( \alpha^{[0,1]} \right) |_{\spaceProbMeas{\left(\baseSpace{\aContField}\right)}}$ when there is no risk of confusion. This induces an action $\alpha_*$ of $\Gamma$ on the C*-algebra $\AofM[\anotContField]$ by \Cref{prop:AofM-action}. 
		Consider the following commutative diagram 
		\[
		\xymatrix{
			K_{*+1}^{\Gam}(\univspfree\Gamma)\otimes_{\zbbd}\qbbd \ar[d]^{[\bottmap]} \ar[r]^{\pi_*} & K_{*+1}^{\Gam}(\univspproper\Gamma)\otimes_{\zbbd}\qbbd \ar[d]^{[\bottmap]} \ar[r]^\mu & K_{*+1}^{\Gam} (C^*(\Gamma) ) \otimes_{\zbbd} \qbbd \ar[d]^{[\bottmap] \rtimesred \Gam} \\
			\kk_{\rbbd,*}^{\Gam}(\univspfree\Gamma, \AofM[\anotContField])\vphantom{\otimes_{\zbbd}\qbbd} \ar[r]^{\pi_*} & \kk_{\rbbd,*}^{\Gam}(\univspproper\Gamma, \AofM[\anotContField])\vphantom{\otimes_{\zbbd}\qbbd} \ar[r]^\mu & \kfunctr_{\rbbd,*}(\AofM[\anotContField] \rtimes_{\operatorname{r}} \Gam)\vphantom{\otimes_{\zbbd}\qbbd} 
		}
		\]
		where the horizontal maps $\pi_*$ are induced by the natural map $\univspfree\Gamma \to \univspproper\Gamma$, the horizontal maps $\mu$ are the Baum-Connes assembly maps (with and without a coefficient), 
		and the vertical maps $[\beta]$ and $[\beta] \rtimesred \Gam$ are induced by taking the Kasparov product with the Bott element in $\kfunctr_{1}^{\Gamma} (\AofM[\anotContField])$ (constructed in \Cref{constr:bott-element}) and then taking the natural maps from the rationalized $KK$-groups to the $KK$-groups with real coefficients. 
		In order to prove the composition of the top row is injective, it suffices to prove each of the three maps on the left column and the bottom row is injective. 
		
		To this end, it follows from \Cref{cor:AofM-action-proper} and the properness of $\alpha \colon \Gamma \curvearrowright \anotContField$ that $\AofM[\anotContField]$ is a proper $\Gamma$-$Y$-C*-algebra, where $Y$ is the spectrum of $Z (\AofM[\anotContField])$. 
		Hence by \cite[Theorem~13.1]{guentnerhigsontrout}, the map
		\[
		\mu \colon \kk_{\rbbd,*}^{\Gam}(\univspproper\Gamma, \AofM[\anotContField]) \to \kfunctr_{\rbbd,*}(\AofM[\anotContField] \rtimes_{\operatorname{r}} \Gam)
		\]
		is an isomorphism. It follows from 
		\Cref{lem:KKR-EGam-inj} 
		that 
		\[
		\pi_* \colon \kk_{\rbbd,*}^{\Gam}(\univspfree\Gamma, \AofM[\anotContField]) \to \kk_{\rbbd,*}^{\Gam}(\univspproper\Gamma, \AofM[\anotContField])\vphantom{\otimes_{\zbbd}\qbbd}
		\]
		is injective. 
		Hence it suffices to show that the map 
		\[
		[\beta] \colon K_{*+1}^{\Gam}(\univspfree\Gamma)\otimes_{\zbbd}\qbbd \to \kk_{\rbbd,*}^{\Gam}(\univspfree\Gamma, \AofM[\anotContField])\vphantom{\otimes_{\zbbd}\qbbd}
		\]
		is injective, for which we are going to apply a trivialization and deformation technique. 
		
		\item
		By \Cref{def:admissible-field}, there is an admissible {\hhs} $X$ such that the randomization $\aContField^{[0,1]}$ is isometrically continuously isomorphic to the constant field $(X)_{\baseSpace{\aContField}}$ over the same base space $\baseSpace{\aContField}$. 
		By 
		\Cref{lem:trivialize-prob-field-specify}, we have
		\[
		\anotContField = \left( \aContField^{[0,1]} \right) |_{\spaceProbMeas{\left(\baseSpace{\aContField}\right)}} \cong (X)_{\baseSpace{\aContField}} |_{\spaceProbMeas{\left(\baseSpace{\aContField}\right)}}  \cong (X)_{\spaceProbMeas{\left(\baseSpace{\aContField}\right)}} = (X)_{\baseSpace{\anotContField}}  \; . 
		\] 
		By \Cref{lem:deformation}, there is a homotopy $\left( \alpha_t \right)_{t \in [0,1]}$ of isometric actions of $\Gamma$ on $\left( \aContField^{[0,1]} \right) |_{\spaceProbMeas{\left(\baseSpace{\aContField}\right)}}$ such that $\alpha_1 = \alpha$, $\baseSpace{\alpha_t} = \baseSpace{\alpha_1}$ for any $t \in [0,1]$, and $\alpha_0$ is a fiberwise trivial action with $\underline{\alpha_0} = \underline{\alpha}$, i.e., $(\alpha_0)_\gamma = \underline{\alpha_\gamma}^*$ for any $\gamma \in \Gamma$. 
		By \Cref{lem:homotopy-actions}, 
		we have an isometric action 
		\[
		\widetilde{\alpha} \colon \Gamma \curvearrowright 
		\anotContField|_{\baseSpace{\anotContField} \times [0,1]}
		\cong (X)_{\baseSpace{\anotContField}} |_{\baseSpace{\anotContField} \times [0,1]} 
		\cong (X)_{\baseSpace{\anotContField} \times [0,1]} 
		\]
		such that for any $t \in [0,1]$ and $\gamma \in \Gamma$, we have 
		\[
		\baseSpace{\widetilde{\alpha}_{\gam}} (z , t)  = \left( \baseSpace{\alpha_{t, \gam}} (z) , t \right) \quad \text{ and } \quad \left( \widetilde{\alpha}_{\gam} \right)_{(z,t)} = \left( \alpha_{t, \gam} \right)_{z} 
		\colon \anotContField_{\baseSpace{\alpha_{\gam}} (z)} \to \anotContField_z
		\; .
		\]
		It thus follows from \Cref{cor:AofM-extension} that, upon identifying $\AofM[\anotContField|_{\baseSpace{\anotContField} \times [0,1]}]$ with $\AofM[\anotContField] \otimes C([0,1])$, we have, for any $t \in [0,1]$, the quotient map 
		\[
		\AofM[\anotContField|_{\baseSpace{\anotContField} \times [0,1]}] \cong \AofM[\anotContField] \otimes C([0,1]) \xrightarrow{\idmap \otimes \operatorname{ev}_{t}} \AofM[\anotContField]
		\]
		intertwines the actions $\widetilde{\alpha}$ and $\alpha_t$. 
		Hence in view of \Cref{lem:botthom-extension}, we have the following commutative diagram of equivariant $KK$-groups, where we write in the superscripts after ``$KK$'' not only the acting group, but also the actions on the coefficient $C^*$-algebras. 
		\[\begin{tikzcd}
		& {K_{*+1}^{\Gam}(\univspfree\Gamma)\otimes_{\zbbd}\qbbd } \\
		& {\kk_{\rbbd,*}^{\Gam,\widetilde{\alpha}}\left(\univspfree\Gamma, \AofM[\anotContField|_{\baseSpace{\anotContField} \times [0,1]}]\right)\vphantom{\otimes_{\zbbd}\qbbd}} \\
		{\kk_{\rbbd,*}^{\Gam,\alpha_1}(\univspfree\Gamma, \AofM[\anotContField])\vphantom{\otimes_{\zbbd}\qbbd}} 
		& {\kk_{\rbbd,*}^{\Gam,\widetilde{\alpha}}\left(\univspfree\Gamma, \AofM[\anotContField] \otimes C([0,1])\right)\vphantom{\otimes_{\zbbd}\qbbd}}
		& {\kk_{\rbbd,*}^{\Gam,\alpha_0}(\univspfree\Gamma, \AofM[\anotContField])\vphantom{\otimes_{\zbbd}\qbbd}}
		\arrow["{\cong}", from=3-2, to=2-2]
		\arrow["{(\idmap \otimes \operatorname{ev}_{1})_*}", from=3-2, to=3-1]
		\arrow["{(\idmap \otimes \operatorname{ev}_{0})_*}"', from=3-2, to=3-3]
		\arrow["{[\bottmap]}"', from=1-2, to=2-2]
		\arrow["{[\bottmap]}"', curve={height=12pt}, from=1-2, to=3-1]
		\arrow["{[\bottmap]}", curve={height=-12pt}, from=1-2, to=3-3]
		\end{tikzcd}\]
		Since both the $*$-homomorphism $\idmap_{\AofM[\anotContField]} \otimes \operatorname{ev}_{0}$ and $\idmap_{\AofM[\anotContField]} \otimes \operatorname{ev}_{1}$ are (non-equivariant) homotopy equivalences, 
		it follows from \Cref{lem:KK-homotopy-eq-nonequivariant} that 		
		the maps $(\idmap_{\AofM[\anotContField]} \otimes \operatorname{ev}_{0})_*$ and $(\idmap_{\AofM[\anotContField]} \otimes \operatorname{ev}_{1})_*$ are isomorphisms. 
		Hence showing the leftmost map 
		$[\bottmap] \colon K_{*+1}^{\Gam}(\univspfree\Gamma)\otimes_{\zbbd}\qbbd \to \kk_{\rbbd,*}^{\Gam,\alpha_1}(\univspfree\Gamma, \AofM[\anotContField])\vphantom{\otimes_{\zbbd}\qbbd}$ is injective is equivalent to showing the rightmost map 
		\[
		[\bottmap] \colon K_{*+1}^{\Gam}(\univspfree\Gamma)\otimes_{\zbbd}\qbbd \to \kk_{\rbbd,*}^{\Gam,\alpha_0}(\univspfree\Gamma, \AofM[\anotContField])\vphantom{\otimes_{\zbbd}\qbbd}
		\]
		is injective. 
		
		\item 
		Since $(\alpha_0)_\gamma = \underline{\alpha_\gamma}^*$ for any $\gamma \in \Gamma$, it follows by \Cref{cor:AofM-extension} that under the isomorphism $\AofM[\anotContField] \cong \AofM[(X)_{\baseSpace{\anotContField}}] \cong \AofM[X] \otimes C(\baseSpace{\anotContField})$, the action $\alpha_0$ is conjugate to $\operatorname{id} \otimes \baseSpace{\alpha}^*$ and thus 
		by \Cref{lem:botthom-extension}, 
		we have a commutative diagram
		\[\begin{tikzcd}
		&& {\kk_{\rbbd,*}^{\Gam, \operatorname{id} } \left(\univspfree\Gamma, \AofM[X]  \right)\vphantom{\otimes_{\zbbd}\qbbd}} \\
		{K_{*+1}^{\Gam} \left(\univspfree\Gamma \right)\otimes_{\zbbd}\qbbd} && {\kk_{\rbbd,*}^{\Gam, \operatorname{id} \otimes \baseSpace{\alpha}^*} \left(\univspfree\Gamma, \AofM[X] \otimes C(\baseSpace{\anotContField}) \right)\vphantom{\otimes_{\zbbd}\qbbd}} \\
		&& {\kk_{\rbbd,*}^{\Gam, \alpha_0 } \left(\univspfree\Gamma, \AofM[\anotContField]  \right)\vphantom{\otimes_{\zbbd}\qbbd}}
		\arrow["\cong", from=2-3, to=3-3]
		\arrow["{\left[\bottmap^X \right] \otimes \left[1_{C(\baseSpace\anotContField)} \right]}", from=2-1, to=2-3]
		\arrow["{\left[ \bottmap^{\anotContField} \right]}"', curve={height=-6pt}, from=2-1, to=3-3]
		\arrow["{\left[ \bottmap^X \right]}", curve={height=-12pt}, from=2-1, to=1-3]
		\arrow["{\left( \operatorname{id}_{\AofM[X]} \otimes 1_{C(\baseSpace\anotContField)} \right)_*}", from=1-3, to=2-3]
		\end{tikzcd}\]
		where we distinguish the Bott homomorphisms into $\AofM[\anotContField]$ and $\AofM[X]$ using superscripts in the notation $[\botthom[]]$. 
		Note that since $\baseSpace{\anotContField} = \operatorname{Prob} (\baseSpace{\aContField})$, which is contractible, the top-right vertical map above induced by the unital embedding $\cbbd \to C(\baseSpace{\anotContField})$ is an isomorphism by \Cref{lem:KK-homotopy-eq-nonequivariant}. 
		Hence it suffices to show 
		\[
		[\beta] \colon K_{*+1}^{\Gam} (\univspfree\Gamma)\otimes_{\zbbd}\qbbd \to \kk_{\rbbd,*}^{\Gam, \operatorname{id} } \left(\univspfree\Gamma, \AofM[X]  \right)\vphantom{\otimes_{\zbbd}\qbbd}
		\]
		is injective. 
		
		\item 
		This last claim above was established in \cite[proof of Proposition~8.8]{GongWuYu2021}. We briefly recall its proof: 
		
		Since the action $\Gamma \curvearrowright \AofM[X]$ is trivial in $\kk_{\rbbd,*}^{\Gam, \operatorname{id} } \left(\univspfree\Gamma, \AofM[X]  \right)$, by \Cref{defn:KK-Gam-trivial-coeff}, there is a commutative diagram
		\[\begin{tikzcd}
		{K_{*+1}^{\Gam} (\univspfree\Gamma)\otimes_{\zbbd}\qbbd} & {\kk_{\rbbd,*}^{\Gam, \operatorname{id} } \left(\univspfree\Gamma, \AofM[X]  \right)} \\
		{K_{*+1} (B\Gamma)\otimes_{\zbbd}\qbbd} & {\kk_{\rbbd,*} \left(B\Gamma, \AofM[X]  \right)} \\
		{\displaystyle \bigoplus_{j \in \zbbd/2\zbbd} K_{j}(B\Gamma)\otimes_{\zbbd}K_{*-j+1}(\cbbd)\otimes_{\zbbd}\qbbd} & {\displaystyle \bigoplus_{j \in \zbbd/2\zbbd} K_{j}(B\Gamma)\otimes_{\zbbd} K_{\rbbd,*-j} \left( \AofM[X] \right)}
		\arrow["{[\beta]}", from=1-1, to=1-2]
		\arrow["{[\beta]}", from=2-1, to=2-2]
		\arrow["\cong"', from=1-1, to=2-1]
		\arrow["\cong", from=1-2, to=2-2]
		\arrow["{\operatorname{id} \otimes_{\zbbd} [\beta]}", from=3-1, to=3-2]
		\arrow["\cong", from=3-1, to=2-1]
		\arrow["\cong"', from=3-2, to=2-2]
		\end{tikzcd}\]
		By \cite[Propositions~7.6 and~8.8]{GongWuYu2021}, the bottom horizontal map is injective. Hence the top horizontal map is also injective, as desired. 
	\end{enumerate}
	This completes the proof. 
\end{proof}

\begin{rmk}
	Note that the coefficient algebra $\AofM[\anotContField]$ we used in the proof above was eventually written, thanks to the triviality of $\anotContField$, simply as $C(Z) \otimes \AofM[X]$, a $C^*$-algebra that can be written down using just the machinery of \cite{GongWuYu2021}, without resorting to the theory we developed in this paper about general continuous fields of {\hhs}s and their $C^*$-algebras. However, the continuous field $\anotContField$ is trivial for nontrivial reasons that necessitate the study of the general theory. 
\end{rmk}

Finally, we prove \Cref{main-theorem} in the introduction regarding groups of diffeomorphisms. 

\begin{proof}[Proof of \Cref{main-theorem}]
	By \Cref{constr:diff-riem}, there is an isometric action $\alpha$ of $\Gamma$ on the continuous field $\operatorname{Riem}(N)$ of {\hhs}s. 
	This induces an isometric action of $\Gamma$ on $\operatorname{Riem}(N)|_{\operatorname{Prob}(N)}$ by \Cref{cor:continuum-product-field-actions}, which we denote by $\alpha$ as well. 
	
	By our assumption, $\Gamma$ is $\mu$-discrete as in \Cref{def:geometrically-discrete}. 
	Thus by \Cref{lem:geometrically-discrete}, 
	if we write 
	\[
	Z := \overline{\operatorname{convex}(\Gamma \cdot \mu)} \subseteq \operatorname{Prob}(N) \; ,
	\]
	then the restriction of the isometric action $\alpha \colon \Gamma \curvearrowright \operatorname{Riem}(N)|_{\operatorname{Prob}(N)}$ to $\operatorname{Riem}(N)|_Z$ is proper. 
	
	In view of \Cref{main-theorem-fields}, it remains to show that $\operatorname{Riem}(N)|_{Z}$ is admissible as in \Cref{def:admissible-field}. To this end, we observe that since $\operatorname{Riem}(N)$ is locally trivial, it follows from \Cref{cor:quasitrivial-locally-trivial} that $\operatorname{Riem}(N)^{[0,1]}$ is trivial, 
	and thus it follows from \Cref{lem:trivialize-prob-field-specify} that  $\operatorname{Riem}(N)^{[0,1]}|_{\operatorname{Prob}(N)}$ is also trivial, 
	the latter continuous field being continuously isometrically isomorphic to $\left(\operatorname{Riem}(N)|_{\operatorname{Prob}(N)}\right)^{[0,1]}$ by \Cref{lem:continuum-product-field-augmentation-maps}\eqref{lem:continuum-product-field-augmentation-maps:variation}, whence $\left(\operatorname{Riem}(N)|_{\operatorname{Prob}(N)}\right)^{[0,1]}$ is trivial and thus $\left(\operatorname{Riem}(N)|_{Z}\right)^{[0,1]}$ is trivial, too. 
	
	On the other hand, for any $\mu \in Z$, the fiber of $\operatorname{Riem}(N)|_{Z}$ at $\mu$ is isometric to the $L^2$-continuum product $L^2(N, \mu, \operatorname{Riem}(N))$ by \Cref{def:continuum-product-field}, the latter being an abbreviation of $L^2(N, \mu, \operatorname{Riem}(N)_{\operatorname{meas}}; g)$ by \Cref{defn:continuous-field-HHS-L2-notations} and \Cref{def:continuum-product}, where $g$ is an arbitrary continuous section of $\operatorname{Riem}(N)$. Upon choosing a Borel trivialization of the locally trivial fiber bundle $\operatorname{Riem}(N)$, we conclude from \Cref{lem:cont-field-HHS-meas} that $\left(\operatorname{Riem}(N)_{\operatorname{meas}} , \mu \right)$ is isometrically measure-preservingly isomorphic to the constant measured field  $P(n)_{(N,g)}$, where $P(n)$ is as in \Cref{constr:symmetric-space-GL}. It thus follows from \Cref{lem:continuum-product-maps} that $L^2(N, \mu, \operatorname{Riem}(N)_{\operatorname{meas}}; g)$ is isometric to $L^2(N, \mu, P(n))$ as in \cite[Construction~3.8]{GongWuYu2021}, which is an admissible Hilbert-Hadamard space by \cite[Proposition~3.13]{GongWuYu2021}. 
	
	Combining the last two paragraphs, we conclude that the continuous field $\operatorname{Riem}(N)|_{Z}$ is admissible, as desired. 
\end{proof}

%%%%%%%%%%%%%%%%%%%%%%%%%main%%%%%%%%%%%%%%%%%%%%%%%%%

%%%%%%%%%%%%%%%%%%%%%%%%%%%%%%%%%%%%%%%%%%%%%%%

\bibliographystyle{alpha}
\bibliography{Novikov-continuous-field}

\end{document}